\documentclass{amsart}
\usepackage{amsmath}
\usepackage{amssymb}
\usepackage{amsthm}
\usepackage{epsfig}
\usepackage{amsfonts}
\usepackage{amscd}
\usepackage{mathrsfs}
\usepackage{enumerate}
\usepackage{latexsym}
\usepackage{url}

\newtheorem{theorem}{Theorem}[section]
\newtheorem{lemma}[theorem]{Lemma}

\newtheorem{corollary}[theorem]{Corollary}

\theoremstyle{definition}
\newtheorem{definition}[theorem]{Definition}
\newtheorem{example}[theorem]{Example}

\newtheorem{notation}[theorem]{Notation}

\theoremstyle{question}
\newtheorem{question}[theorem]{Question}

\numberwithin{equation}{section}

\newtheorem*{xrem}{Remark}
\newtheorem*{xclaim}{Claim}

\begin{document}
\title[Compactification--like extensions]{Compactification--like extensions}

\author{M.R. Koushesh}
\address{Department of Mathematical Sciences, Isfahan University of Technology, Isfahan 84156--83111, Iran}
\address{School of Mathematics, Institute for Research in Fundamental Sciences (IPM), P.O. Box: 19395--5746, Tehran, Iran}
\email{koushesh@cc.iut.ac.ir}
\thanks{This research was in part supported by a grant from IPM (No. 86540012).}




\begin{abstract}
Let $X$ be a space. A space $Y$ is called an {\em extension} of $X$ if $Y$ contains $X$ as a dense subspace. For an extension $Y$ of $X$ the subspace $Y\backslash X$ of $Y$ is called the {\em remainder} of $Y$. Two extensions of $X$ are said to be {\em equivalent} if there is a homeomorphism between them which fixes $X$ pointwise. For two (equivalence classes of)  extensions $Y$ and $Y'$ of $X$ let $Y\leq Y'$ if there is a continuous mapping of $Y'$ into $Y$ which fixes  $X$ pointwise. Let ${\mathcal P}$ be a topological property. An extension $Y$ of $X$ is called a {\em
${\mathcal P}$--extension} of $X$ if it has ${\mathcal P}$. If ${\mathcal P}$ is compactness then  ${\mathcal P}$--extensions are called  {\em compactifications}.

The aim of this article is to introduce and study classes of extensions (which we call {\em compactification--like ${\mathcal P}$--extensions}, where ${\mathcal P}$ is a topological property subject to some mild requirements) which  resemble the classes of compactifications of locally compact spaces. We formally define  compactification--like ${\mathcal P}$--extensions and derive some of their basic properties, and in the case when the remainders are countable, we characterize spaces having such extensions.  We will then consider the classes of compactification--like ${\mathcal P}$--extensions as partially ordered sets. This consideration leads to some interesting results which characterize compactification--like ${\mathcal P}$--extensions of a space among all its  Tychonoff  ${\mathcal P}$--extensions with compact remainder. Furthermore, we study the relations between the order--structure of classes of compactification--like ${\mathcal P}$--extensions of a Tychonoff space $X$ and the topology of a certain subspace of its outgrowth $\beta X\backslash X$. We conclude with some applications, including an answer to an old question of S. Mr\'{o}wka and J.H. Tsai: For what pairs of topological properties ${\mathcal P}$ and ${\mathcal Q}$ is it true that every locally--$\mathcal{P}$ space with $\mathcal{Q}$ has a one--point extension with both $\mathcal{P}$ and $\mathcal{Q}$? An open question is raised.
\end{abstract}

\maketitle

\tableofcontents

\section{Introduction}

Let $X$ be a space. A space $Y$ is called an {\em extension} of $X$ if $Y$ contains $X$ as a dense subspace.  If $Y$ is an extension of $X$ then the subspace $Y\backslash X$ of $Y$ is called the {\em  remainder} of $Y$. Two extensions of $X$ are said to be {\em equivalent} if there exists a homeomorphism between them which fixes $X$ pointwise. This defines an equivalence relation on the class of all extensions of $X$. The equivalence classes will be identified with individuals when this causes no confusion. For two (equivalence classes of) extensions $Y$ and $Y'$ of $X$ we let $Y\leq Y'$ if there exists a continuous  mapping  of $Y'$ into $Y$ which fixes  $X$ pointwise. The relation $\leq$ defines a partial order on the set of all (equivalence classes of) extensions of $X$ (see Section 4.1 of \cite{PW} for more details). Let ${\mathcal P}$ be a topological
property. An extension $Y$ of $X$ is called a {\em ${\mathcal P}$--extension} of $X$ if it has ${\mathcal P}$. If ${\mathcal P}$ is compactness then  ${\mathcal P}$--extensions are called  {\em compactifications}. The aim in this article is to introduce and study classes of extensions (which we call {\em compactification--like ${\mathcal P}$--extensions} where ${\mathcal P}$ is a topological property) which look like  the classes of compactifications of locally compact spaces. These are for a Tychonoff space $X$:
\begin{itemize}
\item The class of {\em minimal ${\mathcal P}$--extensions} of $X$, consisting of those Tychonoff ${\mathcal P}$--extensions $Y$ of $X$ with compact remainder such that $Y$ is minimal (with respect to the subspace relation $\subseteq$) among all Tychonoff ${\mathcal P}$--extensions of $X$ with compact remainder. (In other words, one  cannot construct any other Tychonoff ${\mathcal P}$--extension of $X$ with compact remainder out of $Y$ by deleting points from the space $Y$.)
\item The class of {\em optimal ${\mathcal P}$--extensions} of $X$, consisting of those Tychonoff ${\mathcal P}$--extensions $Y$ of $X$  with compact remainder such that the topology of $Y$ is maximal (with respect to the inclusion relation $\subseteq$) among all topologies on $Y$ which turn $Y$ into a  Tychonoff ${\mathcal P}$--extension of $X$ with compact remainder and $Y$ is minimal (with respect to the subspace relation $\subseteq$) among all Tychonoff  ${\mathcal P}$--extensions  of $X$ with compact remainder. (In other words, one cannot construct any other Tychonoff ${\mathcal P}$--extension of $X$ with compact remainder out of $Y$ either by adding sets to the topology of $Y$ or deleting points from the space $Y$.)
\end{itemize}
Here the topological property ${\mathcal P}$ is subject to some mild restrictions and will include most of the important covering properties (such as compactness, the Lindel\"{o}f property, countable compactness, paracompactness and metacompactness) as special cases.

This article is organized as follows:

In Chapter  2 we give the formal definitions of compactification--like ${\mathcal P}$--extensions and we derive some of their basic properties.

In Chapter 3 we consider the case when the extensions have countable remainders and characterize those Tychonoff spaces which have a  compactification--like ${\mathcal P}$--extension with countable remainder.

In Chapter 4 we consider the classes of compactification--like ${\mathcal P}$--extensions of a Tychonoff space $X$ as partially ordered sets. Besides the standard partial order $\leq$ we consider two other partial orderers $\leq_{inj}$ and $\leq_{surj}$. This considerations lead to some interesting results which characterize compactification--like ${\mathcal P}$--extensions of $X$ among all Tychonoff ${\mathcal P}$--extensions of $X$ with compact remainder. Furthermore, we study the relationships between the order--structure of classes of compactification--like ${\mathcal P}$--extensions of $X$ (partially ordered with $\leq$) and the topology of a certain subspaces of its outgrowth $\beta X\backslash X$. We conclude this chapter with a result which characterizes the largest (with respect to $\leq$) compactification--like ${\mathcal P}$--extension of $X$. This largest element, which we explicitly introduce as a subspace of the Stone--\v{C}ech compactification $\beta X$ of $X$, turns out to be even the largest among all Tychonoff ${\mathcal P}$--extension of $X$ with compact remainder.

In Chapter 5 we give some applications of our study. These applications include the relations between compactification--like ${\mathcal P}$--extensions and tight ${\mathcal P}$--extensions with compact remainder (a {\em tight ${\mathcal P}$--extension} of a space $X$ is a  Tychonoff ${\mathcal P}$--extension of $X$ which does not contain properly any other ${\mathcal P}$--extension of $X$ as a subspace) and an answer to an old question of S. Mr\'{o}wka and J.H. Tsai in \cite{MT}: For what pairs of topological properties ${\mathcal P}$ and ${\mathcal Q}$  is it true that every (Tychonoff)  locally--$\mathcal{P}$ (non--$\mathcal{P}$) space with $\mathcal{Q}$ has a one--point (Tychonoff) extension with both $\mathcal{P}$ and $\mathcal{Q}$?

We  conclude  with an open question which naturally arise in connection with our study.

We now review some of the terminology, notation and well known results which will be used in the sequel. Our definitions mainly  come from the standard  text \cite{E} (thus in particular, compact spaces are Hausdorff, perfect mappings are continuous with Hausdorff domains, etc.). Other useful sources are \cite{GJ}, \cite{PW} and \cite{W}.

The letters  $\mathbf{R}$, $\mathbf{I}$ and $\mathbf{N}$ denote the real line, the closed unit interval and the set of all positive integers, respectively. By $\omega$ and $\Omega$ we denote the first infinite ordinal and the first uncountable ordinal, respectively, and by $\aleph_0$ and $\aleph_1$ we denote their cardinalities. The cardinality of a set $A$ is denoted by $\mbox{card}(A)$. For a subset $A$ of a space $X$ we let $\mbox{cl}_X A$, $\mbox{int}_X A$ and $\mbox{bd}_X A$ denote the closure, the interior and the boundary of $A$ in $X$, respectively. A subset of a space is said to be {\em clopen} if it is simultaneously closed and open. A {\em zero--set} of a space $X$ is a set of the form $Z(f)=f^{-1}(0)$ for some continuous $f:X\rightarrow \mathbf{I}$.  Any set of the form $X\backslash Z$, where $Z$ is a zero--set of a space $X$, is called a {\em cozero--set} of $X$. We denote the set of all zero--sets of $X$ by
${\mathscr Z}(X)$ and the set of all cozero--sets of $X$ by $Coz(X)$. For a Tychonoff space $X$ the {\em Stone--\v{C}ech compactification} of $X$ is the largest (with respect to the partial order $\leq$) compactification of $X$ and is  denoted by $\beta X$.  The Stone--\v{C}ech compactification of a Tychonoff $X$ is characterized among the compactifications of $X$ by either of the following properties:
\begin{itemize}
\item Every continuous  mapping  from $X$ to a compact space is continuously extendible over $\beta X$.
\item Every continuous  mapping  from $X$ to $\mathbf{I}$ is continuously extendible over $\beta X$.
\item For every $Z,S\in {\mathscr Z}(X)$ such that $Z\cap S=\emptyset$ we have
\[\mbox{cl}_{\beta X}Z\cap\mbox{cl}_{\beta X}S=\emptyset.\]
\item For every $Z,S\in {\mathscr Z}(X)$ we have
\[\mbox{cl}_{\beta X}(Z\cap S)=\mbox{cl}_{\beta X}Z\cap\mbox{cl}_{\beta X}S.\]
\end{itemize}
A continuous mapping $f:X\rightarrow Y$ is called {\em perfect} if $X$ is a Hausdorff space, $f$ is closed (not necessarily surjective) and continuous and any fiber $f^{-1}(y)$, where $y\in Y$, is a compact subset of $X$. A topological property $\mathcal{P}$ is said to be {\em invariant under perfect mappings} ({\em  inverse invariant under perfect mappings}, respectively) if for any perfect surjective mapping $f:X\rightarrow Y$ the space $Y$ ($X$, respectively) has $\mathcal{P}$  provided that $X$ ($Y$, respectively) has $\mathcal{P}$. A topological property $\mathcal{P}$  is called {\em perfect} if it is both  invariant and inverse invariant under perfect mappings. A topological property $\mathcal{P}$ is said to be {\em hereditary with respect to closed subsets} ({\em hereditary with respect to open subsets}, respectively) if any closed (open, respectively) subset of a space with $\mathcal{P}$ also has $\mathcal{P}$. A topological property $\mathcal{P}$ is called {\em finitely additive} if whenever $X=X_1\oplus\cdots\oplus X_n$ and each $X_i$ has $\mathcal{P}$ then $X$ also has  $\mathcal{P}$. Let $\mathcal{P}$ be a topological property. A space $X$ is called {\em locally--$\mathcal{P}$} if each $x\in X$ has an open neighborhood $U$ in $X$ whose closure $\mbox{cl}_XU$ has $\mathcal{P}$. Note that if $X$ is a regular (Hausdorff) space and $\mathcal{P}$ is closed hereditary, then $X$ is locally--$\mathcal{P}$ if and only if each point $x$ of $X$ has a local base consisting of open neighborhoods  $U$ of $x$ such that $\mbox{cl}_XU$ has $\mathcal{P}$.

\section{Compactification--like $\mathcal{P}$--extensions}

In this chapter we give definitions  and derive some basic results which will be used throughout.

\begin{definition}\label{4}
Let $X$ be a space, let $\mathcal{P}$ be a topological property and let $Y$ be a Tychonoff $\mathcal{P}$--extension of $X$ with compact remainder.

The extension $Y$ of $X$ is called {\em minimal} if $Y$ is  minimal (with respect to the  subspace relation $\subseteq$) among all Tychonoff ${\mathcal P}$--extensions of $X$ with compact remainder, that is, $Y$ does not contain properly any other  Tychonoff $\mathcal{P}$--extension  of $X$ with compact remainder. In other words, one  cannot obtain any other Tychonoff ${\mathcal P}$--extension of $X$ with compact remainder out of $Y$ by deleting points from the space $Y$.

The extension $Y$ of $X$ is called {\em optimal} if the topology of $Y$ is maximal (with respect to the inclusion relation $\subseteq$) among all topologies on $Y$ which turn $Y$ into a  Tychonoff ${\mathcal P}$--extension  of $X$  with compact remainder, and  $Y$ is minimal (with respect to the subspace relation $\subseteq$) among all Tychonoff  ${\mathcal P}$--extensions  of $X$ with compact remainder. In other words, one  cannot obtain any other  Tychonoff ${\mathcal P}$--extension of $X$ with compact remainder out of $Y$ either by adding sets to the topology of $Y$ or deleting points from the space $Y$.

We refer to either  minimal $\mathcal{P}$--extensions or optimal $\mathcal{P}$--extensions as {\em compactification--like $\mathcal{P}$--extensions}.
\end{definition}

\begin{notation}\label{TR}
Let  $X$ be a space and let $\mathcal{P}$ be a  topological property. Denote by  ${\mathscr E}(X)$ the set of all Tychonoff extensions of $X$ with compact remainder and denote by either ${\mathscr E}^{\mathcal P}(X)$ or ${\mathscr E}_{\mathcal P}(X)$ the set of all elements of ${\mathscr E}(X)$ which have $\mathcal{P}$. Also, let ${\mathscr M}_{\mathcal P}(X)$ and  ${\mathscr O}_{\mathcal P}(X)$ denote the set of all minimal ${\mathcal P}$--extensions of $X$ and the set of all optimal ${\mathcal P}$--extensions of $X$, respectively, and if $\mathcal{Q}$ is a topological property, let
\begin{itemize}
\item ${\mathscr E}^{\mathcal Q}_{\mathcal P}(X)={\mathscr E}^{\mathcal Q}(X)\cap {\mathscr E}_{\mathcal P}(X)$.
\item ${\mathscr M}^{\mathcal Q}_{\mathcal P}(X)={\mathscr E}^{\mathcal Q}(X)\cap {\mathscr M}_{\mathcal P}(X)$.
\item ${\mathscr O}^{\mathcal Q}_{\mathcal P}(X)={\mathscr E}^{\mathcal Q}(X)\cap {\mathscr O}_{\mathcal P}(X)$.
\end{itemize}
\end{notation}

Note that by the definitions
\[{\mathscr O}^{\mathcal Q}_{\mathcal P}(X)\subseteq{\mathscr M}^{\mathcal Q}_{\mathcal P}(X).\]

\begin{xrem}
{\em Topological properties $\mathcal{P}$ considered in this article are assumed to be non--empty, that is, it is assumed that there exists at least one space with $\mathcal{P}$. This in particular implies that for a clopen hereditary topological property $\mathcal{P}$ the empty set has $\mathcal{P}$, or, if a space is non--$\mathcal{P}$ then it is non--empty as well.}
\end{xrem}

The following subspace of $\beta X$ will play a crucial role in our study.

\begin{definition}\label{14}
For a Tychonoff space $X$ and a topological property $\mathcal{P}$ define
\[\lambda_{\mathcal{P}} X=\bigcup\big\{\mbox{int}_{\beta X}\mbox{cl}_{\beta X}Z: Z\in {\mathscr Z}(X) \mbox { has $\mathcal{P}$}\big\}.\]
\end{definition}

Note that any topological property which is hereditary with respect to clopen subsets and inverse invariant under perfect mappings is  hereditary with respect to closed  subsets of Hausdorff spaces (see Theorem 3.7.29 of \cite{E}). This simple fact will be used in a number of places throughout.

\begin{lemma}\label{B}
Let $X$ be a Tychonoff space and let $\mathcal{P}$ be a  clopen hereditary  finitely additive perfect topological property. Then for any subset $A$ of $X$  if $\mbox{\em cl}_{\beta X} A\subseteq\lambda_{\mathcal{P}} X$ then $\mbox{\em cl}_X A$ has $\mathcal{P}$.
\end{lemma}

\begin{proof}
By the compactness of $\mbox{cl}_{\beta X} A$ and the definition of $\lambda_{\mathcal{P}} X$ we have
\[\mbox{cl}_{\beta X} A\subseteq\bigcup_{i=1}^n\mbox{int}_{\beta X}\mbox {cl}_{\beta X}Z_i\subseteq\mbox {cl}_{\beta X}Z\]
where each $Z_1,\ldots,Z_n\in {\mathscr Z}(X)$ has $\mathcal{P}$ and $Z=Z_1\cup\cdots\cup Z_n$. Since $\mathcal{P}$ is finitely additive and invariant under perfect mappings and $Z$ is the finite union of its closed subspaces $Z_i$'s each having $\mathcal{P}$, it follows that $Z$ has $\mathcal{P}$ (see Theorem 3.7.22 of \cite{E}). Now since $\mbox{cl}_X A\subseteq Z$ the set $\mbox{cl}_X A$ has $\mathcal{P}$, as it is closed in $Z$.
\end{proof}

\begin{lemma}\label{22}
Let $\mathcal{P}$ be a topological property which is clopen  hereditary and inverse invariant under perfect mappings and let $f:X\rightarrow Y$ be a perfect mapping. Then if $Y$ is locally--$\mathcal{P}$ and Hausdorff then $X$ is locally--$\mathcal{P}$.
\end{lemma}

\begin{proof}
First note that $f[X]$ is locally--$\mathcal{P}$. To show this let $y\in f[X]$. Since $Y$ is locally--$\mathcal{P}$ there exists an open neighborhood $V$ of $y$ in $Y$ such that $\mbox{cl}_Y V$ has $\mathcal{P}$. Now $V\cap f[X]$ is an open neighborhood of $y$ in $f[X]$, the image $f[X]$ is closed in $Y$ (as $f$ is perfect and thus closed) and
\[\mbox{cl}_{f[X]}\big(V\cap f[X]\big)=\mbox{cl}_Y\big(V\cap f[X]\big)\cap f[X]\subseteq\mbox{cl}_YV.\]
Therefore $\mbox{cl}_{f[X]}(V\cap f[X])$ has $\mathcal{P}$, as it is closed in $\mbox{cl}_YV$. Since $f:X\rightarrow f[X]$ is perfect and surjective we may assume in the statement of the lemma that $f$ is moreover surjective. Let $x\in X$. There exists an open neighborhood $W$ of $f(x)$ in $Y$  such that $\mbox{cl}_YW$ has $\mathcal{P}$. Since
\[f|f^{-1}[\mbox{cl}_YW]:f^{-1}[\mbox{cl}_YW]\rightarrow\mbox{cl}_YW\]
is perfect and surjective and $\mathcal{P}$ is inverse invariant under perfect mappings, $f^{-1}[\mbox{cl}_YW]$ has $\mathcal{P}$. Now
$f^{-1}[W]$ is an open neighborhood of $x$ in $X$, and since $\mbox{cl}_Xf^{-1}[W]\subseteq f^{-1}[\mbox{cl}_YW]$ and the latter has $\mathcal{P}$, its closed subset $\mbox{cl}_Xf^{-1}[W]$ also has $\mathcal{P}$.
\end{proof}

A topological property  $\mathcal{P}$ is said to {\em satisfy Mr\'{o}wka's condition $(\mbox{\em W})$} if it satisfies the following: If $X$ is a Tychonoff  space in which there exists a point  $p$ with an open  base ${\mathscr B}$ for $X$ at $p$ such that $X\backslash  B$ has $\mathcal{P}$ for any $B\in {\mathscr B}$, then $X$ has $\mathcal{P}$ (see \cite{Mr}). If $\mathcal{P}$ is  a topological property which is closed hereditary and productive then Mr\'{o}wka's condition $(\mbox{W})$ is equivalent to the following condition: If a Tychonoff space $X$ is the union of a compact space and a space with $\mathcal{P}$ then $X$ has $\mathcal{P}$ (see \cite{MRW1}). In \cite{Mr}  S. Mr\'{o}wka showed that if  $\mathcal{P}$ is  a topological property which is  closed hereditary, finitely additive with respect to closed subsets (that is, if a space is the union of a finite number of its closed subsets each having $\mathcal{P}$, then it has $\mathcal{P}$) and invariant under continuous mappings then any Tychonoff locally--$\mathcal{P}$ space can be embedded as an open subspace in a Tychonoff space with $\mathcal{P}$ if and only if Mr\'{o}wka's  condition $(\mbox{W})$ holds.

In this article we will be dealing with certain classes of topological properties. For convenience, we make the following definition.

\begin{definition}\label{RDQA}
Let $\mathcal{P}$ be a topological property. Then $\mathcal{P}$ is said to be a {\em compactness--like topological property} if  $\mathcal{P}$ is a clopen hereditary finitely additive perfect topological property which satisfies Mr\'{o}wka's condition $(\mbox{W})$. If $\mathcal{Q}$ also is a topological property, then we say that {\em ${\mathcal P}$ and  ${\mathcal Q}$ is a pair of compactness--like topological properties} (here the order of ${\mathcal P}$ and  ${\mathcal Q}$  is important)
if $\mathcal{P}$ is a compactness--like topological property and ${\mathcal Q}$ is a clopen hereditary topological property which is inverse invariant under perfect mappings and satisfies Mr\'{o}wka's condition $(\mbox{W})$. (Examples of pairs of compactness--like topological properties are given in Example \ref{20UIHG}.)
\end{definition}

\begin{lemma}\label{122}
Let $\mathcal{P}$ be a topological property which is inverse invariant under perfect mappings and satisfies Mr\'{o}wka's condition $(\mbox{\em W})$. Then if $X$  is a Tychonoff space  in which there exists a compact subset $A$ with an open base ${\mathscr B}$ for $X$ at $A$ such that $X\backslash  B$ has $\mathcal{P}$ for any $B\in {\mathscr B}$, then $X$ has $\mathcal{P}$.
\end{lemma}

\begin{proof}
If $A=\emptyset$ then the lemma holds trivially, as in this case $\emptyset\in{\mathscr B}$. Suppose that $A$ is non--empty.  Let $T$ be the space obtained from $X$ by contracting the set $A$ to a point $p$ and let $q:X\rightarrow T$ denote the corresponding quotient mapping. Note that since $A$ is compact $T$ is Tychonoff. Now $\{q[B]:B\in{\mathscr B}\}$ is an open base for $T$ at $p$ such that $T\backslash q[B]=X\backslash B$ has $\mathcal{P}$ for any $B\in {\mathscr B}$. Since $\mathcal{P}$ satisfies Mr\'{o}wka's condition $(\mbox{W})$ the space $T$, and  thus its inverse image $X$ under the perfect surjective mapping $q$ has $\mathcal{P}$.
\end{proof}

Note that if $A$ is a dense subset of a space $X$ and $U$ is an open subset of $X$ then $\mbox{cl}_XU=\mbox{cl}_X(U\cap A)$ and thus  $U\subseteq\mbox{int}_X\mbox{cl}_X(U\cap A)$. In particular, if $X$ is a Tychonoff space, $f:\beta X\rightarrow\mathbf{I}$ is continuous and $r\in(0,1)$ then
\[f^{-1}\big[[0,r)\big]\subseteq \mbox{int}_{\beta X}\mbox{cl}_{\beta X}\big(f^{-1}\big[[0,r)\big]\cap X\big).\]
We use such simple observations frequently in the future.

A Hausdorff space is called {\em zero--dimensional} if the set of all its clopen subsets constitute an open  base for it. A Tychonoff space  is called {\em strongly zero--dimensional} if its Stone--\v{C}ech compactification is zero--dimensional. For a regular space $X$ let  $(EX,k_X)$ denote the absolute of $X$ (see Theorem 6.6(e) of \cite{PW} or  Problem 6.3.20 of \cite{E}). The space $EX$ is (extremely disconnected and) zero--dimensional (thus strongly zero--dimensional; see Theorem 6.4 of \cite{PW})  and $k_X:EX\rightarrow X$ is a perfect (irreducible) surjective mapping.

The following lemma is quite fundamental in our study. We state and prove it in its general form for possible future reference.

\begin{lemma}\label{16}
Let ${\mathcal P}$ and  ${\mathcal Q}$ be a pair of compactness--like topological properties. Let $X$ and $Y$ be Tychonoff spaces such that $Y$ has $\mathcal{Q}$, let $f:X\rightarrow Y$ be a perfect surjective mapping, let $T\in{\mathscr E}(Y)$, let $\alpha T$ be a compactification of $T$ and let $\phi:\beta X\rightarrow\alpha T$ be the continuous extension of $f$. The  following are equivalent:
\begin{itemize}
\item[\rm(1)] $T\in{\mathscr E}^{\mathcal Q}_{\mathcal P}(Y)$.
\item[\rm(2)] $X$ is locally--$\mathcal{P}$ and $\beta X\backslash\lambda_{\mathcal{P}}X\subseteq\phi^{-1}[T\backslash Y]$.
\end{itemize}
\end{lemma}

\begin{proof}
(1) {\em implies} (2). Since $\mathcal{P}$ is hereditary with respect to closed subsets of Hausdorff spaces $Y$, having a $\mathcal{P}$--extension with compact remainder, is locally--$\mathcal{P}$ and therefore by Lemma \ref{22} the space $X$ is  locally--$\mathcal{P}$. Next, we show that $\beta X\backslash \lambda_{\mathcal{P}} X\subseteq\phi^{-1}[T\backslash Y]$. Suppose to the
contrary that there exists an $x\in \beta X\backslash \lambda_{\mathcal{P}} X$ such that $x\notin \phi^{-1}[T\backslash Y]$. Let $g:\beta
X\rightarrow\mathbf{I}$ be continuous with $g(x)=0$ and $g[\phi^{-1}[T\backslash Y]]\subseteq\{1\}$ and let $Z=g^{-1}[[0,1/2]]$. We verify that $Z\cap X\in{\mathscr Z}(X)$ has $\mathcal{P}$. Since $Z\cap\phi^{-1}[T\backslash Y]=\emptyset$ we have
\[\phi[Z]\subseteq\phi\big[\beta X\backslash\phi^{-1}[T\backslash Y]\big]=\phi\big[\phi^{-1}\big[\alpha T\backslash(T\backslash Y)\big]\big]\subseteq\alpha T\backslash(T\backslash Y)\]
and thus $S=\phi[Z]\cap T\subseteq Y$. Therefore $S$ has $\mathcal{P}$, as it is closed in $T$, because $Z$ is compact. Since $f|f^{-1}[S]:f^{-1}[S]\rightarrow S$ is perfect and surjective and $\mathcal{P}$ in inverse invariant under perfect mappings, $f^{-1}[S]$ has $\mathcal{P}$. Thus
\[Z\cap X\subseteq f^{-1}\big[f[Z\cap X]\big]\subseteq f^{-1}\big[\phi[Z\cap X]\cap Y\big]\subseteq f^{-1}\big[\phi[Z]\cap T\big]=f^{-1}[S]\]
which implies that $Z\cap X$ has $\mathcal{P}$, as it is closed in $f^{-1}[S]$. Now
\begin{eqnarray*}
x\in g^{-1}\big[[0,1/2)\big]\subseteq\mbox{int}_{\beta X}\mbox{cl}_{\beta X}\big(g^{-1}\big[[0,1/2]\big]\cap X\big)=\mbox{int}_{\beta X}\mbox{cl}_{\beta X}(Z\cap X)\subseteq\lambda_{\mathcal{P}} X
\end{eqnarray*}
which is a contradiction, as $x\notin\lambda_{\mathcal{P}}X$.

(2) {\em implies} (1). Suppose moreover that $X$ is strongly zero--dimensional. Let
\[{\mathscr B}=\big\{T\backslash f\big[\mbox{cl}_Xf^{-1}\big[T\backslash\phi[U]\big]\big]:U \mbox{ is clopen in }\beta X\mbox{ and }\phi^{-1}[T\backslash Y]\subseteq U\big\}.\]
We verify that ${\mathscr B}$ is an open base for $T$ at $T\backslash Y$ such that $T\backslash B$ has both $\mathcal{P}$ and $\mathcal{Q}$ for any $B\in {\mathscr B}$. By Lemma \ref{122} this will imply that $T$ has $\mathcal{P}$ and $\mathcal{Q}$. Let $U$ be a clopen subset of $\beta X$ such that $\phi^{-1}[T\backslash Y]\subseteq U$. Consider
\[B=T\backslash f\big[\mbox{cl}_Xf^{-1}\big[T\backslash\phi[U]\big]\big]\in{\mathscr B}.\]
Since $\phi$ is surjective (as its image contains $Y=f[X]=\phi[X]$ and $Y$ is dense in $T$) we have  $T\backslash Y=\phi[\phi^{-1}[T\backslash Y]]\subseteq\phi[U]$ and thus $T\backslash \phi[U]\subseteq Y$. Since
\[f^{-1}\big[T\backslash\phi[U]\big]\subseteq\phi^{-1}\big[T\backslash\phi [U]\big]\subseteq\phi^{-1}\big[\alpha T\backslash \phi[U]\big]=\beta X\backslash\phi^{-1}\big[\phi[U]\big]\subseteq\beta X\backslash U\]
we have $\mbox{cl}_Xf^{-1}[T\backslash\phi[U]]\subseteq\beta X\backslash U$ which yields
\begin{eqnarray*}
T\backslash B&=&f\big[\mbox{cl}_Xf^{-1}\big[T\backslash\phi[U]\big]\big]\\&=&\phi\big[\mbox{cl}_Xf^{-1}\big[T\backslash\phi [U]\big]\big]\\&\subseteq&\phi[\beta X\backslash U]\subseteq\phi\big[\beta X\backslash\phi^{-1}[T\backslash Y]\big]=\phi\big[\phi^{-1}\big[\alpha T\backslash(T\backslash Y)\big]\big]\subseteq\alpha T\backslash(T\backslash Y).
\end{eqnarray*}
Therefore since $U$ is clopen in $\beta X$ the set $\phi[\beta X\backslash U]$ is compact and thus $\mbox{cl}_{\alpha T}(T\backslash B)\subseteq\phi[\beta X\backslash U]$. By above this implies that
\[\mbox{cl}_T(T\backslash B)\cap(T\backslash Y)=\mbox{cl}_{\alpha T}(T\backslash B)\cap(T\backslash Y)=\emptyset\]
and therefore since $f$ is closed
\[\mbox{cl}_T(T\backslash B)=\mbox{cl}_Y(T\backslash B)=T\backslash B.\]
This shows that each $B\in{\mathscr B}$ is open in $T$. Obviously, each $B\in{\mathscr B}$ contains $T\backslash Y$. Next, we show that each open neighborhood $W$ of $T\backslash Y$ in $\alpha T$ contains an element of ${\mathscr B}$. Since $X$ is strongly zero--dimensional, $\beta X$ is zero--dimensional. Now since $\phi^{-1}[W]$ is an open neighborhood of the compact set $\phi^{-1}[T\backslash Y]$ in $\beta X$ there exists a clopen subset $U$ of $\beta X$ such that $\phi^{-1}[T\backslash Y]\subseteq U\subseteq \phi^{-1}[W]$ (see Theorem 6.2.4 of \cite{E}). Note that
\begin{eqnarray*}
B=T\backslash f\big[\mbox{cl}_Xf^{-1}\big[T\backslash\phi[U]\big]\big]&\subseteq& T\backslash f\big[f^{-1}\big[T\backslash\phi[U]\big]\big]\\&=&T\backslash\big(T\backslash\phi[U]\big)\subseteq\phi[U]\subseteq\phi\big[\phi^{-1}[W]\big]\subseteq W
\end{eqnarray*}
and that $B\in {\mathscr B}$. This shows that  ${\mathscr B}$ is an open base for $T$ at $T\backslash Y$. Now let $B\in {\mathscr B}$. Then $T\backslash B=f[\mbox{cl}_Xf^{-1}[T\backslash\phi[U]]]$ for some clopen subset $U$ of $\beta X$ containing  $\phi^{-1}[T\backslash Y]$. Since $f$ is closed, $T\backslash B$ is closed in $Y$, and since by our assumption $Y$ has $\mathcal{Q}$ and $\mathcal{Q}$ is hereditary with respect to closed subsets of Hausdorff spaces, $T\backslash B$ has $\mathcal{Q}$. Also, as argued above $f^{-1}[T\backslash\phi[U]]\subseteq \beta X\backslash U$ which implies that
\[\mbox{cl}_{\beta X}f^{-1}\big[T\backslash\phi[U]\big]\subseteq\beta X\backslash U\subseteq\beta X\backslash\phi^{-1}[T\backslash Y]\subseteq\lambda_{\mathcal{P}} X.\]
By Lemma \ref{B} the set $C=\mbox{cl}_Xf^{-1}[T\backslash\phi[U]]$ has $\mathcal{P}$. Now $f|C:C\rightarrow f[C]$ is perfect (as $C$ is closed in $X$) and surjective and $\mathcal{P}$ is invariant under perfect mappings, thus $T\backslash B=f[C]$ has $\mathcal{P}$. This shows that (1) holds in this case.

We now turn to the general case in which $X$ is an arbitrary  Tychonoff space. Let $(EX,k)$ denote the absolute of $X$. By our assumption $X$ is locally--$\mathcal{P}$, and since $k:EX\rightarrow X$ is perfect, by Lemma \ref{22} the space $EX$ is locally--$\mathcal{P}$. Let $K:\beta EX\rightarrow\beta X$ be the continuous extension of $k$. Then $\phi K:\beta EX\rightarrow\alpha T$ continuously extends $fk$ and therefore by above, to show that
$T\in{\mathscr E}^{\mathcal Q}_{\mathcal P}(Y)$ we only need to verify that
\[\beta EX\backslash\lambda_{\mathcal{P}} EX\subseteq(\phi K)^{-1}[T\backslash Y].\]
But by our assumption $\beta X\backslash\lambda_{\mathcal{P}} X\subseteq\phi^{-1}[T\backslash Y]$. Thus as we will see it suffices to show that $K^{-1}[\lambda_{\mathcal{P}} X]\subseteq\lambda_{\mathcal{P}} EX$. Let $t\in K^{-1}[\lambda_{\mathcal{P}} X]$. Let $U$ be an open neighborhood of $K(t)$ in $\beta X$ such that $\mbox{cl}_{\beta X}U\subseteq\lambda_{\mathcal{P}}X$. Let $h:\beta X\rightarrow\mathbf{I}$ be continuous with $h(K(t))=0$ and $h[\beta X\backslash U]\subseteq\{1\}$. Let
\[Z=h^{-1}\big[[0,1/2]\big]\cap X\in{\mathscr Z}(X).\]
Then since
\[\mbox{cl}_{\beta X}Z\subseteq\mbox{cl}_{\beta X}(U\cap X)=\mbox{cl}_{\beta X}U\subseteq\lambda_{\mathcal{P}} X,\]
by Lemma \ref{B} the set $Z$ has $\mathcal{P}$ and therefore its inverse image $k^{-1}[Z]\in{\mathscr Z}(EX)$ under the perfect surjective mapping $k|k^{-1}[Z]:k^{-1}[Z]\rightarrow Z$ has $\mathcal{P}$. By the definition of $\lambda_{\mathcal{P}} EX$ we have $\mbox{int}_{\beta EX}\mbox{cl}_{\beta EX}k^{-1}[Z]\subseteq \lambda_{\mathcal{P}} EX$. By  Theorem 3.7.16 of \cite{E} (or Theorem 1.8(i) of \cite{PW}) and since $K|EX=k$ is perfect, $K[\beta EX\backslash EX]\subseteq\beta X\backslash X$. But $K$ is surjective, as its image contains $X=k[EX]=K[EX]$, and thus $K[\beta EX\backslash EX]= \beta X\backslash X$. We have
\begin{eqnarray*}
k^{-1}[Z]&=& k^{-1}\big[h^{-1}\big[[0,1/2]\big]\cap X\big]\\&=&K^{-1}\big[h^{-1}\big[[0,1/2]\big]\cap X\big]\\&=&K^{-1}\big[h^{-1}\big[[0,1/2]\big]\big]\cap K^{-1}[X]=K^{-1}\big[h^{-1}\big[[0,1/2]\big]\big]\cap EX
\end{eqnarray*}
and therefore
\begin{eqnarray*}
t\in K^{-1}\big[h^{-1}\big[[0,1/2)\big]\big]&\subseteq&\mbox{cl}_{\beta EX}K^{-1}\big[h^{-1}\big[[0,1/2)\big]\big]\\&=&\mbox{cl}_{\beta EX}\big(K^{-1}\big[h^{-1}\big[[0,1/2)\big]\big]\cap EX\big)\\&\subseteq&\mbox{cl}_{\beta EX}\big(K^{-1}\big[h^{-1}\big[[0,1/2]\big]\big]\cap EX\big)=\mbox{cl}_{\beta EX}k^{-1}[Z]
\end{eqnarray*}
which yields $t\in\mbox{int}_{\beta EX}\mbox{cl}_{\beta EX}k^{-1}[Z]$ and thus $t\in \lambda_{\mathcal{P}} EX$. This shows that  $K^{-1}[\lambda_{\mathcal{P}} X]\subseteq\lambda_{\mathcal{P}} EX$. Now
\begin{eqnarray*}
\beta EX\backslash\lambda_{\mathcal{P}} EX&\subseteq&\beta EX\backslash K^{-1}[\lambda_{\mathcal{P}}X]\\&=&K^{-1}[\beta X\backslash
\lambda_{\mathcal{P}}X]\subseteq K^{-1}\big[\phi^{-1}[T\backslash Y]\big]=(\phi K)^{-1}[T\backslash Y]
\end{eqnarray*}
which shows (1).
\end{proof}

The list of topological properties $\mathcal{P}$ and $\mathcal{Q}$ satisfying the assumption of Lemma \ref{16} is quite wide and includes most of the important covering properties (see Example \ref{20UIHG}).

\begin{xrem}
{\em Lemma \ref{16} (and thus its subsequent results) remains valid if one omits $\mathcal{Q}$ from its statement. This is because one can replace $\mathcal{Q}$ by regularity (note that by Theorem 3.7.23 of \cite{E} regularity  is inverse invariant under perfect mappings and satisfies Mr\'{o}wka's condition $(\mbox{W})$ vacuously) and observes that for this specific choice of $\mathcal{Q}$ the terms ``Tychonoff space with $\mathcal{Q}$" and ``${\mathcal P}$ and ${\mathcal Q}$ is a pair of compactness--like topological properties" coincide with the terms ``Tychonoff space" and ``${\mathcal P}$ is a compactness--like topological property", respectively.}
\end{xrem}

\begin{xrem}
{\em Lemma \ref{16} is stronger than what we normally need, as we generally apply Lemma \ref{16} in the special case when $Y=X$, $f=\mbox{id}_X$ and $\alpha T=\beta T$. Lemma \ref{16} is quite fundamental in our study and it is interesting to know whether the requirement ``$\mathcal{P}$ satisfies Mr\'{o}wka's condition $(\mbox{W})$" (implicit in the definition of the compactness--like topological property $\mathcal{P}$) can be omitted from its statement. In Example \ref{20UIHG} we give an example of a Tychonoff space $X$, a topological property $\mathcal{P}$ which does not satisfy Mr\'{o}wka's condition $(\mbox{W})$ and a Tychonoff extension of $X$ with compact remainder, for which the conclusion of (the special case of) Lemma \ref{16} does not hold.}
\end{xrem}

In the sequel we will make frequent use of the following well known result sometimes without explicitly referring to it. The proof is included here for the sake of completeness.

\begin{lemma}\label{j2}
Let $X$ be a Tychonoff space and let $Y$ be a Tychonoff  extension of $X$ with the compact remainder $Y\backslash X=\{p_i:i\in I\}$ where $p_i$'s are
bijectively indexed. Let $\phi:\beta X\rightarrow\beta Y$ be the continuous extension of $\mbox{\em id}_X$. Let $T$ be the space obtained from $\beta X$ by contracting any fiber $\phi^{-1}(p_i)$ where $i\in I$ to a point $a_i$. Then $T=\beta Y$ (identifying each $a_i$ with $p_i$) and $\phi=q$ where $q:\beta X\rightarrow T$ is the quotient mapping.
\end{lemma}

\begin{proof}
We first show that $T$ is a compactification of $Y$. To show that $T$ is Hausdorff let $s,t\in T$ be distinct elements. Consider the following cases:
\begin{description}
\item[{\sc Case 1.}] Suppose that $s,t\in T\backslash \{a_i:i\in I\}$. Then $s,t\in \beta X\backslash\phi^{-1}[Y\backslash X]$ and thus there exist disjoint open neighborhoods $U$ and $V$ of $s$ and $t$ in $\beta X$, respectively, each disjoint from $\phi^{-1}[Y\backslash X]$. The sets $q[U]$ and $q[V]$ are disjoint open neighborhoods of $s$ and $t$ in $T$, respectively.
\item[{\sc Case 2.}] Suppose that $s=a_i$ for some $i\in I$ and $t\in T\backslash \{a_i:i\in I\}$. Then $\phi^{-1}[Y\backslash X]$ is a compact subset of $\beta X$ not containing $t$ and thus there exist disjoint open subsets $U$ and $V$ of $\beta X$ such that $\phi^{-1}[Y\backslash X]\subseteq U$ and $t\in V$. Now $q[U]$ and $q[V]$ are disjoint open neighborhoods of $s$ and $t$ in $T$, respectively. The case when  $s\in T\backslash\{a_i:i\in I\}$ and $t=a_j$ for some  $j\in I$ is analogous.
\item[{\sc Case 3.}] Suppose that $s=a_i$ and $t=a_j$ for some $i,j\in I$. Let $U_i $ and $U_j$ be disjoint open neighborhoods of $p_i$ and $p_j$  in $\beta Y$, respectively. Then since $q^{-1}[q[\phi^{-1}[U_k]]]=\phi^{-1}[U_k]$ where $k=i,j$ are open subsets of $\beta X$ and  $\phi^{-1}(p_k)\subseteq\phi^{-1}[U_k]$ the sets $q[\phi^{-1}[U_k]]$ where $k=i,j$ are disjoint open neighborhoods of $s$ and $t$  in $T$, respectively.
\end{description}
This shows that $T$ is Hausdorff and therefore it is compact, as it is a continuous image of $\beta X$. Note that $Y$ is a subspace of $T$. To show this first note that since $\beta Y$ is also a compactification of $X$, as $X$ is dense in $Y$ and thus in $\beta Y$, and $\phi|X=\mbox{id}_X$, by Theorem 3.5.7 of \cite{E} we have $\phi[\beta X\backslash X]=\beta Y\backslash X$.  Now if $W$ is open in $\beta Y$, since $q^{-1}[q[\phi^{-1}[W]]]=\phi^{-1}[W]$ is open in $\beta X$ the set $q[\phi^{-1}[W]]$ is open in $T$, and therefore $W\cap Y=q[\phi^{-1}[W]]\cap Y$ is open in $Y$ as a subspace of $T$. For the converse note that if $W$ is an open subset of $T$ then
\[W\cap Y=\big(\beta Y\backslash\phi\big[\beta X\backslash q^{-1}[W]\big]\big)\cap Y\]
and therefore (since $\phi[\beta X\backslash q^{-1}[W]]$ is compact and thus closed in $\beta Y$) the set $W\cap Y$ is open in $Y$ in its original topology. Clearly, $Y$ is dense in $T$ and therefore $T$ is a compactification of $Y$. To  show that $T=\beta Y$ it suffices to verify that any continuous $f:Y\rightarrow\mathbf{I}$ can be continuously extended over $T$. Indeed, consider the continuous mapping
\[g=fq:S=X\cup\phi^{-1}[Y\backslash X]\rightarrow\mathbf{I}.\]
Note that since $X\subseteq S\subseteq\beta X$ we have $\beta S=\beta X$ (see Corollary 3.6.9 of \cite{E}). Let $g_\beta:\beta X\rightarrow\mathbf{I}$ be the continuous extension of $g$. Define $F:T\rightarrow\mathbf{I}$ such that $F(x)=g_\beta(x)$ for any $x\in\beta X\backslash\phi^{-1}[Y\backslash X]$ and $F(p_i)=f(p_i)$ for any $i\in I$. Then $F|Y=f$ and since $Fq=g_\beta$ is continuous, $F$ is continuous. This shows that $T=\beta Y$. Note, this also implies that $\phi=q$, as $\phi,q:\beta X\rightarrow\beta Y$ are continuous and $\phi|X=\mbox{id}_X=q|X$.
\end{proof}

The following simple observation will be of frequent use in the future, sometimes with no explicit reference.

\begin{lemma}\label{15}
Let $X$ be a Tychonoff space and let $\mathcal{P}$ be a clopen hereditary topological property which is  inverse invariant under perfect mappings.  Then $X\subseteq\lambda_{\mathcal{P}} X$ if and only if $X$ is locally--$\mathcal{P}$.
\end{lemma}

\begin{proof}
Suppose that $X$ is locally--$\mathcal{P}$. Let $x\in X$ and let $U$ be an open neighborhood of $x$ in $X$ whose closure $\mbox{cl}_XU$ has $\mathcal{P}$. Let $f:X\rightarrow\mathbf{I}$ be continuous with $f(x)=0$ and $f[X\backslash U]\subseteq\{1\}$ and let $f_\beta:\beta X\rightarrow \mathbf{I}$ be the continuous extension of $f$. Let $Z=f^{-1}[[0,1/2]]\in{\mathscr Z}(X)$. Then $Z\subseteq U$ and thus $Z$ has $\mathcal{P}$, as it is  closed in $\mbox{cl}_XU$. Now
\begin{eqnarray*}
x\in f_\beta^{-1}\big[[0,1/2)\big]\subseteq\mbox{int}_{\beta X}\mbox{cl}_{\beta X}f^{-1}\big[[0,1/2]\big]=\mbox{int}_{\beta X}\mbox{cl}_{\beta X}Z\subseteq\lambda_{\mathcal{P}}X
\end{eqnarray*}
and therefore $X\subseteq\lambda_{\mathcal{P}} X$. For the converse suppose that  $X\subseteq\lambda_{\mathcal{P}} X$. Let $x\in X$. Then $x\in \lambda_{\mathcal{P}} X$ and therefore $x\in\mbox{int}_{\beta X}\mbox{cl}_{\beta X}S$ for some $S\in{\mathscr Z}(X)$ which has $\mathcal{P}$. Let $V=(\mbox{int}_{\beta X}\mbox{cl}_{\beta X}S)\cap X$. Then $V$ is an open neighborhood of $x$ in $X$. Since $V\subseteq S$ the set $\mbox{cl}_XV$ has $\mathcal{P}$, as it is closed in $S$. Thus $X$ is locally--$\mathcal{P}$.
\end{proof}

Our next theorem gives characterizations of the elements of ${\mathscr M}^{\mathcal Q}_{\mathcal P}(X)$. Compare with its dual result on ${\mathscr O}^{\mathcal Q}_{\mathcal P}(X)$ (Theorem \ref{HG16}).

\begin{theorem}\label{HUHG16}
Let ${\mathcal P}$ and  ${\mathcal Q}$ be a pair of compactness--like topological properties. Let $X$ be a Tychonoff space with $\mathcal{Q}$ and let $Y\in{\mathscr E}^{\mathcal Q}_{\mathcal P}(X)$. The following are equivalent:
\begin{itemize}
\item[\rm(1)] $Y\in{\mathscr M}^{\mathcal Q}_{\mathcal P}(X)$.
\item[\rm(2)] For any $p\in Y\backslash X$ the set $\phi^{-1}(p)\backslash\lambda_{\mathcal P} X$ is non--empty where $\phi:\beta X\rightarrow\beta Y$ is the continuous extension of $\mbox{\em id}_X$.
\item[\rm(3)] For any open subset $V$ of $Y$ such that $\mbox{\em cl}_X(V\cap  X)$ has ${\mathcal P}$ we have $V\cap(Y\backslash X)=\emptyset$.
\item[\rm(4)] For any $T\in{\mathscr E}^{\mathcal Q}_{\mathcal P}(X)$ and any continuous $f:T\rightarrow Y$ such that $f|X=\mbox{\em id}_X$, the mapping $f$ is surjective.
\item[\rm(5)] For any $T\in{\mathscr E}^{\mathcal Q}_{\mathcal P}(X)$ such that $Y\leq T$ there exists a continuous surjective $f:T\rightarrow Y$ such that $f|X=\mbox{\em id}_X$.
\end{itemize}
\end{theorem}

\begin{proof}
Let $\phi:\beta X\rightarrow\beta Y$ be the continuous extension of $\mbox{id}_X$. (1) {\em implies} (2). Consider the subspace
\[T=X\cup\big\{p\in Y\backslash X:\phi^{-1}(p)\backslash\lambda_{\mathcal P}X\neq\emptyset\big\}\]
of $Y$. We show that $T\backslash X=\phi[\beta X\backslash\lambda_{\mathcal P}X]$. First note that by Lemma \ref{16} we have  $\beta X\backslash\lambda_{\mathcal P} X\subseteq\phi^{-1}[Y\backslash X]$ and that $X$ is locally--${\mathcal P}$. Now if $t\in \beta X\backslash\lambda_{\mathcal P} X$ then $\phi(t)=p\in Y\backslash X$ and thus $\phi^{-1}(p)\backslash\lambda_{\mathcal P} X$ is non--empty, as it contains $t$. Therefore $\phi(t)=p\in T\backslash X$. This shows that $\phi[\beta X\backslash\lambda_{\mathcal P} X]\subseteq T\backslash X$. To show the reverse inclusion note that if $p\in T\backslash X$ then there exists some $t\in\phi^{-1}(p)\backslash\lambda_{\mathcal P} X\subseteq\beta X\backslash\lambda_{\mathcal P} X$ and thus $p=\phi(t)\in\phi[\beta X\backslash\lambda_{\mathcal P} X]$. This shows that $T\in{\mathscr E}(X)$. Now since
\[\phi^{-1}[T\backslash X]=\phi^{-1}\big[\phi[\beta X\backslash\lambda_{\mathcal P} X]\big]\supseteq \beta X\backslash\lambda_{\mathcal P} X\]
by Lemma \ref{16} it follows that $T\in{\mathscr E}_{\mathcal P}(X)$. By the minimality of $Y$ we have $T=Y$ and in particular $T\backslash X=Y\backslash X$.

(2) {\em implies} (1). Let $T\in{\mathscr E}_{\mathcal P}(X)$ be such that $T\subseteq Y$. By (the remark succeeding) Lemma \ref{16} we have $\beta X\backslash\lambda_{\mathcal P} X\subseteq \phi^{-1}[T\backslash X]$. Now if there exists some $p\in Y\backslash T$ then
\[\phi^{-1}(p)\backslash\lambda_{\mathcal P} X\subseteq\phi^{-1}(p)\cap\phi^{-1}[T\backslash X]=\emptyset\]
which contradicts (2). Thus $T=Y$. This shows the minimality of $Y$.

(2) {\em implies} (3). To show (3) let $V$ be an open subset of $Y$ such that $V\cap(Y\backslash X)$ is non--empty. We need to show that $\mbox{cl}_X (V\cap X)$ is non--${\mathcal P}$. Let $V=W\cap Y$ where $W$ is an  open subset of $\beta Y$. Let $p\in V\cap(Y\backslash X)$. Let $g:\beta Y\rightarrow\mathbf{I}$ be continuous with $g(p)=0$ and $g[\beta Y\backslash W]\subseteq\{1\}$ and let
\[Z=(g\phi)^{-1}\big[[0,1/2]\big]\cap X\in {\mathscr Z}(X).\]
Note that
\begin{eqnarray*}
Z=(g\phi)^{-1}\big[[0,1/2]\big]\cap X&=&\phi^{-1}\big[g^{-1}\big[[0,1/2]\big]\big]\cap X\\&=&g^{-1}\big[[0,1/2]\big]\cap X\subseteq W\cap X=V\cap X.
\end{eqnarray*}
Thus if $\mbox{cl}_X(V\cap X)$ has ${\mathcal P}$ then its closed subset $Z$ also has ${\mathcal P}$. Now
\begin{eqnarray*}
\phi^{-1}(p)\subseteq\phi^{-1}\big[g^{-1}\big[[0,1/2)\big]\big]&=&(g\phi)^{-1}\big[[0,1/2)\big]\\&\subseteq&\mbox{int}_{\beta X}\mbox{cl}_{\beta X}\big((g\phi)^{-1}\big[[0,1/2]\big]\cap X\big)\\&=&\mbox{int}_{\beta X}\mbox{cl}_{\beta X}Z\subseteq\lambda_{\mathcal P} X
\end{eqnarray*}
contradicting (2). Therefore $\mbox{cl}_X (V\cap X)$ is non--${\mathcal P}$.

(3) {\em implies} (2). Suppose to the contrary that $\phi^{-1}(p)\backslash\lambda_{\mathcal P} X=\emptyset$  for some $p\in Y\backslash X$. Then $p\notin \phi[\beta X\backslash\lambda_{\mathcal P} X]$. Let $W$ be an open neighborhood of $p$ in $\beta Y$ such that $\mbox{cl}_{\beta Y}W\cap\phi[\beta X\backslash\lambda_{\mathcal P} X]=\emptyset$. We have
\begin{eqnarray*}
\phi^{-1}[\mbox{cl}_{\beta Y}W]\backslash\lambda_{\mathcal P} X&\subseteq&\phi^{-1}[\mbox{cl}_{\beta Y}W]\cap\phi^{-1}\big[\phi[\beta X\backslash\lambda_{\mathcal P} X]\big]\\&=&\phi^{-1}[\mbox{cl}_{\beta Y}W\cap\phi[\beta X\backslash\lambda_{\mathcal P} X]\big]=\emptyset
\end{eqnarray*}
and thus
\[\mbox{cl}_{\beta X}(W\cap X)=\mbox{cl}_{\beta X}\big(\phi^{-1}[W]\cap X\big)=\mbox{cl}_{\beta X}\phi^{-1}[W]\subseteq\phi^{-1}[\mbox{cl}_{\beta Y}W]\subseteq\lambda_{\mathcal P} X.\]
Lemma \ref{B} implies that $\mbox{cl}_X(W\cap X)$ has ${\mathcal P}$. Now $V=W\cap Y$ is an open neighborhood of $p$ in $Y$ such that  $\mbox{cl}_X(V\cap X)=\mbox{cl}_X(W\cap X)$ has ${\mathcal P}$, contradicting (3).

(2) {\em implies} (4). Let $T\in{\mathscr E}^{\mathcal Q}_{\mathcal P}(X)$ and let $f:T\rightarrow Y$ be continuous with $f|X=\mbox{id}_X$. Let $f_\beta:\beta T\rightarrow\beta Y$ and $\psi:\beta X\rightarrow\beta T$ be the continuous extensions of $f$ and $\mbox{id}_X$, respectively. Then since $f_\beta\psi|X=\phi| X$ we have $f_\beta\psi=\phi$. Lemma \ref{16} implies that $\beta X\backslash\lambda_{\mathcal P} X\subseteq\psi^{-1}[T\backslash X]$. Also, for any $p\in Y\backslash X$, since $\phi^{-1}(p)\backslash\lambda_{\mathcal P} X$ is non--empty, $p\in \phi[\phi^{-1}(p)\backslash\lambda_{\mathcal P} X]$. Thus
\begin{eqnarray*}
Y\backslash X\subseteq\bigcup\big\{\phi\big[\phi^{-1}(p)\backslash\lambda_{\mathcal P} X\big]:p\in Y\backslash X\big\}&\subseteq&\phi[\beta X\backslash\lambda_{\mathcal P} X]\\&=&f_\beta\big[\psi[\beta X\backslash\lambda_{\mathcal P} X]\big]\\&\subseteq&f_\beta\big[\psi\big[\psi^{-1}[T\backslash X]\big]\big]\\&\subseteq&f_\beta[T\backslash X]=f[T\backslash X]\subseteq f[T].
\end{eqnarray*}
Since $f|X=\mbox{id}_X$ this shows that $Y\subseteq f[T]$, that is, $f$ is surjective. That (4) implies (5) is trivial.

(5) {\em implies} (2). Consider the subspace $T=X\cup\phi[\beta X\backslash\lambda_{\mathcal P}X]$ of $\beta Y$. By Lemma \ref{16} we have  $\beta X\backslash\lambda_{\mathcal P} X\subseteq\phi^{-1}[Y\backslash X]$ and that $X$ is locally--${\mathcal P}$. Thus
\[T=X\cup\phi[\beta X\backslash\lambda_{\mathcal P}X]\subseteq X\cup\phi\big[\phi^{-1}[Y\backslash X]\big]\subseteq X\cup(Y\backslash X)=Y.\]
By Lemma \ref{15} we have $X\subseteq\lambda_{\mathcal P}X$. Now $T\backslash X=\phi[\beta X\backslash\lambda_{\mathcal P} X]$ is compact, and since
\[\phi^{-1}[T\backslash X]=\phi^{-1}\big[\phi[\beta X\backslash\lambda_{\mathcal P} X]\big]\supseteq\beta X\backslash\lambda_{\mathcal P} X,\]
by Lemma \ref{16} it follows that $T\in{\mathscr E}^{\mathcal Q}_{\mathcal P}(X)$. It is clear that $Y\leq T$, as $T\subseteq Y$. By (5) there exists a continuous surjective $f:T\rightarrow Y$ such that $f|X=\mbox{id}_X$. But $f|X=\mbox{id}_T|X$ which yields $f=\mbox{id}_T$ and therefore  $Y=f[T]=T$. Now it is clear that for any
\[p\in Y\backslash X=T\backslash X=\phi[\beta X\backslash\lambda_{\mathcal P} X]\]
the set $\phi^{-1}(p)\backslash\lambda_{\mathcal P} X$ is non--empty.
\end{proof}

\begin{xrem}
{\em Theorem \ref{HUHG16} fails if one omits the requirement ``$\mathcal{P}$ satisfies Mr\'{o}wka's condition $(\mbox{W})$" (implicit in the definition of the compactness--like topological property $\mathcal{P}$) from its statement (see Example \ref{20UIHG}).}
\end{xrem}

In the next theorem we give characterizations of the elements of ${\mathscr O}^{\mathcal Q}_{\mathcal P}(X)$. We need to prove a few lemmas first.

\begin{notation}
Let $X$ be a Tychonoff space and let $Y$ be a Tychonoff extension of $X$. Let $\phi:\beta X\rightarrow \beta Y$ be the (unique) continuous mapping which extends $\mbox{id}_X$. Denote
\[{\mathscr F}_X(Y)=\big\{\phi^{-1}(p):p\in Y\backslash X\big\}.\]
We may write ${\mathscr F}(Y)$ instead of ${\mathscr F}_X(Y)$ when no confusion arises.
\end{notation}

In \cite{Mag3} the author associated to each compactification $\alpha X$ of a Tychonoff space $X$ a set (called the {\em $\beta$--family} of $\alpha X$)
\[{\mathscr F}_\alpha=\big\{f_\alpha^{-1}(p):p\in \alpha X\backslash X\big\}\]
where $f_\alpha:\beta X\rightarrow \alpha X$ is  the continuous extension of $\mbox{id}_X$. It is then shown that for any compactifications $\alpha_1 X$ and $\alpha_2 X$  of a Tychonoff space $X$ we have $\alpha_1 X\leq\alpha_2 X$ if and only if each set in ${\mathscr F}_{\alpha_2}$  is a subset of a set in ${\mathscr F}_{\alpha_1}$. This provides the motivation for the statement of the next lemma.

\begin{lemma}\label{DFH}
Let $X$ be a Tychonoff space and let $Y_1,Y_2\in{\mathscr E}(X)$. The following are equivalent:
\begin{itemize}
\item[\rm(1)] $Y_1\leq Y_2$.
\item[\rm(2)] Any element of ${\mathscr F}(Y_2)$ is contained in an element of ${\mathscr F}(Y_1)$.
\end{itemize}
\end{lemma}

\begin{proof}
Let $\phi_i:\beta X\rightarrow \beta Y_i$ where $i=1,2$ be the continuous extension of $\mbox{id}_X$. (1) {\em  implies} (2). By definition there exists a continuous $f:Y_2\rightarrow Y_1$ such that $f|X= \mbox{id}_X$. Let $f_\beta:\beta Y_2\rightarrow\beta Y_1$ be the continuous extension of $f$. The continuous mappings $f_\beta\phi_2,\phi_1:\beta X\rightarrow \beta Y_1$ coincide  with $\mbox{id}_X$ on $X$ and thus are identical. Also, since $X$ is dense in $\beta Y_i$, as it is dense in $Y_i$ where $i=1,2$ the space $\beta Y_i$ is a compactification of $X$. Therefore since $f_\beta|X=\mbox{id}_X$, by Theorem 3.5.7 of \cite{E} we have $f_\beta[\beta Y_2\backslash X]=\beta Y_1\backslash X$. Now let $F_2\in {\mathscr F}(Y_2)$. Then $F_2=\phi^{-1}_2(p)$ for some $p\in Y_2\backslash X$. By above $f_\beta(p)\in\beta Y_1\backslash X$ and thus
$f(p)\in Y_1\backslash X$, as $f_\beta(p)=f(p)$. Let $F_1=\phi^{-1}_1(f(p))\in{\mathscr F}(Y_1)$. Then
\begin{eqnarray*}
F_2=\phi^{-1}_2(p)\subseteq\phi^{-1}_2\big[f_\beta^{-1}\big(f_\beta(p)\big)\big]&=&(f_\beta\phi_2)^{-1}\big(f_\beta(p)\big)\\&=&\phi_1^{-1}\big(f_\beta(p)\big)
=\phi_1^{-1}\big(f(p)\big)=F_1.
\end{eqnarray*}

(2) {\em  implies} (1). We define $f:Y_2\rightarrow Y_1$ as follows. If $t\in Y_2\backslash X$ then $\phi^{-1}_2(t)\in{\mathscr F}(Y_2)$ and thus by our assumption  $\phi^{-1}_2(t)\subseteq\phi^{-1}_1(s)$ for some (unique, as  $\phi_2$ is surjective) $s\in  Y_1\backslash X$. Define $f(t)=s$ in this case. If $t\in X$ define $f(t)=t$. We show that $f$ is continuous, this will show that $Y_1\leq Y_2$. By Lemma \ref{j2} the space $\beta Y_2$ is the quotient space of $\beta X$ obtained by contracting each $\phi^{-1}_2(p)$ where $p\in Y_2\backslash X$ to a point and $\phi_2$ is the quotient mapping. Thus in particular $Y_2$ is the quotient space of $X\cup \phi^{-1}_2[Y_2\backslash X]$ with the quotient mapping
\[\phi_2|\big(X\cup\phi^{-1}_2[Y_2\backslash X]\big):X\cup\phi^{-1}_2[Y_2\backslash X]\rightarrow Y_2.\]
Therefore to show that $f$ is continuous it suffices to show that $f\phi_2|(X\cup\phi^{-1}_2[Y_2\backslash X])$ is continuous. We show this by verifying that $f\phi_2(t)=\phi_1(t)$ for any $t\in X\cup\phi^{-1}_2[Y_2\backslash X]$. This obviously holds if $t\in X$. If $t\in \phi^{-1}_2 [Y_2\backslash X]$ then $\phi_2(t)\in Y_2\backslash X$. Let $s\in Y_1\backslash X$ be such that $\phi^{-1}_2(\phi_2(t))\subseteq\phi^{-1}_1(s)$. Then $f\phi_2(t)=s$. But since $t\in\phi^{-1}_2[\phi_2 (t)]$ we have $t\in\phi^{-1}_1(s)$ and thus $\phi_1(t)=s$. Therefore $f\phi_2(t)=\phi_1(t)$ also in this case.
\end{proof}

\begin{lemma}\label{BA27}
Let $X$ be a Tychonoff space and let $\mathcal{P}$ be a clopen hereditary topological property which is inverse invariant under perfect mappings. Suppose that $Z\subseteq C$ where $Z\in {\mathscr Z}(X)$, $C\in Coz(X)$ and $\mbox{\em cl}_XC$ has $\mathcal{P}$. Then $\mbox{\em cl}_{\beta X}Z\subseteq\lambda_{\mathcal{P}} X$.
\end{lemma}

\begin{proof}
The zero--sets $Z$ and $X\backslash C$ of $X$, being disjoint, are completely separated in $X$. Let $f:X\rightarrow\mathbf{I}$ be continuous with $f[Z]\subseteq\{0\}$ and $f[X\backslash C]\subseteq\{1\}$ and let $f_\beta:\beta X\rightarrow\mathbf{I}$ be the continuous extension of $f$. Let $S=f^{-1}[[0,1/2]]\in{\mathscr Z}(X)$. Then $S\subseteq C$ and therefore $S$ has $\mathcal{P}$, as it is closed in $\mbox{cl}_XC$. We have
\[\mbox{cl}_{\beta X}Z\subseteq Z(f_\beta)\subseteq f_\beta^{-1}\big[[0,1/2)\big]\subseteq\mbox{int}_{\beta X}\mbox{cl}_{\beta X}f^{-1}\big[[0,1/2]\big]=\mbox{int}_{\beta X}\mbox{cl}_{\beta X}S\subseteq\lambda_{\mathcal{P}}X.\]
\end{proof}

In the following theorem  we characterize the elements of ${\mathscr O}^{\mathcal Q}_{\mathcal P}(X)$.

\begin{theorem}\label{HG16}
Let ${\mathcal P}$ and  ${\mathcal Q}$ be a pair of compactness--like topological properties. Let $X$ be a Tychonoff space with $\mathcal{Q}$ and let  $Y\in{\mathscr E}^{\mathcal Q}_{\mathcal P}(X)$. The following are equivalent:
\begin{itemize}
\item[\rm(1)] $Y\in{\mathscr O}^{\mathcal Q}_{\mathcal P}(X)$.
\item[\rm(2)] $\phi^{-1}[Y\backslash X]=\beta X\backslash\lambda_{\mathcal P}X$ where $\phi:\beta X\rightarrow\beta Y$ is the continuous extension of $\mbox{\em id}_X$.
\item[\rm(3)] For any $Z\in {\mathscr Z}(X)$ such that $Z\subseteq C$ for some $C\in Coz(X)$ such that $\mbox{\em cl}_X C$ has $\mathcal{P}$ we have $\mbox{\em cl}_Y Z\cap (Y\backslash X)=\emptyset$.
\item[\rm(4)] For any $T\in{\mathscr E}^{\mathcal Q}_{\mathcal P}(X)$ and any continuous injective $f:T\rightarrow Y$ such that $f|X=\mbox{\em id}_X$, the mapping $f$ is a homeomorphism.
\item[\rm(5)] For any $T\in{\mathscr E}^{\mathcal Q}_{\mathcal P}(X)$ if $Y\leq T$ then $T\in{\mathscr M}^{\mathcal Q}_{\mathcal P}(X)$.
\end{itemize}
\end{theorem}

\begin{proof}
Let $\phi:\beta X\rightarrow\beta Y$ be the continuous extension of $\mbox{id}_X$. (1) {\em implies} (2). By Theorem \ref{HUHG16} the set $\phi^{-1}(p)\backslash\lambda_{\mathcal P} X$ is non--empty for any $p\in Y\backslash X$. Let $S$ be the space obtained from $\beta X$ by contracting each $\phi^{-1}(p)\backslash\lambda_{\mathcal P} X$ where $p\in Y\backslash X$ to a point $s_p$ with the quotient mapping $q:\beta X\rightarrow S$. Note that $S$ is a continuous image of $\beta X$. Therefore to prove that $S$ is compact it suffices to show that it is Hausdorff. Suppose that $s,z\in S$ are distinct. Consider the following cases:
\begin{description}
\item[{\sc Case 1.}] Suppose that $s,z\in \lambda_{\mathcal P} X$. Since $s$ and $z$ can be separated in $\lambda_{\mathcal P} X$ by disjoint open subsets and $\lambda_{\mathcal P} X$ is open in $\beta X$ they can be  also separated in $S$.
\item[{\sc Case 2.}] Suppose that $s\in\lambda_{\mathcal P} X$ and $z\in q[\beta X\backslash\lambda_{\mathcal P}X]$. Let $U$ and $V$ be disjoint open neighborhoods of $s$ and $\beta X\backslash\lambda_{\mathcal P} X$ in $\beta X$, respectively. Then $q[U]$ and $q[V]$ are disjoint open subsets of $S$ separating $s$ and $z$.
\item[{\sc Case 3.}] Suppose that $s,z\in q[\beta X\backslash\lambda_{\mathcal P} X]$. Let
    \[s=q[\phi^{-1}(p)\backslash\lambda_{\mathcal P} X]\mbox{ and }z=q[\phi^{-1}(y)\backslash\lambda_{\mathcal P} X]\]
    for some $p,y\in Y\backslash X$. Let $U$ and $V$ be disjoint open neighborhoods of $p$ and $y$ in $\beta Y$, respectively. Then $q[\phi^{-1}[U]]$  and $q[\phi^{-1}[V]]$ are  disjoint open subsets of $S$ separating $s$ and $z$.
\end{description}
Define $f:S\rightarrow\beta Y$ such that $f(x)=p$, if $x\in q[\phi^{-1} (p)]$ for some $p\in Y\backslash X$, and $f(x)=x$ otherwise. Note that this makes sense by the construction of $\beta Y$ and the representation of $\phi$ given in Lemma \ref{j2}. By the definition of $f$ we have $fq=\phi$ and therefore $f$ is continuous. Consider the subspace $T=X\cup q[\beta X\backslash\lambda_{\mathcal P}X]$ of $S$. Note that by Lemma \ref{15} we have $X\subseteq\lambda_{\mathcal P} X$. Thus $T$ contains $X$ as a dense subspace. It is also clear that $X$, and therefore $T$, is dense in $S$. Since $\beta X\backslash\lambda_{\mathcal P} X\subseteq q^{-1}[T\backslash X]$ and $X$ is locally--${\mathcal P}$, as ${\mathscr E}_{\mathcal P}(X)$ is non--empty, by Lemma \ref{16} we have $T\in{\mathscr E}^{\mathcal Q}_{\mathcal P}(X)$. Now $f|T:T\rightarrow Y$ is a continuous bijective mapping such that $f|X=\mbox{id}_X$. Thus by the maximality of the topology of $Y$ we have $T=Y$ (identifying each $s_p$ with $p$ where $p\in Y\backslash X$). Now $S$ is a compactification of $Y$ and therefore there exists a continuous $g:\beta Y\rightarrow S$ such that $g|Y=\mbox{id}_Y$. Since $fg|Y=\mbox{id}_Y$ it follows that $fg=\mbox{id}_{\beta Y}$. On the other hand $gf|Y=\mbox{id}_Y$ which yields $gf=\mbox{id}_S$ and thus $g=f^{-1}$. Now as noted before $fq=\phi$, and therefore for any $p\in Y\backslash  X$ we have
\[\phi^{-1}(p)=(fq)^{-1}(p)=q^{-1}\big[f^{-1}(p)\big]=q^{-1}\big(g(p)\big)=q^{-1}(p)=\phi^{-1}(p)\backslash\lambda_{\mathcal P} X.\]
Thus $\phi^{-1}(p)\subseteq\beta X\backslash\lambda_{\mathcal P} X$ for any $p\in Y\backslash X$ and therefore $\phi^{-1}[Y\backslash X]\subseteq\beta X\backslash\lambda_{\mathcal P}X$. But by Lemma \ref{16} we have $\beta X\backslash\lambda_{\mathcal P} X\subseteq\phi^{-1}[Y\backslash X]$ which shows the equality in the latter.

(2) {\em implies} (4). Let $T\in{\mathscr E}^{\mathcal Q}_{\mathcal P}(X)$ and let $f:T\rightarrow Y$ be a continuous injective mapping which fixes $X$ pointwise. Let $f_\beta:\beta T\rightarrow\beta Y$ and $\psi:\beta X\rightarrow\beta T$ be the continuous extensions of $f$ and $\mbox{id}_X$, respectively. Since $f_\beta\psi|X=\mbox{id}_X=\phi|X$ we have $f_\beta\psi=\phi$. Also $f[T\backslash X]\subseteq Y\backslash X$. To show the latter suppose to the contrary that $f(t)\in X$ for some $t\in T\backslash X$. Let $U$ and $V$ be disjoint open neighborhoods of $f(t)$ and $t$ in $T$, respectively. Since $Y\backslash X$ is compact, $X$ is open in $Y$ and thus $U\cap X$, being open in $X$, is an open neighborhood of $f(t)$ in $Y$. Let $W$ be an open neighborhood of $t$ in $T$ such that $f[W]\subseteq U\cap X$. Since $W\cap V$ is open in $T$ and it is non--empty, as $t\in W\cap V$ and $X$ is dense in $T$, the set $W\cap V\cap X$ is non--empty. But if $x\in W\cap V\cap X$ then $x=f(x)\in U$, which is a contradiction, as $U\cap V=\emptyset$.

\begin{xclaim}
If $t\in T\backslash X$ and $y=f(t)$ then $\psi^{-1}(t)\subseteq\phi^{-1}(y)$.
\end{xclaim}

\subsubsection*{Proof of the claim} We have $y=f(t)=f_\beta(t)$ and  thus $t\in f_\beta^{-1}(y)$. Therefore
\[\psi^{-1}(t)\subseteq\psi^{-1}\big[f_\beta^{-1}(y)\big]=(f_\beta\psi)^{-1}(y)=\phi^{-1}(y).\]

\begin{xclaim}
If $t\in T\backslash X$ and $y=f(t)$ then $\psi^{-1}(t)=\phi^{-1}(y)$.
\end{xclaim}

\subsubsection*{Proof of the claim} By the first claim $\psi^{-1}(t)\subseteq\phi^{-1}(y)$. Let $z\in \phi^{-1}(y)$. By Lemma \ref{16} we have  $\beta X\backslash\lambda_{\mathcal P}X\subseteq\psi^{-1}[T\backslash X]$. Thus since
\[\phi^{-1}(y)\subseteq\phi^{-1}[Y\backslash X]=\beta X\backslash\lambda_{\mathcal P} X\]
we have $z\in\psi^{-1}[T\backslash X]$. Let $\psi(z)=t'\in T\backslash X$ and let $y'=f(t')\in Y\backslash X$. By the first claim $\psi^{-1}(t')\subseteq\phi^{-1}(y')$ and therefore $z\in\phi^{-1}(y')$. Thus $\phi^{-1}(y)\cap\phi^{-1}(y')$ is non--empty and $f(t)=y=y'=f(t')$. But $f$ is injective and  therefore $t=t'$ which yields $z\in \psi^{-1}(t')=\psi^{-1}(t)$. This shows that $\phi^{-1}(y)\subseteq\psi^{-1}(t)$ which together with above proves the claim.

\begin{xclaim}
$\{\psi^{-1}(t):t\in T\backslash X\}=\{\phi^{-1}(y):y\in Y\backslash X\}$.
\end{xclaim}

\subsubsection*{Proof of the claim} By the second claim it suffices to show that for any $y\in Y\backslash X$ we have  $\phi^{-1}(y)=\psi^{-1}(t)$ for some $t\in T\backslash X$. Let $y\in Y\backslash X$ and $z\in \beta X$ be such that $\phi(z)=y$. By Lemma \ref{16} we have $\beta X\backslash\lambda_{\mathcal P} X\subseteq\psi^{-1}[T\backslash X]$ and thus, since
\[\phi^{-1}(y)\subseteq\phi^{-1}[Y\backslash X]=\beta X\backslash\lambda_{\mathcal P} X\]
it follows that $z\in\psi^{-1}[T\backslash X]$. Let $t=\psi(z)\in T\backslash X$. Then $z\in\phi^{-1}(y)\cap\psi^{-1}(t)$. If $y'=f(t)$ then by the second claim $\psi^{-1}(t)=\phi^{-1}(y')$. Thus $\phi^{-1}(y)\cap\phi^{-1}(y')$ is non--empty and therefore $y=y'$. Thus $\phi^{-1}(y)=\phi^{-1}(y')=\psi^{-1}(t)$ which proves the claim.

\medskip

\noindent By Lemma \ref{j2} the spaces  $\beta Y$ and $\beta T$ are respectively obtained from $\beta X$ by contracting the sets $\phi^{-1}(y)$ where $y\in Y\backslash X$ and $\psi^{-1}(t)$ where $t\in T\backslash X$ to points and $\phi$ and $\psi$ are their corresponding quotient mappings. Thus by the third claim $\phi=\psi$ and therefore
\[Y=X\cup\phi\Big[\bigcup\big\{\phi^{-1}(y):y\in Y\backslash X\big\}\Big]= X\cup\psi\Big[\bigcup\big\{\psi^{-1}(t):t\in T\backslash X\big\}\Big]=T.\]
This shows $Y$ and $T$ are equivalent extensions of $X$. Let $g:Y\rightarrow T$ be a homeomorphism such that $g|X=\mbox{id}_X$. Then the continuous mapping $fg:Y\rightarrow Y$ coincides with $\mbox{id}_Y$ on the dense subset $X$ of $Y$. This (since $Y$ is Hausdorff) implies that $fg=\mbox{id}_Y$ and thus $f=g^{-1}$ is a homeomorphism.

(4) {\em implies} (1). Let $Y'$ be the set $Y$ equipped with a topology which is finer than the topology of $Y$ and turns it into an element of ${\mathscr E}_{\mathcal P}(X)={\mathscr E}^{\mathcal Q}_{\mathcal P}(X)$. Since $f:Y'\rightarrow Y$ defined by $f(y')=y'$ for any $y'\in Y'$ is continuous and injective, by our assumption it is a homeomorphism. This shows that the topology of $Y$ is maximal among the topologies on the set $Y$ which turn $Y$ into an element of ${\mathscr E}_{\mathcal P}(X)$. Next, suppose that $T\in{\mathscr E}_{\mathcal P}(X)$ is such that $T\subseteq Y$. Since the inclusion mapping $f:T\rightarrow Y$ defined by $f(t)=t$ for any $t\in T$ is continuous and injective, by our assumption it is a homeomorphism. But this implies that $T=Y$ which proves the minimality of $Y$ among the elements of ${\mathscr E}_{\mathcal P}(X)$.

(2) {\em implies} (3). Let $Z\in{\mathscr Z}(X)$ be  such that $Z\subseteq C$ for some $C\in Coz(X)$ such that $\mbox{cl}_X C$ has ${\mathcal P}$. By Lemma \ref{BA27} we have  $\mbox{cl}_{\beta X}Z\subseteq\lambda_{\mathcal P} X$. Let $U$ be an open neighborhood of $\beta X\backslash\lambda_{\mathcal P} X$ in $\beta X$ which misses $\mbox{cl}_{\beta X}Z$. Now (by the construction of $\beta Y$ and the representation of $\phi$ given in Lemma \ref{j2}) the set $\phi[U]$ is an open neighborhood of $p\in Y\backslash X$ in $\beta Y$ which misses $Z$. Therefore
\[\mbox{cl}_Y Z\cap(Y\backslash X)=\mbox{cl}_{\beta Y} Z\cap(Y\backslash X)=\emptyset.\]

(3) {\em implies} (2).  By Lemma \ref{16} we have $\beta X\backslash\lambda_{\mathcal{P}} X\subseteq\phi^{-1}[Y\backslash X]$. To show the reverse inclusion suppose to the contrary that $t\in \lambda_{\mathcal{P}} X$ for some $t\in \phi^{-1}[Y\backslash X]$. Note that $\lambda_{\mathcal{P}} X$ is open in $\beta X$. Let $U$ be an open neighborhood of $t$ in $\beta X$ such that $\mbox{cl}_{\beta X} U\subseteq\lambda_{\mathcal{P}} X$. Let $f:\beta X\rightarrow \mathbf{I}$ be continuous with $f(t)=0$ and $f[\beta X\backslash U]\subseteq\{1\}$. Define
\[Z=f^{-1}\big[[0,1/3]\big]\cap X\mbox{ and }C=f^{-1}\big[[0,1/2)\big]\cap X.\]
Then $Z\in {\mathscr Z}(X)$, $C\in Coz(X)$ and $Z\subseteq C$. Also, since $\mbox{cl}_{\beta X}(U\cap X)=\mbox{cl}_{\beta X}U\subseteq\lambda_{\mathcal{P}} X$, by Lemma \ref{B} the set $\mbox{cl}_X(U\cap X)$ has $\mathcal{P}$. Therefore, since
\[C=f^{-1}\big[[0,1/2)\big]\cap X\subseteq U\cap X,\]
the set $\mbox{cl}_X C$ has $\mathcal{P}$, as it is closed in $\mbox{cl}_X(U\cap X)$. By our assumption this implies that  $\mbox{cl}_Y Z\cap (Y\backslash X)=\emptyset$. But $\phi(t)\in Y\backslash X$  and thus  $\phi(t)\notin\mbox{cl}_{\beta Y}Z$.  Let $V$ be an open neighborhood of $\phi(t)$ in $\beta Y$ such that $V\cap Z=\emptyset$.  Now since
\[\phi^{-1}[V]\cap \phi^{-1}[Z]=\phi^{-1}[V\cap Z]=\emptyset\]
it follows that
\[\phi^{-1}[V]\cap\mbox{cl}_{\beta X}Z\subseteq\phi^{-1}[V]\cap\mbox{cl}_{\beta X}\phi^{-1}[Z]=\emptyset.\]
Thus $t\notin \mbox{cl}_{\beta X}Z$, which is a contradiction, as
\begin{eqnarray*}
t\in f^{-1}\big[[0,1/3)\big]\subseteq\mbox{cl}_{\beta X}\big(f^{-1}\big[[0,1/3]\big]\cap X\big)=\mbox{cl}_{\beta X}Z.
\end{eqnarray*}

(2) {\em  implies} (5). Let $T\in{\mathscr E}^{\mathcal Q}_{\mathcal P}(X)$ be such that $Y\leq T$. Let $f:T\rightarrow Y$ continuously extends $\mbox{id}_X$. Arguing as in (2) $\Rightarrow$ (4) we have $f[T\backslash X]\subseteq Y\backslash X$. Let $f_\beta:\beta T\rightarrow\beta Y$ and $\psi:\beta X\rightarrow\beta T$ be the continuous extensions of $f$ and $\mbox{id}_X$, respectively. Then $\phi=f_\beta\psi$, as they coincide with $\mbox{id}_X$ on $X$. To show that $T\in{\mathscr M}^{\mathcal Q}_{\mathcal P}(X)$, by Theorem \ref{HUHG16}, it suffices to verify that $\psi^{-1}(p)\backslash\lambda_{{\mathcal P}}X$ is non--empty for any $p\in T\backslash X$. Let $p\in T\backslash X$. Then
\begin{eqnarray*}
\psi^{-1}(p)\subseteq\psi^{-1}\big[f_\beta^{-1}\big(f_\beta(p)\big)\big]&=&(f_\beta\psi)^{-1}\big(f_\beta(p)\big)\\&=&\phi^{-1}\big(f_\beta(p)\big)
=\phi^{-1}\big(f(p)\big)\subseteq\phi^{-1}[Y\backslash X].
\end{eqnarray*}
Now since  $\phi^{-1}[Y\backslash X]=\beta X\backslash\lambda_{{\mathcal P}}X$ and $\psi$ is surjective, $\psi^{-1}(p)\backslash\lambda_{{\mathcal P}}X=\psi^{-1}(p)$ is non--empty.

(5) {\em  implies} (2). Note that (5) in particular implies that $Y\in{\mathscr M}_{\mathcal P}(X)$. Thus by  Theorem \ref{HUHG16} the set $\phi^{-1}(p)\backslash\lambda_{{\mathcal P}}X$ is non--empty for any $p\in Y\backslash X$. Since by Lemma \ref{16} we have $\beta X\backslash\lambda_{{\mathcal P}}X\subseteq\phi^{-1}[Y\backslash X]$ and $X$ is locally--${\mathcal P}$, to show (2) it suffices to verify that $\phi^{-1}[Y\backslash X]\subseteq\beta X\backslash\lambda_{{\mathcal P}}X$. Suppose to the contrary that  $\phi^{-1}(p')\cap\lambda_{{\mathcal P}}X$ is non--empty for some $p'\in Y\backslash X$. Let $t'\in \phi^{-1}(p')\cap\lambda_{{\mathcal P}}X$. Let $Z$ be the quotient space of $\beta X$ obtained by contracting each (non--empty) subset $\phi^{-1} (p)\backslash\lambda_{{\mathcal P}}X$ where $p\in Y\backslash X$ to a point $z_p$ with the quotient mapping $q:\beta X\rightarrow Z$. Then as in (1) $\Rightarrow$ (2) one can verify that $Z$ is compact. Consider the subspace
\[T=q\big[X\cup(\beta X\backslash\lambda_{{\mathcal P}}X)\cup\{t'\}\big]\]
of $Z$. Then $T$ is a Tychonoff extension of $X$ with the compact remainder
\[T\backslash X=q\big[(\beta X\backslash\lambda_{{\mathcal P}}X)\cup\{t'\}\big].\]
Note that $T$ is dense in $Z$ and therefore $Z$ is a compactification of $T$. Let $f:\beta T\rightarrow Z$ and $\psi:\beta X\rightarrow\beta T$ be the continuous extensions of $\mbox{id}_T$ and $\mbox{id}_X$, respectively. Since $f\psi:\beta X\rightarrow Z$ agrees with $q$ on $X$  we have $f\psi=q$. By  Lemma \ref{16} and since  $\beta X\backslash\lambda_{{\mathcal P}}X\subseteq q^{-1}[T\backslash X]$ (and $X$ is locally--${\mathcal P}$) it follows that $T\in{\mathscr E}^{\mathcal Q}_{\mathcal P}(X)$. We verify that $Y\leq T$, our assumption will then imply that $T\in{\mathscr M}^{\mathcal Q}_{\mathcal P}(X)$ from which we will derive a contradiction. By Lemma \ref{DFH} to show that $Y\leq T$ it suffices to verify that each $\psi^{-1}(t)$ where $t\in T\backslash X$ is contained in  $\phi^{-1}(p)$ for some $p\in Y\backslash X$. Let $t\in T\backslash X$. Note that by Theorem 3.5.7 of \cite{E} (and since $f|T=\mbox{id}_T$ and $Z$ is a compactification of $T$) we have $f[\beta T\backslash T]=Z\backslash T$ and thus $f^{-1}(t)=\{t\}$. Therefore
\[\psi^{-1}(t)=\psi^{-1}\big[f^{-1}(t)\big]=(f\psi)^{-1}(t)=q^{-1}(t)\]
and thus by the definition of $Z$ it follows that $\psi^{-1}(t)\subseteq\phi^{-1}(p)$ for some $p\in Y\backslash X$. This shows that
$Y\leq T$. Consider the subspace $T'=T\backslash\{t'\}$ of $T$. Then $T'$ is a Tychonoff extension of $X$ with the compact remainder $T'\backslash X=q[\beta X\backslash\lambda_{{\mathcal P}}X]$. By Lemma \ref{16}  and since  $\beta X\backslash\lambda_{{\mathcal P}}X\subseteq q^{-1}[T'\backslash X]$ (and $X$ is locally--${\mathcal P}$) it follows that $T'\in{\mathscr E}^{\mathcal Q}_{\mathcal P}(X)$. But this contradicts the minimality of $T$, as $T'$ is properly contained in $T$. \end{proof}

\begin{xrem}
{\em Theorem \ref{HG16} fails if one omits the requirement ``$\mathcal{P}$ satisfies Mr\'{o}wka's condition $(\mbox{W})$" (implicit in the definition of the compactness--like topological property $\mathcal{P}$) from its statement (see Example \ref{20UIHG} below).}
\end{xrem}

In the following we provide examples of pairs ${\mathcal P}$ and  ${\mathcal Q}$ of compactness--like topological properties. The topological properties ${\mathcal P}$ and ${\mathcal Q}$ assumed in the statement of Lemma \ref{16} (and thus in the statements of all its corollaries which constitute the main results of this article) are required to be a pair of compactness--like topological properties.

\begin{example}\label{20UIHG}
Let $\alpha$, $\theta$, $\kappa$ and $\mu$ be infinite cardinals and let $X$ be a Hausdorff space. For a collection ${\mathscr A}$ of subsets of $X$ and an $x\in X$ let
\[O(x,{\mathscr A})=\mbox{card}\big(\{A\in{\mathscr A}:x\in A\}\big).\]
For more details on the following definitions see \cite{Bu}, \cite{Steph} and \cite{Va}. The space $X$ is called (2) {\em $\mu$--Lindel\"{o}f} ((3) {\em $[\theta,\kappa]$--compact}, respectively) if every open cover of $X$ (of cardinality $\leq\kappa$, respectively) has a subcover of cardinality $\leq\mu$ ($<\theta$, respectively). The space $X$ is called (4) {\em paracompact} (5) {\em metacompact} (7) {\em subparacompact} (11) {\em para--Lindel\"{o}f} (12) {\em meta--Lindel\"{o}f} (14) {\em screenable} (15) {\em $\sigma$--metacompact} (9) {\em $\sigma$--para--Lindel\"{o}f} if every open cover of $X$ has a (4)$'$ locally finite open (5)$'$ point--finite open (7)$'$ $\sigma$--locally finite closed (11)$'$ locally countable open (12)$'$ point--countable open (14)$'$ $\sigma$--disjoint open (15)$'$ $\sigma$--point--finite open (9)$'$ $\sigma$--locally countable open refinement. The space $X$ is called (16) {\em weakly $\theta$--refinable} (8) {\em $\theta$--refinable} (or {\em submetacompact}) (17) {\em weakly $\delta\theta$--refinable} (13) {\em $\delta\theta$--refinable} (or {\em submeta--Lindel\"{o}f}) if every open cover of $X$ has an open refinement ${\mathscr V}=\bigcup\{{\mathscr V}_n:n\in\mathbf{N}\}$ such that for any $x\in X$ there exists some $n\in\mathbf{N}$ with (16)$'$ $0<O(x,{\mathscr V}_n)<\aleph_0$ (8)$'$ $0<O(x,{\mathscr V}_n)<\aleph_0$ and each ${\mathscr V}_n$ covers $X$ (17)$'$ $0<O(x,{\mathscr V}_n)\leq\aleph_0$ (13)$'$ $0<O(x,{\mathscr V}_n)\leq\aleph_0$ and each ${\mathscr V}_n$ covers $X$. The space $X$ is called (10) {\em $\alpha$--bounded} if any subset of $X$ of cardinality $\leq\alpha$ has compact closure in $X$. Moreover,
let (1) be compactness and (6) be countable paracompactness.

Let ${\mathcal P}=\mbox{regularity}+(i)$ where $i=1,\ldots,10$ and ${\mathcal Q}=\mbox{regularity}+(i)$ where $i=1,\ldots,17$. Then ${\mathcal P}$ and  ${\mathcal Q}$ is a pair of compactness--like topological properties. That ${\mathcal Q}$ is hereditary with respect to clopen subsets follows from Theorem 7.1 of \cite{Bu}. Also, by (modification of) Theorem 3.7.24 and Exercise 5.2.G of \cite{E} and Theorem 5.9 of \cite{Bu}
it follows that $\mathcal{Q}$ is inverse invariant under perfect mappings. (For the case of $\alpha$--boundedness note that for a perfect surjective $f:X\rightarrow Y$, when $Y$ is $\alpha$--bounded, if $A\subseteq X$ has cardinality $\leq\alpha$ then $\mbox{card}(f[A])\leq\alpha$ and thus $\mbox{cl}_Y f[A]$ is compact. But since
\[A\subseteq f^{-1}\big[f[A]\big]\subseteq f^{-1}\big[\mbox{cl}_Y f[A]\big]\]
and the latter is compact (as $f$ is perfect), its closed subset $\mbox{cl}_X A$ also is compact, that is, $X$ is $\alpha$--bounded.) Next, we verify that $\mathcal{Q}$ satisfies Mr\'{o}wka's condition $(\mbox{W})$. We prove this for the cases when $\mathcal{Q}$ is paracompactness and subparacompactness. The remaining cases can be proved analogously.

Let $\mathcal{Q}$ be paracompactness (subparacompactness, respectively). Let $X$ be a Tychonoff space, let $p\in X$ and let ${\mathscr B}$ an open base for $X$  at $p$ such that $X\backslash B$ is paracompact (subparacompact, respectively) for any $B\in{\mathscr B}$. Let ${\mathscr U}$ be an open cover of $X$. Let $p\in B\subseteq\mbox{cl}_X B\subseteq U$ where $U\in {\mathscr U}$ and $B\in {\mathscr B}$. Then ${\mathscr V}=\{V\backslash B:V\in {\mathscr U}\}$ is an open cover of $X\backslash B$. Thus there exists a locally finite open (in $X\backslash B$) refinement ${\mathscr W}$ of ${\mathscr V}$ (a $\sigma$--locally finite closed (in $X\backslash B$) refinement ${\mathscr W}$ of ${\mathscr V}$, respectively). Now if
\[{\mathscr A}=\{X\backslash\mbox{cl}_X B:W\in {\mathscr W}\}\cup\{U\}\]
(${\mathscr A}={\mathscr W}\cup\{\mbox{cl}_X B\}$, respectively) then ${\mathscr A}$ is a locally finite open refinement of ${\mathscr U}$ (${\mathscr A}$ is a $\sigma$--locally finite closed refinement of ${\mathscr U}$, respectively). Thus $X$ is paracompact (subparacompact, respectively).

Note that $\alpha$--boundedness satisfies Mr\'{o}wka's condition $(\mbox{W})$ by Theorem 3.1 of \cite{MRW}. That ${\mathcal P}$ is finitely additive and invariant under perfect mappings follow from Theorems 5.1, 5.5, 7.3 and 7.4 of \cite{Bu} and Exercises 5.2.B and 5.2.G of \cite{E}. Theorem 3.1 of \cite{MRW} provides a few more examples of topological properties satisfying Mr\'{o}wka's condition $(\mbox{W})$. Among them we mention of realcompactness and Dieudonn\'{e} completeness which are hereditary with respect to clopen subsets and inverse invariant under perfect mappings (with Tychonoff domains); see Theorems 3.11.4 and 3.11.14 and Problem 8.5.13 of \cite{E}. That Dieudonn\'{e} completeness is inverse invariant under perfect mappings is well known, however, it can be proved by using the fact that a Tychonoff space $X$ is Dieudonn\'{e} complete if and only if for any $p\in\beta X\backslash X$ there exists a paracompact subset $T$ of $\beta X$ such that $X\subseteq T\subseteq\beta X\backslash\{p\}$ (see Problem 8.5.13 of \cite{E}) and that paracompactness is inverse invariant under perfect mappings.

In addition to the above topological properties the list of topological properties satisfying Mr\'{o}wka's condition $(\mbox{W})$ includes: screenability, $N$--compactness \cite{M1}, almost realcompactness \cite{F} and zero--dimensionality (see \cite{MRW} and \cite{MRW1} for details).
\end{example}

In the following we give an example of a topological property ${\mathcal P}$ which does not satisfy Mr\'{o}wka's condition $(\mbox{W})$. At the same time we show that the requirement ``${\mathcal P}$ satisfies Mr\'{o}wka's condition $(\mbox{W})$" (implicit in the definition of the compactness--like topological property $\mathcal{P}$) cannot be omitted from the statements of Lemma \ref{16} and Theorems \ref{HUHG16} and \ref{HG16} upon them the rest of this article rely.

\begin{example}\label{HGES}
Let $X$ be a locally compact paracompact non--$\sigma$--compact space. Then $X$ can be represented as
\begin{equation}\label{LAUR}
X=\bigoplus_{i\in I}X_i
\end{equation}
for some indexed set $I$, where each $X_i$ where $i\in I$ is $\sigma$--compact and non--compact (see Theorem 5.1.27 and Exercise 3.8.C of \cite{E}). Assume the representation given in (\ref{LAUR}). Let ${\mathcal P}$ be $\sigma$--compactness. Obviously, ${\mathcal P}$ is clopen hereditary, finitely additive  and perfect. We show that ${\mathcal P}$ does not satisfy Mr\'{o}wka's condition $(\mbox{W})$. Note that with the above notation
\[\lambda_{{\mathcal P}}X=\bigcup\Big\{\mbox{cl}_{\beta X}\Big(\bigcup_{i\in J}X_i\Big):J\subseteq I\mbox{ is countable}\Big\}.\]
Also, note that since $X$ is non--$\sigma$--compact,
$\beta X\backslash\lambda_{{\mathcal P}}X$ is non--empty. Contract the compact subset $\beta X\backslash\lambda_{{\mathcal P}}X$ of $\beta X$ to a point $p$ to obtain a space $T$ and denote by $q:\beta X\rightarrow T$ its quotient mapping. Note that $T$ is compact (as it is a Hausdorff continuous image of $\beta X$) and contains $X$ as a dense subspace. Consider the subspace $Y=X\cup\{p\}$ of $T$. We show that for any open neighborhood $V$ of $p$ in $Y$ the set $Y\backslash V$ is $\sigma$--compact while $Y$ itself is not $\sigma$--compact. Let $V$ be an open neighborhood of $p$ in $Y$. Let $V'$ be an open subset of $T$ such that $V'\cap Y=V$. Then since $p\in V'$ we have
\[\beta X\backslash\lambda_{{\mathcal P}}X=q^{-1}(p)\subseteq q^{-1}[V']\]
and thus $\beta X\backslash q^{-1}[V']\subseteq\lambda_{{\mathcal P}}X$. Therefore by compactness
\[\beta X\backslash q^{-1}[V']\subseteq\mbox{cl}_{\beta X}\Big(\bigcup_{i\in J_1} X_i\Big)\cup\cdots\cup\mbox{cl}_{\beta X}\Big(\bigcup_{i\in J_m} X_i\Big)=\mbox{cl}_{\beta X}\Big(\bigcup_{i\in J} X_i\Big)\]
where $m\in\mathbf{N}$, each $J_1,\ldots,J_m\subseteq I$ is countable and $J=J_1\cup\cdots\cup J_m$. Now
\[Y\backslash V=\big(\beta X\backslash q^{-1}[V']\big)\cap X\subseteq\bigcup_{i\in J} X_i\]
being closed in the latter ($\sigma$--compact) set is $\sigma$--compact. To show that $Y$ is not $\sigma$--compact suppose the contrary and let $Y=\bigcup_{n=1}^\infty K_n$ where $K_n$ is compact for any $n\in\mathbf{N}$. Let $p\in K_j$ where $j\in\mathbf{N}$. Then
\[\beta X\backslash\lambda_{{\mathcal P}}X=q^{-1}(p)\subseteq q^{-1}[K_j]\]
and thus $K_n=q^{-1}[K_n]\subseteq\lambda_{{\mathcal P}}X$ for any $j\neq n\in\mathbf{N}$. Arguing as above for any $j\neq n\in\mathbf{N}$ we have
\[K_n\subseteq\mbox{cl}_{\beta X}\Big(\bigcup_{i\in H_n} X_i\Big)\]
where $H_n\subseteq I$ is countable, but (since $K_n\subseteq X$) this  implies that
$K_n\subseteq\bigcup_{i\in H_n} X_i$. Let $H=\bigcup_{j\neq n=1}^\infty H_n$. Then
\begin{equation}\label{JJUY}
\bigcup_{j\neq n=1}^\infty K_n\subseteq\bigcup_{i\in H} X_i.
\end{equation}
Choose some $u\in I\backslash H$. (This is possible, as $H$ is countable and  $I$ is uncountable, because by our assumption $X$ is non--$\sigma$--compact and $X_i$'s are $\sigma$--compact.) Since by our assumption $X_u$ is non--compact, $\mbox{cl}_{\beta X}X_u\backslash X_u$ is non--empty. Let $t\in \mbox{cl}_{\beta X}X_u\backslash X_u\subseteq\lambda_{{\mathcal P}}X$. Then $t\in T\backslash Y$. We show that $t\in\mbox{cl}_T K_j$, contradicting the compactness of $K_j$. Let $W$ be an open neighborhood of $t=q(t)$ in $T$. Then $q^{-1}[W]$ is an open neighborhood of $t$ in $\beta X$ and therefore $X_u\cap q^{-1}[W]$ is non--empty. Let $x\in X_u\cap q^{-1}[W]$. Note that
\begin{eqnarray*}
X\cup(\beta X\backslash\lambda_{{\mathcal P}}X)=q^{-1}[Y]&=&q^{-1}\Big[\bigcup_{n=1}^\infty K_n\Big]\\&=&q^{-1}\Big[\bigcup_{j\neq n=1}^\infty K_n\Big]\cup q^{-1}[K_j]=\bigcup_{j\neq n=1}^\infty K_n\cup q^{-1}[K_j].
\end{eqnarray*}
By (\ref{JJUY}) and since $X_u\cap\bigcup_{i\in H} X_i=\emptyset$ we have $x\notin \bigcup_{j\neq n=1}^\infty K_n$ and thus by the above
$x\in q^{-1}[K_j]$. Therefore $x=q(x)\in W\cap K_j$ and thus $W\cap K_j$ is non--empty. This shows that $t\in \mbox{cl}_T K_j$. Therefore $Y$ is not $\sigma$--compact. Thus for this specific choice of ${\mathcal P}$ Mr\'{o}wka's condition $(\mbox{W})$ fails. By Lemma \ref{j2} we have $T=\beta Y$ and if $\phi:\beta X\rightarrow\beta Y$ is the continuous extension of $\mbox{id}_X$ then $\phi=q$. Also, by the above $Y$ does not have ${\mathcal P}$ while $\phi^{-1}[Y\backslash X]=\beta
X\backslash\lambda_{{\mathcal P}}X$. Therefore in the statements of Lemma \ref{16} and Theorems \ref{HUHG16} and \ref{HG16} the requirement ``${\mathcal P}$ satisfies Mr\'{o}wka's condition $(\mbox{W})$" cannot be omitted.
\end{example}

\section{Compactification--like $\mathcal{P}$--extensions with  countable remainder}

It is a well known result of P. Alexandroff that every locally compact non--compact  space has a compactification with one--point remainder, called the {\em one--point compactification} or the {\em Alexandroff  compactification} of $X$. One can consider  $\mathcal{P}$--extensions with one--point remainder (see \cite{HJW}, \cite{Ko1}, \cite{Ko2}, \cite{Ko3} and \cite{Ko4} for some recent results) or more generally, $\mathcal{P}$--extensions with countable remainder for various topological properties $\mathcal{P}$. Below, after some definitions we sate some known results which motivated our study in this chapter.

\begin{definition}\label{GHF}
Let $n\in\mathbf{N}$. An extension with $n$--point (countable, respectively) remainder  is called an {\em $n$--point} (a {\em countable--point}, respectively) {\em extension}. Similar definitions apply for compactifications.
\end{definition}

countable--point compactifications are also called {\em $\aleph_0$--point compactifications} or {\em countable compactifications}. Throughout this article countable means countable and infinite.

\begin{notation}\label{KLG}
For a Tychonoff space $X$ the set of all compactifications of $X$ is denoted by ${\mathscr K}(X)$.
\end{notation}

In \cite{Mag1} K.D. Magill, Jr. gave the following characterization of those spaces which have an $n$--point compactification and thus generalized the well known result of P. Alexandroff.

\begin{theorem}[Magill \cite{Mag1}]\label{i7}
Let $X$ be a locally compact space and let $n\in\mathbf{N}$. The following are equivalent:
\begin{itemize}
\item[\rm(1)] ${\mathscr K}(X)$ contains an element with $n$--point remainder.
\item[\rm(2)] $X=K\cup U_1\cup\cdots\cup U_n$, where $K,U_1,\ldots,U_n$ are pairwise disjoint, each $U_1,\ldots,U_n$ is open in $X$ such that $K\cup U_i$ is non--compact for any $i=1,\ldots,n$.
\end{itemize}
\end{theorem}

Also, in a separate article \cite{Mag2}, K.D. Magill, Jr. characterized those spaces having a countable--point compactification with compact remainder. Recall that a space $X$ is called {\em totally disconnected} if the (connected) components in $X$ are the one--point sets.

\begin{theorem}[Magill \cite{Mag2}]\label{i0}
Let $X$ be a locally compact space. The following are equivalent:
\begin{itemize}
\item[\rm(1)] ${\mathscr K}(X)$ contains an element with countable remainder.
\item[\rm(2)] ${\mathscr K}(X)$ contains an element with  $n$--point remainder for any $n\in\mathbf{N}$.
\item[\rm(3)] ${\mathscr K}(X)$ contains an element with infinite totally disconnected remainder.
\item[\rm(4)] $\beta X\backslash X$ has an infinite number of (connected) components.
\end{itemize}
\end{theorem}

K.D. Magill, Jr.'s studies were continued by various authors (see e.g. \cite{C}, \cite{Waj} and \cite{Ki} among others). In \cite{Ki} T. Kimura generalized K.D. Magill, Jr.'s result in \cite{Mag1} (Theorem \ref{i7}) and gave the following  characterization of spaces having a countable--point compactification with compact remainder.

\begin{theorem}[Kimura \cite{Ki}]\label{i1}
Let $X$ be a locally compact space. The following are equivalent:
\begin{itemize}
\item[\rm(1)] ${\mathscr K}(X)$ contains an element with countable remainder.
\item[\rm(2)] There exists a bijectively indexed collection $\{U_n: n\in\mathbf{N}\}$ of pairwise disjoint open subsets of $X$ with compact boundary and non--compact closure.
\end{itemize}
\end{theorem}

In \cite{Mc} J.R. McCartney generalized  K.D. Magill, Jr.'s result still further. To state J.R. McCartney's result however, we need some preliminaries.

Let $A$ be an infinite compact countable  space.  As in \cite{MS} we define the successive derived sets $A^{(\zeta)}$ of $A$ for any $\zeta<\Omega$ by $A^{(0)}=A$,  $A^{(\zeta+1)}=(A^{(\zeta)})'$ and
\[A^{(\zeta)}=\bigcap\{A^{(\eta)}:\eta<\zeta\}\]
whenever $\zeta$ is a limit ordinal. Then $A^{(\zeta)}$'s form a decreasing sequence of compact subsets of $A$. Note that if for some  $\zeta<\Omega$ the set $A^{(\zeta)}$ is infinite then $A^{(\zeta)}\backslash A^{(\zeta+1)}$ also is infinite, as otherwise, since
\[A^{(\zeta)}=(A^{(\zeta)}\backslash A^{(\zeta+1)})\cup A^{(\zeta+1)}\]
the set $A^{(\zeta+1)}$ will be a compact non--empty space without isolated points and therefore $A^{(\zeta+1)}\subseteq A$ will be uncountable. Since
\[\bigcup\{A^{(\zeta)}\backslash A^{(\zeta+1)}:\zeta<\Omega\}\subseteq A\]
and $A$ is countable there exists some $\lambda<\Omega$ such that $A^{(\lambda)}\backslash A^{(\lambda+1)}=\emptyset$. Suppose that $\lambda$ is the least with this property. Then by above $A^{(\lambda)}$ is finite and thus it is empty. Note that $\lambda$ is not a limit ordinal, as otherwise, by definition $A^{(\lambda)}$ is non--empty, as it is the intersection of a collection of compact non--empty subsets with the finite intersection property. Let $\lambda=\sigma+1$. Then $A^{(\sigma)}$ is non--empty and since $(A^{(\sigma)})'= A^{(\lambda)}=\emptyset$, it is finite.  We say that an infinite compact countable space  $A$ is {\em of type $(\sigma,n)$}, where  $\sigma<\Omega$ and $n\in\mathbf{N}$, if $\mbox{card}(A^{(\sigma)})=n$. From the above discussion it is clear that the type of $A$ exists and is uniquely determined.

\begin{theorem}[McCartney \cite{Mc}]\label{i2}
Let $X$ be a locally compact space and let $0<\sigma<\Omega$. The following are equivalent:
\begin{itemize}
\item[\rm(1)] ${\mathscr K}(X)$ contains an element with countable remainder of type $(\sigma,1)$.
\item[\rm(2)] There exists a family $\{{\mathscr U}_\zeta:\zeta<\sigma\}$ of infinite collections of pairwise disjoint open subsets of $X$ with compact boundary satisfying the following:
\begin{itemize}
\item[\rm(a)] For any $\zeta<\sigma$, $U\in {\mathscr U}_\zeta$ and finite ${\mathscr W}\subseteq\bigcup\{{\mathscr U}_\eta:\eta<\zeta\}$ the set $\mbox{\em cl}_X U\backslash\bigcup{\mathscr W}$ is non--compact.
\item[\rm(b)] For any distinct $\zeta,\eta<\sigma$, $U\in {\mathscr U}_\zeta$ and $V\in{\mathscr U}_\eta$ there exists a finite ${\mathscr W}\subseteq \bigcup\{{\mathscr U}_\xi:\xi<\zeta\}$ such that either
    \[\mbox{\em cl}_X U\backslash\Big(V\cup\bigcup{\mathscr W}\Big)\mbox{ or }(\mbox{\em cl}_X U\cap\mbox{\em cl}_X V)\backslash\bigcup{\mathscr W}\]
    is compact.
\item[\rm(c)] For any $\zeta<\eta<\sigma$ and $U\in{\mathscr U}_\eta$ there exists an infinite ${\mathscr V}\subseteq{\mathscr U}_\zeta$ such that for any $V\in{\mathscr V}$ there exists a finite ${\mathscr W}\subseteq\bigcup\{{\mathscr U}_\xi:\xi<\zeta\}$ such that $\mbox{\em cl}_X V\backslash(U\cup\bigcup{\mathscr W})$ is compact.
\end{itemize}
\end{itemize}
\end{theorem}

It is worth to mention that the general  problem of characterizing spaces with a countable--point compactification has been remained still open. (See \cite{Ho}, also \cite{Ch}, for a characterization of metrizable spaces having such compactifications.)

Our aim in this chapter is to generalize the above results to compactification--like $\mathcal{P}$--extensions. Note that part (2) of the lemma below generalizes  K.D. Magill,  Jr.'s theorem in \cite{Mag2} (Theorem \ref{i0}) provided that one replaces  $\mathcal{P}$ and $\mathcal{Q}$, respectively, by compactness and regularity, and note that for these specific choices of  $\mathcal{P}$ and $\mathcal{Q}$  and a locally compact space $X$ we have  $\lambda_{\mathcal{P}} X=X$  and the two notions  ``$Y$ is a  minimal ${\mathcal P}$--extension of $X$ with $\mathcal{Q}$" and ``$Y$ is a compactification of $X$" coincide.

\begin{lemma}\label{18}
Let ${\mathcal P}$ and  ${\mathcal Q}$ be a pair of compactness--like topological properties. Let $X$  be a Tychonoff  space with $\mathcal{Q}$.
\begin{itemize}
\item[\rm(1)] Let $n\in\mathbf{N}$. The following are equivalent:
\begin{itemize}
\item[\rm(a)] ${\mathscr M}^{\mathcal Q}_{\mathcal P}(X)$ contains an element with $n$--point remainder.
\item[\rm(b)] ${\mathscr O}^{\mathcal Q}_{\mathcal P}(X)$ contains an element with $n$--point remainder.
\item[\rm(c)] $X$ is locally--$\mathcal{P}$ and $\beta X\backslash \lambda_{\mathcal{P}} X$ contains $n$ pairwise disjoint non--empty clopen subsets.
\item[\rm(d)] $X$ is locally--$\mathcal{P}$ and $\beta X\backslash \lambda_{\mathcal{P}} X$ has at least $n$ (connected) components.
\end{itemize}
\item[\rm(2)] The following are equivalent:
\begin{itemize}
\item[\rm(a)] ${\mathscr M}^{\mathcal Q}_{\mathcal P}(X)$ contains an element with countable remainder.
\item[\rm(b)] ${\mathscr O}^{\mathcal Q}_{\mathcal P}(X)$ contains an element with countable remainder.
\item[\rm(c)] $X$ is locally--$\mathcal{P}$ and $\beta X\backslash\lambda_{\mathcal{P}}X$ contains an infinite number of pairwise disjoint non--empty clopen subsets.
\item[\rm(d)] $X$ is locally--$\mathcal{P}$ and $\beta X\backslash \lambda_{\mathcal{P}} X$ has an infinite number of (connected) components.
\end{itemize}
\item[\rm(3)] Let $0<\sigma<\Omega$ and let $n\in\mathbf{N}$. The following are equivalent:
\begin{itemize}
\item[\rm(a)] ${\mathscr M}^{\mathcal Q}_{\mathcal P}(X)$ contains an element with countable remainder of type $(\sigma,n)$.
\item[\rm(b)] ${\mathscr O}^{\mathcal Q}_{\mathcal P}(X)$ contains an element with countable remainder of type $(\sigma,n)$.
\item[\rm(c)] $X$ is locally--${\mathcal P}$ and there exists a family $\{{\mathscr H}_\zeta:\zeta\leq\sigma\}$ of collections of pairwise disjoint non--empty clopen subset of $\beta X\backslash\lambda_{\mathcal P} X$ satisfying the following:
\begin{itemize}
\item[\rm(i)] For any $\zeta<\sigma$, $\mbox{\em card}({\mathscr H}_\zeta)=\aleph_0$  and $\mbox{\em card}({\mathscr H}_\sigma)=n$.
\item[\rm(ii)] For any $\zeta\leq\sigma$ and $H\in {\mathscr H}_\zeta$ we have
\[H\backslash\bigcup\{G\in {\mathscr H}_\eta:\eta<\zeta\}\neq\emptyset.\]
\item[\rm(iii)] For any $\zeta<\eta\leq\sigma$, $H\in {\mathscr H}_\zeta$ and $G\in {\mathscr H}_\eta$ either
\[H\subseteq G\cup\bigcup\{F\in {\mathscr H}_\xi:\xi<\zeta\}\mbox{ or } H\cap G\subseteq\bigcup\{F\in {\mathscr H}_\xi:\xi<\zeta\}.\]
\item[\rm(iv)] For any $\zeta<\eta\leq\sigma$ and $H\in {\mathscr H}_\eta$ the set
\[\Big\{F\in{\mathscr H}_\zeta:F\subseteq H\cup\bigcup\{G\in {\mathscr H}_\xi:\xi<\zeta\}\Big\}\]
is infinite.
\end{itemize}
\end{itemize}
\end{itemize}
\end{lemma}

\begin{proof} (1). It is clear that (1.c) implies (1.d) and that (1.b) implies (1.a). (1.a) {\em implies} (1.c). Consider some $Y\in{\mathscr M}^{\mathcal Q}_{\mathcal P}(X)$ with an $n$--point remainder $Y\backslash X=\{p_1,\ldots,p_n\}$. Let $\phi:\beta X\rightarrow\beta Y$ be the continuous extension of $\mbox{id}_X$. Let $V_1,\dots,V_n$ be pairwise disjoint open neighborhoods of $p_1,\dots,p_n$ in $\beta Y$, respectively. By Lemma \ref{16} we have $\beta X\backslash\lambda_{\mathcal{P}}X\subseteq\phi^{-1}[Y\backslash X]$ and that $X$ is locally--$\mathcal{P}$. Let $i=1,\dots,n$. Then
\[\phi\big[\phi^{-1}[V_i]\backslash\lambda_{\mathcal{P}}X\big]\subseteq\phi[\beta X\backslash\lambda_{\mathcal{P}}X]\subseteq \phi\big[\phi^{-1}[Y\backslash X]\big]\subseteq Y\backslash X\]
and therefore
\[\phi\big[\phi^{-1}[V_i]\backslash\lambda_{\mathcal{P}}X\big]\cap V_i\subseteq(Y\backslash X)\cap V_i=\{p_i\}.\]
This gives
\begin{eqnarray*}
\phi^{-1}[V_i]\backslash\lambda_{\mathcal{P}} X&\subseteq&\phi^{-1}\big[\phi\big[\phi^{-1}[V_i]\backslash\lambda_{\mathcal{P}} X\big]\big]\cap \phi^{-1}[V_i]\\&=&\phi^{-1}\big[\phi\big[\phi^{-1}[V_i]\backslash\lambda_{\mathcal{P}} X\big]\cap V_i\big]\subseteq\phi^{-1}(p_i)
\end{eqnarray*}
which implies that
\begin{equation}\label{PKWA}
\phi^{-1}[V_i]\backslash\lambda_{\mathcal{P}}X=\phi^{-1}(p_i)\backslash\lambda_{\mathcal{P}}X.
\end{equation}
Let $H_i$ denote the set in (\ref{PKWA}). Then $H_i$ is clopen in $\beta X\backslash\lambda_{\mathcal{P}} X$ and it is non--empty, as otherwise
\[\beta X\backslash\lambda_{\mathcal{P}} X\subseteq\phi^{-1}[Y\backslash X]\backslash\phi^{-1}(p_i)=\phi^{-1}\big[(Y\backslash X)\backslash\{p_i\}\big]=
\phi^{-1}\big[\big(Y\backslash\{p_i\}\big)\backslash X\big]\]
which by Lemma \ref{16}  implies that the subspace $Y\backslash\{p_i\}$ of $Y$ has $\mathcal{P}$, contradicting  the minimality of $Y$. That $H_i$'s are pairwise disjoint follows from their definitions.

(1.d) {\em implies } (1.b). Let $C_1,\ldots,C_n$ be $n$ distinct components of $\beta X\backslash\lambda_{\mathcal{P}} X$.  Then arguing as in the proof of Theorem 2.1 of \cite{Mag2} and since in compact spaces components and quasi--components coincide, each component is the intersection of all clopen subsets containing it (see  Theorem 6.1.23 of \cite{E}). Since
\[C_1\subseteq(\beta X\backslash\lambda_{\mathcal{P}} X)\backslash\bigcup_{k=2}^nC_k\]
with the latter an open subset of $\beta X\backslash\lambda_{\mathcal{P}} X$, by the compactness of $\beta X\backslash\lambda_{\mathcal{P}} X$, there exists a clopen subset $H_1$ of $\beta X\backslash\lambda_{\mathcal{P}} X$  such that
\[C_1\subseteq H_1\subseteq(\beta X\backslash\lambda_{\mathcal{P}} X)\backslash\bigcup_{k=2}^nC_k.\]
Suppose inductively that for some $j=1,\ldots,n-1$ the pairwise disjoint clopen subsets $H_1,\ldots,H_j$ of $\beta X\backslash\lambda_{\mathcal{P}} X$ are defined in such a way that
\[C_i\subseteq H_i\subseteq(\beta X\backslash\lambda_{\mathcal{P}}X)\backslash\Big(\bigcup_{k=1}^{i-1}H_k\cup\bigcup_{k=i+1}^nC_k\Big)\]
for any $i=1,\ldots,j$. Note that $C_{j+1}\cap H_i=\emptyset$ when $i=1,\ldots,j$. Thus
\[C_{j+1}\subseteq(\beta X\backslash\lambda_{\mathcal{P}}X)\backslash\Big(\bigcup_{k=1}^j H_k\cup\bigcup_{k=j+2}^nC_k\Big).\]
Let $H_{j+1}$ be a clopen subset of $\beta X\backslash\lambda_{\mathcal{P}} X$ such that
\[C_{j+1}\subseteq H_{j+1}\subseteq(\beta X\backslash\lambda_{\mathcal{P}}X)\backslash\Big(\bigcup_{k=1}^j H_k\cup\bigcup_{k=j+2}^nC_k\Big).\]
By Lemma \ref{15} we have  $X\subseteq\lambda_{\mathcal{P}} X$. Let $T$ be the quotient space of $\beta X$ obtained by contracting the compact subsets
\[H_1,\ldots,H_{n-1},(\beta X\backslash\lambda_{\mathcal{P}} X)\backslash\bigcup_{k=1}^{n-1}H_k\]
of $\beta X\backslash X$ to points $p_1,\ldots,p_n$, respectively. Then $T$ is Tychonoff and contains $X$ as a dense subspace. Consider the subspace $Y=X\cup\{p_1,\ldots,p_n\}$ of $T$. By Lemma \ref{16} we have $Y\in{\mathscr E}^{\mathcal Q}_{\mathcal P}(X)$. Since $T$ is a compactification of $Y$ there exists a continuous mapping $\gamma:\beta Y\rightarrow T$ such that $\gamma|Y=\mbox{id}_Y$. Let $\psi:\beta X\rightarrow\beta Y$ be the continuous extension of $\mbox{id}_X$. Then since $\gamma\psi|X=\mbox{id}_X=q|X$ we have $\gamma\psi=q$. By Theorem 3.5.7 of \cite{E} we have $\gamma[\beta Y\backslash Y]=T\backslash Y$. Thus
\[\psi^{-1}(p_i)=\psi^{-1}\big[\gamma^{-1}(p_i)\big]=(\gamma\psi)^{-1}(p_i)=q^{-1}(p_i)\]
for any $i=1,\ldots,n$. Therefore
\[\psi^{-1}[Y\backslash X]=q^{-1}[Y\backslash X]=\beta X\backslash\lambda_{\mathcal P} X\]
which by Theorem \ref{HG16} implies that $Y\in{\mathscr O}_{\mathcal P}(X)$.

(2). It is clear that (2.b) implies (2.a).  (2.a) {\em implies} (2.d). Consider some $Y\in{\mathscr M}^{\mathcal Q}_{\mathcal P}(X)$ with countable remainder. Let $\phi:\beta X\rightarrow\beta Y$ be the continuous extension of $\mbox {id}_X$. By Lemma \ref{16} we have $\beta X\backslash \lambda_{\mathcal{P}} X\subseteq\phi^{-1}[Y\backslash X]$ and that $X$ is locally--${\mathcal{P}}$. Let $n\in\mathbf{N}$ and consider some distinct elements
$p_1,\ldots,p_n\in Y\backslash X$. Define a continuous $f:\beta Y\rightarrow [1,n]$ such that $f(p_i)=i$ for any $i=1,\ldots,n$. Since $f[Y\backslash X]$ is countable we can find some real numbers $r_1,\ldots,r_{n+1}$ such that
\[r_1<1<r_2<2<\cdots<r_n<n<r_{n+1}\mbox{ and }r_i\notin f[Y\backslash X]\mbox{ for any }i=1,\ldots,n+1.\]
Let $i=1,\ldots,n$. Define $B_i=f^{-1}[(r_i,r_{i+1})]$. Then
\[\bigcup_{i=1}^n\phi^{-1}[B_i]=\phi^{-1}\Big[\bigcup_{i=1}^n B_i\Big]\supseteq\phi^{-1}[Y\backslash X]\supseteq\beta X\backslash \lambda_{\mathcal{P}}X.\]
Now $\phi^{-1}[B_i]\backslash\lambda_{\mathcal{P}} X$ is non--empty, as $B_i$ is non--empty and by Theorem \ref{HUHG16} the set  $\phi^{-1}(p)\backslash\lambda_{\mathcal{P}}X$ is non--empty for any $p\in Y\backslash X$. Thus $\phi^{-1}[B_i]\backslash\lambda_{\mathcal{P}}X$, where $i=1,\ldots,n$, are $n$ pairwise disjoint non--empty open (and therefore clopen, as their union is $\beta X\backslash \lambda_{\mathcal{P}}X$) subsets of $\beta X\backslash\lambda_{\mathcal{P}} X$ which implies that $\beta X\backslash\lambda_{\mathcal{P}}X$ has at least $n$ components. Since $n$ is arbitrary the result follows.

(2.d) {\em implies}  (2.c). Let $C$ and $D$ be distinct components of $\beta X\backslash \lambda_{\mathcal{P}}X$.
Let $U$ be open in $\beta X\backslash \lambda_{\mathcal{P}} X$ and such that $C\subseteq U$ and $D\cap U=\emptyset$. Then arguing as in
(1.d) $\Rightarrow$ (1.b) and since $U$ is an open subset of $\beta X\backslash \lambda_{\mathcal{P}} X$ containing $C$ there exists a clopen subset $V$ of $\beta X\backslash\lambda_{\mathcal{P}} X$ such that $C\subseteq V\subseteq U$. Define $H_1$ to be either of the (non--empty) sets $V$ or $(\beta X\backslash\lambda_{\mathcal{P}}X)\backslash V$ which misses an infinite number of components of $\beta X\backslash\lambda_{\mathcal{P}}  X$. Now inductively suppose that $H_1,\ldots,H_n$ are defined such that $H_i$'s are pairwise disjoint non--empty clopen subsets of  $\beta X\backslash\lambda_{\mathcal{P}}X$ and $H_1\cup\cdots\cup H_n$ misses an infinite number of components of $\beta X\backslash \lambda_{\mathcal{P}}X$. Let $E$ and $F$ be distinct components of $\beta X\backslash\lambda_{\mathcal{P}}X$ missing $H_1\cup\cdots\cup H_n$ and let $W$ be open
in $\beta X\backslash \lambda_{\mathcal{P}} X$ and  such that $E\subseteq W$ and $F\cap W=\emptyset$. Then since $W\backslash (H_1\cup\cdots\cup
H_n)$ is an open neighborhood of $E$ in $\beta X\backslash\lambda_{\mathcal{P}} X$, arguing as above there exists a non--empty clopen  subset $H_{n+1}$ of $\beta X\backslash \lambda_{\mathcal{P}} X$ such that $H_{n+1}\cap H_i=\emptyset$ for any $i=1,\ldots,n$, and that it
misses  an infinite number of components of $\beta X\backslash \lambda_{\mathcal{P}} X$ contained in $(\beta X\backslash\lambda_{\mathcal{P}} X)\backslash(H_1\cup\cdots\cup H_n)$. The sequence $H_1,H_2,\ldots$ consists of  pairwise disjoint non--empty clopen subsets of $\beta X\backslash \lambda_{\mathcal{P}} X$.

(2.c) {\em implies} (2.b). Suppose that $X$ is locally--$\mathcal{P}$ and  that there exists a bijectively indexed sequence $H_1,H_2,\ldots$ of pairwise disjoint non--empty clopen subsets of $\beta X\backslash \lambda_{\mathcal{P}} X$. By Lemma \ref{15} we have $X\subseteq\lambda_{\mathcal{P}} X$. Let $T$ be the space obtained from $\beta X$ by contracting the sets
\[(\beta X\backslash \lambda_{\mathcal{P}} X)\backslash\bigcup_{k=2}^\infty H_k,H_2,H_3,\ldots\]
into points $p_1,p_2,\ldots$, respectively, with the quotient mapping $q:\beta X\rightarrow T$. Then $T$ is compact, since as we show it is Hausdorff (and a continuous image of $\beta X$). Suppose that $x,y\in T$ are distinct elements. We consider the following three cases:
\begin{description}
\item[{\sc Case 1.}] Suppose that $x,y\in \lambda_{\mathcal{P}} X$. Since $x$ and $y$ can be separated  by disjoint open subsets in $\lambda_{\mathcal{P}} X$ and $\lambda_{\mathcal{P}} X$ is open in $\beta X$ they can also be separated by disjoint open subsets in $T$.
\item[{\sc Case 2.}] Suppose that $x\in\lambda_{\mathcal{P}} X$ and $y=p_j$ for some $j\in\mathbf{N}$. Let $P$ and $Q$ be disjoint open neighborhoods of $x$ and $\beta X\backslash \lambda_{\mathcal{P}} X$ in $\beta X$, respectively. Then $q[P]$ and $q[Q]$ are disjoint open subsets of $T$ separating $x$ and $y$. The case when $x=p_i$ for some $i\in\mathbf{N}$ and $y\in\lambda_{\mathcal{P}} X$ is analogous.
\item[{\sc Case 3.}] Suppose that $x=p_i$ and $y=p_j$ for some $i,j\in\mathbf{N}$. Then either $i\geq 2$ or $j\geq 2$, say $j\geq 2$. Let $P$ and $Q$ be disjoint open neighborhoods of $(\beta X\backslash\lambda_{\mathcal{P}} X)\backslash H_j$ and $H_j$ in $\beta X$, respectively. Then $q[P]$ and $q[Q]$ are disjoint open subsets of $T$ separating $x$ and $y$.
\end{description}
Note that $T$ contains $X$ as a dense subspace. Consider the subspace $Y=X\cup\left\{p_1,p_2,\ldots\right\}$ of $T$. Then
$Y$ is a countable--point  extension of $X$ with the compact remainder $Y\backslash X=q[\beta X\backslash\lambda_{\mathcal{P}} X]$. Now since  $T$ is a compactification of $Y$ and $q^{-1}[Y\backslash X]=\beta X\backslash\lambda_{\mathcal{P}} X$, by Lemma \ref{16}  we have $Y\in{\mathscr E}^{\mathcal Q}_{\mathcal P}(X)$. To complete the proof we only need to verify that $Y$ is optimal. But this follows by an argument similar to the one in (1.d) $\Rightarrow$ (1.b).

(3). It is clear that (3.b) implies (3.a). (3.a) {\em  implies} (3.c). Consider some $Y\in{\mathscr M}^{\mathcal Q}_{\mathcal P}(X)$ with countable remainder of type $(\sigma,n)$. For any $\zeta\leq\sigma$ let
\[(Y\backslash X)^{(\zeta)}\backslash(Y\backslash X)^{(\zeta+1)}=\{p_i^\zeta:i\in J_\zeta\}\]
where $p_i^\zeta$'s are bijectively indexed. Note that if $\zeta<\sigma$ then $ \mbox{card}(J_\zeta)=\aleph_0$, as otherwise, $(Y\backslash X)^{(\zeta)}$ is finite and thus, since
\[(Y\backslash X)^{(\sigma)}\subseteq(Y\backslash X)^{(\zeta+1)}=\big((Y\backslash X)^{(\zeta)}\big)'\]
it follows that $(Y\backslash X)^{(\sigma)}=\emptyset$, contradicting $\mbox{card}((Y\backslash X)^{(\sigma)})=n>0$. Also
\[\mbox{card}(J_\sigma)=\mbox{card}\big((Y\backslash X)^{(\sigma)}\backslash(Y\backslash X)^{(\sigma+1)}\big)=\mbox{card}\big((Y\backslash X)^{(\sigma)}\big)=n.\]
Let $J_\zeta={\mathbf N}$ for any $\zeta<\sigma $ and $J_\sigma=\{1,\ldots,n\}$. For any $\zeta\leq\sigma $ and $i\in J_\zeta$ there exists an open neighborhood  $V_i^\zeta$ of  $p_i^\zeta$ in $\beta Y$ such that $V_i^\zeta\cap(Y\backslash X)^{(\zeta)}=\{p_i^\zeta\}$. Indeed, we prove the following.

\begin{xclaim}
For any $\zeta\leq\sigma$ there exists a collection $\{W_i^\zeta:i\in J_\zeta\}$ of open subsets of $\beta Y$ such that $W_i^\zeta\cap(Y\backslash X)^{(\zeta)}=\{p_i^\zeta\}$ for any $i\in J_\zeta$ and $W_i^\zeta\cap W_j^\zeta=\emptyset$ for any distinct $i,j\in J_\zeta$.
\end{xclaim}

\subsubsection*{Proof of the claim} Let $\zeta<\sigma$. We inductively define $W_i^\zeta$'s for any $i\in J_\zeta$. Let $W_1^\zeta$ be an open neighborhood of $p_1^\zeta$ in $\beta Y$ such that $\mbox{cl}_{\beta Y}W_1^\zeta \subseteq V_1^\zeta$. For an $m\in\mathbf{N}$ suppose that the open subsets $W_1^\zeta,\ldots,W_m^\zeta$ of $\beta Y$ are defined such that  $W_i^\zeta\cap W_j^\zeta=\emptyset$ for any distinct $i,j=1,\ldots,m$ and that
\[p_i^\zeta\in W_i^\zeta\subseteq \mbox{cl}_{\beta Y}W_i^\zeta\subseteq V_i^\zeta\backslash\mbox{cl}_{\beta Y}\Big(\bigcup_{j=1}^{i-1} W_j^\zeta\Big)\]
for any $i=1,\ldots,m$. Since
\[W_i^\zeta\cap(Y\backslash X)^{(\zeta)}\subseteq V_i^\zeta\cap(Y\backslash X)^{(\zeta)}=\{p_i^\zeta\}\]
this implies that $W_i^\zeta\cap(Y\backslash X)^{(\zeta)}=\{p_i^\zeta\}$ for any $i=1,\ldots,m$; also note that $p_{m+1}^\zeta \notin \mbox{cl}_{\beta Y}W_i^\zeta$. Thus $V_{m+1}^\zeta\backslash \mbox{cl}_{\beta Y}( W_1^\zeta\cup\cdots \cup W_m^\zeta )$ is an open neighborhood of  $p_{m+1}^\zeta $ in $\beta Y$. Define $W_{m+1}^\zeta $ to be an open neighborhood of $p_{m+1}^\zeta $ in $\beta Y$ such that
\[\mbox{cl}_{\beta Y} W_{m+1}^\zeta\subseteq V_{m+1}^\zeta\backslash \mbox{cl}_{\beta Y}\Big(\bigcup_{j=1}^{m} W_j^\zeta\Big).\]
The case when $\zeta=\sigma$ is analogous.

\medskip

\noindent Let $\zeta\leq\sigma $ and $i\in J_\zeta$. Let $f_i^\zeta:\beta Y\rightarrow {\mathbf I}$ be continuous with $f_i^\zeta(p_i^\zeta)=0$ and  $f_i^\zeta[\beta Y\backslash W_i^\zeta]\subseteq\{1\}$ and let $r_i^\zeta\in(0,1)\backslash f_i^\zeta[Y\backslash X]$ which exists, as $Y\backslash X$ is countable. For convenience  denote
\[C_i^\zeta=(f_i^\zeta)^{-1}\big[[0,r_i^\zeta)\big]\mbox{ and }D_i^\zeta=(f_i^\zeta)^{-1}\big[[0,r_i^\zeta]\big]\]
and define
\[H_i^\zeta=\phi^{-1}[C_i^\zeta]\backslash\lambda_{\mathcal P} X\]
where $\phi:\beta X\rightarrow\beta Y$ is the continuous extension of $\mbox{id}_X$. For any $\zeta\leq\sigma$ let
\[{\mathscr H}_\zeta=\{H_i^\zeta: i\in J_\zeta\}.\]
We verify that the collection $\{{\mathscr H}_\zeta:\zeta\leq\sigma\}$ has the  desired properties. First note that for any $\zeta\leq\sigma$ and any distinct $i,j\in J_\zeta$ we have
\[H_i^\zeta\cap H_j^\zeta\subseteq\phi^{-1}[C_i^\zeta]\cap\phi^{-1}[C_j^\zeta]\subseteq\phi^{-1}[W_i^\zeta]\cap\phi^{-1}[W_j^\zeta]=\phi^{-1}[W_i^\zeta\cap W_j^\zeta]=\emptyset.\]
By definition $H_i^\zeta$'s are open in $\beta X\backslash\lambda_{\mathcal P} X$.  We show that $H_i^\zeta$'s are closed in $\beta X\backslash\lambda_{\mathcal P} X$ and they are non--empty. Let $\zeta\leq\sigma$ and $i\in J_\zeta$. Note that by the choice of $r_i^\zeta$ we have
\[(f_i^\zeta)^{-1}(r_i^\zeta)\cap (Y\backslash X)=\emptyset\]
and thus
\begin{equation}\label{ADRF}
C_i^\zeta\cap (Y\backslash X)=D_i^\zeta\cap (Y\backslash X).
\end{equation}
Therefore
\[\phi^{-1}[C_i^\zeta]\cap\phi^{-1}[Y\backslash X]=\phi^{-1}\big[C_i^\zeta\cap (Y\backslash X)\big]=\phi^{-1}\big[D_i^\zeta\cap (Y\backslash X)\big]=\phi^{-1}[D_i^\zeta]\cap \phi^{-1}[Y\backslash X].\]
By Lemma \ref{16} we have $\beta X\backslash\lambda_{\mathcal P} X\subseteq\phi^{-1}[Y\backslash X]$ (and that $X$ is locally--${\mathcal P}$) which by above yields
\begin{equation}\label{GASW}
H_i^\zeta=\phi^{-1}[C_i^\zeta]\backslash\lambda_{\mathcal P} X=\phi^{-1}[D_i^\zeta]\backslash\lambda_{\mathcal P} X.
\end{equation}
This shows that $H_i^\zeta$ is closed in $\beta X\backslash\lambda_{\mathcal P} X$. Next, suppose to the contrary that  $H_i^\zeta=\emptyset$. Consider the subspace
\[Y'=X\cup\big((Y\backslash X)\backslash C_i^\zeta\big)\]
of $Y$. Then $Y'$ is an extension of $X$ and $Y'\backslash X$ is compact, as it is closed in $Y\backslash X$. Also
\[\phi^{-1}[Y'\backslash X]=\phi^{-1}\big[(Y\backslash X)\backslash C_i^\zeta\big]=\phi^{-1}[Y\backslash X]\backslash\phi^{-1}[C_i^\zeta]\supseteq\beta X\backslash\lambda_{\mathcal P} X.\]
Thus by Lemma \ref{16} we have $Y'\in{\mathscr E}^{\mathcal Q}_{\mathcal P}(X)$. Note that $Y'$ is properly contained in $Y$, as $p_i^\zeta\notin Y'$, because $f_i^\zeta(p_i^\zeta)=0$ and thus $p_i^\zeta\in C_i^\zeta$. But this contradicts the minimality of $Y$. This shows that each ${\mathscr H}_\zeta$ where $\zeta\leq\sigma$, is a collection of pairwise disjoint non--empty clopen subsets of $\beta X\backslash\lambda_{\mathcal P}X$. Also, note that (3.c.i) holds, as $ \mbox{card}({\mathscr H}_\zeta)=\mbox{card}(J_\zeta)$ for any $\zeta\leq\sigma$.

We now verify (3.c.ii). Let  $\zeta\leq\sigma$ and $H\in {\mathscr H}_\zeta$. Suppose to the contrary that
\[H\backslash\bigcup\{G\in {\mathscr H}_\eta:\eta<\zeta\}=\emptyset.\]
Let $H=H_i^\zeta$ for some $i\in J_\zeta$. Since $H_i^\zeta$ is compact, as it is closed in $\beta X\backslash\lambda_{\mathcal P} X$, there exists some $H_{i_j}^{\zeta_j}\in {\mathscr H}_{\zeta_j}$ with $i_j\in J_{\zeta_j}$ and $\zeta_j<\zeta$, where $j=1,\ldots,k$ and $k\in\mathbf{N}$, such that $H_i^\zeta\subseteq H_{i_1}^{\zeta_1}\cup\cdots\cup H_{i_k}^{\zeta_k}$. Consider the subspace
\[Y'=X\cup\Big((Y\backslash X)\backslash\Big(C_i^\zeta\backslash\bigcup_{j=1}^kD_{i_j}^{\zeta_j}\Big)\Big)\]
of $Y$. Using (\ref{GASW}) we have
\[\phi^{-1}\Big[C_i^\zeta\backslash\bigcup_{j=1}^kD_{i_j}^{\zeta_j}\Big]\backslash\lambda_{\mathcal P} X=\Big(\phi^{-1}[C_i^\zeta]\backslash\bigcup_{j=1}^k\phi^{-1}[D_{i_j}^{\zeta_j}]\Big)
\backslash\lambda_{\mathcal P} X=H_i^\zeta\backslash\bigcup_{j=1}^k H_{i_j}^{\zeta_j}=\emptyset\]
and thus
\begin{eqnarray*}
\phi^{-1}[Y'\backslash X]&=&\phi^{-1}\Big[(Y\backslash X)\backslash\Big(C_i^\zeta\backslash\bigcup_{j=1}^kD_{i_j}^{\zeta_j}\Big)\Big]
\\&=&\phi^{-1}[Y\backslash X]\backslash\phi^{-1}\Big[C_i^\zeta\backslash\bigcup_{j=1}^kD_{i_j}^{\zeta_j}\Big]\supseteq\beta X\backslash\lambda_{\mathcal P}X.
\end{eqnarray*}
Note that $Y'\backslash X$ is compact, as it is closed in $Y\backslash X$, and therefore by Lemma \ref{16} we have $Y'\in{\mathscr E}^{\mathcal Q}_{\mathcal P}(X)$. We show that $Y'$ is properly contained in $Y$. This contradicts the minimality of $Y$ and proves (3.c.ii). Indeed, we verify that $p_i^\zeta\notin Y'$. By the definition of $f_i^\zeta$ we have $p_i^\zeta\in C_i^\zeta$. Also, for any $j=1,\ldots,k$ we have $p_i^\zeta\notin D_{i_j}^{\zeta_j}$, as otherwise since $D_{i_j}^{\zeta_j}\subseteq W_{i_j}^{\zeta_j}$ it follows that $p_i^\zeta\in W_{i_j}^{\zeta_j}$. But since $\zeta_j<\zeta$ and thus $\zeta_j+1\leq\zeta$ we have
\[p_i^\zeta\in(Y\backslash X)^{(\zeta)}\subseteq(Y\backslash X)^{(\zeta_j+1)}=\big((Y\backslash X)^{(\zeta_j)}\big)'\]
and therefore $W_{i_j}^{\zeta_j}$, being an open neighborhood of $p_i^\zeta$ in $\beta Y$, has an infinite intersection with $(Y\backslash X)^{(\zeta_j)}$, contradicting the definition of $W_{i_j}^{\zeta_j}$.

Next, we show (3.c.iii). Suppose that $\zeta<\eta\leq\sigma$, $H\in  {\mathscr H}_\zeta$ and $G\in  {\mathscr H}_\eta$. Let $H=H_i^\zeta$ and $G=H_j^\eta$ for some $i\in J_\zeta$ and $j\in J_\eta$. We have the following cases:

\begin{description}
\item[{\sc Case 1.}] Suppose that $p_i^\zeta\in C_j^\eta$. First note that each $p\in Y\backslash X$ is of the form $p^\xi_k$ for some $\xi\leq\sigma$ and $k\in J_\xi$. To show this let $\alpha<\Omega$ be the least ordinal such that $p\notin (Y\backslash X)^{(\alpha)}$. Such an $\alpha$ exists, as $(Y\backslash X)^{(\sigma+1)}=\emptyset$, and it is necessarily not a limit ordinal, as otherwise
    \[p\in\bigcap\big\{(Y\backslash X)^{(\xi)}:\xi<\alpha\big\}=(Y\backslash X)^{(\alpha)}.\]
    Let $\xi$ be such that $\alpha=\xi+1$. Then $p\in(Y\backslash X)^{(\xi)}\backslash(Y\backslash X)^{(\xi+1)}$ and thus $p=p_k^\xi$ for some $k\in J_\xi$. Since $D_i^\zeta\subseteq W_i^\zeta$, by an argument similar to the one in (3.c.ii), for any $p_k^\xi\in D_i^\zeta$ where $k\in J_\xi$ we have $\xi\leq\zeta$ and if $\xi=\zeta$ then $p_k^\xi=p_i^\zeta$, as by the definition of $W_i^\zeta$ we have $W_i^\zeta\cap(Y\backslash X)^{(\zeta)}=\{p_i^\zeta\}$. Therefore in this case
    \[(D_i^\zeta\backslash C_j^\eta)\cap(Y\backslash X)\subseteq\{p_k^\xi:\xi<\zeta\mbox{ and } k\in J_\xi\}\subseteq\bigcup\{C_k^\xi:\xi<\zeta\mbox{ and }k\in J_\xi\}\]
    which (since $\beta X\backslash\lambda_{\mathcal P} X\subseteq\phi^{-1}[Y\backslash X]$) implies that
    \begin{eqnarray*}
    H_i^\zeta\backslash H_j^\eta&=&\big(\phi^{-1}[D_i^\zeta]\backslash\phi^{-1}[C_j^\eta]\big)\backslash\lambda_{\mathcal P} X
    \\&\subseteq&\big(\big(\phi^{-1}[D_i^\zeta]\backslash\phi^{-1}[C_j^\eta]\big)\cap\phi^{-1}[Y\backslash X]\big)\backslash\lambda_{\mathcal P} X\\&\subseteq&\bigcup\big\{\phi^{-1}[C_k^\xi]\backslash\lambda_{\mathcal P} X:\xi<\zeta \mbox{ and } k\in J_\xi\big\}=\bigcup\{H_k^\xi:\xi<\zeta\mbox{ and }k\in J_\xi\}.
    \end{eqnarray*}
    Thus
    \[H^\zeta_i\subseteq H^\eta_j\cup\bigcup\{F\in{\mathscr H}_\xi:\xi<\zeta\}\]
    and (3.c.iii) holds in this case.
\item[{\sc Case 2.}] Suppose that $p_i^\zeta\notin C_j^\eta$. Arguing as in Case 1 we have
    \[D_i^\zeta\cap C_j^\eta\cap(Y\backslash X)\subseteq\{p_k^\xi:\xi<\zeta\mbox{ and } k\in J_\xi\}\subseteq\bigcup\{C_k^\xi:\xi<\zeta \mbox{ and }k\in J_\xi\}.\]
    Therefore
    \begin{eqnarray*}
    H_i^\zeta\cap
    H_j^\eta&=&\big(\phi^{-1}[D_i^\zeta]\cap\phi^{-1}[C_j^\eta]\big)\backslash\lambda_{\mathcal P} X\\&=&\big(\phi^{-1}[D_i^\zeta]\cap\phi^{-1}[C_j^\eta]\cap\phi^{-1}[Y\backslash X]\big)\backslash\lambda_{\mathcal P} X
    \\&\subseteq&\bigcup\big\{\phi^{-1}[C_k^\xi]\backslash\lambda_{\mathcal P} X:\xi<\zeta \mbox{ and } k\in J_\xi\big\} =\bigcup\{H_k^\xi:\xi<\zeta \mbox{ and } k\in J_\xi\}.
    \end{eqnarray*}
    Thus (3.c.iii) holds in this case as well.
\end{description}

Finally, we verify  (3.c.iv). Suppose that $\zeta<\eta\leq\sigma$ and $H\in {\mathscr H}_\eta$. Let $H=H_j^\eta$ for some $j\in J_\eta$. We first verify that
\[C_j^\eta\cap\big((Y\backslash X)^{(\zeta)}\backslash(Y\backslash X)^{(\zeta+1)}\big)\]
is infinite. Suppose it is finite. Since
\[p_j^\eta\in(Y\backslash X)^{(\eta)}\subseteq(Y\backslash X)^{(\zeta+1)}=\big((Y\backslash X)^{(\zeta)}\big)'\]
and $C_j^\eta$ is an open neighborhood of $p_j^\eta$ in $\beta Y$ the set
\[C_j^\eta\cap(Y\backslash X)^{(\zeta)}=\big(C_j^\eta\cap\big((Y\backslash X)^{(\zeta)}\backslash(Y\backslash X)^{(\zeta+1)}\big)\big)\cup\big(C_j^\eta\cap(Y\backslash X)^{(\zeta+1)}\big)\]
is infinite. But then since by (\ref{ADRF}) we have
\[C_j^\eta\cap(Y\backslash X)^{(\zeta+1)}=D_j^\eta\cap(Y\backslash X)^{(\zeta+1)}\]
the latter set is an infinite compact space without isolated points and therefore uncountable. This contradiction shows that $p_i^\zeta\in C_j^\eta$ for an infinite number of $i\in J_\zeta$. But if $p_i^\zeta\in C_j^\eta$ for some $i\in J_\zeta$, arguing as in Case 1 of (3.c.iii) we have
\[H_i^\zeta\subseteq H_j^\eta \cup\bigcup\{G\in {\mathscr H}_\xi:\xi<\zeta\}.\]
This shows (3.c.iv).

(3.c) {\em implies} (3.b). Consider a family $\{{\mathscr H}_\zeta:\zeta\leq\sigma\}$ of collections of pairwise disjoint non--empty clopen subsets of  $\beta X\backslash\lambda_{\mathcal P} X$ satisfying (3.c.i)--(3.c.iv). For any $\zeta\leq\sigma$ let ${\mathscr H}_\zeta=\{H_i^\zeta:i\in J_\zeta\}$ where $H_i^\zeta$'s are bijectively indexed. Note that if $(\zeta,i)\neq(\eta,j)$, where $\zeta,\eta\leq\sigma$, $i\in J_\zeta$ and $j\in J_\eta$, then $H_i^\zeta\neq H_j^\eta$. This is clear if $\zeta=\eta$, and if $\zeta<\eta$ then it follows from (3.c.ii), as
\[\emptyset\neq H_i^\eta \backslash\bigcup\{G\in {\mathscr H}_\xi:\xi<\eta\}\subseteq H_i^\eta\backslash H_j^\zeta.\]
Similarly, if $\eta<\zeta$. Before we proceed with the main proof we prove the following generalized version of (3.c.iv).

\begin{xclaim}
For any $\zeta,\eta_1,\ldots,\eta_k<\eta\leq\sigma$, where $k$ is a non--negative integer ($\eta_1,\ldots,\eta_k$  may not be distinct) $H\in {\mathscr H}_\eta$ and $H_i\in {\mathscr H}_{\eta_i}$ for any $i=1,\ldots,k$, the set
\begin{equation}\label{UYBJ}
\Big\{F\in{\mathscr H}_\zeta:F\subseteq\Big(H\backslash\bigcup_{i=1}^kH_i\Big)\cup\bigcup\{G\in {\mathscr H}_\xi:\xi<\zeta\}\Big\}
\end{equation}
is infinite.
\end{xclaim}

\subsubsection*{Proof of the claim} If $k=0$ then the claim is simply (3.c.iv). Assume that $k>0$. We use transfinite  induction  on $\eta$. Suppose that $\eta=1$, $k\in\mathbf{N}$, $\zeta,\eta_1,\ldots,\eta_k<\eta$, $H\in {\mathscr H}_\eta$ and  $H_i\in {\mathscr H}_{\eta_i}$ for any $i=1,\ldots,k$. Then  $\zeta,\eta_1,\ldots,\eta_k=0$. By (3.c.iv) the set ${\mathscr F}=\{F\in {\mathscr H}_\zeta:F\subseteq H\}$ is infinite. Now since the elements of ${\mathscr H}_\zeta$ are pairwise disjoint, each $F\in {\mathscr F}\backslash\{H_1,\ldots, H_k\}$ misses $H_i$ for any $i=1,\ldots,k$, and thus $F\subseteq H\backslash (H_1\cup\cdots\cup H_k)$. Therefore (\ref{UYBJ}) holds for $\eta=1$. Now inductively suppose that (\ref{UYBJ}) holds for any $\xi<\eta$. Let $\zeta,\eta_1,\ldots,\eta_k<\eta\leq\sigma$, where $k\in\mathbf{N}$, $H\in {\mathscr H}_\eta$ and $H_i\in {\mathscr H}_{\eta_i}$ for any $i=1,\ldots,k$. We may assume that $\eta_1,\ldots,\eta_{l-1}<\eta_l=\cdots=\eta_k$ for some $l\in\mathbf{N}$ with $l\leq k$. Let
\[{\mathscr K}=\Big\{K\in{\mathscr H}_\zeta:K\subseteq\Big(H\backslash\bigcup_{i=1}^kH_i\Big)\cup\bigcup\{G\in {\mathscr H}_\xi:\xi<\zeta\}\Big\}.\]
We consider the following cases:
\begin{description}
\item[{\sc Case 1.}] Suppose that $\eta_k<\zeta$. Since $\zeta<\eta$, by (3.c.iv) the set
    \[{\mathscr F}=\Big\{F\in{\mathscr H}_\zeta:F\subseteq H\cup\bigcup\{G\in {\mathscr H}_\xi:\xi<\zeta\}\Big\}\]
    is infinite. Now for any $F\in {\mathscr F}$ we have
    \[F\backslash\bigcup\{G\in {\mathscr H}_\xi:\xi<\zeta\}\subseteq H\backslash\bigcup\{G\in {\mathscr H}_\xi:\xi<\zeta\}\subseteq H\backslash\bigcup_{i=1}^kH_i\]
    and thus $F\in{\mathscr K}$. Therefore in this case ${\mathscr F}\subseteq{\mathscr K}$ and thus ${\mathscr K}$ is infinite.
\item[{\sc Case 2.}] Suppose that $\eta_k=\zeta$. By (3.c.iv) the set
    \[{\mathscr F}=\Big\{F\in{\mathscr H}_\zeta:F\subseteq H\cup\bigcup\{G\in{\mathscr H}_\xi:\xi<\zeta\}\Big\}\]
    is infinite. Let $F\in {\mathscr F}\backslash\{H_1,\ldots,H_k\}$. Then each $H_i\in {\mathscr H}_{\eta_i}={\mathscr H}_\zeta$ where $i=l,\ldots,k$ is disjoint from $F\in{\mathscr H}_\zeta$. On the other hand, since $\eta_i<\zeta$ for any $i=1,\ldots,l-1$ we have
    \[F\backslash\{G\in {\mathscr H}_\xi:\xi<\zeta\}\subseteq H\backslash\bigcup_{i=1}^{l-1}H_i\]
    which  combined with above gives $F\in  {\mathscr K}$.  Therefore ${\mathscr F}\backslash\{H_l,\ldots,H_k\}\subseteq{\mathscr K}$ and thus ${\mathscr K}$ is infinite.
\item[{\sc Case 3.}] Suppose that $\eta_k>\zeta$. Since $\eta_k<\eta$, by (3.c.iv) the set
    \[{\mathscr L}=\Big\{L\in{\mathscr H}_{\eta_k}:L\subseteq H\cup\bigcup\{G\in {\mathscr H}_\xi:\xi<\eta_k\}\Big\}\]
    is infinite. Choose some $L\in {\mathscr L}\backslash\{H_l,\ldots,H_k\}$. Then
    \[L\subseteq H\cup\bigcup\{G\in {\mathscr H}_\xi:\xi<\eta_k\}\]
    and thus by compactness $L\subseteq H\cup G_1\cup\cdots\cup G_m$, where $G_i\in{\mathscr H}_{\xi_i}$, $\xi_i<\eta_k$ for any $i=1,\ldots,m$ and $m\in\mathbf{N}$. Now since
    \[\zeta,\xi_1,\ldots,\xi_m,\eta_1,\ldots,\eta_{l-1}<\eta_k<\eta\]
    by our induction assumption the set
    \[{\mathscr F}=\Big\{F\in{\mathscr H}_\zeta:F\subseteq\Big(L\backslash\Big(\bigcup_{i=1}^m G_i\cup\bigcup_{i=1}^{l-1}H_i\Big)\Big)\cup\bigcup\{G\in{\mathscr H}_\xi:\xi<\zeta\}\Big\}\]
    is infinite. If $F\in{\mathscr F}$ then
    \[F\backslash\bigcup\{G\in {\mathscr H}_\xi:\xi<\zeta\}\subseteq L\backslash\Big(\bigcup_{i=1}^m G_i\cup\bigcup_{i=1}^{l-1}H_i\Big)=\Big(L\backslash\bigcup_{i=1}^m G_i\Big)\backslash\bigcup_{i=1}^{l-1}H_i\subseteq H\backslash\bigcup_{i=1}^{l-1}H_i\]
    which together with the fact that $L\in{\mathscr H}_{\eta_k}$ is disjoint from $H_i\in {\mathscr H}_{\eta_i}={\mathscr H}_{\eta_k}$ for any $i=l,\ldots,k$, gives
    \[F\backslash\bigcup\{G\in {\mathscr H}_\xi:\xi<\zeta\}\subseteq H\backslash\bigcup_{i=1}^k H_i.\]
    This shows that $F\in {\mathscr K}$. Therefore  ${\mathscr F}\subseteq {\mathscr K}$ and thus ${\mathscr K}$ is infinite in this case as well.
\end{description}
This proves the claim.

\medskip

\noindent We now return to the main proof. Fix some $k\in J_\sigma$. For any $\zeta\leq\sigma$ and $i\in J_\zeta$ define
\[P_i^\zeta=H_i^\zeta\backslash\bigcup\{H\in {\mathscr H}_\xi:\xi<\zeta\}\]
if $(\zeta,i)\neq(\sigma,k)$ and
\[P_i^\zeta=(\beta X\backslash\lambda_{\mathcal P} X)\backslash\bigcup\{H\in {\mathscr H}_\xi:\xi\leq\sigma\mbox{ and }H\neq H_k^\sigma\}\]
if $(\zeta,i)=(\sigma,k)$.

\begin{xclaim}
The collection $\{P_i^\zeta:\zeta\leq\sigma\mbox{ and }i\in J_\zeta\}$ is bijectively indexed and partitions $\beta X\backslash\lambda_{\mathcal P} X$ into pairwise disjoint non--empty subsets.
\end{xclaim}

\subsubsection*{Proof of the claim}  We first show that
\begin{equation}\label{LOGF}
\bigcup\{P_i^\zeta:\zeta\leq\sigma\mbox{ and }i\in J_\zeta\}=\beta X\backslash\lambda_{\mathcal P} X.
\end{equation}
Let $x\in\beta X\backslash\lambda_{\mathcal P}X$. If $x\notin H_i^\zeta$ for any $\zeta\leq\sigma$ and $i\in J_\zeta$, then clearly $x\in  P_k^\sigma$. If otherwise, then there exists some  $\zeta\leq\sigma$ such that $x\in H_i^\zeta$ for some $i\in J_\zeta$. Let $\zeta$ be the least with this property. If $(\zeta,i)\neq(\sigma,k)$, then by definition it is clear that $x\in P_i^\zeta$. If $(\zeta,i)=(\sigma,k)$ then again $x\in P_i^\zeta$, as the elements of ${\mathscr H}_\zeta$ are pairwise disjoint and $x\notin H$ for any $H\in {\mathscr H}_\eta$ with $\eta<\zeta$.  This shows (\ref{LOGF}). Next, we show that  $P_i^\zeta\cap P_j^\eta=\emptyset$ whenever $\zeta,\eta\leq\sigma$, $i\in J_\zeta$, $j\in J_\eta$ and $(\zeta,i)\neq(\eta,j)$. First suppose that  $(\zeta,i),(\eta,j)\neq(\sigma,k)$. If $\zeta=\eta$ then clearly $P_i^\zeta\cap P_j^\eta\subseteq H_i^\zeta\cap H_j^\eta=\emptyset$. If $\zeta<\eta$ then  $P_i^\zeta\cap P_j^\eta\subseteq H_i^\zeta\cap P_j^\eta=\emptyset$. Similarly, if $\eta<\zeta$. Next, suppose that $(\zeta,i)=(\sigma,k)$. Then $P_i^\zeta\cap P_j^\eta\subseteq P_i^\zeta\cap H_j^\eta=\emptyset$. Similarly, if  $(\eta,j)=(\sigma,k)$. Finally, we verify that $P_i^\zeta$'s are non--empty.  Let $\zeta\leq\sigma$ and $i\in J_\zeta$. If  $(\zeta,i)\neq(\sigma,k)$ then $P_i^\zeta$ is non--empty by (3.c.ii). If  $(\zeta,i)=(\sigma,k)$, then again using (3.c.ii) we have
\[\emptyset\neq H_i^\zeta\backslash\bigcup\{H\in{\mathscr H}_\xi:\xi<\zeta\}\subseteq P_i^\zeta.\]
The fact that  $P_i^\zeta$'s are bijectively indexed is now immediate.

\medskip

\noindent Now let $T$ be the space obtained from $\beta X$ by contracting each $P_i^\zeta$ where $\zeta\leq\sigma$ and $i\in J_\zeta$ to a point $p_i^\zeta$ and denote by $q:\beta X\rightarrow T$ the corresponding quotient mapping. By Lemma \ref{15} we have $X\subseteq\lambda_{\mathcal P} X$. Consider the subspace
\[Y=X\cup\{p_i^\zeta:\zeta\leq\sigma\mbox{ and }i\in J_\zeta\}\]
of $T$. In the remainder of the proof we show that $Y\in{\mathscr O}^{\mathcal Q}_{\mathcal P}(X)$ and that the remainder of  $Y$ is of type $(\sigma,n)$. We first show that $T$ is Hausdorff. Let $s,t\in T$ be distinct. Consider the following cases:
\begin{description}
\item[{\sc Case 1.}] Suppose that $s,t\in T\backslash\{p_i^\zeta:\zeta\leq\sigma\mbox{ and } i\in J_\zeta\}$. Then $s,t\in\lambda_{\mathcal P} X$. Now since $\lambda_{\mathcal P} X$ is open in $\beta X$ and $s$ and $t$ can be separated in $\lambda_{\mathcal P} X$  by disjoint open subsets they can also  be separated by disjoint open subsets in $T$.
\item[{\sc Case 2.}] Suppose that $s\in T\backslash\{p_i^\zeta:\zeta\leq\sigma\mbox{ and } i\in J_\zeta\}$ and $t=p_j^\eta$ for some $\eta\leq\sigma$  and $j\in J_\eta$. Then $s\in\lambda_{\mathcal P} X$. Now if $U$ and $V$ are disjoint open subsets of $\beta X$ containing $s$ and $\beta X\backslash\lambda_{\mathcal P}X$, respectively, then $q[U]$ and $q[V]$ are disjoint open neighborhoods of $s$ and $t$ in $T$, respectively.
\item[{\sc Case 3.}] Suppose that $s=p_i^\zeta$ and $t=p_j^\eta$ for some $\zeta,\eta\leq\sigma$, $i\in J_\zeta$ and $j\in J_\eta$. Without any loss of generality we may assume that $\zeta\leq\eta$ and $(\zeta,i)\neq(\sigma,k)$. Since $H_i^\zeta$ is clopen in  $\beta X\backslash\lambda_{\mathcal P} X$ the sets $H_i^\zeta$ and $(\beta X\backslash\lambda_{\mathcal P} X)\backslash H_i^\zeta$ are compact open subsets of $\beta X\backslash\lambda_{\mathcal P} X$. Let $U$ and $V$ be disjoint open subsets of $\beta X$ such that\[H_i^\zeta=U\cap(\beta X\backslash\lambda_{\mathcal P} X)\mbox{ and }(\beta X\backslash\lambda_{\mathcal P} X)\backslash H_i^\zeta=V\cap(\beta X\backslash\lambda_{\mathcal P}X).\]
    We need to verify the following.

    \begin{xclaim}
    For any $\xi\leq \sigma$ and $l\in J_\xi$ if $P^\xi_l\cap H_i^\zeta$ is non--empty then $P^\xi_l\subseteq H_i^\zeta$.
    \end{xclaim}

    \subsubsection*{Proof of the claim} Suppose that $P^\xi_l\cap H_i^\zeta$ is non--empty for some $\xi\leq \sigma$ and $l\in J_\xi$. This implies that $(\xi,l)\neq(\sigma,k)$, as we are assuming that $(\zeta,i)\neq(\sigma,k)$, and thus by definition $P^\sigma_k\cap H_i^\zeta=\emptyset$. Therefore
    \[P^\xi_l=H^\xi_l\backslash\bigcup\{H\in{\mathscr H}_\alpha:\alpha<\xi\}.\]
    Note that $\zeta<\xi$ implies that $P^\xi_l\cap H_i^\zeta=\emptyset$ and thus $\xi\leq\zeta$. If $\zeta=\xi$, then since $P^\xi_l\cap H_i^\zeta\subseteq H^\xi_l\cap H_i^\zeta$, the latter set is non--empty and therefore $P^\xi_l\subseteq H^\xi_l= H_i^\zeta$. If $\zeta>\xi$, then by (3.c.iii) we either have
    \[H^\xi_l\subseteq H_i^\zeta\cup\bigcup\{H\in{\mathscr H}_\alpha:\alpha<\xi \}\mbox{ or }H^\xi_l\cap H_i^\zeta\subseteq\bigcup\{H\in{\mathscr H}_\alpha:\alpha<\xi \}.\]
    The latter case leads to a contradiction, as $P^\xi_l\cap H_i^\zeta\subseteq H^\xi_l\cap H_i^\zeta$ and
    \[P^\xi_l\cap\bigcup\{H\in{\mathscr H}_\alpha:\alpha<\xi \}=\emptyset.\]
    The first case gives $P^\xi_l\subseteq H_i^\zeta$, which proves the claim.

    \medskip

    \noindent From the claim it follows that $q^{-1}[q[U]]=U$ and $q^{-1}[q[V]]=V$. Thus  $q[U]$ and $q[V]$ are open subsets of $T$ and they are disjoint. It is also clear that $p_i^\zeta\in q[U]$, as $P_i^\zeta\subseteq H_i^\zeta\subseteq U$. It remains to show that $p_j^\eta\in q[V]$. Note that by our assumption  $\zeta\leq\eta$. To show that $P_j^\eta\cap H_i^\zeta=\emptyset$ we consider the following cases:
    \begin{description}
    \item[{\sc Case 3.a.}] Suppose that $\zeta=\eta$ and $(\eta,j)\neq(\sigma,k)$. Then since $p_i^\zeta\neq p_j^\eta$ we have $i\neq j$ and therefore $P_j^\eta\cap H_i^\zeta\subseteq H_j^\eta\cap H_i^\zeta=\emptyset$.
    \item[{\sc Case 3.b.}] Suppose that $\zeta=\eta$ and $(\eta,j)=(\sigma,k)$. By our assumption $(\zeta,i)\neq(\sigma,k)$ or equivalently $H_i^\zeta\neq H_k^\sigma$. Therefore by the definition of $P_j^\eta$ it follows that $P_j^\eta\cap H_i^\zeta=\emptyset$.
    \item[{\sc Case 3.c.}] Suppose that $\zeta<\eta$ and $(\eta,j)\neq(\sigma,k)$. By the definition of $P_j^\eta$ we have  $P_j^\eta\subseteq H_j^\eta\backslash H_i^\zeta$ and thus $P_j^\eta\cap H_i^\zeta=\emptyset$.
    \item[{\sc Case 3.d.}] Suppose that $\zeta<\eta$ and $(\eta,j)=(\sigma,k)$. Then by the definition of $P_j^\eta$ it follows that $P_j^\eta\cap H_i^\zeta=\emptyset$.
    \end{description}
    Thus in each case $P_j^\eta\cap H_i^\zeta=\emptyset$. Therefore $P_j^\eta\subseteq(\beta X\backslash\lambda_{\mathcal P} X)\backslash H_i^\zeta\subseteq V$ and thus $p_j^\eta\in q[V]$.
\end{description}
This shows that $T$ is Hausdorff and therefore compact, being a continuous image of $\beta X$. It is easy to see that $T$ contains $X$ as a dense subspace, and thus since $X\subseteq Y\subseteq T$, it follows that $T$ is a compactification of $Y$ and that $Y$ is a Tychonoff extension of $X$. Also $Y\backslash X=q[\beta X\backslash\lambda_{\mathcal P}X]$ is compact. From these by Lemma \ref{16} we have $Y\in{\mathscr E}^{\mathcal Q}_{\mathcal P}(X)$. Now by Theorem \ref{HG16} and an argument similar to the one in (1.d) $\Rightarrow$ (1.b) it follows that $Y\in{\mathscr O}_{\mathcal P}(X)$. It thus remains to show that $Y\backslash X$ is of type $(\sigma,n)$, that is, $\mbox{card}((Y\backslash X)^{(\sigma)})=n$.  Indeed, we prove the following.

\begin{xclaim}
For any $\zeta\leq\sigma$ we have
\begin{equation}\label{UYNMBJ}
(Y\backslash X)^{(\zeta)}=\{p_j^\eta: \zeta\leq\eta\leq\sigma\mbox{ and  } j\in J_\eta\}.
\end{equation}
\end{xclaim}

\subsubsection*{Proof of the claim} The proof is by transfinite induction on $\zeta$. Note that (\ref{UYNMBJ}) clearly holds  when $\zeta=0$, as by definition  $(Y\backslash X)^{(0)}=Y\backslash X$. Suppose that $0<\alpha\leq\sigma$ and that (\ref{UYNMBJ}) holds for any $\zeta<\alpha$. We show that (\ref{UYNMBJ}) holds  for $\alpha$ as well. Consider the following cases:
\begin{description}
\item[{\sc Case 1.}] Suppose that $\alpha$ is a successor ordinal. Let $\alpha=\gamma+1$. Then by our induction assumption
    \begin{equation}\label{YGH}
    (Y\backslash X)^{(\gamma)}=\{p_j^\eta:\gamma\leq\eta\leq\sigma\mbox{ and }j\in J_\eta\}.
    \end{equation}
    Let $p_j^\eta\in (Y\backslash X)^{(\alpha)}$ where $\eta\leq\sigma$ and $j\in J_\eta$. Since $(Y\backslash X)^{(\alpha)}\subseteq(Y\backslash X)^{(\gamma)}$, by (\ref{YGH}) we have $\gamma\leq\eta$. We show that $\gamma\neq\eta$. Suppose the contrary. Clearly $\eta<\sigma$, as $\eta=\sigma$ implies that $\sigma=\gamma<\alpha$. Let $U$ be an open subset of $\beta X$ such that $H_j^\eta=U\cap(\beta X\backslash\lambda_{\mathcal P} X)$. Then as in the proof of the previous claim, $P_l^\xi\subseteq H_j^\eta$ for any $\xi\leq \sigma$ and $l\in J_\xi$ such that $P_l^\xi\cap H_j^\eta$ is non--empty. Thus $q^{-1}[q[U]]=U$ and therefore $q[U]$ is open in $T$. Since $(\eta,j)\neq(\sigma,k)$, by definition $P_j^\eta\subseteq H_j^\eta\subseteq U$. Thus $q[U]$ is an open neighborhood of $p_j^\eta$ in $T$ and therefore $q[U]\cap(Y\backslash X)^{(\gamma)}$ is infinite. Choose some $s\in q[U]\cap(Y\backslash X)^{(\gamma)}$ such that $s\neq p_j^\eta,p_k^\sigma$. Then by (\ref{YGH}) we have $s=p_l^\xi$ for some $\gamma\leq\xi\leq\sigma$ and $l\in J_\xi$. Since $p_l^\xi\in q[U]$ the set $P_l^\xi\cap H_j^\eta$ is non--empty and thus $P_l^\xi\subseteq H_j^\eta$. Consider the following cases:
    \begin{description}
    \item[{\sc Case 1.a.}] Suppose that $\xi>\gamma$. Then since we are assuming that $\gamma=\eta$, by the definition of $P_l^\xi$ we have $P_l^\xi=P_l^\xi\cap H_j^\eta=\emptyset$, which is a contradiction.
    \item[{\sc Case 1.b.}] Suppose that $\xi=\gamma$. Then $\eta=\gamma=\xi$ and therefore, since by the choice of $s$ we have $p_l^\xi\neq p_j^\eta$, it follows that
        \[P_l^\xi=P_l^\xi\cap H_j^\eta\subseteq H_l^\xi\cap H_j^\eta=\emptyset\]
        which is again a contradiction.
    \end{description}
    Thus in each case we are led to a contradiction which shows that $\gamma<\eta$ or $\alpha=\gamma+1\leq\eta$. Therefore
    \begin{equation}\label{DFG}
    (Y\backslash X)^{(\alpha)}\subseteq\{p_j^\eta: \alpha\leq\eta\leq\sigma\mbox{ and } j\in J_\eta\}.
    \end{equation}
    Next, we show that the reverse inclusion holds in (\ref{DFG}). Consider an element $p_j^\eta$ where $\alpha\leq\eta\leq\sigma$ and $j\in J_\eta$. Let $V$ be an open neighborhood of $p_j^\eta$ in $T$. We show that $V\cap (Y\backslash X)^{(\gamma)}$ is infinite which proves that
    \[p_j^\eta\in\big((Y\backslash X)^{(\gamma)}\big)'=(Y\backslash X)^{(\gamma+1)}=(Y\backslash X)^{(\alpha)}.\]
    First note that
    \begin{equation}\label{YTDFG}
    H_j^\eta\subseteq q^{-1}[V]\cup\bigcup\{H\in {\mathscr H}_\xi:\xi<\eta\}.
    \end{equation}
    This readily follows from the definition of $P_j^\eta$ in the case when $(\eta,j)\neq(\sigma,k)$. If otherwise $(\eta,j)=(\sigma,k)$, note that by the definition of $P_j^\eta$  we have
    \[H_j^\eta\backslash\bigcup\{H\in {\mathscr H}_\xi:\xi<\eta\}\subseteq P_j^\eta\subseteq q^{-1}[V]\]
    From (\ref{YTDFG}) and by compactness there exist $\xi_i<\eta$ and $k_i\in J_{\xi_i}$ where $i=1,\ldots,m$ and $m\in\mathbf{N}$ such that
    \[H^\eta_j\subseteq q^{-1}[V]\cup \bigcup_{i=1}^m H^{\xi_i}_{k_i}.\]
    By the first claim and since $\gamma,\xi_1,\ldots,\xi_m<\eta$ the set
    \[{\mathscr F}=\Big\{F\in{\mathscr H}_\gamma:F\subseteq\Big(H^\eta_j\backslash\bigcup_{i=1}^m H^{\xi_i}_{k_i}\Big)\cup\bigcup\{G\in{\mathscr
    H}_\xi:\xi<\gamma\}\Big\}\]
    is infinite. Now for any $H^\gamma_l\in {\mathscr F}$ where $l\in J_\gamma$, since $(\gamma,l)\neq(\sigma,k)$, as $\gamma<\alpha\leq\sigma$, we have
    \[P^\gamma_l=H^\gamma_l\backslash\bigcup\{G\in{\mathscr H}_\xi:\xi<\gamma\}\subseteq H_j^\eta\backslash\bigcup_{i=1}^m H^{\xi_i}_{k_i}\subseteq q^{-1}[V]\]
    and thus $p^\gamma_l\in V$. Therefore $V\cap(Y\backslash X)^{(\gamma)}$ is infinite. This shows that
    \[p_j^\eta\in\big((Y\backslash X)^{(\gamma)}\big)'=(Y\backslash X)^{(\alpha)}\]
    which proves the reverse inclusion in (\ref{DFG}).
\item[{\sc Case 2.}] Suppose that $\alpha$ is a limit ordinal. We have
    \begin{eqnarray*}
    (Y\backslash X)^{(\alpha)}&=&\bigcap_{\gamma<\alpha}(Y\backslash X)^{(\gamma)}\\&=&\bigcap_{\gamma<\alpha}\{p_j^\eta:\gamma\leq\eta\leq\sigma\mbox{ and } j\in J_\eta\}=\{p_j^\eta:\alpha\leq\eta\leq\sigma\mbox{ and } j\in J_\eta\}.
    \end{eqnarray*}
\end{description}
This completes the inductive proof of the claim.

\medskip

\noindent In particular, we have shown that
\[\mbox{card}\big((Y\backslash X)^{(\sigma)}\big)=\mbox{card}(J_\sigma)=n.\]
Thus $Y\backslash X$ is of type $(\sigma,n)$.
\end{proof}

\begin{xrem}
{\em Note that in Lemma \ref{18} above part (3) implies part (2), however, since the proof for part (3) is quite long and technical, a separate proof is given for part (2).}
\end{xrem}

The characterization given in Lemma \ref{18} is external (to $X$).  Our next theorem is dual to Lemma \ref{18} and gives  an internal characterization of those spaces which have a compactification--like $\mathcal{P}$--extension with finite, countable and countable of type $(\sigma,n)$ remainder. But we need first a few more lemmas.

\begin{lemma}\label{19FG}
Let $X$ be a Tychonoff space and let ${\mathcal P}$ be a clopen hereditary finitely additive perfect topological property. Then for any subset $A$ of $X$ if $\mbox{\em cl}_{\beta X }A\subseteq\lambda_{\mathcal P} X$ then $\mbox{\em cl}_X A\subseteq Z\subseteq C$ for some $Z\in{\mathscr Z}(X)$  and $C\in Coz(X)$ such that  $\mbox{\em cl}_X C$ has ${\mathcal P} $.
\end{lemma}

\begin{proof}
The sets $\mbox{cl}_{\beta X }A$ and $\beta X\backslash\lambda_{\mathcal P} X$ are disjoint closed subsets of $\beta X$ and thus they are completely separated in $\beta X$. Let $f:\beta X\rightarrow\mathbf{I}$ be continuous with $f[\mbox{cl}_{\beta X }A]\subseteq\{0\}$ and $f[\beta X\backslash\lambda_{\mathcal P} X]\subseteq\{1\}$. Let
\[Z=f^{-1}\big[[0,1/3]\big]\cap X\in{\mathscr Z}(X)\mbox{ and }C=f^{-1}\big[[0,1/2)\big]\cap X\in Coz(X).\]
Then $\mbox{cl}_X A\subseteq Z\subseteq C$ and since
\[\mbox{cl}_{\beta X}C=\mbox{cl}_{\beta X}\big(f^{-1}\big[[0,1/2)\big]\cap X\big)=\mbox{cl}_{\beta X}f^{-1}\big[[0,1/2)\big]\subseteq f^{-1}\big[[0,1/2]\big]\subseteq\lambda_{\mathcal P} X\]
by Lemma \ref{B} the set $\mbox{cl}_X C$ has ${\mathcal P} $.
\end{proof}

\begin{lemma}\label{19}
Let $X$ be a Tychonoff space and let $A$ be an infinite compact  countable subset of $X$. Then there exists a bijectively indexed collection $\{V_n:n\in\mathbf{N}\}$  of pairwise disjoint open subsets of $X$ such that $V_n\cap A$ is compact and non--empty for any $n\in\mathbf{N}$.
\end{lemma}

\begin{proof}
We inductively define a sequence $V_1,V_2,\ldots$ of pairwise disjoint open subsets of $X$ such that $V_n\cap A$ is compact and non--empty, $A\backslash \mbox{cl}_X(V_1\cup\cdots\cup V_n)$ is infinite and  $V_n\cap A=\mbox{cl}_XV_n\cap A$ for any $n\in\mathbf{N}$. Let $a,b\in A$ be distinct and let $f:X\rightarrow\mathbf{I}$ be continuous with $f(a)=0$ and $f(b)=1$. Since $f[A]$ is countable there exists some $r\in(0,1)\backslash f[A]$. Either $f^{-1}[[0,r)]\cap A$ or $f^{-1}[(r,1]]\cap A$, say the latter, is infinite. Let $V_1=f^{-1}[[0,r)]$. Since $V_1\cap A=f^{-1}[[0,r]]\cap A$ is closed in $A$, it is compact, and thus since
\[f^{-1}\big[(r,1]\big]\subseteq X\backslash f^{-1}\big[[0,r]\big]\subseteq X\backslash\mbox{cl}_XV_1\]
the set $A\backslash\mbox{cl}_XV_1$ is infinite and
\[V_1\cap A=f^{-1}\big[[0,r]\big]\cap A=\mbox{cl}_X V_1\cap A.\]
Suppose that for some $n\in\mathbf{N}$ the pairwise disjoint open subsets $V_1,\ldots,V_n $ of $X$ are defined such that $A\backslash\mbox{cl}_X(V_1\cup\cdots\cup V_n)$ is infinite, $V_i\cap A$ is compact and non--empty and
$V_i\cap A=\mbox{cl}_X V_i\cap A$ where $i=1,\ldots,n$. Choose some distinct $c,d\in A\backslash\mbox{cl}_X(V_1\cup\cdots\cup V_n)$ and let $g:X\rightarrow\mathbf{I}$ be continuous with $g(c)=0$ and $g(d)=1$. Choose some $s\in(0,1)\backslash g[A]$. Then at least one of
\[\Big(g^{-1}\big[[0,s)\big]\backslash\mbox{cl}_X\Big(\bigcup_{i=1}^n V_i\Big)\Big)\cap A\mbox{ and }\Big(g^{-1}\big[(s,1]\big]\backslash\mbox{cl}_X\Big(\bigcup_{i=1}^n V_i\Big)\Big)\cap A\]
say the latter, is infinite. Define
\[V_{n+1}=g^{-1}\big[[0,s)\big]\backslash\mbox{cl}_X\Big(\bigcup_{i=1}^nV_i\Big).\]
Then $V_1,\ldots,V_{n+1}$ are pairwise disjoint and since $\mbox{cl}_X V_{n+1}\subseteq g^{-1}[[0,s]]$ we have
\[\Big(g^{-1}\big[(s,1]\big]\backslash\mbox{cl}_X\Big(\bigcup_{i=1}^n V_i\Big)\Big)\cap A\subseteq A\backslash\mbox{cl}_X\Big(\bigcup_{i=1}^{n+1} V_i\Big).\]
Therefore $A\backslash\mbox{cl}_X(V_1\cup\cdots\cup V_{n+1})$ is infinite. By the choice of $V_i$'s we have
\[A\backslash\bigcup_{i=1}^nV_i=A\backslash\bigcup_{i=1}^n\mbox{cl}_XV_i.\]
Therefore
\begin{eqnarray*}
\mbox{cl}_XV_{n+1}\cap A&\subseteq&\Big(g^{-1}\big[[0,s]\big]\backslash\bigcup_{i=1}^nV_i\Big)\cap A\\&=&\Big(A\backslash\bigcup_{i=1}^nV_i\Big)\cap g^{-1}\big[[0,s]\big]\\&=&\Big(A\backslash\bigcup_{i=1}^n\mbox{cl}_XV_i\Big)\cap g^{-1}\big[[0,s]\big]\\&=&\Big(A\backslash\bigcup_{i=1}^n\mbox{cl}_XV_i\Big)\cap g^{-1}\big[[0,s)\big]\\&=&\Big(g^{-1}\big[[0,s)\big]\backslash\bigcup_{i=1}^n\mbox{cl}_XV_i\Big)\cap
A=V_{n+1}\cap A
\end{eqnarray*}
which implies that $V_{n+1}\cap A=\mbox{cl}_XV_{n+1}\cap A$ is compact, as it is closed in the compact set $A$. This completes the inductive step.
\end{proof}

Let $X$ be a Tychonoff space and  let $\alpha X$ be a compactification  of $X$. For an open subset $U$ of $X$, the {\em extension of $U$ to $\alpha X$} is defined to be
\[\mbox {Ex}_{\alpha X}U=\alpha X\backslash\mbox {cl}_{\alpha X}(X\backslash U).\]
If $\gamma X$ denotes the Freudenthal compactification of  a rim--compact space  (a space which has a base consisting of open subsets with compact boundary) $X$ then for any open subset $U$ of $X$ we have  $\mbox {cl}_{\gamma X} U\backslash X=\mbox {Ex}_{\gamma X}U\backslash X$ (see \cite{Te}, as mentioned in \cite{Ki}). Using this, in \cite{Ki} the author defined an appropriate upper semicontinuous decomposition  of $\gamma X$ and then  proved  that a locally compact space $X$ has a countable--point compactification if and only if it has a pairwise disjoint sequence $\{U_n:n\in\mathbf{N}\}$ of open subsets each with compact boundary and non--compact closure. Here we only deal with extensions in $\beta X$. Also, we use the simplified notation $\mbox {Ex}_XU$ instead of $\mbox {Ex}_{\beta X} U$. The following lemma is well known (see Lemma 7.1.13 of \cite{E} or  Lemma 3.1 of \cite{vD}).

\begin{lemma}\label{9}
Let $X$ be a Tychonoff space and let $U$ and $V$ be open subsets of $X$. Then
\begin{itemize}
\item[\rm(1)] $X\cap\mbox{\em Ex}_XU=U$ and thus $\mbox{\em cl}_{\beta X}\mbox{\em Ex}_XU=\mbox{\em cl}_{\beta X}U$.
\item[\rm(2)] $\mbox{\em Ex}_X(U\cap V)=\mbox{\em Ex}_XU\cap\mbox{\em Ex}_XV$.
\end{itemize}
\end{lemma}

The following lemma is proved by E.G. Skljarenko in \cite{S}. It is rediscovered by E.K. van Douwen in \cite{vD}.

\begin{lemma}[Skljarenko \cite{S} and van Douwen  \cite{vD}]\label{10}
Let $X$ be a Tychonoff space and let $U$ be an open subset of $X$. Then
\[\mbox {\em bd}_{\beta X}\mbox {\em Ex}_XU=\mbox {\em cl}_{\beta X}\mbox{\em bd}_XU.\]
\end{lemma}

\begin{lemma}\label{j1}
Let $X$ be a Tychonoff space  and let $\mathcal{P}$ be a clopen hereditary topological property which is inverse invariant under perfect mappings. Let $U$ be an open subset of $X$ such that $\mbox{\em bd}_XU\subseteq Z\subseteq C$ where $Z\in {\mathscr Z}(X)$, $C\in Coz(X)$ and $\mbox{\em cl}_XC$ has $\mathcal{P}$. Then
\[\mbox{\em cl}_{\beta X}U\backslash\lambda_{\mathcal{P}} X=\mbox{\em Ex}_X U\backslash\lambda_{\mathcal{P}} X.\]
\end{lemma}

\begin{proof}
By Lemma \ref{BA27} we have $\mbox{cl}_{\beta X}Z\subseteq\lambda_{\mathcal{P}} X$. The result then follows, as by Lemmas \ref{9} and \ref{10} we have
\[\mbox{cl}_{\beta X}U=\mbox{cl}_{\beta X}\mbox{Ex}_XU=\mbox{Ex}_XU\cup\mbox{bd}_{\beta X}\mbox{Ex}_XU=\mbox{Ex}_XU\cup\mbox{cl}_{\beta X}\mbox{bd}_XU\]
and $\mbox{cl}_{\beta X}\mbox{bd}_X U\subseteq\mbox{cl}_{\beta X}Z$.
\end{proof}

In \cite{Mc} the author characterized those  spaces which have a compactification with compact countable remainder of type $(\sigma,n)$ (Theorem \ref{i2}). Indeed, in the proof, for a given space $X$  which satisfies the properties of Theorem \ref{i2}, the author formed a new set $Y$  by adjoining a set of points to $X$ and  he then constructed a topology on $Y$ that turned it into a compactification of $X$ with the desired properties. The proof given in Theorem \ref{20}(3) below can be applied to give an alterative proof to this theorem of J.R. McCartney  in \cite{Mc} (Theorem \ref{i2}). Also, note that parts (1.d) and (2.d) below generalize and give alternative proofs for the theorems of K.D. Magill,  Jr. and T. Kimura in \cite{Mag1} and \cite{Ki}, respectively (Theorems \ref{i7} and \ref{i1}, respectively). One should simply replace  $\mathcal{P}$ and $\mathcal{Q}$, respectively, by compactness and regularity and note that, for any compact subset $A$ of a locally compact space $X$ there exists a continuous $f:X\rightarrow\mathbf{I}$ such that $f[A]\subseteq\{0\}$, and that $f^{-1}[[0,r]]$ is compact for any $r\in(0,1)$.

\begin{theorem}\label{20}
Let ${\mathcal P}$ and  ${\mathcal Q}$ be a pair of compactness--like topological properties. Let $X$  be a Tychonoff  space with $\mathcal{Q}$.
\begin{itemize}
\item[\rm(1)] Let $n\in\mathbf{N}$. The following are equivalent:
\begin{itemize}
\item[\rm(a)] ${\mathscr M}^{\mathcal Q}_{\mathcal P}(X)$ contains an element with $n$--point remainder.
\item[\rm(b)] ${\mathscr O}^{\mathcal Q}_{\mathcal P}(X)$ contains an element with $n$--point remainder.
\item[\rm(c)] $X$ is locally--$\mathcal{P}$ and $X=K\cup U_1\cup\cdots\cup U_n$ where $K,U_1,\ldots,U_n$ are pairwise disjoint, each $U_1,\ldots,U_n$ is open in $X$ with non--$\mathcal{P}$ closure and $\mbox{\em bd}_X K\subseteq Z\subseteq C$ for some $Z\in{\mathscr Z}(X)$ and $C\in Coz(X)$ such that $\mbox{\em cl}_X C$ has $\mathcal{P}$.
\item[\rm(d)] $X$ is locally--$\mathcal{P}$ and $X=U\cup Z_1\cup\cdots\cup Z_n$ where $U,Z_1,\ldots,Z_n$ are pairwise disjoint, $\mbox{\em cl}_X U$ has $\mathcal{P}$ and each $Z_1,\ldots,Z_n\in{\mathscr Z}(X)$ is non--$\mathcal{P}$.
\end{itemize}
\item[\rm(2)] The following are equivalent:
\begin{itemize}
\item[\rm(a)] ${\mathscr M}^{\mathcal Q}_{\mathcal P}(X)$ contains an element with countable remainder.
\item[\rm(b)] ${\mathscr O}^{\mathcal Q}_{\mathcal P}(X)$ contains an element with countable remainder.
\item[\rm(c)] $X$ is locally--$\mathcal{P}$ and there exists a bijectively indexed collection $\{U_n:n\in\mathbf{N}\}$ of pairwise disjoint open subsets of $X$, each with non--$\mathcal{P}$ closure and such that for any $n\in\mathbf{N}$ there exist some $Z_n\in {\mathscr Z}(X)$ and $C_n\in Coz(X)$ such that $\mbox{\em cl}_X C_n$ has $\mathcal{P}$ and $\mbox{\em bd}_X U_n\subseteq Z_n\subseteq C_n$.
\item[\rm(d)] $X$ is locally--${\mathcal P}$ and there exists a bijectively indexed collection $\{Z_n:n\in\mathbf{N}\}$ of non--${\mathcal P}$ zero--sets of $X$ such that $X=Z_1\supseteq Z_2\supseteq\cdots$ and such that for any $n\in\mathbf{N}$ there exist some non--${\mathcal P}$  $S_n\in {\mathscr Z}(X)$ and $K_n\subseteq X$ such that $Z_n\backslash Z_{n+1}=S_n\cup K_n$, where $K_n\subseteq T_n\subseteq C_n$ for some $T_n\in {\mathscr Z}(X)$ and $C_n\in Coz(X)$ such that $\mbox{\em cl}_X C_n$  has ${\mathcal P}$.
\end{itemize}
\item[\rm(3)] Let $0<\sigma<\Omega$ and let $n\in\mathbf{N}$. The following are equivalent:
\begin{itemize}
\item[\rm(a)] ${\mathscr M}^{\mathcal Q}_{\mathcal P}(X)$ contains an element with countable remainder of type $(\sigma,n)$.
\item[\rm(b)] ${\mathscr O}^{\mathcal Q}_{\mathcal P}(X)$ contains an element with countable remainder of type $(\sigma,n)$.
\item[\rm(c)] $X$ is locally--${\mathcal P}$ and there exists a family $\{{\mathscr U}_\zeta:\zeta\leq\sigma\}$ of collections of pairwise disjoint non--empty open subset of $X$ satisfying the following:
\begin{itemize}
\item[\rm(i)] For any $\zeta<\sigma$, $ \mbox{\em card}({\mathscr U}_\zeta)=\aleph_0$ and $ \mbox{\em card}({\mathscr U}_\sigma)=n$.
\item[\rm(ii)] For any $\zeta\leq\sigma$ and $U\in{\mathscr U}_\zeta$ there exist some $Z\in {\mathscr Z}(X)$ and $C\in Coz(X)$ such that $\mbox{\em cl}_X C$ has ${\mathcal P}$ and $\mbox{\em bd}_X U\subseteq Z\subseteq C$.
\item[\rm(iii)] For any $\zeta\leq\sigma$, $U\in{\mathscr U}_\zeta$ and finite ${\mathscr V}\subseteq\bigcup\{{\mathscr U}_\eta:\eta<\zeta\}$ the set $\mbox{\em cl}_X U\backslash\bigcup{\mathscr V}$ is non--${\mathcal P}$.
\item[\rm(iv)] For any $\zeta<\eta\leq\sigma$, $U\in{\mathscr U}_\zeta$ and $V\in{\mathscr U}_\eta$ there exist some $Z\in {\mathscr Z}(X)$ such that $Z$ has ${\mathcal P}$ and a finite ${\mathscr V}\subseteq\bigcup\{{\mathscr U}_\xi:\xi<\zeta\}$ such that either
    \[\mbox{\em cl}_X U\backslash(V\cup\bigcup{\mathscr V})\subseteq Z\mbox{ or }(\mbox{\em cl}_X U\cap\mbox{\em cl}_X V)\backslash\bigcup{\mathscr V}\subseteq Z.\]
\item[\rm(v)] For any $\zeta<\eta\leq\sigma$ and $U\in {\mathscr U}_\eta$ there exists an infinite ${\mathscr V}\subseteq{\mathscr U}_\zeta$ such that for any $V\in{\mathscr V}$ there exist some $Z\in {\mathscr Z}(X)$ which has ${\mathcal P}$ and a finite ${\mathscr W}\subseteq\bigcup\{{\mathscr U}_\xi:\xi<\zeta\}$ such that $\mbox{\em cl}_X V\backslash(U\cup\bigcup{\mathscr W})\subseteq Z$.
\end{itemize}
\end{itemize}
\end{itemize}
\end{theorem}

\begin{proof} (1). By Lemma \ref{18} it follows that (1.a) and (1.b) are equivalent. (1.a) {\em implies} (1.c). Consider some $Y\in{\mathscr M}^{\mathcal Q}_{\mathcal P}(X)$  with an $n$--point remainder $Y\backslash X=\{p_1,\ldots,p_n\}$. Let $V_1,\ldots,V_n$ be pairwise disjoint open neighborhoods of $p_1,\ldots,p_n$ in $\beta Y$, respectively. Let $\phi:\beta X\rightarrow\beta Y$ denote the continuous extension of $\mbox{id}_X$.
Let $i=1,\ldots,n$. Let $f_i:\beta X\rightarrow\mathbf{I}$ be continuous with
\[f_i\big[\phi^{-1}(p_i)\big]\subseteq\{0\}\mbox{ and }f_i\big[\beta X\backslash\phi^{-1}[V_i]\big]\subseteq\{1\}.\]
Let
\[U_i=f_i^{-1}\big[[0,1/2)\big]\cap X\mbox{ and }K=X\backslash\bigcup_{i=1}^n U_i.\]
Then $X=K\cup U_1\cup\cdots\cup U_n$ and  since $U_i\subseteq\phi^{-1}[V_i]$ (and $V_i$'s are pairwise disjoint) the sets $K,U_1,\ldots,U_n$ are pairwise disjoint. To show that $\mbox{cl}_X U_i$ has $\mathcal{P}$ suppose the contrary. Let
\[S=f_i^{-1}\big[[0,1/3]\big]\cap X\in{\mathscr Z}(X).\]
Then $S$ has $\mathcal{P}$, as it is closed in $\mbox{cl}_X U_i$. Therefore
\[\phi^{-1}(p_i)\subseteq f_i^{-1}\big[[0,1/3)\big]\subseteq\mbox{int}_{\beta X}\mbox{cl}_{\beta X}\big(f_i^{-1}\big[[0,1/3]\big]\cap X\big)=\mbox{int}_{\beta X}\mbox{cl}_{\beta X}S\subseteq\lambda_{\mathcal{P}}X.\]
By Lemma \ref{16} the space $X$ is locally--$\mathcal{P}$ and $\beta X\backslash\lambda_{\mathcal{P}} X\subseteq\phi^{-1}[Y\backslash X]$. Now again by Lemma \ref{16} and since by above
\[\beta X\backslash\lambda_{\mathcal{P}} X\subseteq\phi^{-1}[Y\backslash X]\backslash\phi^{-1}(p_i)=\phi^{-1}\big[(Y\backslash X)\backslash\{p_i\}\big]=\phi^{-1}\big[\big(Y\backslash\{p_i\}\big)\backslash X\big],\]
the extension $Y\backslash\{p_i\}$ of $X$ has $\mathcal{P}$. This contradicts the minimality of $Y$. Let
\[Z=X\cap\bigcup_{i=1}^nf_i^{-1}(1/2)\in{\mathscr Z}(X)\mbox{ and }C=X\cap\bigcup_{i=1}^nf_i^{-1}\big[(1/3,2/3)\big]\in Coz(X).\]
Arguing as in the proof of Lemma \ref{18} ((1.a) $\Rightarrow$ (1.c)) we have
\[\phi^{-1}(p_i)\backslash\lambda_{\mathcal{P}} X=\phi^{-1}[V_i]\backslash\lambda_{\mathcal{P}} X.\]
Therefore by the definition of $f_i$ we have
\[\mbox{cl}_{\beta X}C\subseteq\bigcup_{i=1}^nf_i^{-1}\big[[1/3,2/3]\big]\subseteq\bigcup_{i=1}^n\big(\phi^{-1}[V_i]\backslash\phi^{-1}(p_i)\big)
\subseteq\lambda_{\mathcal{P}}X\]
which by Lemma \ref{B} implies that $\mbox{cl}_XC$ has $\mathcal{P}$. Finally
\begin{eqnarray*}
\mbox{bd}_X K&=&\mbox{cl}_XK\cap\mbox{cl}_X(X\backslash K)\\&\subseteq&
X\cap\bigcap_{j=1}^nf_j^{-1}\big[[1/2,1]\big]\cap\bigcup_{i=1}^nf_i^{-1}\big[[0,1/2]\big]\\&=& X\cap\bigcup_{i=1}^n\bigcap_{j=1}^n f_j^{-1}\big[[1/2,1]\big]\cap f_i^{-1}\big[[0,1/2]\big]\subseteq X\cap\bigcup_{i=1}^nf_i^{-1}(1/2)
\end{eqnarray*}
which implies that $\mbox{bd}_X K\subseteq Z\subseteq C$.

(1.c) {\em implies} (1.d). First note that since $U_1,\ldots,U_n$ are pairwise disjoint and open, for any distinct $i,j=1,\ldots,n$ we have  $\mbox{bd}_X U_i\cap U_j=\emptyset$. Let $i=1,\ldots,n$. Then
\[\mbox{bd}_X U_i\subseteq(X\backslash U_i)\cap\bigcap_{i\neq j=1}^n(X\backslash U_j)=X\backslash\bigcup_{j=1}^nU_j=K. \]
which combined with $\mbox{bd}_X U_i\subseteq\mbox{cl}_X U_i\subseteq\mbox{cl}_X (X\backslash K)$ gives $\mbox{bd}_X U_i\subseteq \mbox{bd}_X K$. By Lemma \ref{j1} this implies that $\mbox{cl}_{\beta X}U_i\backslash\lambda_{\mathcal{P}} X=\mbox{Ex}_X U_i\backslash\lambda_{\mathcal{P}} X$; let $H_i$ denote the latter set. By Lemma \ref{B} the set $H_i$ is non--empty, as by our assumption $\mbox{cl}_X U_i$ is non--$\mathcal{P}$. Let $f_i:\beta X\rightarrow\mathbf{I}$ be continuous with
\[f_i[H_i]\subseteq\{0\}\mbox{ and }f_i[\beta X\backslash\mbox{Ex}_X U_i]\subseteq\{1\}\]
if $i=1,\ldots,n-1$ and
\[f_n\Big[(\beta X\backslash\lambda_{\mathcal{P}} X)\backslash\bigcup_{i=1}^{n-1}H_i\Big]\subseteq\{0\}\mbox{ and }f_n\Big[\bigcup_{i=1}^{n-1} \mbox{cl}_{\beta X}U_i\Big]\subseteq\{1\}.\]
Let
\[Z_i=f_i^{-1}\big[[0,1/2]\big]\cap X\in {\mathscr Z}(X)\mbox{ and }U=X\backslash\bigcup_{i=1}^n Z_i.\]
By the definition of $f_i$ we have  $Z_i\subseteq\mbox{Ex}_X U_i$ for any $i=1,\ldots,n-1$. Now since $\mbox{Ex}_X U_i$'s are pairwise disjoint (as  $U_i$'s are; see Lemma \ref{9})  $Z_i$'s are pairwise disjoint when $i=1,\ldots,n-1$, and therefore when $i=1,\ldots,n$, as
\[Z_n\cap\bigcup_{i=1}^{n-1}\mbox{cl}_{\beta X}U_i=\emptyset\]
and $Z_i\subseteq\mbox{cl}_{\beta X}U_i$  for any $i=1,\ldots,n-1$, as $\mbox{Ex}_X U_i\subseteq\mbox{cl}_{\beta X}U_i$ (see Lemma \ref{9}). Since
\[U=X\backslash\bigcup_{i=1}^nf_i^{-1}\big[[0,1/2]\big]\subseteq \beta X\backslash \bigcup_{i=1}^nf_i^{-1}\big[[0,1/2)\big]\subseteq\lambda_{\mathcal{P}}X\]
we have $\mbox{cl}_{\beta X}U\subseteq\lambda_{\mathcal{P}} X$, and thus by Lemma \ref{B} it follows that $\mbox{cl}_XU $ has $\mathcal{P}$. To complete the proof we need to verify that $Z_i$ is non--$\mathcal{P}$. But this follows easily, as otherwise
\begin{eqnarray*}
H_i\subseteq f_i^{-1}\big[[0,1/2)\big]\subseteq\mbox{int}_{\beta X}\mbox{cl}_{\beta X}\big(f_i^{-1}\big[[0,1/2]\big]\cap X\big)=\mbox{int}_{\beta X}\mbox{cl}_{\beta X}Z_i\subseteq\lambda_{\mathcal{P}}X
\end{eqnarray*}
which contradicts the fact $H_i$ is non--empty.

(1.d) {\em  implies} (1.a). For any $i=1,\ldots,n$ let $f_i:X\rightarrow \mathbf{I}$ be continuous with
\[f_i[Z_i]\subseteq\{0\}\mbox{ and }f_i\Big[\bigcup_{k=1}^{i-1}Z_k\cup\bigcup_{k=i+1}^n Z_k\Big]\subseteq\{1\}\]
and let $F_i:\beta X\rightarrow\mathbf{I}$ be the continuous extension of $f_i$.

\begin{xclaim}
Let $r\in(0,1)$ and let  $i,j=1,\ldots,n$ be distinct. Then
\[\mbox{\em cl}_{\beta X}f_i^{-1}\big[[0,r]\big]\cap\mbox{\em cl}_{\beta X}f_j^{-1}\big[[0,r]\big]\subseteq\lambda_{\mathcal{P}} X.\]
\end{xclaim}

\subsubsection*{Proof of the claim} Let $r<s<1$. By the definition of $f_i$'s and since $Z_i\cap Z_j=\emptyset$ we have
\[S=f_i^{-1}\big[[0,s]\big]\cap f_j^{-1}\big[[0,s]\big]\subseteq (Z_i\cup U)\cap (Z_j\cup U)\subseteq U.\]
If $k=i,j$ then
\[\mbox{cl}_{\beta X}f_k^{-1}\big[[0,r]\big]\subseteq F_k^{-1}\big[[0,r]\big]\subseteq F_k^{-1}\big[[0,s)\big]\subseteq\mbox{int}_{\beta X}\mbox{cl}_{\beta X} f_k^{-1}\big[[0,s]\big]\]
and therefore
\begin{eqnarray*}
\mbox{cl}_{\beta X}f_i^{-1}\big[[0,r]\big]\cap\mbox{cl}_{\beta X}f_j^{-1}\big[[0,r]\big]&\subseteq&\mbox{int}_{\beta X}\mbox{cl}_{\beta X} f_i^{-1}\big[[0,s]\big]\cap\mbox{int}_{\beta X}\mbox{cl}_{\beta X} f_j^{-1}\big[[0,s]\big]\\&=&\mbox{int}_{\beta X}\big(\mbox{cl}_{\beta X} f_i^{-1}\big[[0,s]\big]\cap\mbox{cl}_{\beta X} f_j^{-1}\big[[0,s]\big]\big)\\&=&\mbox{int}_{\beta X}\mbox{cl}_{\beta X}\big(f_i^{-1}\big[[0,s]\big]\cap f_j^{-1}\big[[0,s]\big]\big)\\&=&\mbox{int}_{\beta X}\mbox{cl}_{\beta X}S.
\end{eqnarray*}
Note that $S\in{\mathscr Z}(X)$ has $\mathcal{P}$, as it is closed in $\mbox{cl}_X U$, and  thus $\mbox{int}_{\beta X}\mbox{cl}_{\beta X}S\subseteq\lambda_{\mathcal{P}}X$.

\begin{xclaim}
Let $r\in (0,1)$. Then
\[\mbox{\em cl}_{\beta X}\Big(\bigcap_{i=1}^nf_i^{-1}\big[[r,1]\big]\Big)\subseteq\lambda_{\mathcal{P}} X.\]
\end{xclaim}

\subsubsection*{Proof of the claim} Let $0<t<r$ and let
\[T=\bigcap_{i=1}^nf_i^{-1}\big[[t,1]\big]\in{\mathscr Z}(X).\]
Then
\begin{eqnarray*}
\mbox{cl}_{\beta X}\Big(\bigcap_{i=1}^nf_i^{-1}\big[[r,1]\big]\Big)\subseteq\bigcap_{i=1}^nF_i^{-1}\big[[r,1]\big]&\subseteq&\bigcap_{i=1}^nF_i^{-1}\big[(t,1]\big]\\&\subseteq&
\bigcap_{i=1}^n\mbox{int}_{\beta X}\mbox{cl}_{\beta X}f_i^{-1}\big[[t,1]\big]\\&=&\mbox{int}_{\beta X}\Big(\bigcap_{i=1}^n\mbox{cl}_{\beta X}f_i^{-1}\big[[t,1]\big]\Big)\\&=&\mbox{int}_{\beta X}\mbox{cl}_{\beta X}\Big(\bigcap_{i=1}^nf_i^{-1}\big[[t,1]\big]\Big)\\&=&\mbox{int}_{\beta X}\mbox{cl}_{\beta X}T
\end{eqnarray*}
Now since
\[T=\bigcap_{i=1}^nf_i^{-1}\big[[t,1]\big]\subseteq\bigcap_{i=1}^n(X\backslash Z_i)\subseteq X\backslash\bigcup_{i=1}^nZ_i=U\]
and $\mbox{cl}_X U$ has $\mathcal{P}$, its closed subset $T$ has $\mathcal{P}$, and therefore $\mbox{int}_{\beta X}\mbox{cl}_{\beta X}T\subseteq\lambda_{\mathcal{P}} X$.

\medskip

\noindent Now let $r\in (0,1)$ be fixed. Note that
\[\beta X\backslash\lambda_{\mathcal{P}} X=\Big(\mbox{cl}_{\beta X}\Big(\bigcup_{i=1}^nf_i^{-1}\big[[0,r]\big]\Big)\cup\mbox{cl}_{\beta X}\Big(\bigcap_{i=1}^nf_i^{-1}\big[[r,1]\big]\Big)\Big)\backslash\lambda_{\mathcal{P}} X\]
and thus by the above claim
\[\beta X\backslash\lambda_{\mathcal{P}} X=\mbox{cl}_{\beta X}\Big(\bigcup_{i=1}^nf_i^{-1}\big[[0,r]\big]\Big)\backslash\lambda_{\mathcal{P}}X
=\bigcup_{i=1}^n\big(\mbox{cl}_{\beta X}f_i^{-1}\big[[0,r]\big]\backslash\lambda_{\mathcal{P}}X\big).\]
By the first claim it now follows that $\beta X\backslash\lambda_{\mathcal{P}} X$ is the union of $n$ of its pairwise disjoint closed (and thus clopen) subsets which are also non--empty, as $Z_i\subseteq f_i^{-1}[[0,r]]$ for any $i=1,\ldots,n$, and therefore since $Z_i$ is non--$\mathcal{P}$, using Lemma \ref{B} we have
\[\emptyset\neq\mbox{cl}_{\beta X}Z_i\backslash\lambda_{\mathcal{P}} X\subseteq\mbox{cl}_{\beta X}f_i^{-1}\big[[0,r]\big]\backslash\lambda_{\mathcal{P}}X.\]

(2) (2.a) {\em implies} (2.c). Consider some $Y\in{\mathscr M}^{\mathcal Q}_{\mathcal P}(X)$ with countable remainder. By Lemma  \ref{19} there exists a bijectively indexed collection $\{V_n:n\in\mathbf{N}\}$ of pairwise disjoint open subsets of $\beta Y$ such that  $B_n=V_n\cap(Y\backslash X)$ is compact and non--empty for any $n\in\mathbf{N}$. Let $n\in\mathbf{N}$. Let $f_n:\beta X\rightarrow\mathbf{I}$ be continuous with
\[f_n\big[\phi^{-1}[B_n]\big]\subseteq\{0\}\mbox{ and }f_n\big[\beta X\backslash\phi^{-1}[V_n]\big]\subseteq\{1\}\]
where $\phi:\beta X\rightarrow\beta Y$ is the continuous extension of $\mbox{id}_X$ and define $U_n=f_n^{-1}[[0,1/2)]\cap X$. Then $U_n$'s are pairwise disjoint open subsets of $X$. Also, $\mbox{cl}_X U_n$ is non--$\mathcal{P}$, as otherwise, arguing as in (1.a) $\Rightarrow$ (1.c) we have $f^{-1}_n[[0,1/3)]\subseteq\lambda_{\mathcal{P}}X$. Consider the subspace $Y'=Y\backslash B_n$ of $Y$. By Lemma \ref{16} we have $\beta X\backslash\lambda_{\mathcal{P}} X\subseteq\phi^{-1}[Y\backslash X]$ and that $X$ is locally--${\mathcal{P}}$. Since $\phi^{-1}[B_n]\subseteq f^{-1}_n[[0,1/3)]$ we have
\begin{eqnarray*}
\phi^{-1}[Y'\backslash X]=\phi^{-1}\big[(Y\backslash B_n)\backslash X\big]&=&\phi^{-1}\big[(Y\backslash X)\backslash B_n\big]\\&=&\phi^{-1}[Y\backslash X]\backslash\phi^{-1}[B_n]\supseteq\beta X\backslash\lambda_{\mathcal{P}}X.
\end{eqnarray*}
Now since $Y'\backslash X=(Y\backslash X)\backslash V_n$ is compact, Lemma \ref{16} implies that $Y'$ has $\mathcal{P}$, contradicting  the minimality of $Y$. Note that  \[\mbox{bd}_X U_n=\mbox{cl}_X U_n\cap \mbox{cl}_X(X\backslash U_n)\subseteq f_n^{-1}\big[[0,1/2]\big]\cap f_n^{-1}\big[[1/2,1]\big]\cap X=f_n^{-1}(1/2)\cap X.\]
Therefore if
\[Z_n=f_n^{-1}(1/2)\cap X\in{\mathscr Z}(X)\mbox{ and }C_n=f_n^{-1}\big[(1/3,2/3)\big]\cap X\in Coz(X)\]
then $\mbox{bd}_X U_n\subseteq Z_n\subseteq C_n$.
Since $\beta X\backslash\lambda_{\mathcal{P}} X\subseteq\phi^{-1}[Y\backslash X]$, arguing as in the proof of Lemma \ref{18} ((1.a) $\Rightarrow$ (1.c)), we have
\[\phi^{-1}[B_n]\backslash\lambda_{\mathcal{P}} X=\phi^{-1}[V_n]\backslash\lambda_{\mathcal{P}} X\]
and thus by the definition of $f_n$ it follows that
\[\mbox{cl}_{\beta X}C_n\subseteq f_n^{-1}\big[[1/3,2/3]\big]\subseteq\phi^{-1}[V_n]\backslash\phi^{-1}[B_n]\subseteq\lambda_{\mathcal{P}} X.\]
By Lemma \ref{B} this  implies that $\mbox{cl}_XC_n$ has $\mathcal{P}$.

(2.c) {\em  implies} (2.a). Let $\{U_n:n\in\mathbf{N}\}$ satisfy the assumption of the theorem. Let $n\in\mathbf{N}$. By Lemma \ref{j1} we have $\mbox{cl}_{\beta X}U_n\backslash\lambda_{\mathcal{P}} X=\mbox{Ex}_X U_n\backslash\lambda_{\mathcal{P}} X$; let $H_n$ denote the latter set. Note that $H_n$ is clopen in $\beta X\backslash\lambda_{\mathcal{P}} X$ and that $H_n$'s are pairwise disjoint, as $U_n$'s are (see Lemma \ref{9}). Also, $H_n$ is non--empty, as otherwise $\mbox{cl}_{\beta X}U_n\subseteq\lambda_{\mathcal{P}} X$ which by Lemma \ref{B} implies that  $\mbox{cl}_XU_n$ has $\mathcal{P}$, contradicting our assumption.

(2.a) {\em implies} (2.d). By Lemma \ref{18} the space $X$ is locally--${\mathcal P}$ and there exists a bijectively indexed sequence $H_1,H_2,\ldots$ of pairwise  disjoint non--empty clopen subsets of $\beta X\backslash\lambda_{\mathcal P} X$. Let $n\in\mathbf{N}$. Let $f_n:\beta X\rightarrow\mathbf{I}$ be continuous with
\[f_n\big[(\beta X\backslash\lambda_{\mathcal P} X)\backslash H_n\big]\subseteq\{0\}\mbox{ and }f_n[H_n]\subseteq\{1\}\]
and let
\[Z_n=\bigcap_{i=1}^{n-1}f_i^{-1}\big[[0,1/2]\big]\cap X\in {\mathscr Z}(\beta X)\]
with the empty intersection interpreted as $\beta X$. Clearly $X=Z_1\supseteq Z_2\supseteq\cdots$. Note that
\begin{equation}\label{HCO}
H_n\subseteq(\beta X\backslash\lambda_{\mathcal P} X)\backslash\bigcup_{i=1}^{n-1} H_i\subseteq\bigcap_{i=1}^{n-1}\big((\beta X\backslash\lambda_{\mathcal P} X)\backslash H_i\big)\subseteq\bigcap_{i=1}^{n-1}f_i^{-1}\big[[0,1/2)\big].
\end{equation}
To show that $Z_n$ is non--${\mathcal P}$, suppose the contrary. Then
\begin{eqnarray*}
\bigcap_{i=1}^{n-1}f_i^{-1}\big[[0,1/2)\big]&\subseteq&\bigcap_{i=1}^{n-1}\mbox{int}_{\beta X}\mbox{cl}_{\beta X}\big(f_i^{-1}\big[[0,1/2]\big]\cap X\big)\\&=&\mbox{int}_{\beta X}\Big(\bigcap_{i=1}^{n-1}\mbox{cl}_{\beta X}\big(f_i^{-1}\big[[0,1/2]\big]\cap X\big)\Big)\\&=&\mbox{int}_{\beta X}\mbox{cl}_{\beta X}\Big(\bigcap_{i=1}^{n-1}\big(f_i^{-1}\big[[0,1/2]\big]\cap X\big)\Big)=\mbox{int}_{\beta X}\mbox{cl}_{\beta X}Z_n\subseteq\lambda_{\mathcal P} X
\end{eqnarray*}
and therefore $H_n\subseteq\lambda_{\mathcal P} X$, which is a contradiction, as $H_n\subseteq\beta X\backslash\lambda_{\mathcal P} X$ is non--empty. Now
\begin{eqnarray*}
Z_n\backslash Z_{n+1}&=&\Big(\bigcap_{i=1}^{n-1}f_i^{-1}\big[[0,1/2]\big]\cap X\Big)\backslash\Big(\bigcap_{i=1}^nf_i^{-1}\big[[0,1/2]\big]\cap X\Big)\\&=&\big(\beta X\backslash f_n^{-1}\big[[0,1/2]\big]\big)\cap\bigcap_{i=1}^{n-1}f_i^{-1}\big[[0,1/2]\big]\cap X\\&=&f_n^{-1}\big[(1/2,1]\big]\cap\bigcap_{i=1}^{n-1}f_i^{-1}\big[[0,1/2]\big]\cap X\\&=&\big(f_n^{-1}\big[(1/2,2/3)\big]\cup f_n^{-1}\big[[2/3,1]\big]\big)\cap\bigcap_{i=1}^{n-1}f_i^{-1}\big[[0,1/2]\big]\cap X=K_n\cup S_n
\end{eqnarray*}
where
\[K_n=f_n^{-1}\big[(1/3,2/3)\big]\cap\bigcap_{i=1}^{n-1}f_i^{-1}\big[[0,1/2]\big]\cap X\]
and
\[S_n=f_n^{-1}\big[[2/3,1]\big]\cap\bigcap_{i=1}^{n-1}f_i^{-1}\big[[0,1/2]\big]\cap X.\]
Then $S_n\in{\mathscr Z}(X)$ and $K_n\subseteq T_n\subseteq C_n$, where
\[T_n=f_n^{-1}\big[[1/2,2/3]\big]\cap X\in{\mathscr Z}(X)\mbox{ and }C_n=f_n^{-1}\big[(1/3,3/4)\big]\cap X\in Coz(X).\]
Since
\[\mbox{cl}_{\beta X}C_n=\mbox{cl}_{\beta X}\big(f_n^{-1}\big[(1/3,3/4)\big]\cap X\big)\subseteq f_n^{-1}\big[[1/3,3/4]\big]\subseteq\lambda_{\mathcal P}X\]
by Lemma \ref{B} the set $\mbox{cl}_X C_n$  has  ${\mathcal P}$. If $S_n$ has ${\mathcal P}$, then using (\ref{HCO}) and arguing as above
\[H_n\subseteq f_n^{-1}\big[(2/3,1]\big]\cap\bigcap_{i=1}^{n-1}f_i^{-1}\big[[0,1/2)\big]\subseteq\mbox{int}_{\beta X}\mbox{cl}_{\beta X} S_n\subseteq\lambda_{\mathcal P} X\]
which as we argued above is a contradiction. Thus $S_n$ is  non--${\mathcal P}$. Finally, note that as argued above $H_n\subseteq\mbox{cl}_{\beta X}S_n$, and therefore $S_n$ is non--empty. This implies that $Z_n\backslash Z_{n+1}$ is non--empty and thus $Z_n$'s are bijectively indexed.

(2.d) {\em implies} (2.a). Let $n\in\mathbf{N}$. Then
\[X=Z_1=\bigcup_{i=1}^{n-1}(Z_i\backslash Z_{i+1})\cup Z_n=\bigcup_{i=1}^{n-1}(S_i\cup K_i)\cup Z_n\]
and thus
\[\beta X=\bigcup_{i=1}^{n-1}(\mbox{cl}_{\beta X}S_i\cup\mbox{cl}_{\beta X} K_i)\cup\mbox{cl}_{\beta X}Z_n.\]
By Lemma \ref{BA27} we have $\mbox{cl}_{\beta X} K_i\subseteq\mbox{cl}_{\beta X} T_i\subseteq\lambda_{\mathcal P} X$ for any $i\in\mathbf{N}$. Therefore
\[\beta X\backslash\lambda_{\mathcal P} X=\bigcup_{i=1}^{n-1}(\mbox{cl}_{\beta X}S_i\backslash\lambda_{\mathcal P} X)\cup(\mbox{cl}_{\beta X}Z_n\backslash\lambda_{\mathcal P} X).\]
Since $S_1,\ldots,S_{n-1},Z_n\in{\mathscr Z}(X)$ are pairwise disjoint, their closures in $\beta X$ also are pairwise disjoint.
Thus by above $\beta X\backslash\lambda_{\mathcal P} X$ is the union of $n$ of its pairwise disjoint non--empty (as $S_1,\ldots,S_{n-1},Z_n$ are non--${\mathcal P}$; see Lemma \ref{B}) closed (and therefore clopen) subsets. Lemma \ref{18} now completes the proof.

(3). The equivalence of (3.a) and (3.b) follows from Lemma \ref{18}. (3.a) {\em implies} (3.c). By Lemma \ref{18} the space $X$ is locally--${\mathcal P}$ and there exists a family $\{{\mathscr H}_\zeta:\zeta\leq\sigma\}$ of collections of pairwise disjoint non--empty clopen subsets of $\beta X\backslash\lambda_{\mathcal P} X$ satisfying conditions (3.c.i)--(3.c.iv) of Lemma \ref{18}. For any $\zeta\leq\sigma$ let ${\mathscr H}_\zeta=\{ H^\zeta_i:i\in J_\zeta\}$ be bijectively indexed. Let $J_\zeta=\mathbf{N}$ for any $\zeta<\sigma$ and $J_\sigma=\{1,\ldots,n\}$. Also, for any $\zeta\leq\sigma$ and $i\in J_\zeta$ let $A_i^\zeta$ be an open subset of $\beta X$ such that $H_i^\zeta=A_i^\zeta\backslash\lambda_{\mathcal P} X$.

\begin{xclaim}
For any $\zeta\leq\sigma$ there exists a collection $\{W_i^\zeta:i\in J_\zeta\}$ of pairwise disjoint open subsets of $\beta X$ such that $W^\zeta_i\backslash\lambda_{\mathcal P} X=H_i^\zeta$ for any $i\in J_\zeta$.
\end{xclaim}

\subsubsection*{Proof of the claim} Let $\zeta<\sigma$. We inductively define $W_i^\zeta$'s for $i\in J_\zeta$.
Let $W^\zeta_1$ be an open subset of $\beta X$ such that $H_1^\zeta\subseteq W_1^\zeta\subseteq\mbox{cl}_{\beta X} W_1^\zeta\subseteq A_1^\zeta$. For an $m\in\mathbf{N}$, suppose inductively that the open subsets $W_1^\zeta,\ldots,W_m^\zeta$ of $\beta X$ are defined in such a way that
\[H_i^\zeta\subseteq W_i^\zeta\subseteq\mbox{cl}_{\beta X}W_i^\zeta\subseteq A_i^\zeta\backslash\mbox{cl}_{\beta X}\Big(\bigcup_{j=1}^{i-1} W_j^\zeta\Big)\]
for any $i=1,\ldots,m$. Note that $H_{m+1}^\zeta\cap\mbox{cl}_{\beta X} W_i^\zeta=\emptyset$ for any $i=1,\ldots,m$. Let $W_{m+1}^\zeta$ be an open subset of $\beta X$ such that
\[H_{m+1}^\zeta\subseteq W_{m+1}^\zeta\subseteq\mbox{cl}_{\beta X} W_{m+1}^\zeta\subseteq A_{m+1}^\zeta\backslash\mbox{cl}_{\beta X}\Big (\bigcup_{j=1}^{m} W_j^\zeta\Big).\]
Similarly for the case when $\zeta=\sigma$.

\medskip

\noindent Let $ \zeta\leq\sigma$ and $i\in J_\zeta$. Since $W_i^\zeta$ is an open neighborhoods of the compact (and thus closed) subset $H_i^\zeta$ of $\beta X$, there exists a continuous $f_i^\zeta:\beta X\rightarrow\mathbf{I}$ with
\[f_i^\zeta[H_i^\zeta]\subseteq\{0\}\mbox{ and }f_i^\zeta[\beta X\backslash W_i^\zeta]\subseteq\{1\}.\]
Define
\[U_i^\zeta=(f_i^\zeta)^{-1}\big[[0,1/2)\big]\cap X.\]
Let ${\mathscr U}_\zeta=\{U^\zeta_i:i\in J_\zeta\}$ for any $\zeta\leq\sigma$. We verify that the family $\{{\mathscr U}_\zeta:\zeta\leq\sigma\}$ satisfies (3.c.i)--(3.c.v). Let $\zeta\leq\sigma$. Let $i\in J_\zeta$. Since
\[H_i^\zeta\subseteq(f_i^\zeta)^{-1}\big[[0,1/2)\big]\subseteq\mbox{cl}_{\beta X}(f_i^\zeta)^{-1}\big[[0,1/2)\big]=\mbox{cl}_{\beta X}\big((f_i^\zeta)^{-1}\big[[0,1/2)\big]\cap X\big)=\mbox{cl}_{\beta X}U_i^\zeta\]
and $H_i^\zeta$ is non--empty, $U_i^\zeta$  is non--empty. Also, for any distinct $i,j\in J_\zeta$ we have $U_i^\zeta\cap U_j^\zeta\subseteq W_i^\zeta\cap W_j^\zeta=\emptyset$. Thus the collection ${\mathscr U}_\zeta=\{U^\zeta_i:i\in J_\zeta\}$ is bijectively indexed and consists of pairwise disjoint non--empty open subsets of $X$. Note that (3.c.i) holds trivially.

To show (3.c.ii) note that for any $\zeta\leq\sigma$ and $i\in J_\zeta$  we have
\[\mbox{bd}_X U_i^\zeta=\mbox{cl}_X U_i^\zeta\cap\mbox{cl}_X(X\backslash U_i^\zeta)\subseteq(f_i^\zeta)^{-1}\big[[0,1/2]\big]\cap (f_i^\zeta)^{-1}\big[[1/2,1]\big]\cap X\subseteq(f_i^\zeta)^{-1}(1/2).\]
Thus
\[\mbox{cl}_{\beta X}\mbox{bd}_X U_i^\zeta\backslash\lambda_{\mathcal P} X\subseteq(f_i^\zeta)^{-1}(1/2)\backslash\lambda_{\mathcal P} X\subseteq (W_i^\zeta\backslash H_i^\zeta)\backslash\lambda_{\mathcal P} X=\emptyset\]
which by Lemma \ref{19FG} shows (3.c.ii).

Next, we verify (3.c.iii). Let $\zeta\leq\sigma$, $U\in{\mathscr U}_\zeta$ and ${\mathscr V}\subseteq\bigcup\{{\mathscr U}_\eta:\eta<\zeta\}$ be finite. Suppose to the contrary that $\mbox{cl}_X U\backslash\bigcup{\mathscr V}$ has ${\mathcal P}$. Let $U=U_i^\zeta$ for some $i\in J_\zeta$ and ${\mathscr V}=\{U^{\eta_1}_{j_1},\ldots,U^{\eta_m}_{j_m}\}$, where $m$ is a non--negative integer, and $\eta_k<\zeta$ and $j_k\in J_{\eta_k}$ for any $k=1,\ldots,m$. Let
\[S=\Big((f_i^\zeta)^{-1}\big[[0,1/3]\big]\backslash\bigcup_{k=1}^m(f_{j_k}^{\eta_k})^{-1}\big[[0,1/2)\big]\Big)\cap X\in{\mathscr Z}(X).\]
By above $(f_i^\zeta)^{-1}[[0,1/3]]\cap X\subseteq\mbox{cl}_X U_i^\zeta$. Thus $S\subseteq \mbox{cl}_X U\backslash\bigcup{\mathscr V}$ and then $S$, being closed in the latter, has ${\mathcal P}$. This implies that
\[H_i^\zeta\backslash\bigcup_{k=1}^m W_{j_k}^{\eta_k}\subseteq (f_i^\zeta)^{-1}\big[[0,1/3)\big]\backslash\bigcup_{k=1}^m (f_{j_k}^{\eta_k})^{-1}\big[[0,1/2]\big]\subseteq\mbox{int}_{\beta X}\mbox{cl}_{\beta X} S\subseteq\lambda_{\mathcal P} X.\]

\noindent Therefore
\[H_i^\zeta\backslash\bigcup\{G\in {\mathscr H}_\eta:\eta<\zeta\}\subseteq H_i^\zeta\backslash\bigcup_{k=1}^m H_{j_k}^{\eta_k}= \Big(H_i^\zeta\backslash\bigcup_{k=1}^m W_{j_k}^{\eta_k}\Big)\cap(\beta X\backslash\lambda_{\mathcal P} X)=\emptyset\]
which contradicts (3.c.ii) of Lemma \ref{18}. Thus (3.c.iii) holds.

To show (3.c.iv) let $\zeta<\eta\leq\sigma$, $U\in{\mathscr U}_\zeta$ and $V\in  {\mathscr U}_\eta$. Then $U=U_i^\zeta$ and $V=U_j^\eta$ for some $i\in J_\zeta$ and $j\in J_\eta$. By (3.c.iii) of Lemma \ref{18} we either have
\begin{equation}\label{FGHJ}
H_i^\zeta\subseteq H_j^\eta\cup\bigcup\{F\in {\mathscr H}_\xi:\xi<\zeta\}
\end{equation}
or
\begin{equation}\label{HVG}
H_i^\zeta\cap H_j^\eta\subseteq\bigcup\{F\in {\mathscr H}_\xi:\xi<\zeta\}.
\end{equation}
We consider the following cases. For simplicity of notation let $r=1/2$. Note that by the definition of $f_i^\zeta$ for any $\zeta\leq\sigma$ and $i\in J_\zeta$ we have
\[H_i^\zeta\subseteq(f_i^\zeta)^{-1}\big[[0,r)\big]\backslash\lambda_{\mathcal P}X\subseteq(f_i^\zeta)^{-1}\big[[0,r]\big]\backslash\lambda_{\mathcal P}X\subseteq W_i^\zeta\backslash\lambda_{\mathcal P}X=H_i^\zeta\]
and therefore
\[(f_i^\zeta)^{-1}\big[[0,r]\big]\backslash\lambda_{\mathcal P}X=(f_i^\zeta)^{-1}\big[[0,r)\big]\backslash\lambda_{\mathcal P}X=H_i^\zeta.\]
\begin{description}
\item[{\sc Case 1.}] If (\ref{FGHJ}) holds then by compactness $H_i^\zeta\subseteq H_j^\eta\cup H_{k_1}^{\xi_1}\cup\cdots\cup H_{k_m}^{\xi_m}$, where $m$ is a non--negative integer, $\xi_l<\zeta$ and $k_l\in J_{\xi_l}$ for any $l=1,\ldots,m$. We have
    \[\mbox{cl}_{\beta X}\Big(\mbox{cl}_X U_i^\zeta\backslash\Big(U_j^\eta\cup\bigcup_{l=1}^m U_{k_l}^{\xi_l}\Big)\Big)\subseteq(f_i^\zeta)^{-1}\big[[0,r]\big]\backslash\Big((f_j^\eta)^{-1}\big[[0,r)\big]\cup\bigcup_{l=1}^m(f_{k_l}^{\xi_l})^{-1}
    \big[[0,r)\big]\Big)\]
    and thus
    \[\mbox{cl}_{\beta X}\Big(\mbox{cl}_X U_i^\zeta\backslash\Big(U_j^\eta\cup\bigcup_{l=1}^m U_{k_l}^{\xi_l}\Big)\Big)\backslash\lambda_{\mathcal P}X\subseteq H_i^\zeta\backslash\Big(H_j^\eta\cup\bigcup_{l=1}^m H_{k_l}^{\xi_l}\Big)=\emptyset.\]
    Now by Lemma \ref{19FG} it follows that
    \[\mbox{cl}_X U_i^\zeta\backslash\Big(U_j^\eta\cup\bigcup_{l=1}^m U_{k_l}^{\xi_l}\Big)\subseteq Z\subseteq C\]
    for some $Z\in {\mathscr Z}(X)$ and $C\in Coz(X)$, where $\mbox{cl}_X C$, and therefore $Z$, has ${\mathcal P}$, as it is closed in $\mbox{cl}_X C$. Thus
    (3.c.iv) holds in this case.
\item[{\sc Case 2.}] If (\ref{HVG}) holds then by compactness $H_i^\zeta\cap H_j^\eta\subseteq H_{k_1}^{\xi_1}\cup\cdots\cup H_{k_m}^{\xi_m}$,
    where $m$ is a non--negative integer, $\xi_l<\zeta$ and $k_l\in J_{\xi_l}$ for any $l=1,\ldots,m$. We have
    \[\!\mbox{cl}_{\beta X}\Big((\mbox{cl}_X U_i^\zeta \cap\mbox{cl}_X U_j^\eta)\backslash\bigcup_{l=1}^m U_{k_l}^{\xi_l}\Big)\subseteq\big((f_i^\zeta)^{-1}\big[[0,r]\big]\cap(f_j^\eta)^{-1}\big[[0,r]\big]\big)\backslash\bigcup_{l=1}^m(f_{k_l}^{\xi_l})^{-1}
    \big[[0,r)\big]\]
    and thus
    \[\mbox{cl}_{\beta X}\Big((\mbox{cl}_X U_i^\zeta \cap\mbox{cl}_X U_j^\eta)\backslash\bigcup_{l=1}^m U_{k_l}^{\xi_l}\Big)\backslash\lambda_{\mathcal P} X\subseteq(H_i^\zeta\cap H_j^\eta)\backslash\bigcup_{l=1}^m H_{k_l}^{\xi_l}=\emptyset.\]
    Now as in the previous case (3.c.iv) follows.
\end{description}

Finally, we verify (3.c.v). Let $\zeta<\eta\leq\sigma$ and $U\in{\mathscr U}_\eta$. Then $U=U_j^\eta$ for some $j\in J_\eta$. By (3.c.iv) of Lemma \ref{18} there exists an infinite $J\subseteq J_\zeta$ such that
\[H_i^\zeta\subseteq H_j^\eta\cup\bigcup\{G\in{\mathscr H}_\xi:\xi<\zeta\}\]
for any $i\in J$. Now an argument similar to the one above shows that
\[\mbox{cl}_X U_i^\zeta\backslash\Big(U_j^\eta\cup \bigcup_{l=1}^m U_{k_l}^{\xi_l}\Big)\subseteq Z\]
for some $Z\in {\mathscr Z}(X)$  which has ${\mathcal P}$, some non--negative integer $m$, some $\xi_l<\zeta$ and some $k_l\in J_{\xi_l}$ where $l=1,\ldots,m$. Thus (3.c.v) holds.

(3.c) {\em implies} (3.a). To prove (3.a) we verify condition (3.c) of Lemma \ref{18}. Suppose that $X$ is locally--${\mathcal P}$ and there exists a family $\{{\mathscr U}_\zeta:\zeta\leq\sigma\}$ of collections of pairwise disjoint non--empty open subsets of $X$  satisfying  (3.c.i)--(3.c.v). For any $\zeta\leq\sigma$ let ${\mathscr U}_\zeta=\{ U^\zeta_i:i\in J_\zeta\}$ be bijectively indexed. Then  $\mbox{card}(J_\zeta)=\aleph_0$ if $\zeta<\sigma$, and $\mbox{card}(J_\sigma)=n$. Let $\zeta\leq\sigma$ and $i\in J_\zeta$. Define
\[H_i^\zeta=\mbox{cl}_{\beta X}U_i^\zeta\backslash\lambda_{\mathcal P} X.\]
By Lemma \ref{j1} and (3.c.ii) we have $H_i^\zeta=\mbox{Ex}_X U_i^\zeta\backslash\lambda_{\mathcal P} X$ which shows that $H_i^\zeta$ is clopen in $\beta X\backslash\lambda_{\mathcal P}X$. Also, $H_i^\zeta$ is non--empty, as otherwise $\mbox{cl}_{\beta X} U_i^\zeta\subseteq\lambda_{\mathcal P} X$ which by Lemma \ref{B} implies that $\mbox{cl}_X U_i^\zeta$ has ${\mathcal P}$, contradicting (3.c.iii). Since $H_i^\zeta\subseteq\mbox{Ex}_X U_i^\zeta$ and $U_i^\zeta$'s are pairwise disjoint, by Lemma \ref{9} the sets $H_i^\zeta$'s also are pairwise disjoint. For any $\zeta\leq\sigma$ let ${\mathscr H}_\zeta= \{H_i^\zeta: i\in J_\zeta\}$, which is bijectively indexed, as  $H_i^\zeta$'s are non--empty, and for any distinct $i,j\in J_\zeta$ we have $H_i^\zeta\cap H_j^\zeta=\emptyset$. We verify that the family  $\{{\mathscr H}_\zeta:\zeta\leq\sigma\}$ has the desired properties. Condition (3.c.i) of Lemma \ref{18} holds trivially.

To prove condition (3.c.ii) of Lemma \ref{18} let $H\in {\mathscr H}_\zeta$ for some $\zeta\leq\sigma$, and suppose to the contrary that \[H\backslash\bigcup\{ G\in {\mathscr H}_\eta:\eta<\zeta\}=\emptyset.\]
Then $H=H_i^\zeta$ for some $i\in J_\zeta$. By compactness $H_i^\zeta\subseteq H_{k_1}^{\eta_1}\cup\cdots\cup H_{k_m}^{\eta_m}$, where $m\in\mathbf{N}$, $\eta_l<\zeta$ and $k_l\in J_{\eta_l}$ for any $l=1,\ldots,m$. We have
\[\mbox{cl}_{\beta X}\Big(\mbox{cl}_X U_i^\zeta\backslash\bigcup_{l=1}^m U_{k_l}^{\eta_l}\Big)\subseteq\mbox{cl}_{\beta X} U_i^\zeta\backslash\bigcup_{l=1}^m\mbox{Ex}_X U_{k_l}^{\eta_l}\]
and therefore
\[\mbox{cl}_{\beta X}\Big(\mbox{cl}_X U_i^\zeta\backslash\bigcup_{l=1}^m U_{k_l}^{\eta_l}\Big)\backslash\lambda_{\mathcal P} X\subseteq H_i^\zeta\backslash\bigcup_{l=1}^m H_{k_l}^{\eta_l}=\emptyset.\]
Lemma \ref{B} now implies that $\mbox{cl}_X U_i^\zeta\backslash(U_{k_1}^{\eta_1}\cup\cdots\cup U_{k_m}^{\eta_m})$ has ${\mathcal P}$. But this contradicts (3.c.iii).

Next, we show condition (3.c.iii) of Lemma \ref{18}. Suppose that $\zeta<\eta\leq\sigma$, $H\in{\mathscr H}_\zeta$ and  $G\in{\mathscr H}_\eta$. Let $H=H_i^\zeta$ and $G=H_j^\eta$ where $i\in J_\zeta$ and $j\in J_\eta$. By (3.c.iv) there exist a finite ${\mathscr V}\subseteq\bigcup\{{\mathscr U}_\xi:\xi<\zeta\}$ and a $Z\in{\mathscr Z}(X)$ such that $Z$ has ${\mathcal P}$, and either
\begin{equation}\label{HVYUIG}
\mbox{cl}_X U_i^\zeta\backslash\Big(U_j^\eta\cup\bigcup{\mathscr V}\Big)\subseteq Z
\end{equation}
or
\begin{equation}\label{JHHVG}
(\mbox{cl}_X U_i^\zeta \cap\mbox{cl}_X U_j^\eta)\backslash\bigcup{\mathscr V}\subseteq Z.
\end{equation}
Let ${\mathscr V}=\{U_{k_1}^{\xi_1},\ldots,U_{k_m}^{\xi_m}\}$, where $m$ is a non--negative integer, $\xi_l<\zeta$ and $k_l\in J_{\xi_l}$ for any $l=1,\ldots,m$. We consider the following cases:
\begin{description}
\item[{\sc Case 1.}] If (\ref{HVYUIG}) holds then
    \[\mbox{Ex}_X U_i^\zeta\subseteq\mbox{cl}_{\beta X}\mbox{Ex}_X U_i^\zeta=\mbox{cl}_{\beta X} U_i^\zeta\subseteq\mbox{cl}_{\beta X}\Big(U_j^\eta
    \cup \bigcup_{l=1}^m U_{k_l}^{\xi_l}\cup Z\Big)\]
    and thus
    \[\mbox{Ex}_X U_i^\zeta\backslash\Big(\mbox{cl}_{\beta X}U_j^\eta\cup\bigcup_{l=1}^m \mbox{cl}_{\beta X}U_{k_l}^{\xi_l}\Big)\subseteq\mbox{int}_{\beta X}\mbox{cl}_{\beta X}Z\subseteq\lambda_{\mathcal P} X.\]
    From this it follows that
    \[H\subseteq G\cup\bigcup_{l=1}^m H_{k_l}^{\xi_l}\subseteq G\cup\bigcup\{F\in{\mathscr H}_\xi:\xi<\zeta\}.\]
\item[{\sc Case 2.}] If (\ref{JHHVG}) holds then using Lemma \ref{9} we have
    \begin{eqnarray*}
    \mbox{Ex}_X U_i^\zeta\cap\mbox{Ex}_X U_j^\eta &\subseteq&\mbox{cl}_{\beta X}(\mbox{Ex}_X U_i^\zeta\cap\mbox{Ex}_X U_j^\eta)\\&=&\mbox{cl}_{\beta X}\mbox{Ex}_X(U_i^\zeta\cap U_j^\eta)=\mbox{cl}_{\beta X}(U_i^\zeta\cap U_j^\eta)\subseteq\mbox{cl}_{\beta X}\Big(\bigcup_{l=1}^m U_{k_l}^{\xi_l}\cup Z\Big)
    \end{eqnarray*}
    and thus
    \[(\mbox{Ex}_X U_i^\zeta\cap\mbox{Ex}_X U_j^\eta)\backslash\bigcup_{l=1}^m \mbox{cl}_{\beta X}U_{k_l}^{\xi_l}\subseteq\mbox{int}_{\beta X}\mbox{cl}_{\beta X}Z\subseteq\lambda_{\mathcal P}X\]
    which yields
    \[H\cap G\subseteq\bigcup_{l=1}^m H_{k_l}^{\xi_l}\subseteq\bigcup\{F\in{\mathscr H}_\xi:\xi<\zeta\}.\]
\end{description}
This shows condition (3.c.iii) of Lemma \ref{18} in either case.

Finally, to show (3.c.iv) of Lemma  \ref{18} suppose that $\zeta<\eta\leq\sigma$ and $H\in {\mathscr H}_\eta$. Let $H=H_j^\eta$ for some $j\in J_\eta$. By (3.c.v) there exists an infinite $J\subseteq J_\zeta$ such that for any $i\in J$ there exist a $Z\in {\mathscr Z}(X)$ such that $Z$ has $ {\mathcal P}$ and a finite ${\mathscr W}\subseteq \bigcup\{{\mathscr U}_\xi:\xi<\zeta\}$ such that $\mbox{cl}_X U_i^\zeta\backslash(U_j^\eta\cup\bigcup {\mathscr W})\subseteq Z$. Arguing as above
\[H_i^\zeta\subseteq H\cup\bigcup\{F\in{\mathscr H}_\xi:\xi<\zeta\}.\]
Thus
\[\Big\{F\in{\mathscr H}_\zeta:F\subseteq H\cup\bigcup\{G\in{\mathscr H}_\xi:\xi<\zeta\}\Big\}\]
 is infinite.
\end{proof}

The following generalizes a theorem of K.D. Magill,  Jr. in \cite{Mag2} (Theorem \ref{i0}).

\begin{corollary}\label{20UHYJG}
Let ${\mathcal P}$ and  ${\mathcal Q}$ be a pair of compactness--like topological properties. Let $X$  be a  Tychonoff space with $\mathcal{Q}$. The following are equivalent:
\begin{itemize}
\item[\rm(1)] ${\mathscr M}^{\mathcal Q}_{\mathcal P}(X)$ contains an element with $n$--point remainder (equivalently, ${\mathscr O}^{\mathcal Q}_{\mathcal P}(X)$ contains an element with $n$--point remainder) for any $n\in\mathbf{N}$.
\item[\rm(2)] ${\mathscr M}^{\mathcal Q}_{\mathcal P}(X)$ contains an element with countable remainder (equivalently, ${\mathscr O}^{\mathcal Q}_{\mathcal P}(X)$ contains an element with countable remainder).
\end{itemize}
\end{corollary}

\begin{proof}
This follows from Lemma \ref{18} and the observation that $\beta X\backslash \lambda_{\mathcal{P}}X$ has an infinite number of components if and only if $\beta X\backslash\lambda_{\mathcal{P}}X$ has at least $n$ components for any $n\in\mathbf{N}$.
\end{proof}

\begin{theorem}\label{20UHG}
Let ${\mathcal P}$ and  ${\mathcal Q}$ be a pair of compactness--like topological properties. Let $X$  be a Tychonoff  space.
\begin{itemize}
\item[\rm(1)] Let $n\in\mathbf{N}$. If $X$ has a perfect image $Y$ with $\mathcal{Q}$ such that ${\mathscr M}^{\mathcal Q}_{\mathcal P}(Y)$ (${\mathscr O}^{\mathcal Q}_{\mathcal P}(Y)$, respectively) contains an element with $n$--point remainder, then so does $X$.
\item[\rm(2)] If $X$ has a  perfect image $Y$ with $\mathcal{Q}$ such that ${\mathscr M}^{\mathcal Q}_{\mathcal P}(Y)$ (${\mathscr O}^{\mathcal Q}_{\mathcal P}(Y)$, respectively) contains an element with countable remainder, then so does $X$.
\item[\rm(3)] Let $0<\sigma<\Omega$ and let $n\in\mathbf{N}$. If $X$ has a perfect  image $Y$ with $\mathcal{Q}$ such that ${\mathscr M}^{\mathcal Q}_{\mathcal P}(Y)$ (${\mathscr O}^{\mathcal Q}_{\mathcal P}(Y)$, respectively) contains an element with countable remainder of type $(\sigma,n)$, then so does $X$.
\end{itemize}
\end{theorem}

\begin{proof}
We prove the theorem in the case of minimal extensions. From this and  Theorem \ref{20} the result will then follow in the  case of optimal extensions as well.

(1). Suppose that $f:X\rightarrow Y$ is a perfect surjective mapping such that $Y$ has ${\mathcal Q}$ (thus $X$ also has ${\mathcal Q}$, as ${\mathcal Q}$ is inverse invariant under perfect mappings) and that ${\mathscr M}^{\mathcal Q}_{\mathcal P}(Y)$ contains an element with $n$--point remainder where $n\in\mathbf{N}$. Note that $Y$ (having a Tychonoff extension) is Tychonoff and thus by Theorem \ref{20} the space $Y$ is locally--${\mathcal P}$ and $Y=K\cup U_1\cup\cdots\cup U_n$,  where $K,U_1,\ldots,U_n$ are pairwise disjoint, each $U_1,\ldots,U_n$ is open in $Y$ with non--${\mathcal P}$ closure and $\mbox{bd}_YK\subseteq Z\subseteq C$ for some $Z\in{\mathscr Z}(Y)$ and $C\in Coz(Y)$ such that $\mbox{cl}_Y C$ has ${\mathcal P}$. Then
\[X=f^{-1}[Y]=f^{-1}[K]\cup\bigcup_{i=1}^n f^{-1}[U_i]\]
and $f^{-1}[K],f^{-1}[U_1],\ldots,f^{-1}[U_n]$ are pairwise disjoint. We show that the closure of each open subset $f^{-1}[U_1],\ldots,f^{-1}[U_n]$ of $X$ is non--${\mathcal P}$. Suppose to the contrary that $\mbox{cl}_Xf^{-1}[U_i]$ has ${\mathcal P}$ for some $i=1,\ldots,n$. Now since
\[f|\mbox{cl}_Xf^{-1}[U_i]:\mbox{cl}_Xf^{-1}[U_i]\rightarrow f\big[\mbox{cl}_Xf^{-1}[U_i]\big]\]
is a perfect surjective mapping and ${\mathcal P}$ is invariant under perfect mappings, $f[\mbox{cl}_Xf^{-1}[U_i]]$ has ${\mathcal P}$. Since $f$ is surjective we have
\[U_i=f\big[f^{-1}[U_i]\big]\subseteq f\big[\mbox{cl}_Xf^{-1}[U_i]\big]\]
and since $f$ is closed, it follows that $\mbox{cl}_Y U_i\subseteq f[\mbox{cl}_Xf^{-1}[U_i]]$ and thus $\mbox{cl}_Y U_i$, being closed in the latter has ${\mathcal P}$. But this is a contradiction. Also,
\begin{eqnarray*}
\mbox{bd}_Xf^{-1}[K]&=&\mbox{cl}_Xf^{-1}[K]\cap\mbox{cl}_X\big(X\backslash f^{-1}[K]\big)\\&=&\mbox{cl}_Xf^{-1}[K]\cap\mbox{cl}_Xf^{-1}[Y\backslash K]\\&\subseteq& f^{-1}[\mbox{cl}_Y K]\cap f^{-1}\big[\mbox{cl}_Y(Y\backslash K)\big]\\&=&f^{-1}\big[\mbox{cl}_Y K\cap\mbox{cl}_Y(Y\backslash K)\big]=f^{-1}[\mbox{bd}_Y K]\subseteq f^{-1}[Z]\subseteq f^{-1}[C]
\end{eqnarray*}
and $f^{-1}[Z]\in{\mathscr Z}(X)$ and $f^{-1}[C]\in Coz(X)$. Note that since the mapping
\[f|f^{-1}[\mbox{cl}_Y C]:f^{-1}[\mbox{cl}_Y C]\rightarrow\mbox{cl}_Y C\]
is perfect and surjective (since $f$ is surjective), the set $\mbox{cl}_Y C$ has ${\mathcal P}$, and (since ${\mathcal P}$ is inverse invariant under perfect mappings) the set $f^{-1}[\mbox{cl}_Y C]$, and thus its closed subset  $\mbox{cl}_X f^{-1}[C]$, has ${\mathcal P}$. Finally, note that by Lemma \ref{22} it follows that $X$  is locally--${\mathcal P}$. Thus the result follows from Theorem \ref{20}.

(2). This is analogous to part (1) using the characterization given in Theorem \ref{20}.

(3). Suppose that $f:X\rightarrow Y$ is a perfect surjective mapping such that $Y$ has ${\mathcal Q}$ and that ${\mathscr M}^{\mathcal Q}_{\mathcal P}(Y)$ contains an element with countable remainder of type $(\sigma,n)$. Note that as in part (1) it follows that $X$ has ${\mathcal Q}$ and $Y$ is Tychonoff. By Theorem \ref{20} the space $Y$ is locally--${\mathcal P}$ and there exists a family $\{{\mathscr U}_\zeta:\zeta\leq\sigma\}$ of collections of  pairwise disjoint non--empty open subsets of $Y$ satisfying conditions (3.c.i)--(3.c.v) of that theorem. For any $\zeta\leq\sigma$ let ${\mathscr A}_\zeta=\{f^{-1}[U]:U\in{\mathscr U}_\zeta\}$. Then each ${\mathscr A}_\zeta$ where $\zeta\leq\sigma$, consists of pairwise disjoint non--empty (since $f$ is surjective) open subsets of $X$. We verify that $\{{\mathscr A}_\zeta:\zeta\leq\sigma\}$ satisfies condition (3.c.i)--(3.c.v) of Theorem \ref{20}. Condition (3.c.i) holds trivially, as since $f$ is surjective, each ${\mathscr A}_\zeta$ where $\zeta\leq\sigma$, is bijectively indexed and thus $\mbox{card}({\mathscr A}_\zeta)=\mbox{card}({\mathscr U}_\zeta)$ for any $\zeta\leq\sigma$. Condition (3.c.ii) follows by an argument similar to part (1) and the fact that $\{{\mathscr U}_\zeta:\zeta\leq\sigma\}$ satisfies a similar condition. To show condition (3.c.iii) suppose to the contrary that for some $\zeta\leq\sigma$, $U\in{\mathscr U}_\zeta$ and finite ${\mathscr V}\subseteq\bigcup\{{\mathscr U}_\eta:\eta<\zeta\}$ the set $\mbox{cl}_Xf^{-1}[U]\backslash f^{-1}[\bigcup{\mathscr V}]$ has ${\mathcal P}$. Since $f$ is closed and surjective we have
\begin{eqnarray*}
\mbox{cl}_Y U\backslash \bigcup{\mathscr V}&\subseteq&\mbox{cl}_Y f\big[f^{-1}[U]\big]\backslash \bigcup{\mathscr V}\\&\subseteq& f\big[\mbox{cl}_Xf^{-1}[U]\big]\backslash f\Big[f^{-1}\Big[\bigcup{\mathscr V}\Big]\Big]\subseteq f\Big[\mbox{cl}_Xf^{-1}[U]\backslash f^{-1}\Big[\bigcup{\mathscr V}\Big]\Big].
\end{eqnarray*}
But the latter has ${\mathcal P}$, as ${\mathcal P}$ is invariant under perfect mappings, thus its closed subset $\mbox{cl}_Y U\backslash \bigcup{\mathscr V}$ has ${\mathcal P}$, which is a contradiction. To show condition (3.c.iv) suppose that $\zeta<\eta\leq\sigma$, $U\in{\mathscr U}_\zeta$ and $V\in{\mathscr U}_\eta$. Since $\{{\mathscr U}_\zeta:\zeta\leq\sigma\}$ satisfies a similar condition, there exist a $Z\in{\mathscr Z}(Y)$ which has ${\mathcal P}$, and a finite ${\mathscr V}\subseteq\bigcup\{{\mathscr U}_\xi:\xi<\zeta\}$ such that either
\[\mbox{cl}_Y U\backslash\Big(V\cup\bigcup{\mathscr V}\Big)\subseteq Z\mbox{ or }(\mbox{cl}_Y U\cap\mbox{cl}_Y V)\backslash\bigcup{\mathscr V}\subseteq Z.\]
Since $f|f^{-1}[Z]:f^{-1}[Z]\rightarrow Z$ is perfect and surjective (since $f$ is surjective) and ${\mathcal P}$ is inverse invariant under perfect mappings,  $f^{-1}[Z]\in{\mathscr Z}(X)$ has ${\mathcal P}$. In the first case
\[\mbox{cl}_X f^{-1}[U]\backslash\Big(f^{-1}[V]\cup f^{-1}\Big[\bigcup{\mathscr V}\Big]\Big)\subseteq f^{-1}[\mbox{cl}_Y U]\backslash\Big(f^{-1}[V]\cup f^{-1}\Big[\bigcup{\mathscr V}\Big]\Big)\subseteq f^{-1}[Z]\]
and in the second case
\begin{eqnarray*}
\big(\mbox{cl}_X f^{-1}[U]\cap\mbox{cl}_X f^{-1}[V]\big)\backslash f^{-1}\Big[\bigcup{\mathscr V}\Big]&\subseteq&\big(f^{-1}[\mbox{cl}_Y U]\cap f^{-1}[\mbox{cl}_Y V]\big)\backslash f^{-1}\Big[\bigcup{\mathscr V}\Big]\\&\subseteq& f^{-1}[Z].
\end{eqnarray*}
The proof for condition (3.c.v) is analogous. Note that by Lemma \ref{22} the space $X$ is locally--${\mathcal P}$. The result now follows.
\end{proof}

\section{Compactification--like ${\mathcal P}$--extensions as partially ordered sets}

In  this chapter we consider classes of compactification--like ${\mathcal P}$--extensions of a Tychonoff space $X$ as partially ordered sets. We define two partial orders $\leq_{inj}$ and $\leq_{surj}$ (besides $\leq$ itself) on the set of all extensions of $X$. These partial orders behave nicely when restricted to classes of compactification--like ${\mathcal P}$--extensions of $X$ and their introduction lead to some interesting results which characterize compactification--like ${\mathcal P}$--extensions of $X$ among all Tychonoff ${\mathcal P}$--extensions of $X$ with compact remainder. We continue with study of the relationships between the order--structure of classes of compactification--like ${\mathcal P}$--extensions of $X$ (partially ordered with $\leq$) and the topology of the subspace $\beta X\backslash\lambda_{{\mathcal P}}X$ of its outgrowth $\beta X\backslash X$. This generalize a well known result of K.D. Magill, Jr. in \cite{Mag3} which relates the order--structure of the set of all compactifications of a locally compact spaces $X$ and the topology of the  outgrowth $\beta X\backslash X$. We conclude this chapter with a result which characterizes the largest (with respect to $\leq$) compactification--like ${\mathcal P}$--extension of $X$. This largest element (which we explicitly introduce as a subspace of the Stone--\v{C}ech compactification $\beta X$ of $X$) turns out to be also the largest among all Tychonoff ${\mathcal P}$--extension of $X$ with compact remainder.

We start with the following definition.

\begin{definition}
Let $X$ be a space and let $Y$ and $Y'$ be extensions of $X$. We let $Y\leq_{inj} Y'$ if there exists a continuous injective $f:Y'\rightarrow Y$ such that $f|X=\mbox{id}_X$.
\end{definition}

The relation  $\leq_{inj}$ defines  a partial order on the set of all  extensions of a space $X$. The following lemma (see also Lemma  \ref{BYST}) is a counterpart of Lemma \ref{DFH}.

\begin{lemma}\label{YFERFH}
Let $X$ be a Tychonoff space and let $Y_1,Y_2\in{\mathscr E}_{\mathcal P}(X)$ be such that $Y_1\leq Y_2$. The following are equivalent:
\begin{itemize}
\item[\rm(1)] $Y_1\leq_{inj} Y_2$.
\item[\rm(2)] Any element of ${\mathscr F}(Y_1)$ contains at most one element of ${\mathscr F}(Y_2)$.
\end{itemize}
\end{lemma}

\begin{proof}
Let $\phi_i:\beta X\rightarrow\beta Y_i$ where $i=1,2$, be the continuous extension of $\mbox{id}_X$. Since  $Y_1\leq Y_2$ there exists a continuous $f:Y_2\rightarrow Y_1$ such that $f|X=\mbox{id}_X$. Let $f_\beta:\beta Y_2\rightarrow\beta Y_1$ be the continuous extension of $f$. As shown in the proof of Lemma \ref{DFH} we have $f_\beta \phi_2=\phi_1$ and $f[Y_2\backslash X]\subseteq Y_1\backslash X$.

(1) {\em implies} (2). Suppose that $f:Y_2\rightarrow Y_1$ introduced above is moreover injective. Let $p\in Y_1\backslash X$ and let $p_i\in Y_2\backslash X$ where   $i=1,2$, be such that $\phi_2^{-1}(p_i)\subseteq\phi_1^{-1}(p)$. Choose some $s_i\in\phi_2^{-1}(p_i)$ for any $i=1,2$ (such $s_i$'s exist, as  $\phi_2$ is surjective). Then
\[f(p_1)=f_\beta(p_1)=f_\beta\big(\phi_2(s_1)\big)=\phi_1(s_1)=p=\phi_1(s_2)=f_\beta\big(\phi_2(s_2)\big)=f_\beta(p_2)=f(p_2)\]
which implies that $p_1=p_2$. Thus  $\phi_2^{-1}(p_1)=\phi_2^{-1}(p_2)$.

(2) {\em implies} (1). We show that the mapping $f:Y_2\rightarrow Y_1$ introduced above is injective. Let $p_i\in Y_2\backslash X$ where $i=1,2$, be such that $f(p_1)=f(p_2)$ and let $p\in Y_1\backslash X$ denote their common value. Note that
\[\phi_2^{-1}(p_i)\subseteq\phi_2^{-1}\big[f^{-1} (p)\big]\subseteq\phi_2^{-1}\big[f_\beta^{-1} (p)\big]=(f_\beta\phi_2)^{-1}(p)=\phi_1^{-1}(p)\]
for any $i=1,2$, which by (2) implies that $\phi_2^{-1}(p_1)=\phi_2^{-1}(p_2)$ and therefore $p_1=p_2$.
\end{proof}

\begin{notation}\label{HYTE}
Let $R$ be a relation on a set $X$ and let $Y\subseteq X$. Denote
\[R|Y=\big\{(y,x)\in R:y\in Y\big\}.\]
\end{notation}

In the next result we give an order--theoretic characterization of ${\mathscr O}^{\mathcal Q}_{\mathcal P}(X)$. (Compare with its dual result Theorem \ref{UHDER} on ${\mathscr M}^{\mathcal Q}_{\mathcal P}(X)$.) Recall that a subset $A$ of a partially ordered set $(X,\leq)$ is said to be {\em cofinal} if for any $x\in X$ there exists some $a\in A$ with $x\leq a$.

\begin{theorem}\label{LLJLFA}
Let ${\mathcal P}$ and  ${\mathcal Q}$ be a pair of compactness--like topological properties. Let $X$ be a  Tychonoff space with $\mathcal{Q}$. Then
\begin{itemize}
\item[\rm(1)] \[{\mathscr O}^{\mathcal Q}_{\mathcal P}(X)=\big\{Y:Y \mbox{ is  maximal in }\big({\mathscr E}_{\mathcal P}^{\mathcal Q}(X),\leq_{inj}\big)\big\}.\]
\item[\rm(2)] ${\mathscr O}^{\mathcal Q}_{\mathcal P}(X)$ is the smallest cofinal subset of $({\mathscr E}_{\mathcal P}^{\mathcal Q}(X),\leq_{inj})$.
\item[\rm(3)] ${\mathscr O}^{\mathcal Q}_{\mathcal P}(X)$ is the unique cofinal subset of $({\mathscr E}_{\mathcal P}^{\mathcal Q}(X),\leq_{inj})$ on  which the two relations $\leq_{inj}$ and $=$ coincide.
\item[\rm(4)] ${\mathscr O}^{\mathcal Q}_{\mathcal P}(X)$ is the largest subset ${\mathscr E}$ of ${\mathscr E}_{\mathcal P}^{\mathcal Q}(X)$ such that
\begin{equation}\label{FRSY}
(\leq_{inj}|{\mathscr E})\subseteq=.
\end{equation}
\end{itemize}
\end{theorem}

\begin{proof} (2). To show that  ${\mathscr O}^{\mathcal Q}_{\mathcal P}(X)$ is cofinal in ${\mathscr E}_{\mathcal P}^{\mathcal Q}(X)$ with respect to $\leq_{inj}$ let $Y\in{\mathscr E}_{\mathcal P}^{\mathcal Q}(X)$. Let $\phi:\beta X\rightarrow\beta Y$ be the continuous extension of $\mbox{id}_X$. By Lemma \ref{16} the space $X$ is locally--${\mathcal P}$ and  $\beta X\backslash\lambda_{{\mathcal P}}X\subseteq\phi^{-1}[Y\backslash X]$. Also, by Lemma \ref{15} we have $X\subseteq\lambda_{{\mathcal P}}X$. Let
\[P=\big\{p\in Y\backslash X:\phi^{-1}(p)\backslash\lambda_{{\mathcal P}}X\neq\emptyset\big\}.\]
Form the quotient space $T$ of $\beta X$ by contracting each subset $\phi^{-1}(p)\backslash\lambda_{{\mathcal P}}X$ where $p\in P$ to a point $t_p$ and denote by $q:\beta X\rightarrow T$ its quotient mapping. Arguing as in the proof of Theorem \ref{HG16} ((1) $\Rightarrow$ (2)) it follows that $T$ is compact. Consider the subspace $Z=X\cup q[\beta X\backslash\lambda_{{\mathcal P}}X]$ of $T$. Then $Z$ is a Tychonoff extension of $X$ with the compact remainder $Z\backslash X=q[\beta X\backslash\lambda_{{\mathcal P}}X]$. Note that $T$ is a compactification of $Z$. Let $\psi:\beta X\rightarrow \beta Z$ and $f:\beta Z\rightarrow T$ be the continuous extensions of $\mbox{id}_X$ and $\mbox{id}_Z$, respectively. Since the continuous mapping $f\psi:\beta X\rightarrow T$ coincide with $q$ on $X$ we have $f\psi=q$. By Lemma \ref{16} and Theorem \ref{HG16} to show that $Z\in{\mathscr O}_{\mathcal P}^{\mathcal Q}(X)$ it suffices to show that $\psi^{-1}[Z\backslash X]=\beta X\backslash\lambda_{{\mathcal P}}X$. But this follows, as by  Theorem 3.5.7 of \cite{E} (and since $\beta Z$ and $T$ are compactifications of $Z$ and $f$ is continuous with $f|Z=\mbox{id}_Z$) we have $f[\beta Z\backslash Z]=T\backslash Z$ and therefore
\[\psi^{-1}[Z\backslash X]=\psi^{-1}\big[f^{-1}[Z\backslash X]\big]=(f\psi)^{-1}[Z\backslash X]=q^{-1}[Z\backslash X]=\beta X\backslash\lambda_{{\mathcal P}}X.\]
Note that for any $p\in P$ (and again, since $f[\beta Z\backslash Z]=T\backslash Z$) we have
\[\psi^{-1}(t_p)=\psi^{-1}\big[f^{-1}(t_p)\big]=(f\psi)^{-1}(t_p)=q^{-1}(t_p)\subseteq\phi^{-1}(p)\backslash\lambda_{{\mathcal P}}X.\]
Define $g:Z\rightarrow Y$ by $g(t_p)=p$ when $p\in P$ and $g(x)=x$ when $x\in X$. By the proof of Lemma \ref{DFH} ((2) $\Rightarrow$ (1)) (note that $\psi^{-1}(t_p)\subseteq\phi^{-1}(p)$ for any $p\in P$) the mapping $g$ is continuous, and by its definition, it is moreover injective. Thus $Y\leq_{inj} Z$. This shows the cofinality of ${\mathscr O}^{\mathcal Q}_{\mathcal P}(X)$ in ${\mathscr E}_{\mathcal P}^{\mathcal Q}(X)$ with respect to  $\leq_{inj}$.

To complete the proof we need to show that ${\mathscr O}^{\mathcal Q}_{\mathcal P}(X)$ is contained in every subset ${\mathscr S}$ of ${\mathscr E}_{\mathcal P}^{\mathcal Q}(X)$ cofinal with respect to $\leq_{inj}$. Indeed, let $Z\in{\mathscr O}_{\mathcal P}^{\mathcal Q}(X)$. Then $Z\leq_{inj} S$ for some $S\in {\mathscr S}$. We show that $S$ and $Z$ are equivalent extensions of $X$. To show this by Lemma \ref{DFH} it suffices to verify that  ${\mathscr F}(S)={\mathscr F}(Z)$. Let $\psi:\beta X\rightarrow\beta Z$ and $\varphi:\beta X\rightarrow \beta S$ be the continuous extensions of $\mbox{id}_X$. By Theorem \ref{HG16} we have $\psi^{-1}[Z\backslash X]=\beta X\backslash\lambda_{{\mathcal P}}X$. Now since $Z\leq S$ (as $Z\leq_{inj}S$) by  Lemma \ref{DFH}, for any $s\in S\backslash X$ we have $\varphi^{-1} (s)\subseteq\psi^{-1}(z)$ for some $z\in Z\backslash X$. Therefore  $\varphi^{-1}[S\backslash X]\subseteq\beta X\backslash\lambda_{{\mathcal P}}X$ and thus since by  Lemma \ref{16} we have $\beta X\backslash\lambda_{{\mathcal P}}X\subseteq\varphi^{-1}[S\backslash X]$ it follows that  $\varphi^{-1}[S\backslash X]=\beta X\backslash\lambda_{{\mathcal P}}X$. Now let $s'\in S\backslash X$. Then by Lemma \ref{DFH} (and since $Z\leq S$) we have $\varphi^{-1}(s')\subseteq\psi^{-1}(z')$ for some $z'\in Z\backslash X$. Suppose that $\varphi^{-1}(s')\neq \psi^{-1}(z')$. There exists some $s''\in  S\backslash X$ such that $s''\neq s'$ and $\varphi^{-1} (s'')\cap\psi^{-1} (z')$ is non--empty. Thus  $\varphi^{-1}(s'')\subseteq \psi^{-1}(z')$. But by Lemma \ref{YFERFH} this implies that $\varphi^{-1} (s'')=\varphi^{-1} (s')$ which is a contradiction, as $s''\neq s'$ (and $\varphi$ is surjective). This shows that $\varphi^{-1}(s')= \psi^{-1}(z')$. Therefore ${\mathscr F}(S)\subseteq{\mathscr F}(Z)$. To show the reverse inclusion note that for any $z\in Z\backslash X$, since $\psi^{-1}(z)\subseteq\beta X\backslash\lambda_{{\mathcal P}}X$ the set $\psi^{-1}(z)\cap\varphi^{-1}(s)$ is non--empty for some $s\in S\backslash X$, and thus  $\psi^{-1} (z)=\varphi^{-1} (s)$, as ${\mathscr F}(S)\subseteq{\mathscr F}(Z)$ (and the elements of ${\mathscr F}(Z)$ are pairwise disjoint). Therefore ${\mathscr F}(Z)\subseteq{\mathscr F}(S)$ which shows the equality  in the latter. By Lemma \ref{DFH} we have $S\leq Z$ and $Z\leq S$ which implies that $Z$ and $S$ are equivalent. Thus $Z\in {\mathscr S}$. This shows that ${\mathscr O}^{\mathcal Q}_{\mathcal P}(X)\subseteq{\mathscr S}$.

(1). By Theorem \ref{HG16} any element of ${\mathscr O}^{\mathcal Q}_{\mathcal P}(X)$ is maximal in ${\mathscr E}^{\mathcal Q}_{\mathcal P}(X)$ with respect to  $\leq_{inj}$. The converse follows from part (2), as if $Y\in{\mathscr E}^{\mathcal Q}_{\mathcal P}(X)$ is maximal with respect to $\leq_{inj}$ then $Y\leq_{inj} T$ for some $T\in{\mathscr O}^{\mathcal Q}_{\mathcal P}(X)$, which yields $Y=T$ and thus $Y\in{\mathscr O}^{\mathcal Q}_{\mathcal P}(X)$.

(3). Note that by part (2) the set  ${\mathscr O}^{\mathcal Q}_{\mathcal P}(X)$ is cofinal in ${\mathscr E}_{\mathcal P}^{\mathcal Q}(X)$ with respect to $\leq_{inj}$. Also, by part (1) the relations $\leq_{inj}$ and $=$ coincide on ${\mathscr O}^{\mathcal Q}_{\mathcal P}(X)$. Now let ${\mathscr E}$ be a subset of ${\mathscr E}_{\mathcal P}^{\mathcal Q}(X)$ cofinal with respect to $\leq_{inj}$ and such that the relations  $\leq_{inj}$ and $=$ coincide on ${\mathscr E}$. Let $S\in {\mathscr E}$. By the cofinality of ${\mathscr O}^{\mathcal Q}_{\mathcal P}(X)$ (with respect to $\leq_{inj}$) we have $S\leq_{inj} T$ for some $T\in {\mathscr O}^{\mathcal Q}_{\mathcal P}(X)$, and by the cofinality of ${\mathscr E}$ we have $T\leq_{inj} Z$ for some $Z\in {\mathscr E}$. Then  $S\leq_{inj} Z$ and thus (since $S,Z\in{\mathscr E}$) we have $S=Z$. Therefore $S=T$ which implies that $S\in{\mathscr O}^{\mathcal Q}_{\mathcal P}(X)$. This shows that ${\mathscr E}\subseteq{\mathscr O}^{\mathcal Q}_{\mathcal P}(X)$. Note that by part (2) we have also ${\mathscr O}^{\mathcal Q}_{\mathcal P}(X)\subseteq{\mathscr E}$, which together with above proves the equality in the latter.

(4). By part (1) the set ${\mathscr O}^{\mathcal Q}_{\mathcal P}(X)$ satisfies (\ref{FRSY}). Now let ${\mathscr E}$ be a subset of ${\mathscr E}_{\mathcal P}^{\mathcal Q}(X)$ which satisfies (\ref{FRSY}). Let $S\in {\mathscr E}$. By part (2) the set ${\mathscr O}^{\mathcal Q}_{\mathcal P}(X)$ is cofinal in ${\mathscr E}_{\mathcal P}^{\mathcal Q}(X)$ with respect to $\leq_{inj}$. Therefore there exists some $T\in {\mathscr O}^{\mathcal Q}_{\mathcal P}(X)$ such that $S\leq_{inj}T$. By (\ref{FRSY}) we have $S=T$ which implies that $S\in{\mathscr O}^{\mathcal Q}_{\mathcal P}(X)$. Thus ${\mathscr E}\subseteq{\mathscr O}^{\mathcal Q}_{\mathcal P}(X)$.
\end{proof}

\begin{definition}\label{GTDST}
Let $X$ be a space and let $Y$ and $Y'$ be extensions of $X$. We let $Y\leq_{surj} Y'$ if there exists a continuous surjective $f:Y'\rightarrow Y$ such that $f|X=\mbox{id}_X$.
\end{definition}

The relation $\leq_{surj}$ defines a partial order on the set of all extensions of a space $X$.

\begin{lemma}\label{BYST}
Let $X$ be a Tychonoff space and let $Y_1,Y_2\in{\mathscr E}_{\mathcal P}(X)$ be such that $Y_1\leq Y_2$. The  following are equivalent:
\begin{itemize}
\item[\rm(1)] $Y_1\leq_{surj} Y_2$.
\item[\rm(2)] Any element of ${\mathscr F}(Y_1)$ contains at least one element of ${\mathscr F}(Y_2)$.
\end{itemize}
\end{lemma}

\begin{proof}
Let $\phi_i:\beta X\rightarrow\beta Y_i$ where $i=1,2$, be the continuous extension of $\mbox{id}_X$. Since $Y_1\leq Y_2$ there exists a continuous $f:Y_2\rightarrow Y_1$ such that $f|X=\mbox{id}_X$. Let $f_\beta:\beta Y_2\rightarrow\beta Y_1$ be the continuous extension of $f$. As shown in the proof of Lemma 2.12 of \cite{Ko3} we have $f_\beta \phi_2=\phi_1$ and $f[Y_2\backslash X]\subseteq Y_1\backslash X$.

(1) {\em implies} (2). Suppose that $f:Y_2\rightarrow Y_1$ introduced above is moreover surjective. Let $p_1\in Y_1\backslash X$ and let $p_2\in Y_2\backslash X$ be such that $f(p_2)=p_1$. Then
\[\phi_2^{-1}(p_2)\subseteq\phi_2^{-1}\big[f^{-1}(p_1)\big]\subseteq\phi_2^{-1}\big[f_\beta^{-1}(p_1)\big]=(f_\beta\phi_2)^{-1}(p_1)=\phi_1^{-1}(p_1).\]

(2) {\em implies} (1). We show that the mapping $f:Y_2\rightarrow Y_1$ introduced above is surjective. Let $p_1\in Y_1\backslash X$. Let $p_2\in Y_2\backslash X$ be such that $\phi_2^{-1}(p_2)\subseteq\phi_1^{-1}(p_1)$. Choose an $s\in\phi_2^{-1}(p_2)$ (such an  $s$ exists, as $\phi_2$ is surjective). Then since
\[s\in \phi_1^{-1}(p_1)=(f_\beta\phi_2)^{-1}(p_1)=\phi_2^{-1}\big[f_\beta^{-1}(p_1)\big]\]
we have $p_2=\phi_2(s)\in f_\beta^{-1}(p_1)$, which implies that $f(p_2)=f_\beta(p_2)=p_1$.
\end{proof}

\begin{lemma}\label{GHYDSD}
Let ${\mathcal P}$ and  ${\mathcal Q}$ be a pair of compactness--like topological properties. Let $X$ be a Tychonoff space with $\mathcal{Q}$. Let $Y\in {\mathscr M}^{\mathcal Q}_{\mathcal P}(X)$ and let $T\in {\mathscr E}_{\mathcal P}^{\mathcal Q}(X)$ be such that $T\leq_{surj} Y$. Then  $T\in {\mathscr M}^{\mathcal Q}_{\mathcal P}(X)$.
\end{lemma}

\begin{proof}
Let $F\in{\mathscr F}(T)$. By Lemma \ref{BYST} there exists some $G\in {\mathscr F}(Y)$ such that $G\subseteq F$. By Theorem \ref{HUHG16}
the set $G\backslash\lambda_{{\mathcal P}}X$ is non--empty and thus  $F\backslash\lambda_{{\mathcal P}}X$ is non--empty. By Theorem \ref{HUHG16} the result follows.
\end{proof}

In the next result we give an order--theoretic characterization of ${\mathscr M}^{\mathcal Q}_{\mathcal P}(X)$.

\begin{theorem}\label{UHDER}
Let ${\mathcal P}$ and  ${\mathcal Q}$ be a pair of compactness--like topological properties. Let $X$ be a  Tychonoff space with $\mathcal{Q}$. Then
\begin{itemize}
\item[\rm(1)] ${\mathscr M}^{\mathcal Q}_{\mathcal P}(X)$ is the largest cofinal subset of $({\mathscr E}_{\mathcal P}^{\mathcal Q}(X),\leq)$ on  which the two relations $\leq$ and $\leq_{surj}$ coincide.
\item[\rm(2)] ${\mathscr M}^{\mathcal Q}_{\mathcal P}(X)$ is the largest subset of $({\mathscr E}_{\mathcal P}^{\mathcal Q}(X),\leq_{surj})$ in which ${\mathscr O}^{\mathcal Q}_{\mathcal P}(X)$ is cofinal.
\item[\rm(3)] ${\mathscr M}^{\mathcal Q}_{\mathcal P}(X)$ is the largest subset ${\mathscr E}$ of ${\mathscr E}_{\mathcal P}^{\mathcal Q}(X)$ such that
\begin{equation}\label{GYD}
(\leq|{\mathscr E})\subseteq\leq_{surj}.
\end{equation}
\item[\rm(4)] ${\mathscr M}^{\mathcal Q}_{\mathcal P}(X)$ is the smallest cofinal  subset ${\mathscr E}$ of $({\mathscr E}_{\mathcal P}^{\mathcal Q}(X),\leq)$ such that
\begin{equation}\label{GTSYD}
\big(\big({\mathscr E}_{\mathcal P}^{\mathcal Q}(X)\times{\mathscr E}\big)\cap\leq_{surj}\big)\subseteq{\mathscr E}\times{\mathscr E}_{\mathcal P}^{\mathcal Q}(X).
\end{equation}
\end{itemize}
\end{theorem}

\begin{proof}
(1). First we show that ${\mathscr M}_{\mathcal P}^{\mathcal Q}(X)$ is cofinal in ${\mathscr E}_{\mathcal P}^{\mathcal Q}(X)$ with respect to $\leq$. Let $Y\in{\mathscr E}_{\mathcal P}^{\mathcal Q}(X)$. Let $\phi:\beta X\rightarrow\beta Y$ be the continuous extension of $\mbox{id}_X$. Consider the subspace
\[T=X\cup\big\{p\in Y\backslash X:\phi^{-1}(p)\backslash\lambda_{{\mathcal P}}X\neq\emptyset\big\}\]
of $Y$. We show that $T\in{\mathscr M}_{\mathcal P}^{\mathcal Q}(X)$ and $Y\leq T$. Obviously, $T$ is a Tychonoff  extension of $X$. By Lemma \ref{16} we have $\beta X\backslash\lambda_{{\mathcal P}}X\subseteq\phi^{-1}[Y\backslash X]$ (and $X$ is locally--${\mathcal P}$) and thus $T\backslash X=\phi[\beta X\backslash\lambda_{{\mathcal P}}X]$ is compact.  Also, since $\beta Y$ is  a compactification of $T$  and by the definition of $T$ we have $\beta X\backslash\lambda_{{\mathcal P}}X\subseteq\phi^{-1}[T\backslash X]$, again by  Lemma \ref{16} it follows that $T$ has both ${\mathcal P}$ and ${\mathcal Q}$. Let $\psi:\beta X\rightarrow\beta T$ and $f:\beta T\rightarrow\beta Y$ be the continuous extensions of $\mbox{id}_X$ and $\mbox{id}_T$, respectively. The continuous mappings $f\psi$ and $\phi$ agree on $X$, and therefore they are identical. Since $\beta Y$ is a compactification of $T$ (and $f|T=\mbox{id}_T$), by Theorem 3.5.7 of \cite{E} we have $f[\beta T\backslash T]=\beta Y\backslash T$. Thus
\[\psi^{-1}(p)=\psi^{-1}\big[f^{-1}(p)\big]=(f\psi)^{-1}(p)=\phi^{-1}(p)\]
for any $p\in T\backslash X$. By Lemma \ref{DFH} it then follows that $Y\leq T$. By the definition of $T$ we have
\[\psi^{-1}(p)\backslash\lambda_{{\mathcal P}}X=\phi^{-1}(p)\backslash\lambda_{{\mathcal P}}X\neq\emptyset\]
for any $p\in T\backslash X$, which by Theorem \ref{HUHG16} implies that $T\in{\mathscr M}_{\mathcal P}^{\mathcal Q}(X)$.

By Theorem \ref{HUHG16} the relations  $\leq$ and $\leq_{surj}$ coincide on ${\mathscr M}_{\mathcal P}^{\mathcal Q}(X)$. Now let ${\mathscr E}$ be a subset of ${\mathscr E}_{\mathcal P}^{\mathcal Q}(X)$ which is cofinal with respect to $\leq$ and is such that the relations  $\leq$ and $\leq_{surj}$ coincide on ${\mathscr E}$. Let $S\in {\mathscr E}$. By  Theorem \ref{LLJLFA}(2) the set ${\mathscr O}^{\mathcal Q}_{\mathcal P}(X)$ is cofinal in ${\mathscr E}_{\mathcal P}^{\mathcal Q}(X)$ with respect to $\leq$. Therefore there exists some $T\in {\mathscr O}^{\mathcal Q}_{\mathcal P}(X)$ with $S\leq T$. By the cofinality of  ${\mathscr E}$ with respect to $\leq$ there exists some $Z\in{\mathscr E}$ with $T\leq Z$. Then $S\leq Z$ and thus (since $S,Z\in {\mathscr E}$) by our assumption $S\leq_{surj} Z$. But by Theorem \ref{HG16} (since $T\leq Z$) we have $Z\in {\mathscr M}^{\mathcal Q}_{\mathcal P}(X)$ which by Lemma \ref{GHYDSD} yields $S\in{\mathscr M}^{\mathcal Q}_{\mathcal P}(X)$. Thus ${\mathscr E}\subseteq{\mathscr M}^{\mathcal Q}_{\mathcal P}(X)$.

(2). By the definitions we have ${\mathscr O}^{\mathcal Q}_{\mathcal P}(X)\subseteq{\mathscr M}^{\mathcal Q}_{\mathcal P}(X)$. Also, if $Y\in {\mathscr M}^{\mathcal Q}_{\mathcal P}(X)$ then by Theorem \ref{LLJLFA}(2) we have $Y\leq T$ for some $T\in{\mathscr O}^{\mathcal Q}_{\mathcal P}(X)$ and thus by Theorem \ref{HUHG16} we have $Y\leq_{surj} T$. This shows the cofinality of ${\mathscr O}^{\mathcal Q}_{\mathcal P}(X)$ in ${\mathscr M}^{\mathcal Q}_{\mathcal P}(X)$. Now let ${\mathscr E}$ be a subset of ${\mathscr E}_{\mathcal P}^{\mathcal Q}(X)$ in which ${\mathscr O}^{\mathcal Q}_{\mathcal P}(X)$ is cofinal with respect to $\leq_{surj}$. Let $S\in {\mathscr E}$. By the cofinality there exists some $Z\in {\mathscr O}^{\mathcal Q}_{\mathcal P}(X)$ with $S\leq_{surj} Z$.  By Lemma \ref{GHYDSD} we have $S\in {\mathscr M}^{\mathcal Q}_{\mathcal P}(X)$. Thus ${\mathscr E}\subseteq{\mathscr M}^{\mathcal Q}_{\mathcal P}(X)$.

(3). By Theorem \ref{HUHG16}(1.e) the set ${\mathscr E}={\mathscr M}^{\mathcal Q}_{\mathcal P}(X)$ satisfies (\ref{GYD}). Now let ${\mathscr E}$ be a subset of ${\mathscr E}_{\mathcal P}^{\mathcal Q}(X)$ which satisfies (\ref{GYD}). Let $S\in {\mathscr E}$. By part (1) the set ${\mathscr M}^{\mathcal Q}_{\mathcal P}(X)$ is cofinal in ${\mathscr E}_{\mathcal P}^{\mathcal Q}(X)$ with respect to $\leq$. Therefore there exists some $Y\in {\mathscr M}^{\mathcal Q}_{\mathcal P}(X)$ with $S\leq Y$. Thus $S\leq_{surj} Y$ by  (\ref{GYD}). By Lemma \ref{GHYDSD} it follows that $S\in {\mathscr M}^{\mathcal Q}_{\mathcal P}(X)$. Therefore ${\mathscr E}\subseteq{\mathscr M}^{\mathcal Q}_{\mathcal P}(X)$.

(4). By part (1) the set ${\mathscr M}^{\mathcal Q}_{\mathcal P}(X)$ is cofinal  in  ${\mathscr E}_{\mathcal P}^{\mathcal Q}(X)$ with respect to $\leq$. Also, by Lemma \ref{GHYDSD} the set ${\mathscr E}={\mathscr M}^{\mathcal Q}_{\mathcal P}(X)$ satisfies (\ref{GTSYD}). Now let ${\mathscr E}$ be a subset of ${\mathscr E}_{\mathcal P}^{\mathcal Q}(X)$ cofinal with respect to $\leq$ and satisfies (\ref{GTSYD}). Let $Y\in {\mathscr M}^{\mathcal Q}_{\mathcal P}(X)$. By the cofinality of ${\mathscr E}$ we have $Y\leq S$ for some $S\in {\mathscr E}$. By Theorem \ref{HUHG16}(1.e) we have $Y\leq_{surj} S$ and thus by (\ref{GTSYD}) it follows that $Y\in {\mathscr E}$. Therefore ${\mathscr M}^{\mathcal Q}_{\mathcal P}(X)\subseteq{\mathscr E}$.
\end{proof}

Recall that a partially ordered set $(L,\leq)$ is called a {\em lattice} if together with any pair of elements $a,b\in L$ it contains their  least upper bound $a\vee b$ and their greatest lower bound $a\wedge b$. Our next purpose is to generalize the following result of K.D. Magill, Jr. in \cite{Mag3} which relates the order--structure of the lattice of compactifications of a locally compact space $X$ to the topology of the outgrowth $\beta X\backslash X$. (The theorem has been generalized in various directions; see \cite{Me} for a different proof of the theorem; see \cite{R} for generalizations of the theorem to non--locally compact spaces; see \cite{Wo1} and \cite{Do} for a zero--dimensional version of the theorem, and see \cite{PW1} for extension of the theorem to mappings.) Our results here will relate the order--structure of classes of compactification--like ${\mathcal P}$--extensions of a Tychonoff space $X$ to the topology of the subspace $\beta X\backslash\lambda_{{\mathcal P}}X$ of $\beta X$.

\begin{theorem}[Magill \cite{Mag3}]\label{KLFA}
Let $X$ and $Y$ be locally compact non--compact spaces. The following are equivalent:
\begin{itemize}
\item[\rm(1)] $({\mathscr K}(X),\leq)$ and $({\mathscr K}(Y),\leq)$ are order--isomorphic.
\item[\rm(2)] $\beta X\backslash X$ and $\beta Y\backslash Y$ are homeomorphic.
\end{itemize}
\end{theorem}

\begin{xrem}
{\em The above theorem fails if the spaces under consideration are not locally compact (see \cite{Th}).}
\end{xrem}

The following simple observation will be used quite often in the future (sometimes without explicit reference).

\begin{lemma}\label{HFH}
Let $X$ be a Tychonoff locally--${\mathcal P}$ space where ${\mathcal P}$ is a clopen hereditary finitely additive perfect topological property. Then $X$ is non--${\mathcal P}$ if and only if $\lambda_{{\mathcal P}}X$ is non--compact if and only if $\lambda_{{\mathcal P}}X\neq\beta X$.
\end{lemma}

\begin{proof}
If $X$ has ${\mathcal P}$ then by the definition of $\lambda_{{\mathcal P}}X$  (and since obviously $X\in {\mathscr Z}(X)$) we have $\beta X=\mbox{int}_{\beta X}\mbox{cl}_{\beta X}X\subseteq\lambda_{{\mathcal P}}X$. Thus $\lambda_{{\mathcal P}}X=\beta X$  is compact. Note that if $\lambda_{{\mathcal P}}X$ is compact, then since $X\subseteq \lambda_{{\mathcal P}}X$ (as $X$ is locally--${\mathcal P}$; see Lemma \ref{15}) we have  $\mbox{cl}_{\beta X}X\subseteq\lambda_{{\mathcal P}}X$. Therefore by Lemma \ref{B} the space $X$ has ${\mathcal P}$.
\end{proof}

Recall that if $(A,\leq)$ and  $(B,\leq)$ are partially ordered sets, a mapping $f:A\rightarrow B$ is said to be an {\em order--homomorphism} if for any $c,d\in A$ we have $f(c)\leq f(d)$ whenever $c\leq d$. An order--homomorphism $f:A\rightarrow B$ is called an {\em order--isomorphism} if it is bijective and $f^{-1}:B\rightarrow A$ also is an order--homomorphism. Two partially ordered sets $(A,\leq)$ and  $(B,\leq)$ are said to be {\em order--isomorphic} (denoted by $(A,\leq)\cong(B,\leq)$) if there exists an  order--isomorphism between them.

\begin{lemma}\label{HGFA}
Let ${\mathcal P}$ and  ${\mathcal Q}$ be a pair of compactness--like topological properties. Let $X$ be a Tychonoff locally--${\mathcal P}$ non--${\mathcal P}$ space with $\mathcal{Q}$. Then
\[\big({\mathscr O}^{\mathcal Q}_{\mathcal P}(X),\leq\big)\cong\big({\mathscr K}(\lambda_{{\mathcal P}}X),\leq\big).\]
\end{lemma}

\begin{proof}
Let $Y\in {\mathscr O}^{\mathcal Q}_{\mathcal P}(X)$. Let $\phi:\beta X\rightarrow\beta Y$ be the continuous extension of $\mbox{id}_X$. Recall that $\beta Y$ is the quotient space of $\beta X$ obtained by contracting each $\phi^{-1}(p)$ where $p\in Y\backslash X$, to a point, with $\phi$ as the corresponding quotient mapping (see Lemma \ref{j2}). By Theorem \ref{HG16} we have $\phi^{-1}[Y\backslash X]=\beta X\backslash\lambda_{{\mathcal P}}X$ and thus we may assume that $\lambda_{{\mathcal P}}X\subseteq \beta Y$. Also, $X$ is dense in $\beta Y$, as $X$ is dense in $Y$ and by Lemma \ref{15} we have  $X\subseteq\lambda_{{\mathcal P}}X$. Therefore  $\lambda_{{\mathcal P}}X$ is dense in $\beta Y$ and thus $\beta Y$ is a compactification of $\lambda_{{\mathcal P}}X$. Define
\[\Theta:\big({\mathscr O}^{\mathcal Q}_{\mathcal P}(X),\leq\big)\rightarrow\big({\mathscr K}(\lambda_{{\mathcal P}}X),\leq\big)\]
by
\[\Theta(Y)=\beta Y\]
for any $Y\in{\mathscr O}^{\mathcal Q}_{\mathcal P}(X)$. By the above $\Theta$ is well defined. We verify that $\Theta$ is an order--isomorphism.

\begin{xclaim}
$\Theta$ is an order--homomorphism.
\end{xclaim}

\subsubsection*{Proof of the claim} Let $Y_1\leq Y_2$ where $Y_1,Y_2\in {\mathscr O}^{\mathcal Q}_{\mathcal P}(X)$. By definition there exists a continuous  $f:Y_2\rightarrow  Y_1$ such that $f|X=\mbox{id}_X$. Let $f_\beta:\beta Y_2\rightarrow\beta Y_1$ be the continuous extension of $f$. By above $\beta Y_i\in {\mathscr K}(\lambda_{{\mathcal P}}X)$ for any $i=1,2$. Then $f_\beta|\lambda_{{\mathcal P}}X=\mbox{id}_{\lambda_{{\mathcal P}}X}$, as they both coincide with $\mbox{id}_X$ on $X$ and thus by definition $\Theta(Y_1)=\beta Y_1\leq\beta Y_2=\Theta(Y_2)$.

\smallskip

\begin{xclaim}
$\Theta$ is surjective.
\end{xclaim}

\subsubsection*{Proof of the claim} Let $T\in{\mathscr K}(\lambda_{{\mathcal P}}X)$. Consider the subspace $Y=X\cup(T\backslash\lambda_{{\mathcal P}}X)$ of $T$. We verify that $Y\in {\mathscr O}^{\mathcal Q}_{\mathcal P}(X)$ and that $\Theta(Y)=T$. Note that $X$ is dense in $T$ and therefore $X$ is dense in $Y$, as $X$ is dense in $\lambda_{{\mathcal P}}X$ and $\lambda_{{\mathcal P}}X$ is dense in $T$. By definition $\lambda_{{\mathcal P}}X$ is an open subset of $\beta X$ and thus it is locally compact. Also, $X\subseteq\lambda_{{\mathcal P}}X$ and therefore $Y\backslash X=T\backslash\lambda_{{\mathcal P}}X$ is compact. This shows that $Y\in{\mathscr E}(X)$. Also, $\beta\lambda_{{\mathcal P}}X=\beta X$, as $X\subseteq\lambda_{{\mathcal P}}X\subseteq \beta X$. Let $g:\beta X\rightarrow T$ be the continuous extension of  $\mbox{id}_{\lambda_{{\mathcal P}}X}$. By Theorem 3.5.7 of \cite{E} we have $g[\beta X\backslash\lambda_{{\mathcal P}}X]=T\backslash\lambda_{{\mathcal P}}X$. Thus
\[\beta X\backslash\lambda_{{\mathcal P}}X\subseteq g^{-1}\big[g[\beta X\backslash\lambda_{{\mathcal P}}X]\big]=g^{-1}[T\backslash\lambda_{{\mathcal P}}X]=g^{-1}[Y\backslash X].\]
Since $X$ is locally--${\mathcal P}$, by Lemma \ref{16} we have $Y\in{\mathscr E}^{\mathcal Q}_{\mathcal P}(X)$. To show that $Y$ is optimal let $Z\in{\mathscr Z}(X)$ be such that $Z\subseteq C$ for some $C\in Coz(X)$ such that $\mbox{cl}_X C$ has ${\mathcal P}$. By Lemma \ref{BA27}  we have $\mbox{cl}_{\beta X} Z\subseteq\lambda_{{\mathcal P}}X$. Therefore since $Z=g[Z]\subseteq g[\mbox{cl}_{\beta X}Z]$ and the latter is compact, $\mbox{cl}_T Z\subseteq g[\mbox{cl}_{\beta X} Z]$. Since $g[\mbox{cl}_{\beta X} Z]\subseteq g[\lambda_{{\mathcal P}}X]=\lambda_{{\mathcal P}}X$ we have $\mbox{cl}_T Z\subseteq\lambda_{\mathcal P}X$ and thus
\[\mbox{cl}_Y Z\cap (Y\backslash X)\subseteq\mbox{cl}_T Z\cap (T\backslash\lambda_{{\mathcal P}}X)=\emptyset.\]
Theorem \ref{HG16} now implies that $Y\in {\mathscr O}^{\mathcal Q}_{\mathcal P}(X)$. Let $\phi:\beta X\rightarrow\beta Y$ be the continuous extension of $\mbox{id}_X$.
By Theorem \ref{HG16} we have $\phi^{-1}[Y\backslash X]=\beta X\backslash\lambda_{\mathcal P}X$ which implies that $\phi|\lambda_{\mathcal P}X=\mbox{id}_{\lambda_{{\mathcal P}}X}$. (Recall the construction of $\beta Y$ and the representation of $\phi$ given in Lemma \ref{j2}.) Let $h:\beta Y\rightarrow T$ be the continuous extension of $\mbox{id}_Y$. The continuous mapping $h\phi:\beta X\rightarrow T$ is such that $h\phi|X=\mbox{id}_X=g|X$ and therefore $h\phi=g$. Thus (and since $\phi|\lambda_{\mathcal P}X=\mbox{id}_{\lambda_{\mathcal P}X}$) we have $h|\lambda_{{\mathcal P}}X=g|\lambda_{\mathcal P}X=\mbox{id}_{\lambda_{\mathcal P}X}$ and therefore, since $h|Y=\mbox{id}_Y$ and $Y\cup\lambda_{\mathcal P}X=\beta Y$ it follows that  $h=\mbox{id}_{\beta Y}$. In particular, $\mbox{id}_{\beta Y}=h:\beta Y\rightarrow T$ is continuous and it is surjective (as its image contains $X$ and $X$ is dense in $T$) and thus, since $\beta Y$ is compact, it is a homeomorphism. Therefore $T=\beta Y=\Theta (Y)$.

\begin{xclaim}
For any $Y_1,Y_2\in {\mathscr O}^{\mathcal Q}_{\mathcal P}(X)$ if $\Theta(Y_1)\leq \Theta(Y_2)$ then $Y_1\leq Y_2$.
\end{xclaim}

\subsubsection*{Proof of the claim} Let $\Theta(Y_1)\leq \Theta(Y_2)$ for some $Y_1,Y_2\in {\mathscr O}^{\mathcal Q}_{\mathcal P}(X)$. Since $\beta Y_1\leq \beta Y_2$, by definition there exists a continuous $l:\beta Y_2\rightarrow \beta Y_1$ such that $l| \lambda_{{\mathcal P}}X=\mbox{id}_{\lambda_{{\mathcal P}}X}$. By Theorem 3.5.7 of \cite{E} we have $l[\beta Y_2\backslash\lambda_{{\mathcal P}}X]=\beta Y_1\backslash\lambda_{{\mathcal P}}X$. Note that $Y_i\backslash X=\beta Y_i\backslash\lambda_{{\mathcal P}}X$ for any $i=1,2$. To  see this, observe that if $\phi_i:\beta X\rightarrow\beta Y_i$ where $i=1,2$, denotes the continuous extension of $\mbox{id}_X$ then $\beta Y_i$  is the quotient space of $\beta X$ obtained by contracting the fibers $\phi_i^{-1}(p)$'s where $p\in Y_i\backslash X$ to points with the quotient mapping $\phi_i$, and by Theorem \ref{HG16} we have $\phi_i^{-1}[Y_i\backslash X]=\beta X\backslash\lambda_{{\mathcal P}}X$. Therefore $l[Y_2\backslash X]=Y_1\backslash X$. Thus $l|Y_2: Y_2\rightarrow Y_1$ and obviously it continuously extends $\mbox{id}_X$. Therefore by the definition $Y_1\leq Y_2$.

\smallskip

\noindent The third claim  implies  that $\Theta$ is injective and that $\Theta^{-1}$ is an order--homomorphism. This shows that $\Theta$ is an order--isomorphism.
\end{proof}

Recall that a partially ordered set $(L,\leq)$ is called a {\em complete upper semilattice} ({\em complete lower semilattice}, respectively)  if for any non--empty subset $A$ of $L$ the least upper bound $\bigvee A$ (the greatest lower bound $\bigwedge A$, respectively) exists in $L$. A partially ordered set $(L,\leq)$ is called a {\em complete lattice} if it is both a complete upper semilattice and a complete lower semilattice. It is well known that for any Tychonoff $X$ the partially ordered set ${\mathscr K}(X)$ of its all compactifications, partially ordered with $\leq$, is a complete upper  semilattice, and it is a complete lattice if and only if $X$ is locally compact (see Propositions 4.2(a) and 4.3(e) of \cite{PW}). The following corollary of Lemma \ref{HGFA} is now immediate.

\begin{corollary}\label{YGFAS}
Let ${\mathcal P}$ and  ${\mathcal Q}$ be a pair of compactness--like topological properties. Let $X$ be a  Tychonoff locally--${\mathcal P}$ non--${\mathcal P}$ space with $\mathcal{Q}$. Then $({\mathscr O}^{\mathcal Q}_{\mathcal P}(X),\leq)$ is a complete lattice.
\end{corollary}

The following theorem relates the order--structure of the set of optimal ${\mathcal P}$--extensions of a Tychonoff locally--${\mathcal P}$ space $X$ and the topology of the subspace $\beta X\backslash\lambda_{{\mathcal P}}X$ of $\beta X$. This generalizes K.D. Magill,  Jr.'s theorem in \cite{Mag3} (Theorem \ref{KLFA}) provided that one replaces $\mathcal{P}$ and $\mathcal{Q}$, respectively, by compactness and regularity, and note that for these specific choices of  $\mathcal{P}$ and $\mathcal{Q}$ and a locally compact space $X$ we have $\lambda_{\mathcal{P}} X=X$  and ${\mathscr O}^{\mathcal Q}_{\mathcal P}(X)={\mathscr K}(X)$.

\begin{theorem}\label{KKLFA}
Let ${\mathcal P}$ and  ${\mathcal Q}$ be a pair of compactness--like topological properties. Let $X$ and $Y$ be Tychonoff locally--${\mathcal P}$ non--${\mathcal P}$ spaces with $\mathcal{Q}$.
The following are equivalent:
\begin{itemize}
\item[\rm(1)] $({\mathscr O}^{\mathcal Q}_{\mathcal P}(X),\leq)$ and $({\mathscr O}^{\mathcal Q}_{\mathcal P}(Y),\leq)$ are order--isomorphic.
\item[\rm(2)] $\beta X\backslash\lambda_{{\mathcal P}}X$ and $\beta Y\backslash\lambda_{{\mathcal P}}Y$ are homeomorphic.
\end{itemize}
\end{theorem}

\begin{proof}
Note that $\lambda_{{\mathcal P}}X$ is locally compact, as it is open in $\beta X$, and by Lemma \ref{HFH} (since $X$ is non--${\mathcal P}$) it is non--compact. Also, since $X$ is locally--${\mathcal P}$, by Lemma \ref{15} we have  $X\subseteq\lambda_{{\mathcal P}}X$ and thus (since $\lambda_{{\mathcal P}}X\subseteq \beta X$) we have $\beta\lambda_{{\mathcal P}}X=\beta X$. Similar statements hold for $Y$. By Theorem \ref{KLFA} the partially ordered sets ${\mathscr K}(\lambda_{{\mathcal P}}X)$ and ${\mathscr K}(\lambda_{{\mathcal P}}Y)$ are order--isomorphic if and only if  $\beta\lambda_{{\mathcal P}}X\backslash\lambda_{{\mathcal P}}X$ ($=\beta X\backslash\lambda_{{\mathcal P}}X$) and $\beta\lambda_{{\mathcal P}}Y\backslash\lambda_{{\mathcal P}}Y$ ($=\beta Y\backslash\lambda_{{\mathcal P}}Y$) are homeomorphic. Now Lemma \ref{HGFA}  shows the equivalence of (1) and (2).
\end{proof}

Our next purpose is to state and prove a result for minimal ${\mathcal P}$--extensions which is analogous to (1) $\Rightarrow$ (2) in Theorem \ref{KKLFA}. As we will see, there is no counterpart for (2) $\Rightarrow$ (1) in Theorem \ref{KKLFA} in the case of minimal ${\mathcal P}$--extensions. This will be shown by means of an example.  (This is the first place in this article where the duality between minimal ${\mathcal P}$--extensions and optimal ${\mathcal P}$--extensions disappears.) The example, however, is long and quite technical, and requires several lemmas. The reader who is not interested in the construction of the example may skip reading Lemmas \ref{KHFA}, \ref{YTRO}, \ref{DSEY}, \ref{GSA}, \ref{KJSQ}, \ref{FSAEW}, \ref{AYLF}, \ref{FSWR} and \ref{KLYS} and replaces the role of Lemma \ref{FSWR} by Lemma \ref{PKJF} in the proof of Theorem \ref{LGLH} ((1) $\Rightarrow$ (2)).

The following lemma is a counterpart of Lemma 4 in \cite{Mag3}.

\begin{lemma}\label{LKA}
Let ${\mathcal P}$ and  ${\mathcal Q}$ be a pair of compactness--like topological properties. Let $X$ be a Tychonoff locally--${\mathcal P}$ non--${\mathcal P}$ space with $\mathcal{Q}$. For an $n\in\mathbf{N}$, let $K_1,\ldots,K_n$ be $n$ pairwise disjoint compact subsets of $\beta X\backslash X$ such that $K_i\backslash\lambda_{{\mathcal P}}X$ is non--empty  for any $i=1,\ldots,n$. Then there exists a unique $Y$ in ${\mathscr M}^{\mathcal Q}_{\mathcal P}(X)$ such that
\begin{equation}\label{UY}
{\mathscr F}(Y)=\Big\{\{p\}:p\in (\beta X\backslash\lambda_{{\mathcal P}}X)\backslash\bigcup_{i=1}^n K_i\Big\}\cup\{K_1,\ldots,K_n\}.
\end{equation}
\end{lemma}

\begin{proof}
Let $T$ be the space obtained from $\beta X$ by contracting the sets  $K_1,\ldots,K_n$ to points $p_1,\ldots,p_n$, respectively, and denote by $q:\beta X\rightarrow T$ its  quotient mapping. Since $K_i$'s where $i=1,\dots,n$ are compact, $T$ is Hausdorff and thus compact, being a continuous image of $\beta X$. Consider the subspace
\[Y=q\Big[X\cup(\beta X\backslash\lambda_{{\mathcal P}}X)\cup\bigcup_{i=1}^n K_i\Big]\]
of $T$. Then $Y$ is a Tychonoff extension of $X$ with the compact remainder
\[Y\backslash X=q\Big[(\beta X\backslash\lambda_{{\mathcal P}}X)\cup\bigcup_{i=1}^n K_i\Big].\]
Note that $T$ is a compactification of $Y$ and  thus by Lemma \ref{16}  we have $Y\in {\mathscr E}_{\mathcal P}^{\mathcal Q}(X)$. Also, by Lemma \ref{j2} if $\phi:\beta X\rightarrow \beta Y$ continuously extends $\mbox{id}_X$, then $\beta Y$ coincides with the quotient space of $\beta X$  obtained by contracting each fiber $\phi^{-1}(p)$ where $p\in Y\backslash X$ to a point, that is, $\beta Y=T$. This shows  (\ref{UY}). The fact that $Y\in {\mathscr M}^{\mathcal Q}_{\mathcal P}(X)$ follows from Theorem \ref{HUHG16}.

For the uniqueness part, let $Y'\in{\mathscr M}^{\mathcal Q}_{\mathcal P}(X)$ be such that  ${\mathscr F}(Y')={\mathscr F}(Y)$. Let $\psi:\beta X\rightarrow\beta Y'$ be the continuous extension of $\mbox{id}_X$. By Lemma \ref{j2} we have $\beta Y'=T$ and $\psi=q$. Thus
\[Y'=q\Big[X\cup(\beta X\backslash\lambda_{{\mathcal P}}X)\cup\bigcup_{i=1}^n K_i\Big]=Y.\]
\end{proof}

\begin{notation}\label{PKWES}
Let ${\mathcal P}$ and  ${\mathcal Q}$ be a pair of compactness--like topological properties. Let $X$ be a Tychonoff locally--${\mathcal P}$ non--${\mathcal P}$ space with $\mathcal{Q}$. Let $n\in\mathbf{N}$ and let $K_1,\ldots,K_n$ be $n$ pairwise disjoint compact subsets of $\beta X\backslash X$ such that $K_i\backslash\lambda_{{\mathcal P}}X$ is non--empty for any $i=1,\ldots,n$. Denote by $e_X(K_1,\ldots,K_n)$ the unique element of ${\mathscr M}^{\mathcal Q}_{\mathcal P}(X)$ such that
\[{\mathscr F}\big(e_X(K_1,\ldots,K_n)\big)=\Big\{\{p\}:p\in (\beta X\backslash\lambda_{{\mathcal P}}X)\backslash\bigcup_{i=1}^n K_i\Big\}\cup\{K_1,\ldots,K_n\}.\]
\end{notation}

The next lemma is a counterpart of Lemma 6 in \cite{Mag3}.

\begin{lemma}\label{OTH}
Let ${\mathcal P}$ and  ${\mathcal Q}$ be a pair of compactness--like topological properties. Let $X$ be a Tychonoff locally--${\mathcal P}$ non--${\mathcal P}$ space with $\mathcal{Q}$. Let $K_i$ where $i=1,2$, be a compact subset of $\beta X\backslash X$ such that $K_i\backslash\lambda_{{\mathcal P}}X$ is non--empty. Then
\begin{itemize}
\item[\rm(1)] $e_X(K_1)\wedge e_X(K_2)=e_X(K_1,K_2)$, if $K_1\cap K_2=\emptyset$.
\item[\rm(2)] $e_X(K_1)\wedge e_X(K_2)=e_X(K_1\cup K_2)$, if $K_1\cap K_2\neq\emptyset$.
\end{itemize}
Here $\wedge$ is the operation in ${\mathscr M}^{\mathcal Q}_{\mathcal P}(X)$.
\end{lemma}

\begin{proof}
This  follows from Lemma \ref{DFH}. In part (2) note that if $Y\in {\mathscr M}^{\mathcal Q}_{\mathcal P}(X)$ is such that $Y\leq e_X(K_i)$ for any $i=1,2$, then by Lemma \ref{DFH} we have $K_i\subseteq F_i$ for some $F_i\in  {\mathscr F}(Y)$. But by our assumption $K_1\cap K_2$ is non--empty and thus $F_1\cap F_2$ is non--empty, which implies that $F_1=F_2$. Therefore $K_1\cup K_2\subseteq F_1$ and thus again by Lemma \ref{DFH} it follows that $Y\leq e_X(K_1\cup K_2)$.
\end{proof}

Let $(X,\leq)$ be a partially ordered set with the largest element $u$. An element $a\in X$ is called an {\em anti--atom} in $X$ if $a\neq u$ and there exists  no $x\in X$ with $a<x<u$.

The following lemma is a counterpart of Lemma 9 in \cite{Mag3}.

\begin{lemma}\label{KJH}
Let ${\mathcal P}$ and  ${\mathcal Q}$ be a pair of compactness--like topological properties. Let $X$ be a Tychonoff locally--${\mathcal P}$ non--${\mathcal P}$ space with $\mathcal{Q}$. The following are equivalent:
\begin{itemize}
\item[\rm(1)] $Y$ is an anti--atom in ${\mathscr M}^{\mathcal Q}_{\mathcal P}(X)$.
\item[\rm(2)] $Y=e_X(\{a,b\})$ for some distinct $a,b\in \beta X\backslash X$ such that either $a\notin\lambda_{{\mathcal P}}X$ or  $b\notin\lambda_{{\mathcal P}}X$.
\end{itemize}
\end{lemma}

\begin{proof}
Note that  ${\mathscr M}^{\mathcal Q}_{\mathcal P}(X)$ has the largest element $\zeta_{\mathcal{P}} X=X\cup(\beta X\backslash\lambda_{{\mathcal P}}X)$. To show this first note that by Lemma \ref{15} we have $X\subseteq\lambda_{{\mathcal P}}X$. Thus $\zeta_{\mathcal{P}} X$ is a Tychonoff extension of $X$ which (since $\lambda_{{\mathcal P}}X$ is open in $\beta X$) has a compact remainder. Since $X\subseteq\zeta_{\mathcal{P}} X\subseteq \beta X$ we have $\beta \zeta_{\mathcal{P}} X=\beta X$ (see Corollary 3.6.9 of \cite{E}). It follows from Lemma \ref{16} (with $Y=X$, $f=\mbox{id}_X$, $T=\zeta_{\mathcal{P}} X$,  $\alpha T=\beta X$ and $\phi=\mbox{id}_{\beta X}$ in its statement) that  $\zeta_{\mathcal{P}} X$ has both ${\mathcal P}$  and ${\mathcal Q}$ and from  Theorem \ref{HUHG16} that  $\zeta_{\mathcal{P}} X\in {\mathscr M}^{\mathcal Q}_{\mathcal P}(X)$. That $\zeta_{\mathcal{P}} X$ is the largest element of ${\mathscr M}^{\mathcal Q}_{\mathcal P}(X)$ now follows from Theorem \ref{HUHG16} and Lemma \ref{DFH}.

That (2) implies (1) is trivial.  (1) {\em  implies} (2). Suppose that $Y$ is an anti--atom in ${\mathscr M}^{\mathcal Q}_{\mathcal P}(X)$. We show that except for a 2--element set the rest of the sets in ${\mathscr F}(Y)$ are singletons, the uniqueness part of Lemma \ref{LKA} will then imply (2). Suppose to the contrary that there exist some distinct $F_1,F_2\in {\mathscr F}(Y)$ such that $\mbox{card}(F_i)\geq 2$ for any $i=1,2$. By Theorem \ref{HUHG16} the set $F_i\backslash\lambda_{{\mathcal P}}X$ is non--empty for any $i=1,2$; choose some distinct $a_i,b_i\in F_i$ such that $a_i\notin \lambda_{{\mathcal P}}X$. Then $e_X(\{a_i,b_i\})$, where $i=1,2$, are distinct elements of  ${\mathscr M}^{\mathcal Q}_{\mathcal P}(X)$ and  $Y\leq e_X(\{a_i,b_i\})$, which contradicts (1). Thus there is at most one set in  ${\mathscr F}(Y)$ which is not a singleton, and  since $Y\neq\zeta_{\mathcal{P}} X$, there is at least one such set. Let $F\in {\mathscr F}(Y)$ be such that $\mbox{card}(F)\geq 2$. Suppose that $\mbox{card}(F)> 2$. By Theorem \ref{HUHG16} the set $F\backslash\lambda_{{\mathcal P}}X$ is non--empty. Let $a\in F\backslash\lambda_{{\mathcal P}}X$ and let $b,c\in F$ be distinct elements distinct from $a$. Then $e_X(\{a,b\})$ and $e_X(\{a,c\})$ are distinct  elements of ${\mathscr M}^{\mathcal Q}_{\mathcal P}(X)$ and  $Y\leq e_X(\{a,b\})$  and $Y\leq e_X(\{a,c\})$. This contradiction proves that $\mbox{card}(F)=2$.
\end{proof}

\begin{definition}\label{PGRHJ}
Let ${\mathcal P}$ and  ${\mathcal Q}$ be a pair of compactness--like topological properties. Let $X$ be a Tychonoff locally--${\mathcal P}$ non--${\mathcal P}$ space with $\mathcal{Q}$. An anti--atom $Y=e_X(\{a,b\})$ of ${\mathscr M}^{\mathcal Q}_{\mathcal P}(X)$ is said to  be {\em of type $(\mbox{\em I})$} if $\{a,b\}\cap\lambda_{{\mathcal P}}X$ is non--empty, otherwise, $Y$ is said to be {\em of type $(\mbox{\em II})$}.
\end{definition}

The purpose of the next two lemmas is to give an order--theoretic characterization of anti--atoms of type $(\mbox{I})$ (and thus anti--atoms of type $(\mbox{II})$ as well) in ${\mathscr M}^{\mathcal Q}_{\mathcal P}(X)$.

\begin{lemma}\label{GGFDS}
Let ${\mathcal P}$ and  ${\mathcal Q}$ be a pair of compactness--like topological properties.  Let $X$ be a Tychonoff locally--${\mathcal P}$ non--${\mathcal P}$ space with $\mathcal{Q}$ such that $\mbox{\em card}(\lambda_{{\mathcal P}}X\backslash X)\geq 2$. Then
\begin{itemize}
\item[\rm(1)] $\mbox{\em card}(\beta X\backslash\lambda_{{\mathcal P}}X)=1$ if and only if for any pair of distinct anti--atoms $Y$ and $Y'$ in ${\mathscr M}^{\mathcal Q}_{\mathcal P}(X)$ we have
    \[\mbox{\em card}\big(\big\{T:T\mbox{ is an anti--atom in }{\mathscr M}^{\mathcal Q}_{\mathcal P}(X) \mbox{ and }Y\wedge Y'\leq T\big\}\big)=2.\]
\item[\rm(2)] $\mbox{\em card}(\beta X\backslash\lambda_{{\mathcal P}}X)=2$ if and only if there exists an anti--atom $Y$ in ${\mathscr M}^{\mathcal Q}_{\mathcal P}(X)$ such that for any anti--atom $Y'$ in ${\mathscr M}^{\mathcal Q}_{\mathcal P}(X)$ with $Y'\neq Y$ we have
    \[\mbox{\em card}\big(\big\{T:T\mbox{ is an anti--atom in }{\mathscr M}^{\mathcal Q}_{\mathcal P}(X) \mbox{ and }Y\wedge Y'\leq T\big\}\big)=3.\]
\item[\rm(3)] $\mbox{\em card}(\beta X\backslash\lambda_{{\mathcal P}}X)\geq3$ if and only if there exist some anti--atoms $Y$, $Y'$ and $Y''$ in  ${\mathscr M}^{\mathcal Q}_{\mathcal P}(X)$ such that
    \[\mbox{\em card}\big(\big\{T:T\mbox{ is an anti--atom in }{\mathscr M}^{\mathcal Q}_{\mathcal P}(X) \mbox{ and }Y\wedge Y'\wedge Y''\leq T\big\}\big)=6.\]
\end{itemize}
Here $\wedge$ is the operation in ${\mathscr M}^{\mathcal Q}_{\mathcal P}(X)$.
\end{lemma}

\begin{proof}
Since $X$ is locally--${\mathcal P}$ by Lemma \ref{15} we have $X\subseteq\lambda_{{\mathcal P}}X$, and since $X$ is moreover non--${\mathcal P}$, by Lemma \ref{HFH} the set  $\beta X\backslash\lambda_{{\mathcal P}}X$ is non--empty.

(1). Suppose that $\mbox{card}(\beta X\backslash\lambda_{{\mathcal P}}X)=1$. Let $\beta X\backslash\lambda_{{\mathcal P}}X=\{a\}$.

Let $Y$ and $Y'$ be distinct anti--atoms in ${\mathscr M}^{\mathcal Q}_{\mathcal P}(X)$. Then by Lemma \ref{KJH} we have
\[Y=e_X(\{a,b\})\mbox{ and }Y'=e_X(\{a,c\})\]
for some $b,c\in\lambda_{{\mathcal P}}X\backslash X$. By Lemmas \ref{DFH} and \ref{LKA} the elements $b$ and $c$ are distinct. By Lemma \ref{OTH} we have
\[Y\wedge Y'=e_X\big(\{a,b\}\big)\wedge e_X\big(\{a,c\}\big)=e_X\big(\{a,b,c\}\big).\]
Using Lemmas \ref{DFH}, \ref{LKA} and \ref{KJH} it now follows that there are only 2 anti--atoms $T$ in ${\mathscr M}^{\mathcal Q}_{\mathcal P}(X)$ with $Y\wedge Y'\leq T$, namely,
\[e_X\big(\{a,b\}\big)\mbox{ and }e_X\big(\{a,c\}\big).\]

To show the converse, suppose that $\mbox{card}(\beta X\backslash\lambda_{{\mathcal P}}X)\neq1$. Choose some distinct $a,b\in\beta X\backslash\lambda_{{\mathcal P}}X$ and some $c\in\lambda_{{\mathcal P}}X\backslash X$. By Lemma \ref{KJH} the elements
\[Y=e_X\big(\{a,b\}\big),Y'=e_X\big(\{a,c\}\big)\mbox{ and }Y''=e_X\big(\{b,c\}\big)\]
are anti--atoms in ${\mathscr M}^{\mathcal Q}_{\mathcal P}(X)$ and by Lemmas \ref{DFH} and \ref{LKA} they are distinct.  By Lemma  \ref{OTH} we have
\[Y\wedge Y'=e_X\big(\{a,b\}\big)\wedge e_X\big(\{a,c\}\big)=e_X\big(\{a,b,c\}\big)\]
and if $T$ is either $Y$, $Y'$ or $Y''$ then by Lemmas \ref{DFH}, \ref{LKA} and \ref{KJH} we have $Y\wedge Y'\leq T$.

(2). Suppose that $\mbox{card}(\beta X\backslash\lambda_{{\mathcal P}}X)=2$. Let $\beta X\backslash\lambda_{{\mathcal P}}X=\{a,b\}$.

Let $Y=e_X(\{a,b\})$. Then by Lemma \ref{KJH} the element $Y$ is an anti--atom in ${\mathscr M}^{\mathcal Q}_{\mathcal P}(X)$. Now let $Y'$ be an anti--atom in ${\mathscr M}^{\mathcal Q}_{\mathcal P}(X)$ with $Y'\neq Y$. By Lemma \ref{KJH} we have $Y'=e_X(\{c,d\})$ for some distinct $c,d\in\beta X\backslash X$ with either $c\notin\lambda_{{\mathcal P}}X$ or $d\notin\lambda_{{\mathcal P}}X$. Without any loss of generality we may assume that $c\notin\lambda_{{\mathcal P}}X$ and that $c=a$. By Lemma \ref{OTH} we have \[Y\wedge Y'=e_X\big(\{a,b\}\big)\wedge e_X\big(\{a,d\}\big)=e_X\big(\{a,b,c\}\big).\]
Now using Lemmas \ref{DFH}, \ref{LKA} and \ref{KJH} if  $T$ is either
\begin{equation}\label{GGRF}
e_X\big(\{a,b\}\big),e_X\big(\{a,d\}\big)\mbox{ or }e_X\big(\{b,d\}\big)
\end{equation}
then $T$ is an anti--atom in ${\mathscr M}^{\mathcal Q}_{\mathcal P}(X)$ and $Y\wedge Y'\leq T$ and conversely any anti--atom $T$ in ${\mathscr M}^{\mathcal Q}_{\mathcal P}(X)$ with $Y\wedge Y'\leq T$ is of the above form. By Lemmas \ref{DFH} and \ref{LKA} the elements in (\ref{GGRF}) are distinct.

To show the converse, suppose that $\mbox{card}(\beta X\backslash\lambda_{{\mathcal P}}X)\neq2$. Either $\mbox{card}(\beta X\backslash\lambda_{{\mathcal P}}X)=1$ or $\mbox{card}(\beta X\backslash\lambda_{{\mathcal P}}X)\geq 3$. Consider the following cases:
\begin{description}
\item[{\sc Case 1.}] Suppose that $\mbox{card}(\beta X\backslash\lambda_{{\mathcal P}}X)=1$. Let $\beta X\backslash\lambda_{{\mathcal P}}X=\{a\}$. Let $Y$ be an anti--atom in ${\mathscr M}^{\mathcal Q}_{\mathcal P}(X)$. By Lemma \ref{KJH}  we have  $Y=e_X(\{a,b\})$ for some $b\in\lambda_{{\mathcal P}}X\backslash X$. Let $c\in\lambda_{{\mathcal P}}X\backslash X$ be distinct from $b$. (Such a $c$ exists, as we are assuming that $\mbox{card}(\lambda_{{\mathcal P}}X\backslash X)\geq 2$.) Let $Y'=e_X(\{a,c\})$. Then by Lemma \ref{KJH} the element $Y'$ is an anti--atom in ${\mathscr M}^{\mathcal Q}_{\mathcal P}(X)$ and by Lemmas \ref{DFH} and \ref{LKA} we have $Y'\neq Y$. By Lemma \ref{OTH} we have
    \[Y\wedge Y'=e_X\big(\{a,b\}\big)\wedge e_X\big(\{a,c\}\big)=e_X\big(\{a,b,c\}\big).\]
    Now using Lemmas \ref{DFH}, \ref{LKA} and \ref{KJH} it follows that the anti--atoms $T$ in ${\mathscr M}^{\mathcal Q}_{\mathcal P}(X)$ with $Y\wedge Y'\leq T$ are exactly
    \[e_X\big(\{a,b\}\big)\mbox{ and }e_X\big(\{a,c\}\big).\]
\item[{\sc Case 2.}] Suppose that $\mbox{card}(\beta X\backslash\lambda_{{\mathcal P}}X)\geq 3$. Let $Y$ be an anti--atom in ${\mathscr M}^{\mathcal Q}_{\mathcal P}(X)$. By Lemma \ref{KJH} we have $Y=e_X(\{a,b\})$ for some distinct $a,b\in\beta X\backslash X$ with either $a\notin\lambda_{{\mathcal P}}X$ or $b\notin\lambda_{{\mathcal P}}X$. Choose some $c\in\beta X\backslash\lambda_{{\mathcal P}}X$ and some $d\in\lambda_{{\mathcal P}}X\backslash X$ such that neither $c\notin\{a,b\}$ nor $d\notin\{a,b\}$. (Again, we are using the fact that $\mbox{card}(\lambda_{{\mathcal P}}X\backslash X)\geq 2$.) Let $Y'=e_X(\{c,d\})$. Then by Lemma \ref{KJH} the element $Y'$ is an anti--atom in ${\mathscr M}^{\mathcal Q}_{\mathcal P}(X)$ and  by Lemmas \ref{DFH} and \ref{LKA} we have $Y'\neq Y$. By Lemma \ref{OTH}  we have
    \[Y\wedge Y'=e_X\big(\{a,b\}\big)\wedge e_X\big(\{c,d\}\big)=e_X\big(\{a,b\},\{c,d\}\big).\]
    Now as above it follows that the anti--atoms $T$ in ${\mathscr M}^{\mathcal Q}_{\mathcal P}(X)$ with $Y\wedge Y'\leq T$ are exactly
    \[e_X\big(\{a,b\}\big)\mbox{ and }e_X\big(\{c,d\}\big).\]
\end{description}
Thus in either case for a given anti--atom $Y$ in ${\mathscr M}^{\mathcal Q}_{\mathcal P}(X)$ we can find an anti--atom $Y'$ in ${\mathscr M}^{\mathcal Q}_{\mathcal P}(X)$ distinct from $Y$ with at most 2 anti--atoms $T$ in ${\mathscr M}^{\mathcal Q}_{\mathcal P}(X)$ with $Y\wedge Y'\leq T$.

(3). Suppose that $\mbox{card}(\beta X\backslash\lambda_{{\mathcal P}}X)\geq3$.

Choose some distinct $a,b,c\in\beta X\backslash\lambda_{{\mathcal P}}X$ and some $d\in\lambda_{{\mathcal P}}X\backslash X$. By Lemma \ref{KJH} the elements
\[Y=e_X\big(\{a,b\}\big),Y'=e_X\big(\{b,c\}\big)\mbox{ and }Y''=e_X\big(\{c,d\}\big)\]
are anti--atoms in ${\mathscr M}^{\mathcal Q}_{\mathcal P}(X)$. By Lemma \ref{OTH} we have
\begin{eqnarray*}
Y\wedge Y'\wedge Y''&=&e_X\big(\{a,b\}\big)\wedge e_X\big(\{b,c\}\big)\wedge e_X\big(\{c,d\}\big)\\&=&e_X\big(\{a,b,c\}\big)\wedge e_X\big(\{c,d\}\big)=e_X\big(\{a,b,c,d\}\big)
\end{eqnarray*}
and therefore, using Lemmas \ref{DFH}, \ref{LKA} and \ref{KJH} it follows that there are 6 anti--atoms $T$ in ${\mathscr M}^{\mathcal Q}_{\mathcal P}(X)$ with $Y\wedge Y'\wedge Y''\leq T$, namely,
\[e_X\big(\{a,b\}\big),e_X\big(\{a,c\}\big),e_X\big(\{a,d\}\big),e_X\big(\{b,c\}\big),e_X\big(\{b,d\}\big)\mbox{ and }e_X\big(\{c,d\}\big).\]

To show the converse, suppose that $\mbox{card}(\beta X\backslash\lambda_{{\mathcal P}}X)\leq2$. Consider the following cases:
\begin{description}
\item[{\sc Case 1.}] Suppose that $\mbox{card}(\beta X\backslash\lambda_{{\mathcal P}}X)=1$. Let $\beta X\backslash\lambda_{{\mathcal P}}X=\{a\}$. Let $Y$, $Y'$ and $Y''$ be anti--atoms in ${\mathscr M}^{\mathcal Q}_{\mathcal P}(X)$. By Lemma \ref{KJH} we have
    \[Y=e_X\big(\{a,b\}\big),Y'=e_X\big(\{a,c\}\big)\mbox{ and }Y''=e_X\big(\{a,d\}\big)\]
    for some $b,c,d\in\lambda_{{\mathcal P}}X\backslash X$.
    Using Lemma \ref{OTH} we have
    \begin{eqnarray*}
    Y\wedge Y'\wedge Y''&=&e_X\big(\{a,b\}\big)\wedge e_X\big(\{a,c\}\big)\wedge e_X\big(\{a,d\}\big)\\&=&e_X\big(\{a,b,c\}\big)\wedge e_X\big(\{a,d\}\big)=e_X\big(\{a,b,c,d\}\big)
    \end{eqnarray*}
    and therefore, using Lemmas \ref{DFH}, \ref{LKA} and \ref{KJH} it follows that the only anti--atoms $T$ in ${\mathscr M}^{\mathcal Q}_{\mathcal P}(X)$ with $Y\wedge Y'\wedge Y''\leq T$ are
    \[e_X\big(\{a,b\}\big),e_X\big(\{a,c\}\big)\mbox{ and }e_X\big(\{a,d\}\big).\]
\item[{\sc Case 2.}] Suppose that $\mbox{card}(\beta X\backslash\lambda_{{\mathcal P}}X)=2$. Let $\beta X\backslash\lambda_{{\mathcal P}}X=\{a,b\}$. Let $Y$, $Y'$ and $Y''$ be anti--atoms in ${\mathscr M}^{\mathcal Q}_{\mathcal P}(X)$.  Consider the following cases:
    \begin{description}
    \item[{\sc Case 2.a.}] Suppose that $e_X(\{a,b\})\notin\{Y,Y',Y''\}$. Consider the following cases:
        \begin{description}
        \item[{\sc Case 2.a.i.}] Suppose that
        \[Y=e_X\big(\{c,d\}\big),Y'=e_X\big(\{c,e\}\big)\mbox{ and }Y''=e_X\big(\{c,f\}\big)\]
        where $c$ is either $a$ or $b$ and $d,e,f\in\lambda_{{\mathcal P}}X\backslash X$. By Lemma \ref{OTH} we have
        \begin{eqnarray*}
        Y\wedge Y'\wedge Y''&=&e_X\big(\{c,d\}\big)\wedge e_X\big(\{c,e\}\big)\wedge e_X\big(\{c,f\}\big)\\&=&e_X\big(\{c,d,e\}\big)\wedge e_X\big(\{c,f\}\big)=e_X\big(\{c,d,e,f\}\big)
        \end{eqnarray*}
        and therefore, again using Lemmas \ref{DFH}, \ref{LKA} and \ref{KJH} it follows that the only  anti--atoms $T$ in ${\mathscr M}^{\mathcal Q}_{\mathcal P}(X)$ with $Y\wedge Y'\wedge Y''\leq T$ are
        \[e_X\big(\{c,d\}\big),e_X\big(\{c,e\}\big)\mbox{ and }e_X\big(\{c,f\}\big).\]
    \item[{\sc Case 2.a.ii.}] Suppose that either
        \[Y=e_X\big(\{a,d\}\big),Y'=e_X\big(\{a,e\}\big)\mbox{ and }Y''=e_X\big(\{b,f\}\big)\]
        or
        \[Y=e_X\big(\{b,d\}\big),Y'=e_X\big(\{b,e\}\big)\mbox{ and }Y''=e_X\big(\{a,f\}\big)\]
        where $d,e,f\in\lambda_{{\mathcal P}}X\backslash X$. Without any loss of generality we may assume that the first of the above cases occurs. Suppose that $f\notin\{d,e\}$. Then by Lemma \ref{OTH} we have
        \begin{eqnarray*}
        Y\wedge Y'\wedge Y''&=&e_X\big(\{a,d\}\big)\wedge e_X\big(\{a,e\}\big)\wedge e_X\big(\{b,f\}\big)\\&=&e_X\big(\{a,d,e\}\big)\wedge e_X\big(\{b,f\}\big)=e_X\big(\{a,d,e\},\{b,f\}\big)
        \end{eqnarray*}
        and therefore as above it then follows that the anti--atoms $T$ in ${\mathscr M}^{\mathcal Q}_{\mathcal P}(X)$ with $Y\wedge Y'\wedge Y''\leq T$ are exactly
        \[e_X\big(\{a,d\}\big),e_X\big(\{a,e\}\big)\mbox{ and }e_X\big(\{b,f\}\big).\]
        Now suppose that $f\in\{d,e\}$, say $f=d$. Then by Lemma \ref{OTH} we have
        \begin{eqnarray*}
        Y\wedge Y'\wedge Y''&=&e_X\big(\{a,d\}\big)\wedge e_X\big(\{a,e\}\big)\wedge e_X\big(\{b,f\}\big)\\&=&e_X\big(\{a,d,e\}\big)\wedge e_X\big(\{b,d\}\big)=e_X\big(\{a,b,d,e\}\big)
        \end{eqnarray*}
        and therefore as above it follows that the anti--atoms $T$ in ${\mathscr M}^{\mathcal Q}_{\mathcal P}(X)$ such that $Y\wedge Y'\wedge Y''\leq T$ are exactly
        \[e_X\big(\{a,b\}\big),e_X\big(\{a,d\}\big),e_X\big(\{a,e\}\big),e_X\big(\{b,d\}\big)\mbox{ and }e_X\big(\{b,e\}\big).\]
    \end{description}
    \item[{\sc Case 2.b.}] Suppose that $e_X(\{a,b\})\in\{Y,Y',Y''\}$, say $Y=e_X(\{a,b\})$. Consider the following cases:
        \begin{description}
        \item[{\sc Case 2.b.i.}] Suppose that
            \[Y'=e_X\big(\{c,d\}\big)\mbox{ and }Y''=e_X\big(\{c,e\}\big)\]
            where $c$ is either $a$ or $b$, say $c=a$, and $d,e\in\lambda_{{\mathcal P}}X\backslash X$. Then by Lemma \ref{OTH} we have
            \begin{eqnarray*}
            Y\wedge Y'\wedge Y''&=&e_X\big(\{a,b\}\big)\wedge e_X\big(\{a,d\}\big)\wedge e_X\big(\{a,e\}\big)\\&=&e_X\big(\{a,b,d\}\big)\wedge e_X\big(\{a,e\}\big)=e_X\big(\{a,b,d,e\}\big)
            \end{eqnarray*}
            and therefore using Lemmas \ref{DFH}, \ref{LKA} and \ref{KJH} it follows that the anti--atoms $T$ in ${\mathscr M}^{\mathcal Q}_{\mathcal P}(X)$ such that $Y\wedge Y'\wedge Y''\leq T$ are exactly
            \[e_X\big(\{a,b\}\big),e_X\big(\{a,d\}\big),e_X\big(\{a,e\}\big),e_X\big(\{b,d\}\big)\mbox{ and }e_X\big(\{b,e\}\big).\]
        \item[{\sc Case 2.b.ii.}] Suppose that
            \[Y'=e_X\big(\{a,d\}\big)\mbox{ and }Y''=e_X\big(\{b,e\}\big)\]
            where $d,e\in\lambda_{{\mathcal P}}X\backslash X$. Then by Lemma \ref{OTH} we have
            \begin{eqnarray*}
            Y\wedge Y'\wedge Y''&=&e_X\big(\{a,b\}\big)\wedge e_X\big(\{a,d\}\big)\wedge e_X\big(\{b,e\}\big)\\&=&e_X\big(\{a,b,d\}\big)\wedge e_X\big(\{b,e\}\big)=e_X\big(\{a,b,d,e\}\big)
            \end{eqnarray*}
            and therefore as above it follows that the anti--atoms $T$ in ${\mathscr M}^{\mathcal Q}_{\mathcal P}(X)$ such that $Y\wedge Y'\wedge Y''\leq T$ are exactly
            \[e_X\big(\{a,b\}\big),e_X\big(\{a,d\}\big),e_X\big(\{a,e\}\big),e_X\big(\{b,d\}\big)\mbox{ and }e_X\big(\{b,e\}\big).\]
        \end{description}
    \end{description}
\end{description}
Thus in either case for any given anti--atoms $Y$, $Y'$ and $Y''$ of ${\mathscr M}^{\mathcal Q}_{\mathcal P}(X)$ there are at most 5 anti--atoms $T$ in ${\mathscr M}^{\mathcal Q}_{\mathcal P}(X)$ with $Y\wedge Y'\wedge Y''\leq T$.
\end{proof}

\begin{lemma}\label{HVFS}
Let ${\mathcal P}$ and  ${\mathcal Q}$ be a pair of compactness--like topological properties.  Let $X$ be a Tychonoff locally--${\mathcal P}$ non--${\mathcal P}$ space with $\mathcal{Q}$ such that $\mbox{\em card}(\lambda_{{\mathcal P}}X\backslash X)\geq 2$.
\begin{itemize}
\item[\rm(1)] Suppose that $\mbox{\em card}(\beta X\backslash\lambda_{{\mathcal P}}X)=1$. Then any anti--atom of ${\mathscr M}^{\mathcal Q}_{\mathcal P}(X)$ is of type $(\mbox{\em I})$.
\item[\rm(2)] Suppose that $\mbox{\em card}(\beta X\backslash\lambda_{{\mathcal P}}X)=2$. Then an anti--atom $Y$ of ${\mathscr M}^{\mathcal Q}_{\mathcal P}(X)$ is of type $(\mbox{\em I})$ if and only if there exists an anti--atom $Y'$ of ${\mathscr M}^{\mathcal Q}_{\mathcal P}(X)$ such that
    \[\mbox{\em card}\big(\big\{T:T\mbox{ is an anti--atom in }{\mathscr M}^{\mathcal Q}_{\mathcal P}(X) \mbox{ and }Y\wedge Y'\leq T\big\}\big)=2.\]
\item[\rm(3)] Suppose that $\mbox{\em card}(\beta X\backslash\lambda_{{\mathcal P}}X)\geq3$. Then an anti--atom $Y$ of ${\mathscr M}^{\mathcal Q}_{\mathcal P}(X)$ is of type $(\mbox{\em I})$ if and only if there exists an anti--atom $Y'$ of ${\mathscr M}^{\mathcal Q}_{\mathcal P}(X)$ with $Y'\neq Y$ such that for any anti--atom $Y''$ of ${\mathscr M}^{\mathcal Q}_{\mathcal P}(X)$ we have
    \[\mbox{\em card}\big(\big\{T:T\mbox{ is an anti--atom in }{\mathscr M}^{\mathcal Q}_{\mathcal P}(X) \mbox{ and }Y\wedge Y'\wedge Y''\leq T\big\}\big)\leq5.\]
\end{itemize}
Here $\wedge$ is the operation in ${\mathscr M}^{\mathcal Q}_{\mathcal P}(X)$.
\end{lemma}

\begin{proof}
By Lemmas \ref{15} and \ref{HFH} we have $X\subseteq\lambda_{{\mathcal P}}X$ and that $\beta X\backslash\lambda_{{\mathcal P}}X$ is non--empty.

(1). This is obvious. (2). Suppose that $\mbox{card}(\beta X\backslash\lambda_{{\mathcal P}}X)=2$. Let $\beta X\backslash\lambda_{{\mathcal P}}X=\{a,b\}$.

Suppose that $Y$ is an anti--atom in ${\mathscr M}^{\mathcal Q}_{\mathcal P}(X)$ of type $(\mbox{I})$. Then $Y=e_X(\{c,e\})$ where $c$ is either $a$ or $b$, say $c=a$, and $e\in\lambda_{{\mathcal P}}X\backslash X$. Choose some $d\in\lambda_{{\mathcal P}}X\backslash X$ such that $d\neq e$. (Such a $d$ exists, as we are assuming that $\mbox{card}(\lambda_{{\mathcal P}}X\backslash X)\geq 2$.)  Then by Lemma \ref{KJH} the element $Y'=e_X(\{b,d\})$ is an anti--atom in ${\mathscr M}^{\mathcal Q}_{\mathcal P}(X)$. By Lemma \ref{OTH} we have
\[Y\wedge Y'=e_X\big(\{a,e\}\big)\wedge e_X\big(\{b,d\}\big)=e_X\big(\{a,e\},\{b,d\}\big)\]
and thus the anti--atoms $T$ in ${\mathscr M}^{\mathcal Q}_{\mathcal P}(X)$ with $Y\wedge Y'\leq T$ are exactly
\[e_X\big(\{a,e\}\big)\mbox{ and }e_X\big(\{b,d\}\big)\]
which by Lemmas \ref{DFH} and \ref{LKA} are distinct.

To show the converse, suppose that an anti--atom $Y$ of ${\mathscr M}^{\mathcal Q}_{\mathcal P}(X)$ is not of type $(\mbox{I})$. Then necessarily $Y=e_X(\{a,b\})$. Let $Y'$ be an anti--atom of ${\mathscr M}^{\mathcal Q}_{\mathcal P}(X)$. If $Y'=Y$ then the only anti--atom $T$ of ${\mathscr M}^{\mathcal Q}_{\mathcal P}(X)$ with $Y=Y\wedge Y'\leq T$ is $Y$ itself. Suppose that $Y'\neq Y$. Using Lemmas \ref{DFH}, \ref{LKA} and \ref{KJH} we have $Y'=e_X(\{c,e\})$, where $c$ is either $a$ or $b$, say $c=a$, and $e\in \lambda_{{\mathcal P}}X\backslash X$. By Lemma \ref{OTH} we have
\[Y\wedge Y'=e_X\big(\{a,b\}\big)\wedge e_X\big(\{a,e\}\big)=e_X\big(\{a,b,e\}\big)\]
and therefore, again using Lemmas \ref{DFH}, \ref{LKA} and \ref{KJH} there are exactly 3 anti--atoms $T$ of ${\mathscr M}^{\mathcal Q}_{\mathcal P}(X)$ with $Y\wedge Y'\leq T$, namely, \[e_X\big(\{a,b\}\big),e_X\big(\{a,e\}\big)\mbox{ and }e_X\big(\{b,e\}\big).\]
Thus in either case the number of  anti--atoms $T$ in ${\mathscr M}^{\mathcal Q}_{\mathcal P}(X)$ with $Y\wedge Y'\leq T$ is not 2.

(3). Suppose that $\mbox{card}(\beta X\backslash\lambda_{{\mathcal P}}X)\geq3$.

Suppose that $Y$ is an anti--atom in ${\mathscr M}^{\mathcal Q}_{\mathcal P}(X)$ of type $(\mbox{I})$. By Lemma \ref{KJH} we have $Y=e_X(\{a,b\})$ for some distinct $a,b\in\beta X\backslash X$ such that either $a\notin\lambda_{{\mathcal P}}X$ or $b\notin\lambda_{{\mathcal P}}X$. Choose some $c\in \lambda_{{\mathcal P}}X\backslash X$ distinct from $b$ (this is possible as we are assuming that  $\mbox{card}(\lambda_{{\mathcal P}}X\backslash X)\geq 2$) and let $Y'=e_X(\{a,c\})$. Then $Y'$ is an anti--atom in ${\mathscr M}^{\mathcal Q}_{\mathcal P}(X)$ by Lemma \ref{KJH}, and $Y'\neq Y$ by Lemmas \ref{DFH} and \ref{LKA}. Let $Y''$ be an anti--atom in ${\mathscr M}^{\mathcal Q}_{\mathcal P}(X)$. By Lemma \ref{KJH} we have $Y''=e_X(\{d,e\})$ where $d,e\in\beta X\backslash X$ are distinct and either $d\notin\lambda_{{\mathcal P}}X$ or $e\notin\lambda_{{\mathcal P}}X$. Consider the following cases:
\begin{description}
\item[{\sc Case 1.}] Suppose that $\{a,b,c\}\cap\{d,e\}=\emptyset$. By Lemma \ref{OTH} we have
    \begin{eqnarray*}
    Y\wedge Y'\wedge Y''&=&e_X\big(\{a,b\}\big)\wedge e_X\big(\{a,c\}\big)\wedge e_X\big(\{d,e\}\big)\\&=&e_X\big(\{a,b,c\}\big)\wedge e_X\big(\{d,e\}\big)=e_X\big(\{a,b,c\},\{d,e\}\big)
    \end{eqnarray*}
    and therefore using Lemmas \ref{DFH}, \ref{LKA} and \ref{KJH} the anti--atoms $T$ in ${\mathscr M}^{\mathcal Q}_{\mathcal P}(X)$ such that $Y\wedge Y'\wedge Y''\leq T$ are exactly
    \[e_X\big(\{a,b\}\big),e_X\big(\{a,c\}\big)\mbox{ and }e_X\big(\{d,e\}\big).\]
\item[{\sc Case 2.}] Suppose that $\{a,b,c\}\cap\{d,e\}\neq\emptyset$. By Lemma \ref{OTH} we have
    \begin{eqnarray*}
    Y\wedge Y'\wedge Y''&=&e_X\big(\{a,b\}\big)\wedge e_X\big(\{a,c\}\big)\wedge e_X\big(\{d,e\}\big)\\&=&e_X\big(\{a,b,c\}\big)\wedge e_X\big(\{d,e\}\big)=e_X\big(\{a,b,c,d,e\}\big).
    \end{eqnarray*}
    Consider the following cases:
    \begin{description}
    \item[{\sc Case 2.a.}] Suppose that $a\in\{d,e\}$, say $a=d$.
        Consider the following cases:
        \begin{description}
        \item[{\sc Case 2.a.i.}] Suppose that $\{d,e\}\cap\lambda_{{\mathcal P}}X=\emptyset$. Now using Lemmas \ref{DFH}, \ref{LKA} and \ref{KJH} the anti--atoms $T$ in ${\mathscr M}^{\mathcal Q}_{\mathcal P}(X)$ with $Y\wedge Y'\wedge Y''\leq T$ are exactly
            \[e_X\big(\{a,b\}\big),e_X\big(\{a,c\}\big),e_X\big(\{a,e\}\big),e_X\big(\{b,e\}\big)\mbox{ and }e_X\big(\{c,e\}\big).\]
        \item[{\sc Case 2.a.ii.}] Suppose that $\{d,e\}\cap\lambda_{{\mathcal P}}X\neq\emptyset$. Then necessarily $e\in\lambda_{{\mathcal P}}X$
            and therefore as above the anti--atoms $T$ in ${\mathscr M}^{\mathcal Q}_{\mathcal P}(X)$ such that $Y\wedge Y'\wedge Y''\leq T$ are exactly
            \[e_X\big(\{a,b\}\big),e_X\big(\{a,c\}\big)\mbox{ and }e_X\big(\{a,e\}\big).\]
        \end{description}
    \item[{\sc Case 2.b.}] Suppose that $a\notin\{d,e\}$. Then $\{b,c\}\cap\{d,e\}$ is non--empty. Without any loss of generality we may assume that $c=d$. This implies that $e\notin\lambda_{{\mathcal P}}X$ and therefore again using Lemmas \ref{DFH}, \ref{LKA} and \ref{KJH} the anti--atoms $T$ in ${\mathscr M}^{\mathcal Q}_{\mathcal P}(X)$ with $Y\wedge Y'\wedge Y''\leq T$ are exactly
        \[e_X\big(\{a,b\}\big),e_X\big(\{a,c\}\big),e_X\big(\{a,e\}\big),e_X\big(\{b,e\}\big)\mbox{ and }e_X\big(\{c,e\}\big).\]
    \end{description}
\end{description}
Thus for this choice of $Y'$, for any anti--atom $Y''$ of ${\mathscr M}^{\mathcal Q}_{\mathcal P}(X)$ the number of anti--atoms $T$ of ${\mathscr M}^{\mathcal Q}_{\mathcal P}(X)$ with $Y\wedge Y'\wedge Y''\leq T$ are at most 5.

To show the converse, suppose that an anti--atom $Y$ of ${\mathscr M}^{\mathcal Q}_{\mathcal P}(X)$ is not of type $(\mbox{I})$. Then $Y=e_X(\{a,b\})$ for some distinct $a,b\in\beta X\backslash\lambda_{{\mathcal P}}X$. Let $Y'$ be an anti--atom in ${\mathscr M}^{\mathcal Q}_{\mathcal P}(X)$ distinct from $Y$. By Lemma \ref{KJH} we have $Y'=e_X(\{c,d\})$ for some distinct $c,d\in\beta X\backslash X$ with either $c\notin\lambda_{{\mathcal P}}X$ or $d\notin\lambda_{{\mathcal P}}X$. Consider the following cases:
\begin{description}
\item[{\sc Case 1.}] Suppose that $\{a,b\}\cap\{c,d\}=\emptyset$. Let $Y''=e_X(\{b,c\})$. By Lemma \ref{KJH} the element $Y''$ is an anti--atom in ${\mathscr M}^{\mathcal Q}_{\mathcal P}(X)$. By Lemma \ref{OTH} we have
    \begin{eqnarray*}
    Y\wedge Y'\wedge Y''&=&e_X\big(\{a,b\}\big)\wedge e_X\big(\{c,d\}\big)\wedge e_X\big(\{b,c\}\big)\\&=& e_X\big(\{a,b\}\big)\wedge e_X\big(\{b,c,d\}\big)=e_X\big(\{a,b,c,d\}\big)
    \end{eqnarray*}
    and therefore using Lemmas \ref{DFH}, \ref{LKA} and \ref{KJH} there are exactly 6 anti--atoms $T$ in ${\mathscr M}^{\mathcal Q}_{\mathcal P}(X)$ such that $Y\wedge Y'\wedge Y''\leq T$, namely,            \[e_X\big(\{a,b\}\big),e_X\big(\{a,c\}\big),e_X\big(\{a,d\}\big),e_X\big(\{b,c\}\big),e_X\big(\{b,d\}\big)\mbox{ and }e_X\big(\{c,d\}\big).\]
\item[{\sc Case 2.}] Suppose that $\{a,b\}\cap\{c,d\}\neq\emptyset$. Without any loss of generality we may assume that $b=c$. Consider the following cases:
    \begin{description}
    \item[{\sc Case 2.a.}] Suppose that $d\notin\lambda_{{\mathcal P}}X$. Choose some $e\in \lambda_{{\mathcal P}}X\backslash X$. (This is possible as we are assuming that $\mbox{card}(\lambda_{{\mathcal P}}X\backslash X)\geq2$.)
    \item[{\sc Case 2.b.}] Suppose that $d\in\lambda_{{\mathcal P}}X$. Choose some $e\in \beta X\backslash\lambda_{{\mathcal P}}X$ distinct from both $a$ and $b$. (This is possible as we are assuming that $\mbox{card}(\beta X\backslash\lambda_{{\mathcal P}}X)\geq3$.)
    \end{description}
    Now let $Y''=e_X(\{a,e\})$. By Lemma \ref{KJH} the element $Y''$ is an anti--atom in ${\mathscr M}^{\mathcal Q}_{\mathcal P}(X)$. By Lemma \ref{OTH} we have
    \begin{eqnarray*}
    Y\wedge Y'\wedge Y''&=&e_X\big(\{a,b\}\big)\wedge e_X\big(\{b,d\}\big)\wedge e_X\big(\{a,e\}\big)\\&=& e_X\big(\{a,b,d\}\big)\wedge e_X\big(\{a,e\}\big)=e_X\big(\{a,b,d,e\}\big)
    \end{eqnarray*}
    and therefore as above there are exactly 6 anti--atoms $T$ in ${\mathscr M}^{\mathcal Q}_{\mathcal P}(X)$ such that $Y\wedge Y'\wedge Y''\leq T$, namely,
    \[e_X\big(\{a,b\}\big),e_X\big(\{a,d\}\big),e_X\big(\{a,e\}\big),e_X\big(\{b,d\}\big),e_X\big(\{b,e\}\big)\mbox{ and }e_X\big(\{d,e\}\big).\]
\end{description}
Thus in either case for a given anti--atom $Y'$ of ${\mathscr M}^{\mathcal Q}_{\mathcal P}(X)$ distinct from $Y$ there is an anti--atom $Y''$ in ${\mathscr M}^{\mathcal Q}_{\mathcal P}(X)$ with exactly 6 anti--atoms $T$ in ${\mathscr M}^{\mathcal Q}_{\mathcal P}(X)$ such that $Y\wedge Y'\wedge Y''\leq T$.
\end{proof}

The following lemma together with Lemmas \ref{GGFDS} and \ref{HVFS} above gives an order--theoretic characterization of one--point extensions in ${\mathscr M}^{\mathcal Q}_{\mathcal P}(X)$.

\begin{lemma}\label{KHFA}
Let ${\mathcal P}$ and  ${\mathcal Q}$ be a pair of compactness--like topological properties.  Let $X$ be a Tychonoff locally--${\mathcal P}$ non--${\mathcal P}$ space with $\mathcal{Q}$ and let $Y\in{\mathscr M}^{\mathcal Q}_{\mathcal P}(X)$. The following are equivalent:
\begin{itemize}
\item[\rm(1)] $Y$ is a one--point extension of $X$.
\item[\rm(2)] $Y\leq T$ for any anti--atom $T$ of ${\mathscr M}^{\mathcal Q}_{\mathcal P}(X)$ of type $(\mbox{\em II})$.
\end{itemize}
\end{lemma}

\begin{proof}
Let $\phi:\beta X\rightarrow\beta Y$ be the continuous extension of $\mbox{id}_X$. By Lemma \ref{16} we have $\beta X\backslash\lambda_{\mathcal{P}}X\subseteq\phi^{-1}[Y\backslash X]$. Also, by Lemmas \ref{15} and \ref{HFH} we have $X\subseteq\lambda_{{\mathcal P}}X$ and that $\beta X\backslash\lambda_{{\mathcal P}}X$ is non--empty.

(1) {\em implies} (2). Note that ${\mathscr F}(Y)=\{\phi^{-1}[Y\backslash X]\}$. Let $T$ be an anti--atom in ${\mathscr M}^{\mathcal Q}_{\mathcal P}(X)$ of type $(\mbox{II})$. Then $Y=e_X(\{a,b\})$ for some distinct $a,b\in\beta X\backslash\lambda_{{\mathcal P}}X$. Since $\{a,b\}\subseteq\phi^{-1}[Y\backslash X]$ it follows from Lemmas \ref{DFH} and \ref{LKA} that $Y\leq T$.

(2) {\em implies} (1). Note that $Y\backslash X$ is non--empty, as $\beta X\backslash\lambda_{\mathcal{P}}X\subseteq\phi^{-1}[Y\backslash X]$ and $\beta X\backslash\lambda_{\mathcal{P}}X$ is non--empty. Suppose that $\mbox{card}(Y\backslash X)\geq 2$. Let $F,G\in {\mathscr F}(Y)$ be distinct. By Theorem \ref{HUHG16} both $F\backslash\lambda_{\mathcal P} X$ and $G\backslash\lambda_{\mathcal P} X$ are non--empty. Let $a\in F\backslash\lambda_{\mathcal P} X$ and $b\in G\backslash\lambda_{\mathcal P} X$. Then $T=e_X(\{a,b\})$ is an anti--atom in ${\mathscr M}^{\mathcal Q}_{\mathcal P}(X)$ of type $(\mbox{II})$, while $Y\nleq T$. This shows that  $Y\backslash X$ is a one--point set.
\end{proof}

The following lemma is well known. It is included here for the sake of completeness.

\begin{lemma}\label{YTRO}
Let $X$ be a Tychonoff space. Then for any compact non--empty subset $C$ of $\beta X\backslash X$ there exists a unique one--point Tychonoff extension $Y$ of $X$ with
$C=\phi^{-1}[Y\backslash X]$, where $\phi:\beta X\rightarrow\beta Y$ is the continuous extension of $\mbox{\em id}_X$.
\end{lemma}

\begin{proof}
Let $Z$ be the quotient space of $\beta X$ obtained by contracting $C$ to a point $p$ with the quotient mapping $q:\beta X\rightarrow Z$. Note that $Z$ is compact, being a  Hausdorff continuous image of $\beta X$. Consider the subspace $Y=X\cup\{p\}$ of $Z$. Then $Y$ is a one--point Tychonoff extension of $X$. We show that $Z=\beta Y$ and $q=\phi$ where $\phi:\beta X\rightarrow\beta Y$ is the continuous  extension of $\mbox{id}_X$. Note that $Z$ is a compactification of $Y$, as it contains  $Y$ as a dense subspace. Thus to show that  $Z=\beta Y$ it suffices to verify that any continuous $h:Y\rightarrow\mathbf{I}$ is continuously extendable over $Z$. Indeed, let $G:\beta X\rightarrow\mathbf{I}$ be the continuous extension of $hq|(X\cup C):X\cup C\rightarrow\mathbf{I}$. (Note that $\beta(X\cup C)=\beta X$, as $X\subseteq X\cup C\subseteq\beta X$; see  Corollary 3.6.9 of \cite{E}.) Define $H:T\rightarrow\mathbf{I}$ such that $H|(\beta X\cup C)=G|(\beta X\cup C)$ and $H(p)=h(p)$. Then $H|Y=h$, and since $Hq=G$ is continuous, $H$ is continuous. This shows that $Z=\beta Y$. That $q=\phi$ follows easily, as they are both continuous and coincide with $\mbox{id}_X$ on the dense subset $X$ of $\beta X$.  We have
\[C=q^{-1}(p)=q^{-1}[Y\backslash X]=\phi^{-1}[Y\backslash X].\]
For the uniqueness part, note that if $T$ also is a one--point Tychonoff extension of $X$ such that $C=\psi^{-1}[T\backslash X]$, where $\psi:\beta X\rightarrow\beta T$ is the continuous extension of $\mbox{id}_X$, then ${\mathscr F}(T)=\{C\}={\mathscr F}(Y)$ and thus $T=Y$ by Lemma \ref{DFH}.
\end{proof}

\begin{notation}\label{DFQA}
For a Tychonoff space $X$ and a compact non--empty subset $C$ of $\beta X\backslash X$ denote by $e_CX$ the unique one--point Tychonoff extension $Y$ of $X$ with $C=\phi^{-1}[Y\backslash X]$, where $\phi:\beta X\rightarrow\beta Y$ is the continuous extension of $\mbox{id}_X$.
\end{notation}

\begin{notation}\label{FFDS}
Let ${\mathcal P}$ and ${\mathcal Q}$ be a pair of compactness--like topological properties. Let $X$ be a Tychonoff locally--${\mathcal P}$ non--${\mathcal P}$ space with $\mathcal{Q}$. Denote
\[M^X_{\mathcal P}=e_CX\]
where $C=\beta X\backslash\lambda_{\mathcal{P}}X$.
\end{notation}

\begin{xrem}
{\em In Notation \ref{FFDS} the set $C=\beta X\backslash\lambda_{\mathcal{P}}X$ is a compact subset of $\beta X\backslash X$ and it is non--empty; see Lemma \ref{HFH}. Therefore the above definition of $e_CX$ makes sense.}
\end{xrem}

In the following we associated to any element $Y$ in ${\mathscr M}^{\mathcal Q}_{\mathcal P}(X)$ a certain one--point extension $Y_U$ in ${\mathscr M}^{\mathcal Q}_{\mathcal P}(X)$. This will be used in Lemma \ref{AYLF} when we order--theoretically characterize the locally compact elements of ${\mathscr M}^{\mathcal Q}_{\mathcal P}(X)$.

\begin{notation}\label{GFDA}
Let ${\mathcal P}$ and  ${\mathcal Q}$ be a pair of compactness--like topological properties. Let $X$ be a Tychonoff locally--${\mathcal P}$ non--${\mathcal P}$ space with $\mathcal{Q}$ and let $Y\in{\mathscr M}^{\mathcal Q}_{\mathcal P}(X)$. Denote
\[Y_U=e_CX\]
where $C=\phi^{-1}[Y\backslash X]$ and $\phi:\beta X\rightarrow\beta Y$ is the continuous extension of $\mbox{id}_X$.
\end{notation}

\begin{xrem}
{\em The definition in Notation \ref{GFDA} makes sense, as $C$ is a compact subset of $\beta X\backslash X$ and since by Lemma \ref{16} we have $\beta X\backslash\lambda_{\mathcal{P}}X\subseteq\phi^{-1}[Y\backslash X]$, it is non--empty (as $\beta X\backslash\lambda_{\mathcal{P}}X$ is non--empty; see Lemma \ref{HFH}).}
\end{xrem}

The following lemma together with Lemmas \ref{GGFDS}, \ref{HVFS} and \ref{KHFA} gives an order--theoretic characterization of the element $Y_U$ we already associated to any $Y$ in ${\mathscr M}^{\mathcal Q}_{\mathcal P}(X)$.

\begin{lemma}\label{DSEY}
Let ${\mathcal P}$ and  ${\mathcal Q}$ be a pair of compactness--like topological properties.  Let $X$ be a Tychonoff locally--${\mathcal P}$ non--${\mathcal P}$ space with $\mathcal{Q}$ and let $Y\in{\mathscr M}^{\mathcal Q}_{\mathcal P}(X)$. Then $Y_U$ is the largest $T\in{\mathscr M}^{\mathcal Q}_{\mathcal P}(X)$ satisfying  the following:
\begin{itemize}
\item[\rm(1)] $T$ is a one--point extension of $X$.
\item[\rm(2)] $T\leq Y'$ for any anti--atom $Y'$ of ${\mathscr M}^{\mathcal Q}_{\mathcal P}(X)$ of type $(\mbox{\em I})$ such that $Y\leq Y'$.
\end{itemize}
\end{lemma}

\begin{proof}
Let $\phi:\beta X\rightarrow\beta Y$ and $\phi_U:\beta X\rightarrow\beta Y_U$ denote the continuous extensions of $\mbox{id}_X$. By Lemma \ref{YTRO} we have $\phi^{-1}[Y\backslash X]=\phi_U^{-1}[Y_U\backslash X]$ and thus $Y_U\in{\mathscr M}^{\mathcal Q}_{\mathcal P}(X)$ by Lemma \ref{16} and Theorem \ref{HUHG16}, as $\beta X\backslash\lambda_{\mathcal{P}}X\subseteq\phi^{-1}[Y\backslash X]$ and by Lemma \ref{HFH} the set $\beta X\backslash\lambda_{\mathcal{P}}X$ is non--empty. Obviously, $Y_U$ satisfies (1). To show that $Y_U$ satisfies (2), let $Y'$ be an anti--atom of ${\mathscr M}^{\mathcal Q}_{\mathcal P}(X)$ of type $(\mbox{I})$ such that $Y\leq Y'$. Let $Y'=e_X(\{a,b\})$. Then by Lemmas \ref{DFH} and \ref{LKA} we have $\{a,b\}\subseteq F$ for some $F\in{\mathscr F}(Y)$. Since $F\subseteq\phi^{-1}[Y\backslash X]$ we have  $\{a,b\}\subseteq\phi_U^{-1}[Y_U\backslash X]$ and therefore, again by Lemmas \ref{DFH} and \ref{LKA} it follows that $Y_U\leq Y'$. Now  we show that $Y_U$ is the largest element of ${\mathscr M}^{\mathcal Q}_{\mathcal P}(X)$ satisfying conditions (1)--(2). Let for some $T\in{\mathscr M}^{\mathcal Q}_{\mathcal P}(X)$ conditions (1)--(2) hold. Let $\psi:\beta X\rightarrow\beta T$ be the continuous extension of $\mbox{id}_X$. To show that $T\leq Y_U$, by Lemma \ref{DFH}, it suffices to show that $\phi_U^{-1}[Y_U\backslash X]\subseteq\psi^{-1}[T\backslash X]$. Let $c\in \phi_U^{-1}[Y_U\backslash X]$. Suppose that $c\in\beta X\backslash\lambda_{\mathcal{P}}X$. Then obviously $c\in\psi^{-1}[T\backslash X]$, as by Lemma \ref{16} we have $\beta X\backslash\lambda_{\mathcal{P}}X\subseteq\psi^{-1}[T\backslash X]$. Now suppose that $c\in\lambda_{\mathcal{P}}X$. Since $c\in\phi^{-1}[Y\backslash X]$ there exists some $G\in{\mathscr F}(Y)$ such that $c\in G$. By Theorem \ref{HUHG16} the set $G\backslash\lambda_{\mathcal{P}}X$ is non--empty. Let $d\in G\backslash\lambda_{\mathcal{P}}X$. Then $Y'=e_X(\{c,d\})$ is an  anti--atom in ${\mathscr M}^{\mathcal Q}_{\mathcal P}(X)$ of type $(\mbox{I})$ which by Lemmas \ref{DFH} and \ref{LKA} is such that $Y\leq Y'$. Thus by our assumption $T\leq Y'$. Therefore again using Lemmas \ref{DFH} and \ref{LKA} it follows that $\{c,d\}\subseteq\psi^{-1}[T\backslash X]$. Thus $c\in\psi^{-1}[T\backslash X]$ in this case as well. This shows that $\phi_U^{-1}[Y_U\backslash X]\subseteq\psi^{-1}[T\backslash X]$ which completes the proof.
\end{proof}

In the following we define, and then order--theoretically characterize, certain elements of ${\mathscr M}^{\mathcal Q}_{\mathcal P}(X)$. This will have an application in the order--theoretic characterization of locally compact elements of ${\mathscr M}^{\mathcal Q}_{\mathcal P}(X)$ given in Lemma \ref{AYLF}.

\begin{definition}\label{GFXS}
Let ${\mathcal P}$ and  ${\mathcal Q}$ be a pair of compactness--like topological properties. Let $X$ be a Tychonoff locally--${\mathcal P}$ non--${\mathcal P}$ space with $\mathcal{Q}$. An element $Y\in{\mathscr M}^{\mathcal Q}_{\mathcal P}(X)$ is called {\em almost optimal} provided that $\lambda_{\mathcal{P}}X\cap\phi^{-1}[Y\backslash X]$ is compact, where $\phi:\beta X\rightarrow\beta Y$ is the continuous extension of $\mbox{id}_X$.
\end{definition}

\begin{lemma}\label{GSA}
Let ${\mathcal P}$ and  ${\mathcal Q}$ be a pair of compactness--like topological properties.  Let $X$ be a Tychonoff locally--${\mathcal P}$ non--${\mathcal P}$ space with $\mathcal{Q}$. Let $\{Y_i:i\in I\}\subseteq{\mathscr M}^{\mathcal Q}_{\mathcal P}(X)$ be a non--empty collection of one--point extensions of $X$. Then
\begin{itemize}
\item[\rm(1)] The least upper bound $\bigvee_{i\in I}Y_i$ exists in ${\mathscr M}^{\mathcal Q}_{\mathcal P}(X)$.
\item[\rm(2)] If $Y=\bigvee_{i\in I}Y_i$ then $Y$ is a one--point extension of $X$ and
\[\phi^{-1}[Y\backslash X]=\bigcap_{i\in I}\phi_i^{-1}[Y_i\backslash X]\]
where $\phi:\beta X\rightarrow\beta Y$ and $\phi_i:\beta X\rightarrow\beta Y_i$ for any $i\in I$ denote the continuous extensions of $\mbox{\em id}_X$.
\end{itemize}
Here $\vee$ is the operation in ${\mathscr M}^{\mathcal Q}_{\mathcal P}(X)$.
\end{lemma}

\begin{proof}
Let $\phi_i:\beta X\rightarrow\beta Y_i$ for any $i\in I$ be the continuous extension of $\mbox{id}_X$. Let
\[C=\bigcap_{i\in I}\phi_i^{-1}[Y_i\backslash X].\]
Then $C$ is compact, as it is closed in $\beta X$, and obviously $C\subseteq\beta X\backslash X$, as $\phi_i|X=\mbox{id}_X$ for any $i\in I$ (and $I$ is non--empty). By Lemma \ref{16} we have $\beta X\backslash\lambda_{\mathcal{P}}X\subseteq\phi_i^{-1}[Y_i\backslash X]$ for any $i\in I$. Therefore $\beta X\backslash\lambda_{\mathcal{P}}X\subseteq C$ which implies that  $C$ is non--empty, as by Lemma \ref{HFH} the set $\beta X\backslash\lambda_{\mathcal{P}}X$ is non--empty. Let $Y=e_CX$. Then $Y$ is a one--point Tychonoff extension of $X$ and $Y\in {\mathscr M}^{\mathcal Q}_{\mathcal P}(X)$ by Lemma \ref{16} and Theorem \ref{HUHG16}, as if $\phi:\beta X\rightarrow\beta Y$ denotes the continuous extensions of $\mbox{id}_X$, then using Lemma \ref{YTRO} we have $\beta X\backslash\lambda_{\mathcal{P}}X\subseteq C=\phi^{-1}[Y\backslash X]$. By Lemma \ref{DFH} it is obvious that $Y_i\leq Y$ for any $i\in I$, as $\phi^{-1}[Y\backslash X]=C\subseteq\phi_i^{-1}[Y_i\backslash X]$. We only need to show that $Y\leq Y'$, for any $Y'\in {\mathscr M}^{\mathcal Q}_{\mathcal P}(X)$ which satisfies $Y_i\leq Y'$ for any $i\in I$. Indeed, let $F\in {\mathscr F}(Y')$. By Lemma \ref{DFH} we have $F\subseteq\phi_i^{-1}[Y_i\backslash X]$ for any $i\in I$ and thus
\[F\subseteq\bigcap_{i\in I}\phi_i^{-1}[Y_i\backslash X]=C=\phi^{-1}[Y\backslash X].\]
Therefore $Y\leq Y'$ again by Lemma \ref{DFH}.
\end{proof}

The following lemma together with Lemmas \ref{GGFDS}, \ref{HVFS} and \ref{KHFA} gives an order--theoretic characterization of almost optimal elements of ${\mathscr M}^{\mathcal Q}_{\mathcal P}(X)$. Recall that a Tychonoff space $X$ is locally compact if and only if it is open in $\beta X$ (see Theorem 3.5.8 of \cite{E}). We use this well know fact in the proof of the following lemma.

\begin{lemma}\label{KJSQ}
Let ${\mathcal P}$ and  ${\mathcal Q}$ be a pair of compactness--like topological properties. Let $X$ be a locally compact locally--${\mathcal P}$ non--${\mathcal P}$ space with $\mathcal{Q}$ and let $Y\in{\mathscr M}^{\mathcal Q}_{\mathcal P}(X)$. The following are equivalent:
\begin{itemize}
\item[\rm(1)] $Y$ is almost optimal.
\item[\rm(2)] For any collection $\{Y_i:i\in I\}\subseteq{\mathscr M}^{\mathcal Q}_{\mathcal P}(X)$ of one--point extensions of $X$ such that $Y_U\vee\bigvee_{i\in I}Y_i=M^X_{\mathcal P}$ it follows that $Y_U\vee\bigvee_{j=1}^kY_{i_j}=M^X_{\mathcal P}$ for some $k\in\mathbf{N}$ and some $i_1,\ldots,i_k\in I$.
\end{itemize}
Here $\vee$ is the operation in ${\mathscr M}^{\mathcal Q}_{\mathcal P}(X)$.
\end{lemma}

\begin{proof}
Let $\phi:\beta X\rightarrow\beta Y$ and $\phi_U:\beta X\rightarrow\beta Y_U$ denote the continuous extensions of $\mbox{id}_X$. By Lemma \ref{YTRO} we have $\phi^{-1}[Y\backslash X]=\phi_U^{-1}[Y_U\backslash X]$.

(1) {\em implies} (2). Suppose that $\lambda_{\mathcal{P}}X\cap\phi^{-1}[Y\backslash X]$ is compact. Let $\{Y_i:i\in I\}\subseteq{\mathscr M}^{\mathcal Q}_{\mathcal P}(X)$ be a collection of one--point extensions of $X$ with $Y_U\vee\bigvee_{i\in I}Y_i=M^X_{\mathcal P}$. By Lemmas \ref{YTRO} and \ref{GSA} we have
\[\phi_U^{-1}[Y_U\backslash X]\cap\bigcap_{i\in I}\phi_i^{-1}[Y_i\backslash X]=\beta X\backslash\lambda_{\mathcal{P}}X\]
where $\phi_i:\beta X\rightarrow\beta Y_i$ for any $i\in I$, denotes the continuous extension of $\mbox{id}_X$. Now
\[\lambda_{\mathcal{P}}X\cap\phi^{-1}[Y\backslash X]\cap\bigcap_{i\in I}\phi_i^{-1}[Y_i\backslash X]=\lambda_{\mathcal{P}}X\cap\phi_U^{-1}[Y_U\backslash X]\cap\bigcap_{i\in I}\phi_i^{-1}[Y_i\backslash X]=\emptyset\]
and therefore by the compactness of $\lambda_{\mathcal{P}}X\cap\phi^{-1}[Y\backslash X]$ it follows that
\[\lambda_{\mathcal{P}}X\cap\phi^{-1}[Y\backslash X]\cap\bigcap_{j=1}^k\phi_{i_j}^{-1}[Y_{i_j}\backslash X]=\emptyset\]
for some $k\in\mathbf{N}$ and some $i_1,\ldots,i_k\in I$. This implies that
\[\phi^{-1}[Y\backslash X]\cap\bigcap_{j=1}^k\phi_{i_j}^{-1}[Y_{i_j}\backslash X]\subseteq\beta X\backslash\lambda_{\mathcal{P}}X.\]
But by Lemma \ref{16} we have $\beta X\backslash\lambda_{\mathcal{P}}X\subseteq\phi^{-1}[Y\backslash X]$ and $\beta X\backslash\lambda_{\mathcal{P}}X\subseteq\phi_i^{-1}[Y_i\backslash X]$ for any $i\in I$. Thus from above
\[\phi_U^{-1}[Y_U\backslash X]\cap\bigcap_{j=1}^k\phi_{i_j}^{-1}[Y_{i_j}\backslash X]=\phi^{-1}[Y\backslash X]\cap\bigcap_{j=1}^k\phi_{i_j}^{-1}[Y_{i_j}\backslash X]=\beta X\backslash\lambda_{\mathcal{P}}X.\]
Lemma \ref{GSA} now implies that $Y_U\vee\bigvee_{j=1}^kY_{i_j}=M^X_{\mathcal P}$.

(2) {\em implies} (1). First note that $X\subseteq\lambda_{\mathcal{P}}X$ (see Lemma \ref{15}) and that $\beta X\backslash\lambda_{\mathcal{P}}X$ is non--empty (see Lemma \ref{HFH}). To show (1) we have to verify that $\lambda_{\mathcal{P}}X\cap\phi^{-1}[Y\backslash X]$ is compact. Note that $\lambda_{\mathcal{P}}X\cap\phi^{-1}[Y\backslash X]\subseteq\lambda_{\mathcal{P}}X\backslash X$, as obviously $\phi^{-1}[Y\backslash X]\subseteq\beta X\backslash X$, because $\phi|X=\mbox{id}_X$. Let $\{U_i:i\in I\}$ be an open cover of  $\lambda_{\mathcal{P}}X\cap\phi^{-1}[Y\backslash X]$ in $\lambda_{\mathcal{P}}X\backslash X$. Note that each $U_i$, where $i\in I$, is open in $\beta X\backslash X$, as  $U_i$ is open in $\lambda_{\mathcal{P}}X\backslash X$ and the latter is open in $\beta X\backslash X$, because $\lambda_{\mathcal{P}}X$ is open in $\beta X$. Let $C_i=(\beta X\backslash X)\backslash U_i$ where $i\in I$. Then $C_i$ is closed in $\beta X\backslash X$ and thus it is compact, as $\beta X\backslash X$ is compact (because $X$ is locally compact) and it is non--empty, as it contains $\beta X\backslash\lambda_{\mathcal{P}}X$. Let $Y_i=e_{C_i}X$ and let $\phi_i:\beta X\rightarrow\beta Y_i$ denotes the continuous extension of $\mbox{id}_X$. By Lemma \ref{YTRO} we have
\begin{equation}\label{JJWA}
\phi_i^{-1}[Y_i\backslash X]=C_i=(\beta X\backslash X)\backslash U_i.
\end{equation}
We have
\begin{eqnarray*}
\phi^{-1}[Y\backslash X]\cap\bigcap_{i\in I}\phi_i^{-1}[Y_i\backslash X]&=&\phi^{-1}[Y\backslash X]\cap\bigcap_{i\in I}\big((\beta X\backslash X)\backslash U_i\big)\\&=&\phi^{-1}[Y\backslash X]\cap\Big((\beta X\backslash X)\backslash \bigcup_{i\in I}U_i\Big)\\&\subseteq&\phi^{-1}[Y\backslash X]\cap\big((\beta X\backslash X)\backslash\big(\lambda_{\mathcal{P}}X\cap\phi^{-1}[Y\backslash X]\big)\big)\\&\subseteq&\beta X\backslash\lambda_{\mathcal{P}}X
\end{eqnarray*}
and therefore
\[\phi_U^{-1}[Y_U\backslash X]\cap\bigcap_{i\in I}\phi_i^{-1}[Y_i\backslash X]=\phi^{-1}[Y\backslash X]\cap\bigcap_{i\in I}\phi_i^{-1}[Y_i\backslash X]=\beta X\backslash\lambda_{\mathcal{P}}X\]
as by Lemma \ref{16} we have $\beta X\backslash\lambda_{\mathcal{P}}X\subseteq\phi^{-1}[Y\backslash X]$ and $\beta X\backslash\lambda_{\mathcal{P}}X\subseteq\phi_i^{-1}[Y_i\backslash X]$ for any $i\in I$. Lemma \ref{GSA} now implies that $Y_U\vee\bigvee_{i\in I}Y_i=M^X_{\mathcal P}$ which yields $Y_U\vee\bigvee_{j=1}^kY_{i_j}=M^X_{\mathcal P}$ for some $k\in\mathbf{N}$ and some $i_1,\ldots,i_k\in I$. Again, using Lemma \ref{GSA} we have
\[\phi^{-1}[Y\backslash X]\cap\bigcap_{j=1}^k\phi_{i_j}^{-1}[Y_{i_j}\backslash X]=\phi_U^{-1}[Y_U\backslash X]\cap\bigcap_{j=1}^k\phi_{i_j}^{-1}[Y_{i_j}\backslash X]=\beta X\backslash\lambda_{\mathcal{P}}X\]
and thus by (\ref{JJWA}) it then follows that
\begin{eqnarray*}
\phi^{-1}[Y\backslash X]\cap\Big((\beta X\backslash X)\backslash \bigcup_{j=1}^kU_{i_j}\Big)&=&\phi^{-1}[Y\backslash X]\cap\bigcap_{j=1}^k\big((\beta X\backslash X)\backslash U_{i_j}\big)\\&=&\phi^{-1}[Y\backslash X]\cap\bigcap_{j=1}^k\phi_{i_j}^{-1}[Y_{i_j}\backslash X]=\beta X\backslash\lambda_{\mathcal{P}}X.
\end{eqnarray*}
Therefore
\[\lambda_{\mathcal{P}}X\cap\phi^{-1}[Y\backslash X]\cap\Big((\beta X\backslash X)\backslash \bigcup_{j=1}^kU_{i_j}\Big)=\emptyset\]
which implies that
\[\lambda_{\mathcal{P}}X\cap\phi^{-1}[Y\backslash X]\subseteq\bigcup_{j=1}^kU_{i_j}.\]
That is $\lambda_{\mathcal{P}}X\cap\phi^{-1}[Y\backslash X]$ is compact.
\end{proof}

\begin{lemma}\label{FSAEW}
Let $X$ be a locally compact space and let $Y\in {\mathscr E}(X)$. The following are equivalent:
\begin{itemize}
\item[\rm(1)] $Y$ is locally compact.
\item[\rm(2)] $\phi^{-1}[Y\backslash X]$ is open in $\beta X\backslash X$, where $\phi:\beta X\rightarrow\beta Y$ is the continuous extension of $\mbox{\em id}_X$.
\end{itemize}
\end{lemma}

\begin{proof}
Recall that by Lemma \ref{j2} the space $\beta Y$ is the quotient space of $\beta X$ obtained by contracting each $\phi^{-1}(p)$ where $p\in Y\backslash X$, to a point with the continuous extension $\phi:\beta X\rightarrow\beta Y$ of $\mbox{id}_X$ as the quotient mapping.

(1) {\em implies} (2). Note that $Y$ is open in $\beta Y$ and thus $\phi^{-1}[Y]$ is open in $\beta X$. Therefore
\[\phi^{-1}[Y\backslash X]=\phi^{-1}[Y]\cap(\beta X\backslash X)\]
is open in $\beta X\backslash X$.

(2) {\em implies} (1). Let $U$ be an open subset of $\beta X$ such that  $\phi^{-1}[Y\backslash X]=U\cap(\beta X\backslash X)$. Since $X$ is locally compact, $X$ is open in $\beta X$. Thus $U\cup X$ is open in $\beta X$ and therefore $Y=\phi[U\cup X]$ is open in $\beta Y$. This shows that $Y$ is locally compact.
\end{proof}

The following lemma together with Lemmas \ref{GGFDS}, \ref{HVFS}, \ref{KHFA}, \ref{DSEY} and \ref{KJSQ} gives an order--theoretic characterization of locally compact  elements of ${\mathscr M}^{\mathcal Q}_{\mathcal P}(X)$.

\begin{lemma}\label{AYLF}
Let ${\mathcal P}$ and  ${\mathcal Q}$ be a pair of compactness--like topological properties.  Let $X$ be a locally compact locally--${\mathcal P}$ non--${\mathcal P}$ space with $\mathcal{Q}$ and let $Y\in{\mathscr M}^{\mathcal Q}_{\mathcal P}(X)$. The following are equivalent:
\begin{itemize}
\item[\rm(1)] $Y$ is locally compact.
\item[\rm(2)] There exists an almost optimal one--point extension $Y'\in{\mathscr M}^{\mathcal Q}_{\mathcal P}(X)$ such that for any anti--atom $T$ of ${\mathscr M}^{\mathcal Q}_{\mathcal P}(X)$ of type $(\mbox{\em I})$, either $Y_U\leq T$ or $Y'\leq T$, but not both.
\end{itemize}
\end{lemma}

\begin{proof}
Let $\phi:\beta X\rightarrow\beta Y$ and $\phi_U:\beta X\rightarrow\beta Y_U$ denote the continuous extensions of $\mbox{id}_X$. By Lemmas \ref{16} and \ref{YTRO} we have $\beta X\backslash\lambda_{\mathcal{P}}X\subseteq\phi^{-1}[Y\backslash X]$ and $\phi^{-1}[Y\backslash X]=\phi_U^{-1}[Y_U\backslash X]$. Also, by Lemmas \ref{15} and \ref{HFH} we have $X\subseteq\lambda_{{\mathcal P}}X$ and that $\beta X\backslash\lambda_{{\mathcal P}}X$ is non--empty.

(1) {\em implies} (2). By Lemma \ref{FSAEW} the set $\phi^{-1}[Y\backslash X]$ is open in $\beta X\backslash X$. Let
\[C=\big((\beta X\backslash X)\backslash\phi^{-1}[Y\backslash X]\big)\cup(\beta X\backslash\lambda_{\mathcal{P}}X).\]
Then $C\subseteq\beta X\backslash X$ is compact, as it is the union of two compact subspaces, and it is obviously non--empty, as it contains $\beta X\backslash\lambda_{{\mathcal P}}X$. Let $Y'=e_CX$. Then $Y'$ is a Tychonoff one--point extension of $X$ and by Lemma \ref{YTRO}, if $\psi:\beta X\rightarrow\beta Y'$ denotes the continuous extensions of $\mbox{id}_X$ then $\psi^{-1}[Y'\backslash X]=C$. Therefore by Lemma \ref{16} and Theorem \ref{HUHG16} we have  $Y'\in{\mathscr M}^{\mathcal Q}_{\mathcal P}(X)$, as $\beta X\backslash\lambda_{\mathcal{P}}X\subseteq C$. Also, $Y'$ is almost optimal, as
\[\lambda_{\mathcal{P}}X\cap\psi^{-1}[Y'\backslash X]=\lambda_{\mathcal{P}}X\cap C=(\beta X\backslash X)\backslash\phi^{-1}[Y\backslash X]\]
is compact. (Note that $(\beta X\backslash X)\backslash\phi^{-1}[Y\backslash X]\subseteq\lambda_{\mathcal{P}}X$, as $\beta X\backslash\lambda_{\mathcal{P}}X\subseteq\phi^{-1}[Y\backslash X]$.) Now consider an anti--atom $T$ of ${\mathscr M}^{\mathcal Q}_{\mathcal P}(X)$ of type $(\mbox{I})$. Then $T=e_X(\{a,b\})$ for some $a\in\beta X\backslash\lambda_{\mathcal{P}}X$ and some $b\in\lambda_{\mathcal{P}}X\backslash X$. Consider the following cases:
\begin{description}
\item[{\sc Case 1.}] Suppose that $b\in\phi^{-1}[Y\backslash X]$. Then $a\in\phi^{-1}[Y\backslash X]$, as $a\in\beta X\backslash\lambda_{\mathcal{P}}X$ and thus $\{a,b\}\subseteq\phi_U^{-1}[Y_U\backslash X]$. Lemmas \ref{DFH} and \ref{LKA} now imply that $Y_U\leq T$.
\item[{\sc Case 2.}] Suppose that $b\notin\phi^{-1}[Y\backslash X]$. Then necessarily $b\in C=\psi^{-1}[Y'\backslash X]$. But also $a\in\psi^{-1}[Y'\backslash X]$, as $a\in\beta X\backslash\lambda_{\mathcal{P}}X$ and $\beta X\backslash\lambda_{\mathcal{P}}X\subseteq\psi^{-1}[Y'\backslash X]$. Thus $\{a,b\}\subseteq\psi^{-1}[Y'\backslash X]$. Again Lemmas \ref{DFH} and \ref{LKA} imply that $Y'\leq T$.
\end{description}
Obviously, the two $Y_U\leq T$ and $Y'\leq T$ cannot simultaneously hold, as this implies both $\{a,b\}\subseteq\phi_U^{-1}[Y_U\backslash X]$ and  $\{a,b\}\subseteq\psi^{-1}[Y'\backslash X]$. But $b\in\psi^{-1}[Y'\backslash X]=C$ implies, by the choice of $b$, that $b\in(\beta X\backslash X)\backslash\phi^{-1}[Y\backslash X]$, which is not possible, as $b\in\phi^{-1}[Y\backslash X]$.

(2) {\em implies} (1). Let $\psi:\beta X\rightarrow\beta Y'$ denote the continuous extension of $\mbox{id}_X$. By Lemma \ref{FSAEW} to prove that $Y$ is locally compact it suffices to show that $\phi^{-1}[Y\backslash X]$ is open in $\beta X\backslash X$. We show this by verifying that
\begin{equation}\label{SAQP}
(\beta X\backslash X)\backslash\phi^{-1}[Y\backslash X]=\lambda_{\mathcal{P}}X\cap\psi^{-1}[Y'\backslash X].
\end{equation}
Choose some $a\in\beta X\backslash\lambda_{\mathcal{P}}X$. Let $b\in(\beta X\backslash X)\backslash\phi^{-1}[Y\backslash X]$ and suppose to the contrary that $b\notin\psi^{-1}[Y'\backslash X]$. Let $T=e_X(\{a,b\})$. Then $T$ is an anti--atom in ${\mathscr M}^{\mathcal Q}_{\mathcal P}(X)$ of type $(\mbox{I})$, and by Lemmas \ref{DFH} and \ref{LKA} neither $Y_U\leq T$ nor $Y'\leq T$ holds, as neither $\{a,b\}\subseteq\phi^{-1}[Y\backslash X]$ nor $\{a,b\}\subseteq\psi^{-1}[Y'\backslash X]$. This is a contradiction. Therefore
\begin{equation}\label{ASKL}
(\beta X\backslash X)\backslash\phi^{-1}[Y\backslash X]\subseteq\lambda_{\mathcal{P}}X\cap\psi^{-1}[Y'\backslash X].
\end{equation}
To show the reverse inclusion in (\ref{ASKL}), let $c\in\lambda_{\mathcal{P}}X\cap\psi^{-1}[Y'\backslash X]$. Suppose to the contrary that $c\notin(\beta X\backslash X)\backslash\phi^{-1}[Y\backslash X]$, or equivalently that $c\in\phi^{-1}[Y\backslash X]$, as $c\in\beta X\backslash X$, because $c\in\psi^{-1}[Y'\backslash X]$ and $\psi^{-1}[Y'\backslash X]\subseteq\beta X\backslash X$, since $\psi|X=\mbox{id}_X$. Let $T'=e_X(\{a,c\})$. Then $T'$ is an  anti--atom in ${\mathscr M}^{\mathcal Q}_{\mathcal P}(X)$ of type $(\mbox{I})$ and by Lemmas \ref{DFH} and \ref{LKA} both $Y_U\leq T'$ and $Y'\leq T'$, as both $\{a,c\}\subseteq\phi^{-1}[Y\backslash X]$ and $\{a,c\}\subseteq\psi^{-1}[Y'\backslash X]$, because by Lemma \ref{16} we have $\beta X\backslash\lambda_{\mathcal{P}}X\subseteq\phi^{-1}[Y\backslash X]$ and $\beta X\backslash\lambda_{\mathcal{P}}X\subseteq\psi^{-1}[Y'\backslash X]$ and $a\in\beta X\backslash\lambda_{\mathcal{P}}X$. This contradicts our assumption and proves (\ref{SAQP}). Now since $\lambda_{\mathcal{P}}X\cap\psi^{-1}[Y'\backslash X]$ is compact, as $Y'$ is almost optimal, $(\beta X\backslash X)\backslash\phi^{-1}[Y\backslash X]$ is compact and thus closed in $\beta X\backslash X$. Equivalently, $\phi^{-1}[Y\backslash X]$ is open in $\beta X\backslash X$, as $\phi^{-1}[Y\backslash X]\subseteq\beta X\backslash X$, because $\phi|X=\mbox{id}_X$.
\end{proof}

The following lemma together with Lemmas \ref{GGFDS} and \ref{HVFS} gives an order--theoretic characterization of optimal elements of ${\mathscr M}^{\mathcal Q}_{\mathcal P}(X)$.

\begin{lemma}\label{PKJF}
Let ${\mathcal P}$ and  ${\mathcal Q}$ be a pair of compactness--like topological properties.  Let $X$ be a Tychonoff locally--${\mathcal P}$ non--${\mathcal P}$ space with $\mathcal{Q}$ and let $Y\in{\mathscr M}^{\mathcal Q}_{\mathcal P}(X)$. The following are equivalent:
\begin{itemize}
\item[\rm(1)] $Y\in{\mathscr O}^{\mathcal Q}_{\mathcal P}(X)$.
\item[\rm(2)] There exists no anti--atom $T$ in ${\mathscr M}^{\mathcal Q}_{\mathcal P}(X)$ of type $(\mbox{\em I})$ with $Y\leq T$.
\end{itemize}
\end{lemma}

\begin{proof}
Let $\phi:\beta X\rightarrow\beta Y$ denote the continuous extension of $\mbox{id}_X$.

(1) {\em implies} (2). Let $T$ be an anti--atom in ${\mathscr M}^{\mathcal Q}_{\mathcal P}(X)$ such that $Y\leq T$. Let $T=e_X(\{a,b\})$ where $a,b\in\beta X\backslash X$ are distinct such that either $a\notin\lambda_{{\mathcal P}}X$ or $b\notin\lambda_{{\mathcal P}}X$. By Lemmas \ref{DFH} and \ref{LKA} we have $\{a,b\}\subseteq F$ for some $F\in {\mathscr F}(Y)$. But $F\subseteq\phi^{-1}[Y\backslash X]$ and $\phi^{-1}[Y\backslash X]=\beta X\backslash\lambda_{\mathcal{P}}X$ by Theorem \ref{HG16}. Therefore $\{a,b\}\subseteq\beta X\backslash\lambda_{\mathcal{P}}X$ which shows that $T$ is of type $(\mbox{II})$.

(2) {\em implies} (1). By Theorem \ref{HG16} to show (1) it suffices to show that $\phi^{-1}[Y\backslash X]=\beta X\backslash\lambda_{\mathcal{P}}X$. Suppose otherwise. By Lemma \ref{16} we have $\beta X\backslash\lambda_{\mathcal{P}}X\subseteq\phi^{-1}[Y\backslash X]$. Thus $\phi^{-1}[Y\backslash X]\nsubseteq\beta X\backslash\lambda_{\mathcal{P}}X$. Choose some $b\in\phi^{-1}[Y\backslash X]$ such that $b\notin\beta X\backslash\lambda_{\mathcal{P}}X$. Then $b\in\phi^{-1}(p)$ for some $p\in Y\backslash X$. By Theorem \ref{HUHG16} the set $\phi^{-1}(p)\backslash\lambda_{\mathcal{P}}X$ is non--empty. Choose an $a\in\phi^{-1}(p)\backslash\lambda_{\mathcal{P}}X$. Note that $a,b\in\beta X\backslash X$, as $a,b\in\phi^{-1}[Y\backslash X]$ and $\phi|X=\mbox{id}_X$. Consider the  anti--atom $T=e_X(\{a,b\})$ of ${\mathscr M}^{\mathcal Q}_{\mathcal P}(X)$. Then $T$ is of type $(\mbox{I})$, and since $\{a,b\}\subseteq\phi^{-1}(p)$, by Lemmas \ref{DFH} and \ref{LKA} we have $Y\leq T$. This is a contradiction.
\end{proof}

The following is an immediate corollary of our previous lemmas.

\begin{lemma}\label{FSWR}
Let ${\mathcal P}$ and  ${\mathcal Q}$ be a pair of compactness--like topological properties.  Let $X$ and $Y$ be Tychonoff locally--${\mathcal P}$ non--${\mathcal P}$ spaces with $\mathcal{Q}$ such that $\mbox{\em card}(\lambda_{{\mathcal P}}X\backslash X)\geq 2$ and $\mbox{\em card}(\lambda_{{\mathcal P}}Y\backslash Y)\geq 2$. Let
\[\Theta:\big({\mathscr M}^{\mathcal Q}_{\mathcal P}(X),\leq\big)\rightarrow\big({\mathscr M}^{\mathcal Q}_{\mathcal P}(Y),\leq\big)\]
be an order--isomorphism.  Let $T\in{\mathscr M}^{\mathcal Q}_{\mathcal P}(X)$. Then
\begin{itemize}
\item[\rm(1)] If $T$ is an anti--atom in ${\mathscr M}^{\mathcal Q}_{\mathcal P}(X)$ (an anti--atom in ${\mathscr M}^{\mathcal Q}_{\mathcal P}(X)$ of type $(\mbox{\em I})$, an anti--atom in ${\mathscr M}^{\mathcal Q}_{\mathcal P}(X)$ of type $(\mbox{\em II})$, respectively), then so is $\Theta(T)$.
\item[\rm(2)] If $T$ is optimal, then so is $\Theta(T)$.
\item[\rm(3)] If $T$ is a one--point extension, then so is $\Theta(T)$.
\end{itemize}
Suppose that $X$ and $Y$ are moreover locally compact. Then
\begin{itemize}
\item[\rm(4)] If $T$ is almost optimal, then so is $\Theta(T)$.
\item[\rm(5)] If $T$ is locally compact, then so is $\Theta(T)$.
\end{itemize}
\end{lemma}

\begin{proof}
This follows from the previous lemmas, as Lemmas \ref{GGFDS} and \ref{HVFS} imply (1), part (1) and Lemma \ref{PKJF} imply (2), part (1) and Lemma \ref{KHFA} imply (3), part (3) and Lemma \ref{KJSQ} imply (4), and finally parts (1), (3), (4) and Lemma \ref{AYLF} imply (5), noting that by Lemma \ref{DSEY} (and parts (1) and (3)) we have $\Theta(S_U)=(\Theta(S))_U$ for any $S\in{\mathscr M}^{\mathcal Q}_{\mathcal P}(X)$.
\end{proof}

\begin{lemma}\label{KLYS}
Let ${\mathcal P}$ and  ${\mathcal Q}$ be a pair of compactness--like topological properties. Let $X$ be a Tychonoff locally--${\mathcal P}$ non--${\mathcal P}$ space with $\mathcal{Q}$. Let $T=e_X(\{a,b\})$ be an anti--atom in ${\mathscr M}^{\mathcal Q}_{\mathcal P}(X)$ of type $(\mbox{\em I})$ where $a\notin\lambda_{\mathcal{P}}X$ and $b\in\lambda_{\mathcal{P}}X$. Let $T'=e_X(\{c,d\})$ be an anti--atom in ${\mathscr M}^{\mathcal Q}_{\mathcal P}(X)$ of type $(\mbox{\em I})$ such that $T\neq T'$. The following are equivalent:
\begin{itemize}
\item[\rm(1)] $b\notin\{c,d\}$.
\item[\rm(2)]
\[\mbox{\em card}\big(\big\{T'':T''\mbox{ is an anti--atom in }{\mathscr M}^{\mathcal Q}_{\mathcal P}(X) \mbox{ and }T\wedge T'\leq T''\big\}\big)=2.\]
\end{itemize}
\end{lemma}

\begin{proof}
(1) {\em implies} (2). Consider the following cases:
\begin{description}
\item[{\sc Case 1.}] Suppose that $a\in\{c,d\}$, say $a=c$. By Lemma \ref{OTH} we have
    \[T\wedge T'=e_X\big(\{a,b\}\big)\wedge e_X\big(\{a,d\}\big)=e_X\big(\{a,b,d\}\big).\]
    Now using Lemmas \ref{DFH}, \ref{LKA} and \ref{KJH} there are only 2 anti--atoms $T''$ in ${\mathscr M}^{\mathcal Q}_{\mathcal P}(X)$ with $T\wedge T'\leq T''$, namely
    \[e_X\big(\{a,b\}\big)\mbox{ and }e_X\big(\{a,d\}\big).\]
\item[{\sc Case 2.}] Suppose that $a\notin\{c,d\}$. Again by Lemma \ref{OTH} we have
    \[T\wedge T'=e_X\big(\{a,b\}\big)\wedge e_X\big(\{c,d\}\big)=e_X\big(\{a,b\},\{c,d\}\big)\]
    and thus as above there are only 2 anti--atoms $T''$ in ${\mathscr M}^{\mathcal Q}_{\mathcal P}(X)$ with $ T\wedge T'\leq T''$, namely
    \[e_X\big(\{a,b\}\big)\mbox{ and }e_X\big(\{c,d\}\big).\]
\end{description}
Therefore (2) holds in either case.

(2) {\em implies} (1). Suppose to the contrary that $b\in\{c,d\}$, say $b=c$. Note that using Lemmas \ref{DFH}, \ref{LKA} and \ref{KJH} it follows that  $a\neq d$, as $T\neq T'$, and thus there are exactly 3 anti--atoms $T''$ in ${\mathscr M}^{\mathcal Q}_{\mathcal P}(X)$ with $T\wedge T'\leq T''$, namely
\[e_X\big(\{a,b\}\big),e_X\big(\{a,d\}\big)\mbox{ and }e_X\big(\{b,d\}\big).\]
This is a contradiction.
\end{proof}

The following lemma gives an internal (to $X$) characterization of spaces $X$ with $\mbox{card}(\lambda_{{\mathcal P}}X\backslash X)\geq 2$. This assumption has been used before in the statements of a couple of lemmas.

\begin{lemma}\label{PLFA}
Let $X$ be a Tychonoff space and let ${\mathcal P}$ be a  clopen hereditary finitely additive perfect topological property. The following are equivalent:
\begin{itemize}
\item[\rm(1)] $\mbox{\em card}(\lambda_{{\mathcal P}}X\backslash X)\geq 2$.
\item[\rm(2)] There exist a pair of disjoint non--compact zero--sets of $X$ each contained in a cozero--set of $X$ whose closure (in $X$) has ${\mathcal P}$.
\end{itemize}
\end{lemma}

\begin{proof}
(1) {\em  implies} (2). Let $z_i\in\lambda_{{\mathcal P}}X\backslash X$ where $i=1,2$ be distinct and let $U_i$ be an open neighborhoods of $z_i$ in $\lambda_{{\mathcal P}}X$ (and therefore in $\beta X$, as $\lambda_{{\mathcal P}}X$ is open in $\beta X$) such that $U_1\cap U_2=\emptyset$. Let $f_i:\beta X\rightarrow\mathbf{I}$ be continuous with $f(z_i)=0$ and $f_i[\beta X\backslash U_i]\subseteq\{1\}$ and let
\[Z_i=f_i^{-1}\big[[0,1/3]\big]\cap X\in{\mathscr Z}(X)\mbox{ and }C_i=f_i^{-1}\big[[0,1/2)\big]\cap X\in Coz(X).\]
Note that $Z_i\subseteq f_i^{-1}[[0,1/3]]\subseteq U_i$ and thus $Z_1\cap Z_2=\emptyset$. Also, since
\[z_i\in f_i^{-1}\big[[0,1/3)\big]\subseteq\mbox{int}_{\beta X}\mbox{cl}_{\beta X}\big(f_i^{-1}\big[[0,1/3]\big]\cap X\big)\subseteq\mbox{cl}_{\beta X}\big(f_i^{-1}\big[[0,1/3]\big]\cap X\big)=\mbox{cl}_{\beta X}Z_i\]
the set $\mbox{cl}_{\beta X}Z_i\backslash X$ is non--empty and therefore $Z_i$ is non--compact. Then
\[\mbox{cl}_{\beta X} C_i=\mbox{cl}_{\beta X}\big(f_i^{-1}\big[[0,1/2)\big]\cap X\big)=\mbox{cl}_{\beta X}f_i^{-1}\big[[0,1/2)\big]\subseteq f_i^{-1}\big[[0,1/2]\big]\subseteq U_i\subseteq\lambda_{{\mathcal P}}X\]
and thus by Lemma \ref{B} the closure $\mbox{cl}_X C_i$ has ${\mathcal P}$.

(2) {\em  implies} (1). Let $Z_i\in{\mathscr Z}(X)$ be non--compact with $Z_i\subseteq C_i$ where $C_i\in Coz(X)$ and $\mbox{cl}_X C_i$ has ${\mathcal P}$ and $Z_1\cap Z_2=\emptyset$. By Lemma \ref{BA27} we have $\mbox{cl}_{\beta X}Z_i\subseteq\lambda_{{\mathcal P}}X$. Since $Z_i$ is non--compact, $\mbox{cl}_{\beta X}Z_i\backslash X$ is non--empty. Also
\[\mbox{cl}_{\beta X}Z_1\cap\mbox{cl}_{\beta X}Z_2=\mbox{cl}_{\beta X}(Z_1\cap Z_2)=\emptyset.\]
Therefore $\mbox{card}(\lambda_{{\mathcal P}}X\backslash X)\geq 2$.
\end{proof}

\begin{theorem}\label{LGLH}
Let ${\mathcal P}$ and  ${\mathcal Q}$ be a pair of compactness--like topological properties.  Let $X$ and $Y$ be Tychonoff locally--${\mathcal P}$ non--${\mathcal P}$ spaces with $\mathcal{Q}$ such that each space contains a pair of disjoint non--compact zero--sets each contained in a cozero--set  whose closure has ${\mathcal P}$. Consider the following:
\begin{itemize}
\item[\rm(1)] $({\mathscr M}^{\mathcal Q}_{\mathcal P}(X),\leq)$ and $({\mathscr M}^{\mathcal Q}_{\mathcal P}(Y),\leq)$ are order--isomorphic.
\item[\rm(2)] $\beta X\backslash\lambda_{{\mathcal P}}X$ and $\beta Y\backslash\lambda_{{\mathcal P}}Y$ are homeomorphic.
\end{itemize}
Then $(1)$ implies $(2)$, while, $(2)$ does not necessarily imply $(1)$.
\end{theorem}

\begin{proof}
By Lemma \ref{PLFA} we have
\[\mbox{card}(\lambda_{{\mathcal P}}X\backslash X)\geq 2\mbox{ and }\mbox{card}(\lambda_{{\mathcal P}}Y\backslash Y)\geq 2.\]
To show that (1) implies (2) let
\[\Theta:\big({\mathscr M}^{\mathcal Q}_{\mathcal P}(X),\leq\big)\rightarrow\big({\mathscr M}^{\mathcal Q}_{\mathcal P}(Y),\leq\big)\]
denote an order--isomorphism. By Lemma \ref{FSWR} we have $\Theta({\mathscr O}^{\mathcal Q}_{\mathcal P}(X))\subseteq{\mathscr O}^{\mathcal Q}_{\mathcal P}(Y)$. Now, since  \[\Theta^{-1}:\big({\mathscr M}^{\mathcal Q}_{\mathcal P}(Y),\leq\big)\rightarrow\big({\mathscr M}^{\mathcal Q}_{\mathcal P}(X),\leq\big)\]
also is an order--isomorphism, again, using Lemma \ref{FSWR} we have $\Theta^{-1}({\mathscr O}^{\mathcal Q}_{\mathcal P}(Y))\subseteq{\mathscr O}^{\mathcal Q}_{\mathcal P}(X)$, or equivalently, that ${\mathscr O}^{\mathcal Q}_{\mathcal P}(Y)\subseteq\Theta({\mathscr O}^{\mathcal Q}_{\mathcal P}(X))$. Therefore $\Theta({\mathscr O}^{\mathcal Q}_{\mathcal P}(X))={\mathscr O}^{\mathcal Q}_{\mathcal P}(Y)$. Thus
\[\Theta|{\mathscr O}^{\mathcal Q}_{\mathcal P}(X):\big({\mathscr O}^{\mathcal Q}_{\mathcal P}(X),\leq\big)\rightarrow\big({\mathscr O}^{\mathcal Q}_{\mathcal P}(Y),\leq\big)\]
is an order--isomorphism. By Theorem \ref{KKLFA} this now implies (2).

Next, by means of an example, we show that (2) does not necessarily imply (1). Let
\[X=\mathbf{N}\oplus\bigoplus_{i<\Omega}R_i\mbox{ and }Y=\bigoplus_{i<\Omega}R_i\]
where $R_i$ for any $i<\Omega$ is the subspace $[0,\infty)$ of $\mathbf{R}$. Let ${\mathcal P}$ be the Lindel\"{o}f property  and let ${\mathcal Q}$ be regularity. Then
${\mathcal P}$ and  ${\mathcal Q}$ is a pair of compactness--like topological properties (see Example \ref{20UIHG}) and $X$ and $Y$ are locally compact non---Lindel\"{o}f spaces each containing a pair of disjoint non--compact clopen Lindel\"{o}f subsets. (Just consider $R_i$ and $R_j$ for some distinct $i,j<\Omega$ as the desired pair.) Note that
\begin{equation}\label{AWF}
\lambda_{{\mathcal P}}X=\mbox{cl}_{\beta X}\mathbf{N}\cup\bigcup\Big\{\mbox{cl}_{\beta X}\Big(\bigcup_{i\in J}R_i\Big):J\subseteq[0,\Omega)\mbox{ is countable}\Big\}
\end{equation}
and
\begin{equation}\label{PGDS}
\lambda_{{\mathcal P}}Y=\bigcup\Big\{\mbox{cl}_{\beta Y}\Big(\bigcup_{i\in J}R_i\Big):J\subseteq[0,\Omega)\mbox{ is countable}\Big\}
\end{equation}
as, for example, in the first case, for any Lindel\"{o}f $Z\in{\mathscr Z}(X)$ we have
\[Z\subseteq\mathbf{N}\cup\bigcup_{i\in J}R_i\]
for some countable $J\subseteq[0,\Omega)$, and conversely, if $J\subseteq[0,\Omega)$ is countable then
\[S=\mathbf{N}\cup\bigcup_{i\in J}R_i\]
is a clopen Lindel\"{o}f subset of $X$, and thus $\mbox{cl}_{\beta X}S=\mbox{int}_{\beta X}\mbox{cl}_{\beta X}S\subseteq\lambda_{{\mathcal P}}X$. (Note that a clopen subset of a Tychonoff space has a clopen closure in its Stone--\v{C}ech compactification; see Corollary 3.6.5 of \cite{E}.) We now verify that $\beta X\backslash\lambda_{{\mathcal P}}X$ and $\beta Y\backslash\lambda_{{\mathcal P}}Y$ are homeomorphic. Since $X$ contains $Y$ as a closed subspace (and it is normal) the spaces $\mbox{cl}_{\beta X}Y$ and $\beta Y$ are equivalent compactifications of $Y$ (see Corollary 3.6.8 of \cite{E}). Therefore for any countable $J\subseteq[0,\Omega)$ we have
\[\mbox{cl}_{\beta Y}\Big(\bigcup_{i\in J}R_i\Big)=\mbox{cl}_{(\mbox{cl}_{\beta X}Y)}\Big(\bigcup_{i\in J}R_i\Big)=\mbox{cl}_{\beta X}\Big(\bigcup_{i\in J}R_i\Big)\cap\mbox{cl}_{\beta X}Y=\mbox{cl}_{\beta X}\Big(\bigcup_{i\in J}R_i\Big).\]
Thus by (\ref{PGDS}) we have
\begin{equation}\label{PJEW}
\lambda_{{\mathcal P}}Y=\bigcup\Big\{\mbox{cl}_{\beta X}\Big(\bigcup_{i\in J}R_i\Big):J\subseteq[0,\Omega)\mbox{ is countable}\Big\}.
\end{equation}
We then by (\ref{AWF}) have
\[\lambda_{{\mathcal P}}X=\mbox{cl}_{\beta X}\mathbf{N}\cup\lambda_{{\mathcal P}}Y\]
and also, since $X=\mathbf{N}\cup Y$ we have
\[\beta X=\mbox{cl}_{\beta X}(\mathbf{N}\cup Y)=\mbox{cl}_{\beta X}\mathbf{N}\cup\mbox{cl}_{\beta X} Y=\mbox{cl}_{\beta X}\mathbf{N}\cup\beta Y.\]
Note that for any countable $J\subseteq[0,\Omega)$ we have
\[\mbox{cl}_{\beta X}\mathbf{N}\cap\mbox{cl}_{\beta X}\Big(\bigcup_{i\in J}R_i\Big)=\emptyset\]
as $\mathbf{N}$ and $\bigcup_{i\in J}R_i$ are disjoint zero--sets (in fact clopen subsets) of $X$. Therefore by (\ref{PJEW}) we have
\[\mbox{cl}_{\beta X}\mathbf{N}\cap\lambda_{{\mathcal P}}Y=\emptyset.\]
Also
\[\mbox{cl}_{\beta X}\mathbf{N}\cap\beta Y=\mbox{cl}_{\beta X}\mathbf{N}\cap\mbox{cl}_{\beta X}Y=\emptyset.\]
It now follows that
\[\beta X\backslash\lambda_{{\mathcal P}}X=(\mbox{cl}_{\beta X}\mathbf{N}\cup\beta Y)\backslash(\mbox{cl}_{\beta X}\mathbf{N}\cup\lambda_{{\mathcal P}}Y)=\beta Y\backslash\lambda_{{\mathcal P}}Y.\]
This shows (2). Next, we show that the partially ordered sets ${\mathscr M}^{\mathcal Q}_{\mathcal P}(X)$ and ${\mathscr M}^{\mathcal Q}_{\mathcal P}(Y)$ are not order--isomorphic. But first, we need to prove the following.

\begin{xclaim}
Let $D$ be a non--empty clopen subset of $\beta Y\backslash Y$. Then $\mbox{\em cl}_{\beta Y}R_i\backslash Y\subseteq D$ for some $i<\Omega$.
\end{xclaim}

\subsubsection*{Proof of the claim} Let $g:\beta Y\backslash Y\rightarrow\mathbf{I}$ be continuous with
\[g[D]\subseteq\{0\}\mbox{ and }g\big[(\beta Y\backslash Y)\backslash D\big]\subseteq\{1\}.\]
Since $Y$ is locally compact, $\beta Y\backslash Y$ is closed in (the normal space) $\beta Y$ and thus by The Tietze--Urysohn Theorem $g=G|(\beta Y\backslash Y)$ for some continuous $G:\beta Y\rightarrow\mathbf{I}$. Let
\[V=G^{-1}\big[[0,1/2)\big]\cap Y.\]
Then $V$ is an open subset of $Y$. Since
\begin{eqnarray*}
G^{-1}\big[[0,1/2)\big]\backslash Y&\subseteq&\mbox{cl}_{\beta Y}G^{-1}\big[[0,1/2)\big]\backslash Y\\&=&\mbox{cl}_{\beta Y}\big(G^{-1}\big[[0,1/2)\big]\cap Y\big)\backslash Y\subseteq G^{-1}\big[[0,1/2]\big]\backslash Y
\end{eqnarray*}
and
\[G^{-1}\big[[0,1/2)\big]\backslash Y=g^{-1}\big[[0,1/2)\big]=D=g^{-1}\big[[0,1/2]\big]=G^{-1}\big[[0,1/2]\big]\backslash Y\]
it follows that
\[D=\mbox{cl}_{\beta Y}\big(G^{-1}\big[[0,1/2)\big]\cap Y\big)\backslash Y=\mbox{cl}_{\beta Y}V\backslash Y.\]
Also, $\mbox{bd}_YV$ is compact, as
\begin{eqnarray*}
\mbox{bd}_YV=\mbox{cl}_YV\backslash V&\subseteq&\big(G^{-1}\big[[0,1/2]\big]\cap Y\big)\backslash\big(G^{-1}\big[[0,1/2)\big]\cap Y\big)\\&=&\big(G^{-1}\big[[0,1/2]\big]\backslash G^{-1}\big[[0,1/2)\big]\big)\cap Y\\&=&G^{-1}(1/2)\cap Y\subseteq G^{-1}(1/2)
\end{eqnarray*}
which implies that
\[\mbox{cl}_{\beta Y}\mbox{bd}_YV\backslash Y\subseteq G^{-1}(1/2)\backslash Y= g^{-1}(1/2)=\emptyset.\]
Therefore $\mbox{cl}_{\beta Y}\mbox{bd}_YV\subseteq Y$ and thus $\mbox{bd}_YV=\mbox{cl}_{\beta Y}\mbox{bd}_YV\cap Y=\mbox{cl}_{\beta Y}\mbox{bd}_YV$ is compact, as it is closed in $\beta Y$. Let
\[H=\{i<\Omega:\mbox{bd}_YV\cap R_i\neq\emptyset\}.\]
Note that $H$ is finite, as $\mbox{bd}_YV$ is compact. To prove the claim suppose to the contrary that $\mbox{cl}_{\beta Y}R_i\backslash Y\nsubseteq D$ for any $i<\Omega$. But
\[\mbox{cl}_{\beta Y}R_i\backslash Y=\mbox{cl}_{\beta Y}R_i\backslash R_i=\beta R_i\backslash R_i\]
as $\mbox{cl}_{\beta Y}R_i$ and $\beta R_i$ are equivalent compactifications of $R_i$, because $R_i$ is closed in $Y$ (and $Y$ is normal) and, therefore since $\beta R_i\backslash R_i$ is connected (see Problem 6L of \cite{GJ}), we have
\begin{equation}\label{GFDW}
\mbox{cl}_{\beta Y}R_i\cap D=(\mbox{cl}_{\beta Y}R_i\backslash Y)\cap D=\emptyset
\end{equation}
for any $i<\Omega$. Now let $i<\Omega$ be such that $V\cap R_i$ is non--empty. If $\mbox{bd}_{R_i}(V\cap R_i)=\emptyset$ then $V\cap R_i$ is clopen in $R_i$, and since $R_i$ is connected, $V\cap R_i=R_i$, that is, $R_i\subseteq V$. But then
\[\emptyset\neq\beta R_i\backslash R_i=\mbox{cl}_{\beta Y}R_i\backslash Y\subseteq\mbox{cl}_{\beta Y}V\backslash Y=D\]
which by (\ref{GFDW}) cannot be true. Thus
\[\mbox{bd}_YV\cap R_i=\mbox{bd}_{R_i}(V\cap R_i)\neq\emptyset\]
that is, $i\in H$. Therefore $V\subseteq\bigcup_{i\in H}R_i$. Now
\[D=\mbox{cl}_{\beta Y}V\backslash Y\subseteq\mbox{cl}_{\beta Y}V\subseteq\mbox{cl}_{\beta Y}\Big(\bigcup_{i\in H}R_i\Big)=\bigcup_{i\in H}\mbox{cl}_{\beta Y}R_i\]
which again contradicts (\ref{GFDW}), as $D$ is non--empty. This proves the claim.

\medskip

\noindent Now we prove that ${\mathscr M}^{\mathcal Q}_{\mathcal P}(X)$ and ${\mathscr M}^{\mathcal Q}_{\mathcal P}(Y)$ are not order--isomorphic. Suppose the contrary and let
\[\Theta:\big({\mathscr M}^{\mathcal Q}_{\mathcal P}(X),\leq\big)\rightarrow\big({\mathscr M}^{\mathcal Q}_{\mathcal P}(Y),\leq\big)\]
denote an order--isomorphism. Let
\[C=(\beta X\backslash X)\backslash\mbox{cl}_{\beta X}\mathbf{N}.\]
Then $C$ is a clopen non--empty subset of $\beta X\backslash X$. (Note that $\beta X\backslash X$ is closed in $\beta X$, as $X$ is locally compact, and $\mbox{cl}_{\beta X}\mathbf{N}$ is clopen in  $\beta X$, as $\mathbf{N}$ is clopen in $X$.) Let $T=e_CX$. Then $T$ is a one--point Tychonoff extension of $X$, which by Lemmas \ref{YTRO} and \ref{FSAEW} is locally compact. By Lemmas \ref{16} and \ref{YTRO} and Theorem \ref{HUHG16} we have $T\in{\mathscr M}^{\mathcal Q}_{\mathcal P}(X)$. Let $S=\Theta(T)$. Then by Lemma \ref{FSWR} the element $S$ is a one--point locally compact extension of $Y$. Denote by $\psi:\beta Y\rightarrow\beta S$ the continuous extension of $\mbox{id}_Y$. Then $D=\psi^{-1}(S\backslash Y)$ is clopen in $\beta Y\backslash Y$ by Lemma \ref{FSAEW}, and obviously $D\neq\beta Y\backslash Y$, as $S$ is not the smallest element in ${\mathscr M}^{\mathcal Q}_{\mathcal P}(Y)$, because $T$ is not the smallest element in ${\mathscr M}^{\mathcal Q}_{\mathcal P}(X)$, since $C\neq\beta X\backslash X$. By the above claim, $\mbox{cl}_{\beta Y}R_i\backslash Y\subseteq(\beta Y\backslash Y)\backslash D$ for some $i<\Omega$. Choose some distinct $b',c'\in\mbox{cl}_{\beta Y}R_i\backslash Y$ (which exist, as by above $\mbox{cl}_{\beta Y}R_i\backslash Y=\beta R_i\backslash R_i$) and choose some $a'\in\beta Y\backslash\lambda_{{\mathcal P}}Y$ (which exists, as $Y$ is non---Lindel\"{o}f; see Lemma \ref{HFH}). Let
\[S'=e_Y\big(\{a',b'\}\big)\mbox{ and }S''=e_Y\big(\{a',c'\}\big)\]
and let
\[T'=\Theta^{-1}(S')\mbox{ and }T''=\Theta^{-1}(S'').\]
By Lemma \ref{FSWR} both $T'$ and $T''$ are anti--atoms in ${\mathscr M}^{\mathcal Q}_{\mathcal P}(X)$ of type $(\mbox{I})$, as $S'$ and $S''$ are anti--atoms in ${\mathscr M}^{\mathcal Q}_{\mathcal P}(Y)$ of type $(\mbox{I})$. Let
\[T'=e_X\big(\{a,b\}\big)\mbox{ and }T''=e_X\big(\{c,d\}\big)\]
where $b,d\in\lambda_{{\mathcal P}}X$. Note that $b'\notin\{a',c'\}$, which by Lemma \ref{KLYS} yields $b\notin\{c,d\}$. We neither have $T\leq T'$ nor $T\leq T''$, as neither
\[\Theta(T)=S\leq S'=\Theta(T')\mbox { nor }\Theta(T)=S\leq S''=\Theta(T'')\]
because $b',c'\notin D$; see Lemmas \ref{DFH} and \ref{LKA}. Thus again by Lemmas \ref{DFH}, \ref{LKA} and \ref{YTRO} neither $\{a,b\}\subseteq C$ nor $\{c,d\}\subseteq C$, which implies that $b,d\notin C$, or $b,d\in\mbox{cl}_{\beta X}\mathbf{N}$, as $a,c\in\beta X\backslash\lambda_{{\mathcal P}}X\subseteq C$. But, since $\mbox{cl}_{\beta X}\mathbf{N}\backslash\mathbf{N}$ is zero--dimensional, there exists a clopen subset $E$ of $\beta X\backslash X$ containing $b$, but not $d$, such that it contains $\beta X\backslash\lambda_{{\mathcal P}}X$. This by Lemmas \ref{YTRO} and \ref{FSWR} corresponds to a one--point locally compact element of ${\mathscr M}^{\mathcal Q}_{\mathcal P}(X)$, namely, $e_EX$. By Lemma \ref{FSWR} the element $\Theta(e_EX)$ is a one--point locally compact element in ${\mathscr M}^{\mathcal Q}_{\mathcal P}(Y)$.  Now if ${\mathscr F}_Y(\Theta(e_EX))=\{G\}$, then by Lemma \ref{FSAEW} the set $G$ is a clopen subset of $\beta Y\backslash Y$, and neither
\[\mbox{cl}_{\beta Y}R_i\backslash Y\subseteq G\mbox { nor }(\mbox{cl}_{\beta Y}R_i\backslash Y)\cap G=\emptyset\]
as $b'\in G$ and $c'\notin G$, because $\{a',b'\}\subseteq G$ and $\{a',c'\}\nsubseteq G$, since $\Theta(e_EX)\leq S'$ and $\Theta(e_EX)\nleq S''$, as
\[e_EX\leq \Theta^{-1}(S')=T'\mbox{ and }e_EX\nleq\Theta^{-1}(S'')=T''\]
because $\{a,b\}\subseteq E$ and $\{c,d\}\nsubseteq E$, again by Lemmas \ref{DFH}, \ref{LKA} and \ref{YTRO}. This contradicts the fact that $\mbox{cl}_{\beta Y}R_i\backslash Y$ is connected.
\end{proof}

In the final result of this chapter we introduces the largest (with respect to the partial order $\leq$) compactification--like $\mathcal{P}$--extension of a Tychonoff space $X$. This largest element (also introduced in the proof of Lemma \ref{KJH}) turns out to be a familiar subspace of the Stone--\v{C}ech compactification of $X$. We formally define this element and prove some of its properties which characterize it among all $\mathcal{P}$--extensions of $X$ with compact remainder.

\begin{theorem}\label{A207}
Let ${\mathcal P}$ and  ${\mathcal Q}$ be a pair of compactness--like topological properties. Let $X$ be a Tychonoff locally--$\mathcal{P}$ non--${\mathcal P}$ space with $\mathcal{Q}$. Consider the subspace
\[\zeta_{\mathcal{P}} X=X\cup(\beta X\backslash\lambda_{\mathcal{P}} X)\]
of $\beta X$. Then
\begin{itemize}
\item[\rm(1)] $\zeta_{\mathcal{P}}X$ is the largest element (with respect to $\leq$) of either ${\mathscr E}^{\mathcal Q}_{\mathcal P}(X)$, ${\mathscr M}^{\mathcal Q}_{\mathcal P}(X)$ or ${\mathscr O}^{\mathcal Q}_{\mathcal P}(X)$.
\item[\rm(2)] For any $Y\in{\mathscr E}^{\mathcal Q}_{\mathcal P}(X)$ consider the following properties:
\begin{itemize}
\item[\rm(a)] For any $S,Z\in{\mathscr Z}(X)$ where $S\cap Z\subseteq C$ for some $C\in Coz(X)$ such that $\mbox{\em cl}_X C$ has $\mathcal{P}$, we have $\mbox{\em cl}_Y S\cap\mbox{\em cl}_Y Z\subseteq X$.
\item[\rm(b)] $Y$ satisfies the following:
\begin{itemize}
\item[\rm(i)] For any $S,Z\in{\mathscr Z}(X)$ we have
\[\mbox{\em cl}_Y (S\cap Z)\backslash X=(\mbox{\em cl}_Y S\cap\mbox{\em cl}_Y Z)\backslash X.\]
\item[\rm(ii)] For any $Z\in{\mathscr Z}(X)$ where $Z\subseteq C$ for some $C\in Coz(X)$ such that $\mbox{\em cl}_X C$ has $\mathcal{P}$, we have $\mbox{\em cl}_Y Z\subseteq X$.
\end{itemize}
\end{itemize}
Then $\zeta_{\mathcal{P}}X$ is characterized in ${\mathscr E}^{\mathcal Q}_{\mathcal P}(X)$ by either of the above properties.
\end{itemize}
\end{theorem}

\begin{proof}
(1). By  Lemma \ref{15} we have  $X\subseteq\lambda_{\mathcal{P}} X$. Thus $\zeta_{\mathcal{P}} X$ is a Tychonoff extension of $X$ with the compact remainder $\zeta_{\mathcal{P}} X\backslash X=\beta X\backslash \lambda_{\mathcal{P}} X$. Since $X\subseteq\zeta_{\mathcal{P}} X\subseteq \beta X$ we have $\beta\zeta_{\mathcal{P}} X=\beta X$ (see Corollary 3.6.9 of \cite{E}). Therefore by Lemma \ref{16} (with $f=\mbox{id}_X$ and $\phi=\mbox{id}_{ \beta X}$) we have $\zeta_{\mathcal{P}} X\in{\mathscr E}^{\mathcal Q}_{\mathcal P}(X)$ and by Theorem \ref{HG16} it follows that $\zeta_{\mathcal{P}} X\in{\mathscr O}^{\mathcal Q}_{\mathcal P}(X)$ and thus $\zeta_{\mathcal{P}} X\in{\mathscr M}^{\mathcal Q}_{\mathcal P}(X)$.
Now let $Y\in{\mathscr E}^{\mathcal Q}_{\mathcal P}(X)$ and let $\phi:\beta X\rightarrow \beta Y$ be the continuous extension of $\mbox{id}_X$.  By Lemma \ref{16} we have $\beta X\backslash \lambda_{\mathcal{P}}X\subseteq\phi^{-1}[Y\backslash X]$. Therefore
\begin{eqnarray*}
\phi[\zeta_{\mathcal{P}}X]&=&\phi\big[X\cup(\beta X\backslash\lambda_{\mathcal{P}}X)\big]\\&=&\phi[X]\cup\phi[\beta X\backslash\lambda_{\mathcal{P}}X]\\&=&X\cup\phi[\beta X\backslash\lambda_{\mathcal{P}}X]\subseteq X\cup\phi\big[\phi^{-1}[Y\backslash X]\big]\subseteq X\cup(Y\backslash X)=Y
\end{eqnarray*}
and thus $\phi|\zeta_{\mathcal{P}} X:\zeta_{\mathcal{P}} X\rightarrow Y$. Since the latter fixes $X$ pointwise this shows that $Y\leq \zeta_{\mathcal{P}}X$. Therefore $Y$ is the largest element of ${\mathscr E}^{\mathcal Q}_{\mathcal P}(X)$.

(2). Let  $Y=\zeta_{\mathcal{P}} X$. We show that $Y$ satisfies (2.a) and (2.b). To show (2.a) suppose that $S,Z\in {\mathscr Z}(X)$ are such that $S\cap Z\subseteq C$ for some $C\in Coz(X)$ such that $\mbox{cl}_X C$ has $\mathcal{P}$. Then by Lemma \ref{BA27} we have $\mbox{cl}_{\beta X}(S\cap Z)\subseteq\lambda_{\mathcal{P}} X$ and thus
\[\mbox{cl}_Y S\cap\mbox{cl}_Y Z=\mbox{cl}_{\beta X} S\cap\mbox{cl}_{\beta X}Z\cap Y=\mbox{cl}_{\beta X}(S\cap Z)\cap Y\subseteq\lambda_{\mathcal{P}} X\cap Y\subseteq X.\]
To show (2.b) note that for any $S,Z\in{\mathscr Z}(X)$ we have
\[\mbox{cl}_Y (S\cap Z)=\mbox{cl}_{\beta X}(S\cap Z)\cap Y=\mbox{cl}_{\beta X} S\cap\mbox{cl}_{\beta X}Z\cap Y=\mbox{cl}_Y S\cap\mbox{cl}_Y Z.\]
Therefore (2.b.i) holds. Note that since (2.a) holds, (2.b.ii) holds as well.

Now suppose that some $Y\in{\mathscr E}^{\mathcal Q}_{\mathcal P}(X)$ satisfies (2.a). Let $\phi:\beta X\rightarrow\beta Y$ be the continuous extension of $\mbox{id}_X$.
Recall the construction of $\beta Y$ and the representation of $\phi$ given in Lemma \ref{j2}. Note that (2.a) in particular implies that $\mbox{cl}_Y Z\cap(Y\backslash X)=\emptyset$ for any $Z\in {\mathscr Z}(X)$ such that $Z\subseteq C$ for some $C\in Coz(X)$ such that $\mbox{cl}_X C$ has $\mathcal{P}$. Thus by Theorem \ref{HG16} we have $Y\in{\mathscr O}_{\mathcal P}(X)$ and therefore $\phi^{-1}[Y\backslash X]=\beta X\backslash\lambda_{\mathcal{P}} X$. We show that for any $p\in Y\backslash X$ the set $\phi^{-1}(p)$ consists of a single point from this it then follows that $Y=\zeta_{\mathcal{P}} X$. Suppose to the contrary that for some $p\in Y\backslash X$ there exist distinct $a,b\in \phi^{-1}(p)$. Let $f:\beta X\rightarrow\mathbf{I}$ be continuous with $f(a)=0$ and $f(b)=1$. Let
\[S=f^{-1}\big[[0,1/3]\big]\cap X\mbox{ and }Z=f^{-1}\big[[2/3,1]\big]\cap X.\]
Then $S,Z\in {\mathscr Z}(X)$ and $S\cap Z=\emptyset$. Thus (and since $X$ is a non--empty Tychonoff locally--${\mathcal{P}}$ space, and ${\mathcal{P}}$ is hereditary with respect to closed subsets of Hausdorff spaces and thus containing some $C\in Coz(X)$ such that $\mbox{cl}_X C$ has $\mathcal{P}$) by our assumption $\mbox{cl}_Y S\cap\mbox{cl}_Y Z\subseteq X$. We show that $p\in\mbox{cl}_Y S\cap\mbox{cl}_Y Z$, this contradiction proves that $Y=\zeta_{\mathcal{P}} X$. Let $V$ be an open neighborhood of $p$ in $Y$ and let the open subset $V'$ of $\beta Y$ be such that $V=V'\cap Y$. Then $\phi^{-1}[V']\cap f^{-1}[[0,1/3)]$ and $\phi^{-1}[V']\cap f^{-1}[(2/3,1]]$ are open neighborhoods of $a$ and $b$ in $\beta X$, respectively, and therefore, have non--empty intersection with $X$. Note that
\[\phi^{-1}[V']\cap f^{-1}\big[[0,1/3)\big]\cap X\subseteq S\cap V\mbox{ and }\phi^{-1}[V']\cap f^{-1}\big[(2/3,1]\big]\cap X\subseteq Z\cap V.\]
Thus $p\in \mbox{cl}_Y S\cap\mbox{cl}_Y Z$. Finally, we show that (2.b) implies (2.a). This together with the above proves the theorem. Suppose that (2.b) holds. Let $S,Z\in{\mathscr Z}(X)$ be such that $S\cap Z\subseteq C$ for some $C\in Coz(X)$ such that $\mbox{cl}_X C$ has $\mathcal{P}$. Then by (2.b.ii) we have $\mbox{cl}_Y(S\cap Z)\subseteq X$. Therefore using (2.b.i) we have
\[(\mbox{cl}_Y S\cap \mbox{cl}_Y Z)\backslash X=\mbox{cl}_Y (S\cap Z)\backslash X=\emptyset\]
and thus $\mbox{cl}_Y S\cap \mbox{cl}_Y Z\subseteq X$.
\end{proof}

\section{Applications}

\subsection{Tight $\mathcal{P}$--extensions}

In \cite{Mag3} K.D. Magill, Jr. proved the following theorem relating the order--structure of the set of all compactifications of a locally compact space $X$ to the topology of $\beta X\backslash X$. Recall that order--isomorphic lattices are called {\em lattice--isomorphic}.

\begin{theorem}[Magill \cite{Mag3}]\label{KLFA}
Let $X$ and $Y$ be locally compact non--compact spaces. The following are equivalent:
\begin{itemize}
\item[\rm(1)] $({\mathscr K}(X),\leq)$ and $({\mathscr K}(Y),\leq)$ are lattice--isomorphic.
\item[\rm(2)] $\beta X\backslash X$ and $\beta Y\backslash Y$ are homeomorphic.
\end{itemize}
\end{theorem}

The idea of generalizing the above result led J. Mack, M. Rayburn and R.G. Woods in \cite{MRW1} to introduce and study a new class of extensions. We state some definitions together with some results from \cite{MRW1} below. The reader may find it useful to compare these results with those we have already obtained in the previous chapter.

Let $X$ be a Tychonoff space and let $\mathcal{P}$ be a topological property. A Tychonoff  $\mathcal{P}$--extension of $X$ is called {\em tight} if it does not contain properly any other $\mathcal{P}$--extension of $X$. Now let ${\mathcal P}$ be a topological property which is  closed hereditary, productive and is  such that if a Tychonoff space is the union of a compact space and a space with ${\mathcal P}$ then it has ${\mathcal P}$. Let $X$ be a Tychonoff space. Define the {\em ${\mathcal P}$--reflection} $\gamma_{\mathcal P} X$ of $X$ by
\[\gamma_{\mathcal P}X=\bigcap\{T:T \mbox{ has }{\mathcal P}\mbox{ and } X\subseteq T\subseteq\beta X\}.\]
If ${\mathcal P}$ is compactness then $\gamma_{\mathcal P} X=\beta X$ and if ${\mathcal P}$ is realcompactness then $\gamma_{\mathcal P}X=\upsilon X$ (the Hewitt realcompactification of $X$). Also, by Corollary 2.4 of \cite{MRW1} the space $\gamma_{\mathcal P}X$ has
$\mathcal{P}$. Denote by ${\mathscr P}(X)$ the set of all tight ${\mathcal P}$--extensions of $X$. As remarked in \cite{MRW1}, for a Tychonoff locally--${\mathcal P}$ non--${\mathcal P}$ space $X$ there is the largest one--point extension $X^*$ in ${\mathscr P}(X)$. Let
\[{\mathscr P}^*(X)=\big\{T\in{\mathscr P}(X):X^*\leq T\big\}\]
and for any $T\in{\mathscr P}^*(X)$, if $f_T:\beta X\rightarrow\beta T$ denotes the continuous extension of $\mbox{id}_X$, let
\[{\mathscr D}^*(X)=\big\{T\in {\mathscr P}^*(X):f_T[\gamma_{\mathcal P} X]=T\big\}.\]

\begin{theorem}[Mack, Rayburn and  Woods \cite{MRW1}]\label{JlFA}
Let $X$ and $Y$ be Tychonoff locally--${\mathcal P}$ non--${\mathcal P}$ spaces. If  $({\mathscr P}^*(X),\leq)$ and $({\mathscr P}^*(Y),\leq)$  are lattice--isomorphic then  $\gamma_{\mathcal P} X\backslash X$ and $\gamma_{\mathcal P} Y\backslash Y$ are homeomorphic.
\end{theorem}

The following main result of \cite{MRW1} generalizes K.D. Magill, Jr.'s theorem in \cite{Mag3} (Theorem \ref{KLFA}).

\begin{theorem}[Mack, Rayburn and  Woods \cite{MRW1}]\label{GGF}
Let $X$ and $Y$ be Tychonoff locally--${\mathcal P}$ non--${\mathcal P}$ spaces and suppose that ${\mathscr D}^*(X)={\mathscr P}^*(X)$ and ${\mathscr D}^*(Y)={\mathscr P}^*(Y)$. Suppose moreover that $\gamma_{\mathcal P}X\backslash X$ and $\gamma_{\mathcal P}Y\backslash Y$ are $C^*$--embedded in
$\gamma_{\mathcal P}X$ and $\gamma_{\mathcal P}Y$, respectively. The following are equivalent:
\begin{itemize}
\item[\rm(1)] $({\mathscr P}^*(X),\leq)$ and $({\mathscr P}^*(Y),\leq)$ are lattice--isomorphic.
\item[\rm(2)] $\gamma_{\mathcal P}X\backslash X$ and $\gamma_{\mathcal P}Y\backslash Y$ are homeomorphic.
\end{itemize}
\end{theorem}

Topological properties considered in \cite{MRW1} are all assumed to be productive while topological properties we have considered are hardly productive. (As it is shown in Example \ref{20UIHG} specific examples of compactness--like topological properties are mostly covering properties which are normally not expected to be productive.) However, there are topological properties which satisfy the two sets of requirements. We need to know the relation between the classes of compactification--like $\mathcal{P}$--extensions of a Tychonoff space $X$ and the class of its tight $\mathcal{P}$--extensions with compact remainder, in particular, we need to know if these two coincide. In this section we apply some of our previous results to obtain analogous results in the context of tight $\mathcal{P}$--extensions with compact countable remainder. Also, we give examples to show that the concepts ``tight $\mathcal{P}$--extension with compact remainder", ``minimal $\mathcal{P}$--extension" and ``optimal $\mathcal{P}$--extension" in general do not coincide.

We start with the following result which together with Lemma \ref{18} and Theorem \ref{20} characterizes spaces having a tight $\mathcal{P}$--extension with compact countable remainder. Note that by definitions, the two terms ``$n$--point minimal $\mathcal{P}$--extension" and ``$n$--point tight $\mathcal{P}$--extension" coincide for any $n\in\mathbf{N}$. Thus Lemma \ref{18} and Theorem \ref{20} also characterize spaces with an $n$--point tight $\mathcal{P}$--extension.

\begin{theorem}\label{HYDE}
Let ${\mathcal P}$ and  ${\mathcal Q}$ be a pair of compactness--like topological properties. Let $X$  be a Tychonoff  space with $\mathcal{Q}$. The following are equivalent:
\begin{itemize}
\item[\rm(1)] $X$ has a  countable--point minimal $\mathcal{P}$--extension with $\mathcal{Q}$.
\item[\rm(2)] $X$ has a  countable--point optimal $\mathcal{P}$--extension with $\mathcal{Q}$.
\item[\rm(3)] $X$ has a  countable--point tight $\mathcal{P}$--extension with compact remainder with $\mathcal{Q}$.
\end{itemize}
\end{theorem}

\begin{proof}
The equivalence of (1) and (2) follows from Theorem \ref{20}.  (1) {\em implies} (3). By Lemma \ref{18}(2.c) the space $X$ is locally--$\mathcal{P}$ and $\beta X\backslash\lambda_{\mathcal{P}}X$ contains an infinite bijectively indexed sequence $H_1,H_2,\ldots$ of pairwise disjoint non--empty clopen subsets. Let $p_i$'s, $T$, $q$ and $Y$ be as in Lemma \ref{18} ((2.c) $\Rightarrow$ (2.b)). Then $Y$ is a countable--point Tychonoff  $\mathcal{P}$--extension of $X$ with compact remainder with $\mathcal{Q}$. We show that $Y$ is also a tight $\mathcal{P}$--extension. Suppose to the contrary that there exists a $\mathcal{P}$--extensions $Y'$ of $X$ properly contained in $Y$. Choose some $p_n\in Y\backslash Y'$ where $n\in\mathbf{N}$. The sets $H_n$ and $(\beta X\backslash\lambda_{\mathcal{P}}X)\backslash H_n$ are closed in $\beta X$, as they are  closed in $\beta X\backslash\lambda_{\mathcal{P}} X$. Let $f_n:\beta X\rightarrow\mathbf{I}$ be continuous with
\[f_n[H_n]\subseteq\{0\}\mbox{ and }f_n\big[(\beta X\backslash\lambda_{\mathcal{P}}X)\backslash H_n\big]\subseteq\{1\}.\]
When $n=1$ note that $p_1$ is obtained by contracting a set containing $H_1$. The set
\[Z_n=f_n^{-1}\big[[0,1/2]\big]\cap X=q\big[f_n^{-1}\big[[0,1/2]\big]\big]\cap Y'\in{\mathscr Z}(X)\]
has $\mathcal{P}$, as it is closed in $Y'$. Therefore
\[H_n\subseteq f_n^{-1}\big[[0,1/2)\big]\subseteq\mbox{int}_{\beta X}\mbox{cl}_{\beta X}\big(f_n^{-1}\big[[0,1/2]\big]\cap X\big)=\mbox{int}_{\beta X}\mbox{cl}_{\beta X}Z_n\subseteq\lambda_{\mathcal{P}}X\]
which is a contradiction. This shows that $Y$ is also a tight $\mathcal{P}$--extension. That (3) implies (1) is trivial and follows from definitions.
\end{proof}

The following is a counterpart for Corollary \ref{20UHYJG}. Thus it too (besides Corollary \ref{20UHYJG}) may be considered as a generalization of the K.D. Magill, Jr.'s theorem in \cite{Mag2} (Theorem \ref{i0}).

\begin{theorem}\label{UGHFR}
Let ${\mathcal P}$ and  ${\mathcal Q}$ be a pair of compactness--like topological properties. Let $X$  be a  Tychonoff space with $\mathcal{Q}$. The following are equivalent:
\begin{itemize}
\item[\rm(1)] $X$ has an $n$--point  tight ${\mathcal P}$--extension  with $\mathcal{Q}$ for any $n\in\mathbf{N}$.
\item[\rm(2)] $X$ has a countable--point tight ${\mathcal P}$--extension with compact remainder with $\mathcal{Q}$.
\end{itemize}
\end{theorem}

\begin{proof}
As noted before, by definitions the terms ``$n$--point minimal ${\mathcal P}$--extension" and ``$n$--point tight ${\mathcal P}$--extension" coincide for any $n\in\mathbf{N}$. The result now follows from Theorems \ref{20} and \ref{HYDE}.
\end{proof}

\begin{theorem}\label{HHTR}
Let ${\mathcal P}$ and  ${\mathcal Q}$ be a pair of compactness--like topological properties. Let $X$  be a Tychonoff  space.
\begin{itemize}
\item[\rm(1)] Let $n\in\mathbf{N}$. If $X$ has a perfect image  with $\mathcal{Q}$ which has an $n$--point tight  ${\mathcal P}$--extension with $\mathcal{Q}$, then so does $X$.
\item[\rm(2)] If $X$ has a  perfect image  with $\mathcal{Q}$ which has a countable--point tight ${\mathcal P}$--extension with compact remainder with $\mathcal{Q}$, then so does $X$.
\end{itemize}
\end{theorem}

\begin{proof}
This follows from Theorems \ref{20UHG} and \ref{HYDE}.
\end{proof}

In the following we give examples of a topological property ${\mathcal P}$ and Tychonoff spaces $X$ for which the notion ``tight $\mathcal{P}$--extension with compact remainder" differ from both ``minimal $\mathcal{P}$--extension" and ``optimal $\mathcal{P}$--extension". For convenience, for a space $X$ and a topological property ${\mathcal P}$, denote by ${\mathscr T}_{\mathcal P}(X)$ the set of all tight $\mathcal{P}$--extensions of $X$ with compact remainder. Observe that by definitions ${\mathscr T}_{\mathcal P}(X)\subseteq{\mathscr M}_{\mathcal P}(X)$.

\begin{example}\label{FDG}
Let $\mathcal{P}$ be $\aleph_0$--boundedness (see Example \ref{20UIHG} for the definition) and let $X=D(\aleph_1)$ (the discrete space of cardinality $\aleph_1$). Note that ${\mathcal P}$ is closed hereditary, productive, finitely additive, perfect and satisfies Mr\'{o}wka's condition $(\mbox{W})$ (thus by Corollary 2.6 of \cite{MRW1} is such that if a Tychonoff  space is the union of a compact space and a space with ${\mathcal P}$, then it has ${\mathcal P}$). Then $\lambda_{\mathcal{P}} X=X$, as any $\aleph_0$--bounded $Z\in {\mathscr Z}(X)$ is finite. Therefore
\[{\mathscr O}_{\mathcal P}(X)={\mathscr M}_{\mathcal P}(X)={\mathscr K}(X)\]
where as before ${\mathscr K}(X)$ is the set of all compactifications of $X$. In \cite{Wo} it is shown that
\[\gamma_{\mathcal{P}} X=\bigcup\{\mbox{cl}_{\beta X}A:A\subseteq X\mbox{ is countable}\}.\]
Now $\gamma_{\mathcal{P}} X$ is a $\mathcal{P}$--extension of $X$ (as $\gamma_{\mathcal{P}} X$ always has ${\mathcal P}$; see Corollary 2.4 of \cite{MRW1}) and obviously it is contained properly in $\beta X$. Therefore $\beta X\in {\mathscr O}_{\mathcal P}(X)$, while $\beta X\notin{\mathscr T}_{\mathcal P}(X)$.
\end{example}

\begin{example}\label{JKTXA}
Let $\mathcal{P}$ be $\aleph_0$--boundedness and let $X=[0,\Omega)\backslash\{\omega\}$. Note that $X$ is locally compact; denote by $X^*$ the one--point compactification of $X$. Then $X^*\in {\mathscr T}_{\mathcal P}(X)$, as the only extension of $X$ contained properly in $X^*$ is $X$ itself which does not have $\mathcal{P}$, because $\omega\notin X$. Now let $\phi:\beta X\rightarrow X^*$ be the continuous extension of $\mbox{id}_X$. Then
\begin{equation}\label{UHES}
\phi^{-1}[X^*\backslash X]=\beta X\backslash X\neq\beta X\backslash\lambda_{\mathcal{P}} X
\end{equation}
as $\lambda_{\mathcal{P}} X\backslash X$ is non--empty. To show the latter simply let $Z=(\omega,\Omega)$ and observe that $Z$ is clopen in $X$ (thus it is a zero--set in $X$) and it has  $\mathcal{P}$. Since $Z$ is non--compact
\[\emptyset\neq\mbox{cl}_{\beta X}Z\backslash X=\mbox{int}_{\beta X}\mbox{cl}_{\beta X}Z\backslash X\subseteq\lambda_{\mathcal{P}} X\backslash X.\]
Now by Theorem \ref{HG16} from (\ref{UHES}) it follows that $X^*\notin{\mathscr O}_{\mathcal P}(X)$.
\end{example}

\subsection{On a question of  S. Mr\'{o}wka  and J.H. Tsai}

Let $X$ and $E$ be Hausdorff spaces. The space $X$ is said to be {\em $E$--completely regular} if $X$ is homeomorphic to a subspace of a product $E^\alpha$ for some cardinal $\alpha$ (see \cite{EM} and \cite{Mr1}). In \cite{MT} (also see \cite{T}) the authors proved that for a topological property $\mathcal{P}$ which is regular--closed hereditary, finitely additive with respect to closed subsets (that is, if a Hausdorff space is the finite union of its closed subsets with ${\mathcal P}$, then it has ${\mathcal P}$) and satisfy  Mr\'{o}wka's condition $(\mbox{W})$, every $E$--completely regular (where $E$ is regular and subject to some restrictions) locally--$\mathcal{P}$ space has a one--point $E$--completely regular $\mathcal{P}$--extension (see \cite{Ma} for related results). The authors then posed the following more general question: {\em For what pairs of topological properties ${\mathcal P}$ and ${\mathcal Q}$ is it true that every locally--$\mathcal{P}$ space  with $\mathcal{Q}$ has a one--point extension with both $\mathcal{P}$ and $\mathcal{Q}$?} Indeed, the systematic study of this sort of questions dates back to earlier times when P. Alexandroff proved that every locally compact (Hausdorff) non--compact space has a one--point compact (Hausdorff) extension (thus answering the question in the case when ${\mathcal P}$ is compactness and ${\mathcal Q}$ is the property of being Hausdorff). Since then the question has been considered by various authors for specific choices of topological properties ${\mathcal P}$ and ${\mathcal Q}$. The following result which is a corollary of Lemma \ref{16} is to provide an answer to the above question of S. Mr\'{o}wka and J.H. Tsai (see also Theorem 4.1 of \cite{MRW} for a related result).

\begin{theorem}\label{HGAB}
Let  $\mathcal{P}$ be a compactness--like topological property. Let $\mathcal{Q}$ be a topological property which is either
\begin{itemize}
\item clopen hereditary, inverse invariant under perfect mappings and satisfying Mr\'{o}wka's condition $(\mbox{\em W})$, or
\item strong zero--dimensionality.
\end{itemize}
Let $X$ be a Tychonoff non--$\mathcal{P}$ space with $\mathcal{Q}$. The following are equivalent:
\begin{itemize}
\item[\rm(1)] $X$ is locally--$\mathcal{P}$.
\item[\rm(2)] $X$ has a one--point Tychonoff extension with both $\mathcal{P}$ and $\mathcal{Q}$.
\end{itemize}
\end{theorem}

\begin{proof}
That (2) implies (1) is obvious. (1) {\em implies} (2). By Lemma \ref{15} we have $X\subseteq\lambda_{\mathcal{P}} X$. Let $T$ be the space obtained from $\beta X$ by contracting the compact subset $\beta X\backslash\lambda_{\mathcal{P}} X$ to a point $p$
(note that $\beta X\backslash\lambda_{\mathcal{P}} X$ is non--empty by Lemma \ref{HFH}). Then $T$ is Tychonoff. Consider the subspace $Y=X\cup\{p\}$ of $T$.
\begin{description}
\item[{\sc Case 1.}] Suppose that $\mathcal{Q}$ is hereditary with respect to clopen subsets, inverse invariant under perfect mappings and satisfies   Mr\'{o}wka's condition $(\mbox{W})$. Lemma \ref{16} then implies that $Y$ is a one--point Tychonoff extension of $X$ with both $\mathcal{P}$ and $\mathcal{Q}$.
\item[{\sc Case 2.}] Suppose that $\mathcal{Q}$ is  strong zero--dimensionality. By Lemma \ref{16} (with $\mathcal{Q}$ being regularity in its statement) the space $Y$ is a one--point Tychonoff $\mathcal{P}$--extension of $X$. We verify that $Y$ is strongly zero--dimensional. Note that $T$ is a compactification of $Y$. Let $\phi:\beta X\rightarrow\beta Y$ and $\gamma:\beta Y\rightarrow T$ be the continuous extensions of $\mbox{id}_X$ and $\mbox{id}_Y$, respectively. Since $\gamma\phi:\beta X\rightarrow T$ agrees with $q$ on $X$ we have $\gamma\phi=q$. Since $T$ is a compactification of $Y$ (and $\gamma|Y=\mbox{id}_Y$), by Theorem 3.5.7 of \cite{E} we have $\gamma[\beta Y\backslash Y]=T\backslash Y$. Thus $\gamma^{-1}(p)=\{p\}$ and
    \[\phi^{-1}(p)=\phi^{-1}\big[\gamma^{-1}(p)\big]=(\gamma\phi)^{-1}(p)=q^{-1}(p)=\beta X\backslash\lambda_{\mathcal{P}} X.\]
    By Lemma \ref{j2} we have $\beta Y=T$ and $\phi=q$. Using zero--dimensionality of $\beta X$ it is easy to verify that $\beta Y$ is zero--dimensional, that is,  $Y$ is strongly zero--dimensional.
\end{description}
Therefore (2) holds in either case.
\end{proof}

\section{Question}

We conclude with a question which naturally arose in connection with our study.

\begin{question}
{\em Let $X$ be a space, let $\mathcal{P}$ be a topological property and let $Y$ be a Tychonoff $\mathcal{P}$--extension of $X$ with compact remainder. The extension $Y$ of $X$ is called {\em maximal} if the topology of $Y$ is maximal (with respect to the inclusion relation $\subseteq$) among all topologies on $Y$ which turn $Y$ into a  Tychonoff ${\mathcal P}$--extension  of $X$  with compact remainder. Thus $Y$ is an optimal ${\mathcal P}$--extension of $X$ if it is both a minimal and a maximal ${\mathcal P}$--extension of $X$. Now to what extent and how the results of this article can be rephrased in order to remain valid in the new context?}
\end{question}



\begin{theindex}

{\em Items are listed according to the order of appearance in the text.}\\

Minimal ${\mathcal P}$--extension; \hfill {\em 4}

Optimal ${\mathcal P}$--extension; \hfill{\em 4}

Compactification--like ${\mathcal P}$--extension; \hfill{\em 4}

${\mathscr E}(X)$; \hfill{\em 4}

${\mathscr E}^{\mathcal P}(X)$, ${\mathscr E}_{\mathcal P}(X)$, ${\mathscr E}^{\mathcal Q}_{\mathcal P}(X)$; \hfill{\em 4}

${\mathscr M}_{\mathcal P}(X)$, ${\mathscr M}^{\mathcal Q}_{\mathcal P}(X)$; \hfill{\em 4}

${\mathscr O}_{\mathcal P}(X)$, ${\mathscr O}^{\mathcal Q}_{\mathcal P}(X)$; \hfill{\em 4}

$\lambda_{\mathcal{P}} X$; \hfill{\em 5}

Mr\'{o}wka's condition $(\mbox{W})$; \hfill{\em 5}

compactness--like topological property; \hfill{\em 6}

Pair of compactness--like topological properties; \hfill{\em 6}

$(EX,k_X)$; \hfill{\em 6}

${\mathscr F}_X(Y)$, ${\mathscr F}(Y)$; \hfill{\em 12, 13}

$\beta$--family; \hfill{\em 13}

${\mathscr F}_\alpha$; \hfill{\em 13}

${\mathscr K}(X)$; \hfill{\em 21}

$A^{(\zeta)}$; \hfill{\em 21}

Of type $(\sigma,n)$; \hfill{\em 21, 22}

Extension of $U$ to $\alpha X$; \hfill{\em 36}

$\mbox {Ex}_{\alpha X}U$; \hfill{\em 36}

$\mbox {Ex}_XU$; \hfill{\em 36}

$\leq_{inj}$; \hfill{\em 49}

$R|Y$; \hfill{\em 50}

$\leq_{surj}$; \hfill{\em 52}

$e_X(K_1,\ldots,K_n)$; \hfill{\em 57}

anti--atom in ${\mathscr M}^{\mathcal Q}_{\mathcal P}(X)$ of type $(\mbox{I})$; \hfill{\em 58}

anti--atom in ${\mathscr M}^{\mathcal Q}_{\mathcal P}(X)$ of type $(\mbox{II})$; \hfill{\em 58}

$e_CX$; \hfill{\em 66}

$M^X_{\mathcal P}$; \hfill{\em 66}

$Y_U$; \hfill{\em 66}

Almost optimal element of ${\mathscr M}^{\mathcal Q}_{\mathcal P}(X)$; \hfill{\em 67}

Tight ${\mathcal P}$--extension; \hfill{\em 79}

${\mathcal P}$--reflection; \hfill{\em 79}

$\gamma_{\mathcal P}X$; \hfill{\em 79}

${\mathscr P}(X)$; \hfill{\em 79}

$X^*$; \hfill{\em 79}

${\mathscr P}^*(X)$; \hfill{\em 79}

$f_T$; \hfill{\em 79}

${\mathscr D}^*(X)$; \hfill{\em 79}

${\mathscr T}_{\mathcal P}(X)$; \hfill{\em 81}

Maximal ${\mathcal P}$--extension; \hfill{\em 82}

\end{theindex}


\begin{thebibliography}{HD}

\bibitem{Bu} D.K. Burke, \emph{Covering properties}, in: K. Kunen and J.E. Vaughan (Eds.), Handbook of Set--theoretic Topology, Elsevier, Amsterdam, 1984, pp. 347--422.

\bibitem{C} G.L. Cain, Jr., \emph{Continuous preimages of spaces with finite compactifications}, Canad. Math. Bull. 24 (1981), 177--180.

\bibitem{Ch} M.G. Charalambous, \emph{Compactifications with countable remainder}, Proc. Amer. Math. Soc. 78 (1980), 127--131.

\bibitem{Do} J.M. Dom\'{\i}nguez, \emph{A generating family for the Freudenthal compactification of a class of rimcompact spaces}, Fund. Math. 178 (2003), 203--215.

\bibitem{E} R. Engelking, \emph{General Topology}, Second edition. Heldermann Verlag, Berlin, 1989.

\bibitem{EM} R. Engelking and S. Mr\'{o}wka, \emph{On $E$--compact spaces}, Bull. Acad. Polon. Sci. S\'{e}r. Sci. Math. Astronom. Phys. 6 (1958), 429--436.

\bibitem{F} Z. Frol\'{\i}k, \emph{A generalization of realcompact spaces}, Czechoslovak Math. J. 13 (1963), 127--138.

\bibitem{GJ} L. Gillman and M. Jerison, \emph{Rings of Continuous Functions}, Springer--Verlag, New York--Heidelberg, 1976.

\bibitem{HJW} M. Henriksen, L. Janos and R.G. Woods, \emph{Properties of one--point completions of a non--compact metrizable space}, Comment. Math. Univ. Carolin. 46 (2005), 105--123.

\bibitem{Ho} T. Hoshina,\, \emph{countable--points compactifications for metric spaces},\, Fund. Math. \,103 \,(1979), 123--132.

\bibitem{Ki} T. Kimura, \emph{$\aleph_0$--point compactifications of locally compact spaces and product spaces}, Proc. Amer. Math. Soc. 93 (1985), 164--168.

\bibitem{Ko1} M.R. Koushesh, \emph{On one--point metrizable extensions of locally compact metrizable spaces}, Topology Appl. 154 (2007), 698--721.

\bibitem{Ko2} M.R. Koushesh, \emph{On order--structure of the set of one--point Tychonoff extensions of a locally compact space}, Topology Appl. 154 (2007), 2607--2634.

\bibitem{Ko3} M.R. Koushesh, \emph{The partially ordered set of one--point extensions},  Topology Appl. 158 (2011), 509--532.

\bibitem{Ko4} M.R. Koushesh, \emph{one--point extensions of locally compact paracompact spaces}, manuscript.

\bibitem{MRW} J. Mack, M. Rayburn and R.G. Woods, \emph{Local topological properties and one--point extensions}, Canad. J. Math. 24 (1972), 338--348.

\bibitem{MRW1} J. Mack, M. Rayburn and R.G. Woods, \emph{Lattices of topological extensions}, Trans. Amer. Math. Soc. 189 (1974), 163--174.

\bibitem{Mag1} K.D. Magill, Jr., \emph{$n$--point compactifications}, Amer. Math. Monthly 72 (1965), 1075--1081.

\bibitem{Mag2} K.D. Magill, Jr., \emph{Countable compactifications}, Canad. J. Math. 18 (1966), 616--620.

\bibitem{Mag3} K.D. Magill, Jr., \emph{The lattice of compactifications of a locally compact space}, Proc. London Math. Soc. 18 (1968), 231--244.

\bibitem{Ma} F.L. Marin, \emph{A note on $E$--compact spaces}, Fund. Math. 76 (1972), 195--206.

\bibitem{MS} S. Mazurkiewicz and W. Sierpinski, \emph{Contribution \`{a} la topologie des ensembles d\'{e}nombrabl-\\es}, Fund. Math. 1 (1920), 17--27.

\bibitem{Mc} J.R. McCartney, \emph{Maximum countable compactifications of locally compact spaces}, Proc. London Math. Soc. 22 (1971), 369--384.

\bibitem{Me} F. Mendivil, \emph{Function algebras and the lattice of compactifications}, Proc. Amer. Math. Soc. 127 (1999), 1863--1871.

\bibitem{Mr} S. Mr\'{o}wka, \emph{On local topological properties}, Bull. Acad. Polon. Sci. 5 (1957), 951--956.

\bibitem{Mr1} S. Mr\'{o}wka, \emph{Further results on $E$--compact spaces. I}, Acta Math. 120 (1968), 161--185.

\bibitem{M1} S. Mr\'{o}wka, \emph{Some comments on the author's example of a non--$R$--compact space}, Bull. Acad. Polon. Sci. S\'{e}r. Sci. Math. Astronom. Phys. 18 (1970), 443--448.

\bibitem{MT} S. Mr\'{o}wka and J. H. Tsai, \emph{On local topological properties. II}, Bull. Acad. Polon. Sci. S\'{e}r. Sci. Math. Astronom. Phys. 19 (1971), 1035--1040.

\bibitem{PW} J.R. Porter and R.G. Woods, \emph{Extensions and Absolutes of Hausdorff Spaces}, Springer--Verlag, New York, 1988.

\bibitem{PW1} J.R. Porter and R.G. Woods, \emph{The poset of perfect irreducible images of a space}, Canad. J. Math. 41 (1989), 213--233.

\bibitem{R} M.C. Rayburn, \emph{On Hausdorff compactifications}, Pacific J. Math. 44 (1973), 707--714.

\bibitem{S} E.G. Skljarenko, \emph{Some questions in the theory of bicompactifications}, Izv. Akad. Nauk SSSR Ser. Mat. 26 (1962), 427--452 (in Russian).

\bibitem{Steph} R.M. Stephenson, Jr., \emph{Initially $\kappa$--compact and related spaces}, in: K. Kunen and J.E. Vaughan (Eds.), Handbook of Set--theoretic Topology, Elsevier, Amsterdam, 1984, pp. 603--632.

\bibitem{Te} T. Terada, \emph{On countable discrete compactifications}, General Topology Appl. 7 (1977), 321--327.

\bibitem{Th} T. Thrivikraman, \emph{On the lattices of compactifications}, J. London Math. Soc. 4 (1972), 711--717.

\bibitem{T} J.H. Tsai, \emph{More on local topological properties}, Bull. Acad. Polon. Sci. S\'{e}r. Sci. Math. Astronom. Phys. 22 (1974), 49--51.

\bibitem{vD} E.K. van Douwen, \emph{Remote points}, Dissertationes Math. (Rozprawy Mat.) 188 (1981), 45 pp.

\bibitem{Va} J.E. Vaughan, \emph{Countably compact and sequentially compact spaces}, in: K. Kunen and J.E. Vaughan (Eds.), Handbook of Set--theoretic Topology, Elsevier, Amsterdam, 1984, pp. 569--602.

\bibitem{Waj} E. Wajch, \emph{Compactifications with finite remainders}, Comment. Math. Univ. Carolin. 29 (1988), 559--565.

\bibitem{W} R.C. Walker, \emph{The Stone--\v{C}ech Compactification}, Springer--Verlag, Berlin, 1974.

\bibitem{Wo} R.G. Woods, \emph{Some $\aleph_0$--bounded subsets of Stone--\v{C}ech Compactifications}, Israel J. Math. 9 (1971), 250--256.

\bibitem{Wo1} R.G. Woods, \emph{zero--dimensional compactifications of locally compact spaces}, Canad. J. Math. 26 (1974), 920--930.

\end{thebibliography}
\end{document}